\documentclass[reqno]{amsart}
\usepackage[all,line]{xy}
\usepackage{color}
\usepackage{mathrsfs}
\usepackage{hyperref}
\usepackage{enumerate}
\usepackage{amsmath}
\usepackage{amsthm}
\usepackage{amssymb}
\usepackage{amscd}
\usepackage{graphicx}
\usepackage{epsfig}
\usepackage[english]{babel}
\usepackage[stable]{footmisc}
\usepackage{stackengine}

\newcommand{\qq}{{\mathfrak{q}}} % the projections in Morita equivalences
\newcommand{\pp}{{\mathfrak{p}}} % the projections in a principal bundle
\definecolor{darkgreen}{rgb}{0.0, 0.5, 0.0}

\newcommand{\n}{\nabla}

\renewcommand{\d}{\mathrm d}               % differential of functions, forms, etc.
\newcommand{\Lie}{\boldsymbol{\pounds}}    % Lie derivative
\newcommand{\X}{\ensuremath{\mathfrak{X}}} % Vector fields
\newcommand{\red}{{\mathrm{red}}} % subscript to denote reduced geometric structure
\renewcommand{\top}{{\mathrm{top}}} % superscript to denote top degree classes
\newcommand{\bas}{{\mathrm{bas}}} % subscript to denote basic forms
\newcommand{\can}{{\mathrm{can}}} % subscript to denote canonical simplectic form
\newcommand{\lin}{{\mathrm{lin}}} % subscript to denote linear model
\newcommand{\mon}{{\mathrm{mon}}} % subscript to denote monodromy
\newcommand{\Per}{{\mathrm{Per}}} % group of periods
\newcommand{\Lagr}{{\mathrm{Lagr}}} % subscript to denote Lagrangian
\newcommand{\cl}{{\mathrm{cl}}} % subscript to denote closed forms
\newcommand{\hol}{{\mathrm{hol}}} % subscript to denote holonomy
\newcommand{\opp}{{\mathrm{opp}}} % subscript to denote opposite
\newcommand{\reg}{{\mathrm{reg}}} % subscript to regular
\newcommand{\princ}{{\mathrm{princ}}} % subscript to principal
\renewcommand{\ss}{{\mathrm{ss}}} % subscript to semsimple

   \renewcommand{\a}{\alpha}
   \renewcommand{\b}{\beta}

  \newcommand{\w}{\omega}
  
\newcommand{\N}{\mathbb{N}}
\newcommand{\R}{\mathbb{R}}

\newcommand{\Q}{\mathbb{Q}}
\renewcommand{\S}{\mathbb{S}}
\newcommand{\T}{\mathbb{T}}
\newcommand{\Z}{\mathbb{Z}}

\newcommand{\Oga}{\Omega_{X}}
\newcommand{\DD}{\mathrm{c_2}}
\newcommand{\dd}{\mathrm{c_2}}
  \newcommand{\ucT}{\underline{\mathcal{T}}}

     \newcommand{\cA}{\mathcal{A}}
     \newcommand{\cB}{\mathcal{B}}
     \newcommand{\cC}{\mathcal{C}}
   \newcommand{\cD}{\mathcal{D}}
      \newcommand{\cE}{\mathcal{E}}
   \newcommand{\cF}{\mathcal{F}}
   \newcommand{\cG}{\mathcal{G}}
   
        \newcommand{\cK}{\mathcal{K}}
    \newcommand{\cH}{\mathcal{H}}

\newcommand{\cM}{\mathcal{M}}
  \newcommand{\cN}{\mathcal{N}}

  \newcommand{\cO}{\mathcal{O}}

  \newcommand{\cS}{\mathcal{S}}
    \newcommand{\cT}{\mathcal{T}}
  \newcommand{\cU}{\mathcal{U}}

    \newcommand{\eE}{\mathscr{E}}

\newcommand{\G}{\mathcal{G}}            % Lie groupoid
\renewcommand{\O}{\mathcal{O}}             % Orbit of a Lie groupoid
\DeclareMathOperator{\Mon}{Mon}         % Monodromy groupoid
\newcommand{\s}{\mathbf{s}}             % source map
\renewcommand{\H}{\mathcal{H}}          % Lie subgroupoid
\renewcommand{\gg}{\mathfrak{g}}        % Lie algebra
\newcommand{\hh}{\mathfrak{h}}          % Lie subalgebra
\renewcommand{\tt}{\mathfrak{t}}        % Lie subalgebra of maximal torus
  \renewcommand{\aa}{\mathfrak{a}}          % interior of the Weyl alcove
\newcommand{\cc}{\mathfrak{c}}          % interior of the Weyl chamber
\newcommand{\zz}{\mathfrak{z}}          % Center of theLie algebra 

\newcommand{\tto}{\rightrightarrows}    % Arrows of a groupoid

\DeclareMathOperator{\Ker}{Ker}           % kernel
\DeclareMathOperator{\im}{Im}           % kernel
\DeclareMathOperator{\Ad}{Ad}           % Adjoint
\DeclareMathOperator{\Aff}{Aff}         % Affine transformations
\DeclareMathOperator{\GL}{GL}           % General Linear...
\DeclareMathOperator{\Hol}{Hol}         % Holonomy
\DeclareMathOperator{\dev}{dev}           % developing map
\DeclareMathOperator{\Span}{Span}    %span
\DeclareMathOperator{\Cospan}{Cospan}    %span
\DeclareMathOperator{\pr}{pr}      % projection
\DeclareMathOperator{\Tor}{Tor}         % for torsors
\DeclareMathOperator{\Bun}{Bun}         % for the space of bundles
\DeclareMathOperator{\Gerbes}{Gerb}         % for the space of gerbes

\newcommand{\Emb}{\mathrm{\mathbb{E}mb}} % the Embedding category
\DeclareMathOperator{\vol}{vol}           % volume

\DeclareMathOperator{\Haar}{Haar}    % Haar density

\newcommand{\action}{\curvearrowright} % left action
\DeclareMathOperator{\Dh}{DH}         % Duistermaat-Heckman
\renewcommand{\DH}{\Dh}

\renewcommand{\Gauge}[2]{#2\stackMath\stackunder[1pt]{\star}{\scriptscriptstyle #1}#2} 

\newcommand{\varnu}{\mathrm{var}_{\nu}} % for the variation (when defining symplectic area) 

\newcommand{\helpI}{\mathrm{var}}
\newcommand{\Iaff}{\helpI_{\varpi}^{\Aff}} 
\newcommand{\Ilin}{\helpI_{\varpi}^{\lin}} % \newcommand{\Ilin}{{\mathrm{Lin}}_{\varpi}}

\newcommand{\helpVar}{\mathrm{var}} 
\newcommand{\Var}{\helpVar_{\varpi}} 
\newcommand{\Varb}{\helpVar_{0}} % this is for when we fix a base point                        

\newcommand{\helpVspace}{\mathscr{Var}}
\newcommand{\Vspace}{\helpVspace_{\varpi}} % \newcommand{\VV}{\mathcal{V}} 
\newcommand{\VspaceZ}{\helpVspace_{\varpi, \Z}} % THIS IS FOR THE INTEGRAL LATTICE INSIDE $\Vspace$
\newcommand{\Vspaceb}{\helpVspace_{\varpi, b}} % THIS IS FOR THE INTEGRAL FIBER OF $\Vspace$ AT $b$

\newcommand{\HH}{\mathcal{H}} 
\newcommand{\cBG}{\mathcal{B}} 
\newcommand{\cHE}{\mathcal{H}^{\cBG}}

\DeclareMathOperator{\Graph}{Graph}    %span
\newcommand{\HolX}{\Hol_X(M, \pi)}          
\newcommand{\HolXD}{\Hol_X(M, L)}

\numberwithin{equation}{section}
\allowdisplaybreaks

\newtheorem{theorem}{Theorem}[subsection]
\newtheorem{lemma}[theorem]{Lemma}
\newtheorem{proposition}[theorem]{Proposition}
\newtheorem{corollary}[theorem]{Corollary}

\newtheorem{conjecture}[theorem]{Conjecture}
\theoremstyle{definition}
\newtheorem{definition}[theorem]{Definition}
\newtheorem{example}[theorem]{Example}

\newtheorem{remark}[theorem]{Remark}

\begin{document}
%%%%%%%%%%%%%%%%%%%%%%%%%%%%%%%%%%%%%%%%%%
\title[Regular PMCTs]{Regular Poisson manifolds of compact types}

% author one information
\author{Marius Crainic}
\address{Depart. of Math., Utrecht University, 3508 TA Utrecht,
The Netherlands}
\email{m.crainic@uu.nl}

% author two information
\author{Rui Loja Fernandes}
\address{Department of Mathematics, University of Illinois at Urbana-Champaign, 1409 W. Green Street, Urbana, IL 61801 USA}
\email{ruiloja@illinois.edu}

% author three information
\author{David Mart\'inez Torres}
\address{Departamento de Matemtica, PUC-Rio, R. Mq. S. Vicente 225, Rio de Janeiro 22451-900, Brazil}
\email{dfmtorres@gmail.com}

\thanks{MC and DMT were partially supported by the NWO Vici Grant no. 639.033.312.
RLF was partially supported by NSF grant DMS-1710884, a Simons Fellowship in Mathematics and FCT/Portugal}

\begin{abstract}
This is the second paper of a series dedicated to the study of Poisson structures of compact types (PMCTs). In this paper, we focus on regular PMCTs, exhibiting a rich transverse geometry.  We show that their leaf spaces are integral affine orbifolds. We prove that the cohomology class of the leafwise symplectic form varies linearly and that there is a distinguished polynomial function describing the leafwise sympletic volume. The leaf space of a PMCT carries a natural Duistermaat-Heckman measure and a Weyl type integration formula holds. We introduce the notion of a symplectic gerbe, and we show that they obstruct realizing PMCTs as the base of a symplectic complete isotropic fibration (a.k.a.~a non-commutative integrable system).
\end{abstract}



\maketitle

%{\bf French title:} Vari\'et\'es de Poisson r\'eguli\`eres de types compacts.
%
%\medskip
%
%{\bf French abstract:}  Nous consacrons une suite d'articles aux vari\'et\'es de Poisson de types compacts (nous emploierons simplement l'acronyme PMCTs). Ce travail, qui est le second de cette suite, se concentre sur les PMCTs r\'eguli\`eres, et explore leur riche g\'eom\'etrie transverse. Nous montrons que l'espace de leurs feuilles sont des orbi-vari\'et\'es affines enti\`eres. Nous \'etablissons une d\'ependance lin\'eaire de la classe de cohomologie de la structure symplectique dont h\'eritent les feuilles et exhibons un polyn\^ome qui d\'ecrit le volume symplectique des feuilles. Nous \'equipons l'espace des feuilles d'un PMCT d'une mesure de Duistermaat-Heckman naturelle et donnons une formule d'int\'egration de type Weyl.
%Nous introduisons enfin la notion de gerbe symplectique et montrons que celles-ci sont l'obstruction \`a la construction de la PMCT comme la base d'une fibration symplectique compl\`ete \`a fibres isotropes (autrement dit, un syst\`eme int\'egrable non-commutatif).
%
%\medskip
%
%{\bf French keywords:} vari\'et\'e de Poisson, groupo\"{\i}de symplectique, structure affine enti\`ere, gerbe symplectique.
%
%\medskip

\setcounter{tocdepth}{1}
\tableofcontents

%%%%%%%%%%%%%%%%%%%%%%%%
%%%%%%%%%%%%%%%%%%%%%%%%
%%%%%%%%%%%%%%%%%%%%%%%%
%%%%%%%%%%%%%%%%%%%%%%%%
%%%%%%%%%%%%%%%%%%%%%%%%
%%%%%%%%%%%%%%%%%%%%%%%%
\section{Introduction}
%%%%%%%%%%%%%%%%%%%%%%%%
%%%%%%%%%%%%%%%%%%%%%%%%
%%%%%%%%%%%%%%%%%%%%%%%%
%%%%%%%%%%%%%%%%%%%%%%%%
%%%%%%%%%%%%%%%%%%%%%%%%
%%%%%%%%%%%%%%%%%%%%%%%%

This is the second manuscript of a series of works devoted to the study of \emph{Poisson structures of compact types} (PMCTs). 
These are the analogues in Poisson Geometry of compact Lie groups in Lie theory. In the first paper of this series \cite{CFMa} 
we have discussed general properties, described several examples, and outlined our general plan. In this paper, which is self-contained, we focus on regular PMCTs
and we discover a very rich transverse geometry, where several structures, both classical and new, interact with each other in a non-trivial way.
These include orbifold structures, integral affine structures, symplectic gerbes, etc. Moreover, we find that celebrated results, like the Duistermaat-Heckman 
Theorem on the linear variation of the symplectic class in the cohomology of reduced spaces, the polynomial behavior of the Duistermaat-Heckman measure, the Atiyah-Guillemin-Sternberg Convexity Theorem, or the Weyl Integration Formula, fit perfectly into the world of PMCTs, arising as particular statements of general results concerning PMCTs.

% What is a PMCT?
Given a Poisson manifold $(M, \pi)$ we will look at s-connected integrations $(\cG, \Omega)$,
which are symplectic Lie groupoids of compact type. At the level of Lie groupoids, there are several compact types $\cC$
characterized by possible conditions on $\cG$:
\begin{equation}
\label{types}
\mathcal{C}\in \{ \textrm{proper,\ s-proper,\ compact}\},
\end{equation}
that is,  Hausdorff Lie groupoids with proper anchor map, proper source map, and compact manifold of arrows, respectively.
For example, when $\cG= G\times M$ comes from a Lie group acting on a manifold $M$, the three conditions correspond to the properness of the action, the compactness of $G$, and the compactness of both $G$ and $M$, respectively. Therefore, one says that  the Poisson manifold $(M, \pi)$ is of:
\begin{itemize}
\item \emph{$\mathcal{C}$-type} if it has an s-connected integration $(\cG, \Omega)$ with property $\cC$;
\item \emph{strong $\cC$-type} if its canonical integration $\Sigma(M, \pi)$ has property $\cC$.
\end{itemize}

% What is the class of PMCTs that we discuss in this paper?
A Poisson manifold $(M, \pi)$ comes with a partition into symplectic leaves, generalizing the partition by coadjoint orbits from Lie theory.
In this paper, we consider PMCTs where the dimension of the leaves is constant, leaving the non-regular case to the next paper 
in the series \cite{CFMc}. This gives rise to a regular foliation $\cF_{\pi}$ on $M$, so, in some sense, we are looking at symplectic 
foliations from the perspective of Poisson Geometry. 

% First result: the rich geometries found in the leaf space
For a general regular Poisson manifold, the leaf space
\[ B= M/\cF_{\pi}\]
is very pathological. However, for us, the first immediate consequence of any of the compactness conditions is that $B$ is Hausdorff. Moreover, we will see that it comes with a very rich geometry, illustrated in the following theorem, which collects several results spread throughout the paper:

\begin{theorem}\label{thm:main1}
Given a regular Poisson manifold $(M, \pi)$ of proper type and an s-connected, proper symplectic integration $(\cG, \Omega)$:
\begin{enumerate}[(a)]
\item The space $B$ of symplectic leaves comes with an {\rm orbifold structure} $\cBG= \cBG(\cG)$; 
%which is smooth when the leaves are 1-connected;
\item There is an induced {\rm integral affine structure} $\Lambda$ on $\cBG$;
\item The classical effective orbifold underlying $\cBG$ is good;
\item There is a {\rm symplectic $\cT$-gerbe} over $\cBG$, where $\cT$ is the symplectic torus bundle induced by $\Lambda$. This gerbe is classified by the {\rm Lagrangian Dixmier-Douady class}:
\[ \DD(\cG, \Omega)\in H^2(\cBG, \cT_{\Lagr}).\]
\item The class $\DD(\cG,\Omega)$ vanishes if and only if $(M, \pi)$ admits a {\rm proper isotropic realization} $q: (X, \Oga)\to (M, \pi)$ for which $\cG\cong \cBG_X(M, \pi)$, a natural symplectic integration constructed from $X$ and the orbifold structure $\cBG$.
\end{enumerate}
\end{theorem}

% Remark about smooth vs orbifold
The presence of an orbifold structure on the leaf space which, in general, is non-effective,
gives rise to several technical difficulties throughout the discussion. When the symplectic leaves are 1-connected,
then $B$ is just a smooth manifold, and no further complications arise from orbifolds. In this case, all the other main features of PMCTs are already present, and it includes interesting examples, such as the regular coadjoint orbits or the principal conjugacy classes of a compact Lie group. For that reason, in the general discussion we will often consider this case first.

% Second result: linear variation of cohomology
The different geometric structures present on the leaf space of a PMCT, mentioned in the previous theorem, interact nicely with the leafwise symplectic geometry. One illustration of this interaction is the \emph{linear variation of symplectic forms in cohomology}, generalizing the classical Duistermaat-Heckman Theorem. For simplicity, we concentrate on the smooth case, where the leaves are 1-connected. Then to each $b\in B$ corresponds a symplectic leaf $(S_b, \omega_b)$, and the cohomologies  $H^2(S_b)$ yield a bundle $\cH\to B$. The cohomology class of the leafwise symplectic form defines a section of this bundle:
\[ B\ni b \mapsto [\omega_b]\in \cH_b= H^2(S_b).\]

In the s-proper case, the leaves are compact and $\cH$ is a smooth flat vector bundle over $B$. The flat connection is the so called \emph{Gauss-Manin connection} and arises from the underlying integral cohomology. Using parallel transport, one can compare classes $[\omega_b]$ at distinct points $b\in B$, once a path has been fixed. On the other hand, the integral affine structure on $B$ of the previous theorem determines a developing map, defined on the universal cover of $B$:
\[ \dev: \widetilde{B} \to \R^q\quad (q=\dim B).\]
Denoting the Chern classes of the principal torus bundle $t:s^{-1}(x_0)\to S_{b_0}$, where $s$ and $t$ are the source/fiber of the s-proper integration, by
\[ c_1, \ldots, c_q\in H^2(S_{b_0}),\]
the linear variation theorem can be stated as follows:

\begin{theorem}\label{thm:main2} If $(M, \pi)$ is a regular, s-proper Poisson manifold, with 1-connected symplectic leaves, then for any path $\gamma$ in $B$ starting at $b_0$ one has
\[ \gamma^*([\w_{\gamma(1)}])= [\omega_{b_0}]+ \dev^{1}(\gamma)c_1+ \ldots + \dev^{q}(\gamma)c_q.\]
Similar formulas hold for a general Poisson manifold of s-proper type.
\end{theorem}

% Third result: Duistermaat-Heckman measure
One can also look at volume forms instead. Assume as before that we have an s-proper integration $(\cG, \Omega)$ of $(M, \pi)$. Pushing forward the Liouville measure associated to $\Omega$, one obtains the \emph{Duistermaat-Heckman measure} on the leaf space:
\[ \mu_{DH}\in \cM(B).\]
On the other hand, the integral affine structure on $B$ gives rise to another measure, $\mu_{\Aff}\in \cM(B)$. The classical result on the polynomial behavior of the Duistermaat-Heckman measure is a special case of the following general result for PMCTs:

\begin{theorem}\label{thm:main3}
If $(M, \pi)$ is a regular Poisson manifold, with s-connected, s-proper integration $(\cG, \Omega)$, then:
\[
\mu_{\DH}^\Omega= (\iota \cdot \vol)^2 \mu_{\Aff},
\]
where $\vol:B\to\R$ is the leafwise symplectic volume function and $\iota:B\to \N$ counts the number of connected components of the isotropy group $\cG_x$ ($x\in S_b$). Moreover, $(\iota \cdot \vol)^2$ is a polynomial relative to the orbifold integral affine structure on $B$.
\end{theorem}

The previous theorem has an interesting version already on $M$, where we obtain two measures, $\mu_{\DH}^M$ and $\mu_{M}^{\Aff}= \mu_{M}$, both induced by densities $\rho^{M}_{{\DH}}$ and $\rho_{M}$, which are invariant under all Hamiltonian flows. Our study of such invariant densities yields the following Fubini type theorem:

\begin{theorem}\label{thm:main4}
If $(M, \pi)$ is a regular Poisson manifold, with proper integration $(\cG, \Omega)$, then for any $f\in C_{c}^{\infty}(M)$:
\[ \int_M f(x) \,\d\mu_{M}(x) =\int_B \left(\iota(b) \int_{S_b}f(y)\,\d\mu_{S_b}(y) \right)\,\d\mu_{\Aff}(b),\]
where $\mu_{S_b}$ is the Liouville measure of the symplectic leaf $S_b$, and $\iota:B\to \N$ is the function that for each $b\in B$ counts the number of connected components of the isotropy group $\G_x$ ($x\in S_b$).
\end{theorem}

We shall see in \cite{CFMc} that a similar theorem is valid for all, including non-regular, PMCTs. This theorem includes, as a special instance, the classical Weyl Integration Formula.

% Organization of the paper along with some other results

The rest of this paper is organized into 8 sections and 2 appendices.

Section \ref{sec:foliations} is devoted to foliations and orbifolds, recalling Haefliger's approach to transversal geometry, fixing the necessary framework, but also illustrating the various compactness properties (\ref{types}) in the simpler context of foliations. In this section, the orbifold structure on the leaf space of a PMCT, stated in part (a) of Theorem \ref{thm:main1}, is shown to exist.

Section \ref{sec:integral:affine} includes some basics on Integral Affine Geometry and describes its relationship with Poisson Geometry. Besides proving part (b) of Theorem \ref{thm:main1}, we discover new Poisson invariants, the so-called \emph{extended monodromy groups} which give rise to obstructions to s-properness, but which are interesting also for general Poisson manifolds. 

Sections  \ref{sec:lin-var-1} and Section \ref{sec:lin-var-2} concern Theorem \ref{thm:main2}, on the linear variation of the cohomology class of the leafwise symplectic form. We first treat the case of smooth leaf space and then the orbifold case. Both these sections start by revisiting the developing map for integral affine structures from a novel groupoid perspective. That allows for a global formulation, free of choices, which is more appropriate for our purposes. We also obtain a decomposition result for Poisson manifolds of s-proper type which, from the point of view of classification, indicates two types of building blocks: (i) the strong proper ones with full variation, and (ii) the ones with no variation, corresponding to symplectic fibrations over integral affine manifolds.

Section \ref{sec:measures} discusses the Duistermaat-Heckman measures on PMCTs and on their leaf spaces, its relationship with the measures determined by the integral affine structures, and the interaction with the Liouville measure on the symplectic leaves, leading to proofs of Theorem \ref{thm:main3}, on the polynomial nature of the Duistermaat-Heckman measure, and the integration formula of Theorem \ref{thm:main4}.

Section \ref{sec:realizations} explains the relationship between PMCTs and proper isotropic realizations, which appears in part (e) of Theorem \ref{thm:main1}.
For any proper isotropic realization $q: (X, \Oga)\to (M, \pi)$ we introduce a ``holonomy groupoid relative to $X$'', $\Hol_X(M, \pi)$, which is usually smaller than the canonical integration $\Sigma(M, \pi)$, and hence has better chances to be proper. The groupoids $\Hol_X(M, \pi)$ not only arise in many examples, but are an important concept. Indeed, recall that foliations come with two standard s-connected integrations: the largest one which is the monodromy groupoid $\Mon(M, \cF)$ and the smallest one which is the holonomy groupoid $\Hol(M, \cF)$. In Poisson geometry, the integration $\Sigma(M, \pi)$ is the analogue of $\Mon(M, \cF)$ but, in general, there is no analogue of the holonomy groupoid. Our results suggest that, in Poisson Geometry, instead of looking for the smallest integration, one should look for the smallest one that acts on a given symplectic realization. This property characterizes $\Hol_X(M, \pi)$ uniquely.

Sections  \ref{sec:gerbes:manifolds} and \ref{sec:gerbes:orbifolds} describe our theory of \emph{symplectic gerbes}, first in the smooth case and then in the orbifold case, proving parts (d) and (e) of Theorem \ref{thm:main1}. Our departure point is the usual theory of $\S^1$-gerbes, which we first extend to $\cT$-gerbes, where $\cT$ is an arbitrary torus bundle over a manifold $B$. For the symplectic theory, we need to look at a \emph{symplectic torus bundles} $(\cT, \omega_{\cT})$ over an orbifold $B$ or, equivalently, integral affine structures on $B$. In the more standard theory ones looks at central extensions of $\cT$ by Lie groupoids, while in the symplectic theory we look at central extensions of $(\cT, \omega_{\cT})$ by symplectic groupoids.
The main conclusion is that, while $\cT$-gerbes are classified by their Dixmier-Douady classes, living in $H^2(B,\ucT)$, for symplectic gerbes one obtains a Lagrangian Dixmier-Douady class which gives rise to a group isomorphism
\[ \DD: \Gerbes_{B}(\cT, \omega_{\cT}) \to H^2(B, \ucT_{\Lagr}),\]
where $\ucT_{\Lagr}$ is the sheaf of Lagrangian sections of $(\cT, \omega_{\cT})$.

Appendix \ref{appendix:moment:maps} gives some background on actions of symplectic groupoids, Hamiltonian $\cG$-spaces, 
and symplectic Morita equivalence, which are relevant for the paper. Appendix \ref{appendix:molino} is of a very different nature: we show there how one can adapt (part of) Molino's approach of Riemannian foliations to the context of integral affine geometry, to prove that integral affine orbifolds are good, i.e. quotients of a discrete integral affine group action. While this is relevant for PMCTs and we make good use of it, we believe it may be of independent interest.

\vspace*{.1in}

%%% Comments on examples and twisted Dirac  %%%%
As we develop the theory of PMCTs, we will explain how to adapt it to {\it Dirac manifolds}. The first motivation for this arises from the extension of the results of this paper from regular to arbitrary PMCTs, since we will introduce in \cite{CFMc} a desingularization procedure which will turn a PMCT into a {\it regular, Dirac} manifold (without changing the leaf space or the compactness type!). The second motivation comes from Lie theory and the striking similarity between the geometry of (co)adjoint orbits and the one of conjugacy classes (see e.g. \cite{DK}). While coadjoint orbits fit into Poisson Geometry, conjugacy classes belong to the world of (twisted) Dirac Geometry. Hence the Dirac framework allows us to (re)cover even more fundamental examples.

However, in order to get a faster grasp of the results and new techniques introduced here, the reader may choose to skip, in a first reading, all sections concerning Dirac Geometry. The same applies to the sections on orbifolds, since the rich geometry that comes with PMCTs is present already when the leaf space is smooth.

% Throughout the paper we give many examples illustrating our theory, including examples related to Lie theory, both at the Lie algebra level and at
% the group level. The latter examples actually belong to the world of (twisted) Dirac Geometry, and  we also develop
% a corresponding theory of \emph{Dirac manifolds of compact type}. Besides covering very fundamental examples, the theory
% of DMCTs will be crucial in \cite{CFMc} to extend all the results on regular PMCTs obtained in these paper to arbitrary PMCTS. In fact, in \cite{CFMc}, we will introduce
% a desingularization procedure which turns a PMCT into a regular DMCT (without changing the leaf space!).

% The reader who wishes to get a faster grasp of the results and new techniques introduced here can skip, in a first reading, all sections concerning either orbifolds or DMCTs, without affecting the logic dependence of later sections. 
% %These include the following sections: \ref{ssec:Orbifolds}, \ref{ssec:IAS-orbifol} (up to Definition \ref{def:transverse:int:affine}), \ref{ssec:The twisted case}, \ref{sec:lin-var-2} and \ref{sec:gerbes:orbifolds}. Section \ref{sec:measures} requires some material from \ref{ssec:Orbifolds}.

{\bf Acknowledgments.} The work of N.-T.~Zung \cite{Zu} on proper symplectic grou\-poids should be considered as a precursor of the theory of PMCTs. 
However, Zung focuses his attention on the symplectic groupoid, instead of the underlying Poisson manifold. A.~Weinstein's work on measures on stacks \cite{We} was a source of inspiration for our study of measures. Our theory of symplectic gerbes can be viewed as a symplectic version of I.~Moerdijk's work on regular proper groupoids \cite{Moe03}, but with a richer geometric flavor that includes the connection to the Delzant-Dazord theory of isotropic fibrations \cite{DaDe}.
% although we follow here a more direct approach that avoids using the stack encoded by the symplectic groupoids (the precise relationship is explained in \cite{CrMe}. 
% We will also make intensive use of Haefliger's philosophy on the transversal geometry of foliations and on orbifolds \cite{Hae}. 
We would also like to acknowledge the gracious support of IMPA, the University of Utrecht and the University of Illinois at Urbana-Champagne, at various stages of these project.

%%%%%%%%%%%%%%%%%%%%%%%%
%%%%%%%%%%%%%%%%%%%%%%%%
%%%%%%%%%%%%%%%%%%%%%%%%
%%%%%%%%%%%%%%%%%%%%%%%%
%%%%%%%%%%%%%%%%%%%%%%%%
%%%%%%%%%%%%%%%%%%%%%%%%
%%%%%%%%%%%%%%%%%%%%%%%%
%%%%%%%%%%%%%%%%%%%%%%%%
\section{PMCTs, foliations of compact types and orbifolds}
\label{sec:foliations}
%%%%%%%%%%%%%%%%%%%%%%%%
%%%%%%%%%%%%%%%%%%%%%%%%
%%%%%%%%%%%%%%%%%%%%%%%%
%%%%%%%%%%%%%%%%%%%%%%%%
%%%%%%%%%%%%%%%%%%%%%%%%
%%%%%%%%%%%%%%%%%%%%%%%%
%%%%%%%%%%%%%%%%%%%%%%%%
%%%%%%%%%%%%%%%%%%%%%%%%

Recall that given a foliation $\cF$ on $M$ the associated distribution can be thought of as a  Lie algebroid 
 with anchor the inclusion and bracket the restriction of the Lie bracket of vector fields. This Lie algebroid is well-known to be integrable --
 for example by the holonomy groupoid (see Section \ref{ssec:foliation-groupoids}). Therefore, for any of the compactness types (\ref{types}), the notion of $\mathcal{C}$-type (respectively, strong $\mathcal{C}$-type)  makes 
sense for any foliated manifold $(M, \cF)$: one requires the existence of a source connected (respectively, source 1-connected) Hausdorff Lie groupoid
integrating $\cF$  having property $\mathcal{C}$. 

In this section we shall make a detailed study of these compactness types of foliations. 
On the one hand, foliations of compact types are easier to handle than regular Poisson manifolds of compact types, but they
still exhibit phenomena/properties that will persist in the Poisson case. 
Thus the analysis of the former will play a guiding role in the analysis of the latter. On the other hand, 
if a regular Poisson manifolds is of (strong) $\mathcal{C}$-type, then so is the underlying symplectic foliation. Hence, 
the results in this section have immediate applications to the Poisson case. 

As (rough) main goal of this section, we mention here:

\begin{theorem}\label{thm-reg-fol} If $(M,\pi)$ is a regular Poisson manifold of $\mathcal{C}$-type, then the symplectic foliation $\cF_\pi$
is of $\mathcal{C}$-type.  As a consequence, the space of symplectic leaves
\[ B= M/\cF_{\pi}\]
is an  orbifold.  More precisely, any integration $\cG$ of $(M, \pi)$ of $\mathcal{C}$-type gives rise to an integration $\cBG(\cG)$ of $\cF_\pi$ of $\mathcal{C}$-type, which induces an orbifold structure on $B$. 
\end{theorem}

We shall see in the next sections that \emph{symplectic} integrations of $(M, \pi)$ induce several geometric structures on the orbifold $B$. For that reason this section pays special attention to geometric structures on leaf spaces of foliations and on orbifolds.

% We will see in the next sections, that \emph{symplectic} integrations $(M,\cF_\pi)$ induce several geometric structures on the orbifold $B$. For that reason, in this section we also discuss geometric structures on leaf spaces of foliations and on orbifolds.

\begin{remark}[Classical compact foliations]\label{classical-compact} In classical Foliation Theory the notion of a compact foliation
refers to a foliation all whose leaves are compact (see, e.g., \cite{EMS,Ep}). This property does not refer to any of the integrations of $\cF$.
We will clarify later how this classical notion is related to our compactness types.
In this regard, since the pioneering work of A.~Haefliger, Lie groupoids (in particular, the holonomy  groupoid) have been extensively used in the study of
the transverse geometry of foliations. Our approach follows the same spirit, defining compactness type of foliations in terms of
groupoids integrating them.
\end{remark}

%%%%%%%%%%%%%%%%%%%%%%%
%%%%%%%%%%%%%%%%%%%
%%%%%%%%%%%%%%%%%%%%%
\subsection{The monodromy and holonomy groupoids}\label{ssec:foliation-groupoids}
%%%%%%%%%%%%%%%%%%%%%%%
%%%%%%%%%%%%%%%%%%%
%%%%%%%%%%%%%%%%%%%%%
The Lie groupoids that integrate foliations are called \textbf{foliation groupoids}. 
They are easy to characterize, since, in general, the Lie algebra of the isotropy group of a Lie groupoid
is precisely the kernel of the anchor of its Lie algebroid:

\begin{proposition}[\cite{CrMo}]\label{fol-groupoid-isotropy} A Lie groupoid $\G$ is a foliation groupoid iff all the iso\-tro\-py groups $\cG_x$ are discrete.
\end{proposition}

Since any foliation $\cF$ on $M$ is integrable as a Lie algebroid, it has a unique (smooth) source $1$-connected integration, called the
\textbf{the monodromy groupoid} of $\cF$ and denoted by
\[ \Mon(M, \cF) \tto M .\]
The arrows in this groupoid are the leafwise homotopy classes (relative to the end-points) of leafwise curves in $M$ (see \cite{MM}).

Every foliation has yet another s-connected canonical integration: the {\bf holonomy groupoid} of $\cF$, denoted by
\[\Hol(M, \cF)\tto M.\]
The arrows are now equivalence classes of leafwise paths where two paths are identified if they induce the same germ of holonomy transformation. 
From their definitions, we have a morphism of Lie groupoids which is a local diffeomorphism:
\begin{equation}\label{mon-hol} 
\hol: \Mon(M, \cF) \to \Hol(M, \cF).
\end{equation}

The relevance of the holonomy groupoid in studying the transverse geometry of foliations stems from
 the fact that any other s-connected Lie groupoid integrating $\cF$ lies above it. More
precisely:

\begin{theorem}[\cite{CrMo,Ph}]\label{seq-mon-G-hol}  For any s-connected integration $\cE$ of a foliation $\cF$ on $M$,
there is a natural factorization of (\ref{mon-hol}) into a composition of surjective submersions compatible with the groupoid structure:
\[ 
\xymatrix{ \Mon(M, \cF) \ar[r]^-{h_{\cE}} & \cE \ar[r]^-{\hol_{\cE}}\ar[r] &  \Hol(M, \cF)} .
\]
\end{theorem}

Recall that we are only interested in Hausdorff Lie groupoids, although even very elementary foliations can have
non-Hausdorff monodromy and/or holonomy groupoids. Moreover, one can have one of them being Hausdorff while the other one is not, and even
both not being Hausdorff but there exists a Hausdorff one in between them! The monodromy groupoid is Hausdorff iff the foliation does not
have vanishing cycles \cite{AH1}, but no geometric criteria characterizing the Hausdorffness
of other foliation groupoids (e.g., the holonomy groupoid) is known. 

\begin{example}[Simple foliations]\label{ex-simple-foliations} If $p: M\to B$ is a submersion with connected fibers, 
then the fibers of $p$ define a foliation $\mathcal{F}$ on $M$, called a \emph{simple foliation}.
The holonomy groupoid of $\mathcal{F}$ is the {\bf submersion groupoid of $p$}, 
consisting of pairs of points in $M$ that are in the same fiber of $p$:
 \[\Hol(M, \cF)= M\times_{B} M\tto M,\]
where $(x,y)\in M\times_B M$ is thought of as an arrow from $y$ to $x$.
This groupoid is Hausdorff but the monodromy groupoid of $\cF$ may fail to be Hausdorff:
e.g., the fiber above $0$ of the first projection $p:\R^{3}\setminus \{0\}\to\R$ contains a vanishing cycle.
% if one considers, e.g., $p:\R^{3}\setminus \{0\}\to\R$ the first projection then the fiber above $0$ contains a vanishing cycle.
\end{example}

\begin{example}[One sided holonomy] 
\label{ex-one-sided-hol}
On the cylinder $M=\S^1\times\R$ consider the foliation $\cF$ given by the orbits of the vector field 
$X=\frac{\partial}{\partial \theta}+f(t) \frac{\partial}{\partial t}$,
where $f(t)$ is a smooth function with $f(t)=0$ for $t\le 0$ and $f(t)>0$ for $t>0$. $\cF$ has closed leaves $\S^1\times \{t\}$ for
$t\le 0$ and open leaves for $t>0$. It follows that there are no vanishing cycles, so $\Mon(M,\cF)$ is Hausdorff. 
The leaf $\S^1\times \{0\}$ has one-sided holonomy, so the leaves with $t<0$ give cycles with trivial
holonomy that converge to a cycle at $t=0$ with non-trivial holonomy. Hence, $\Hol(M,\cF)$ is non-Hausdorff.
\end{example}

Recall that the {\bf linear holonomy} of a foliation $\cF$ on $M$ along a leafwise path $\gamma\subset S$ from $x$ to $y$ 
is, by definition, the linearization of the holonomy parallel transport along $\gamma$.
Identifying the tangent spaces of the transversals with the normal spaces $\nu_x(S)= T_xM/T_xS$,
the linear holonomy becomes a map:
\begin{equation}\label{lin-hol-ref} 
\hol_{\gamma}^{\lin}:=\d_x \hol_\gamma: \nu_x(S) \to \nu_y(S) .
\end{equation}
It can also be described directly as the parallel transport associated to the so-called Bott connection.
The linear holonomy groups are then defined similarly, by identifying loops that induce the same linear holonomies:
\[ \Hol_{x}^{\lin}(M, \cF):= \Mon_x(M, \cF)/\textrm{linear\ holonomy\ equivalence} .\]

Similarly, one can also define the linear holonomy groupoid $\Hol^{\lin}(M, \cF)$. The resulting quotient map
\[ \hol^{\lin}: \Mon(M, \cF)\to \Hol^{\lin}(M, \cF)\]
will factor through the holonomy groupoid, giving rise to a morphism of groupoids
$\lin: \Hol(M, \cF)\to \Hol^{\lin}(M, \cF)$.
However, in general, $\Hol^{\lin}(M, \cF)$ will only be a topological groupoid: it follows from Theorem \ref{seq-mon-G-hol} that, for 
$\Hol^{\lin}(M, \cF)$ to admit a Lie groupoid structure such that $\lin$ is smooth, the holonomy must coincide with 
the linear holonomy, i.e. $\lin$ must be 1-1. When this happens, we say that $(M, \cF)$ \textbf{has linear holonomy}.
As a consequence of Bochner's linearization theorem, this is the case whenever the holonomy groups are finite. 
On the other hand, the foliation in Example \ref{ex-one-sided-hol} does not have linear holonomy.

\begin{example}[Linear foliations]\label{ex-local-linear-models} A class of examples that is relevant for us, since 
they provide the (linear)  local  models that appear in local Reeb stability and are intimately
related to our compactness types, is obtained as follows. One starts with a connected manifold $S$ and:
\begin{enumerate}[(i)]
\item $\hat{S}\to S$ a 
%{ (regular}\footnote{{All covering spaces in this paper will be regular, so they come with a group of Deck transformations which is the quotient of $\pi_1(S)$ by the \emph{normal} subgroup $\pi_1(\hat{S})$}}{)} 
covering space with group $\Gamma$;
\item a representation $\Gamma\to\GL(V)$ on a vector space $V$ of dimension $q$.
\end{enumerate}
The associated {\bf linear local model} $(\hat{S}\times_{\Gamma} V,\mathcal{F}_{\lin})$ is the foliation of the quotient:
\[ \hat{S}\times_{\Gamma} V:=(\hat{S}\times V)/\Gamma\]
obtained from the trivial codimension $q$ product foliation $\{\hat{S}\times\{v\}\}_{v\in V}$. Note that $S$ sits canonically inside the linear local model 
as the leaf corresponding to $0\in V$.

This construction has a groupoid version which gives us an integrating foliation groupoid for $\mathcal{F}_{\lin}$. More precisely,  
$\hat{S}$ is replaced by the pair groupoid $\hat{S}\times \hat{S} \tto \hat{S}$. The product of this groupoid with $V$ 
(viewed as a groupoid with units only) gives rise to a groupoid $\hat{S}\times \hat{S}\times V\tto \hat{S}\times V$, where
$\Gamma$ acts freely and properly by groupoid automorphisms (again by the diagonal action). Hence, we have a quotient groupoid:
\begin{equation}\label{fol-gpds-loc-mod}
 (\hat{S}\times \hat{S})\times_{\Gamma} V\tto \hat{S}\times_{\Gamma} V .
\end{equation}
One readily checks that this is a foliation groupoid, and that the induced foliation on its base is precisely $\mathcal{F}_{\lin}$. 
However, this groupoid may sit strictly between the monodromy and holonomy groupoids.
In fact, the monodromy groupoid is obtain as a special case of this construction: 
\begin{equation}\label{mon-loc-mod}
\Mon(M,\cF_\lin)=(\widetilde{S}\times \widetilde{S})\times_{\pi_1(S)} V\tto \widetilde{S}\times_{\pi_1(S)} V,
\end{equation}
where $\pi_1(S)$ acts on $V$ via the homomorphism $\pi_1(S)\to \Gamma\to\GL(V)$.

We summarize the previous discussion in the following result (for item (ii) see Example \ref{ex-local-linear-models-doi}):

\begin{proposition}\label{gpd-of-the-local-model} The  linear local model $\mathcal{F}_{\lin}$ is a foliation with linear holonomy and s-connected integration the Lie groupoid  (\ref{fol-gpds-loc-mod}). Moreover, this groupoid:
\begin{enumerate}[(i)]
\item coincides with $\Mon(M, \mathcal{F}_{\lin})$ iff $\hat{S}$ is simply connected.
\item coincides with $\Hol(M, \mathcal{F}_{\lin})$ iff the action of $\Gamma$ on $V$ is effective.
\end{enumerate}
\end{proposition}
\end{example}

\subsection{Foliation versus \'etale groupoids} Recall that an {\bf \'etale groupoid} is a Lie groupoid whose source map 
is a local diffeomorphism. Typical examples of \'etale groupoids include:
\begin{itemize}
\item the identity groupoid $M\tto M$ of a manifold, 
\item the action groupoid associated to a discrete group action on a manifold. 
\end{itemize}
The fundamental example coming from foliation theory is the restriction 
of the holonomy groupoid of $(M, \cF)$ to a complete transversal $T$ (i.e., a transversal intersecting all the leaves):
\[ \left.\Hol(M, \cF)\right|_T \tto T.\]

Historically, \'etale groupoids associated with a foliation were introduced via pseudogroups, as the objects that encode the transverse geometry of the foliation (see Remark \ref{rmk-Transversal geometric structures} below). The main point about \'etale groupoids is that they can be handled very much as usual manifolds. 
The resulting theory should be viewed as a study of ``singular spaces", namely, the orbit spaces of the \'etale groupoids. The role of the \'etale groupoid is to provide a ``desingularization" of the singular space. 

A foliation groupoid $\cE\tto M$ and the \'etale groupoid 
\begin{equation}\label{ET} 
\eE_T:= \left( \left.\cE\right|_T \tto T\right)
\end{equation}
obtained by restricting $\cE$ to a complete transversal $T$ for $\cF$ have the same leaf space. This passage to the \'etale groupoid depends on the choice of a transversal $T$ but, modulo the appropriate notion of equivalence, called Morita equivalence (see Section \ref{ssec:Morita}), this choice is irrelevant.  An entirely similar story holds for any Lie groupoid $\cG$, with the exception that the restriction to a complete transversal is not \'etale unless $\cG$ is a foliation groupoid.
% Actually, an entirely similar story holds in greater generality for any Lie groupoid $\cG$, with the exception that the restriction to a complete transversal is not \'etale unless $\cG$ is a foliation groupoid, and this will be relevant for us later.

A fundamental property that will be used repeatedly is the following: in an \'etale groupoid $\eE_T\tto T$, any arrow $g: x\to y$ induces a germ of diffeomorphisms:
\begin{equation}\label{germ-of-bisection}
\sigma_{g}: (T, x) \to (T, y).
\end{equation}  
To define it choose a neighborhood $U$ of $g$ in $\eE_T$ where both $s$ and $t$ restrict to local
diffeomorphisms and then take the germ at $x$ of $\sigma_g:= (t|_U)\circ (s|_U)^{-1}$.

Given a foliation groupoid $\cE\tto M$, if the restriction $\eE_{T}$ to some complete transversal $T$ is effective then the same
holds for any other transversal. In such case we say that $\cE$ is an {\bf effective foliation groupoid}. One can characterize holonomy groupoids as follows:

\begin{proposition}[\cite{CrMo}]
\label{prop-red-compl-transv} 
A foliation groupoid $\cE\tto M$ is the holonomy groupoid of the induced
foliation on the base iff $\cE$ is s-connected and effective. 
\end{proposition}

\begin{example}\label{ex-local-linear-models-doi} Consider the linear foliation $\cF_\lin$ in Example
\ref{ex-local-linear-models} associated with a $\Gamma$-cover $\hat{S}\to S$ and a linear action $\Gamma\action V$. 
A  complete transversal to $\cF_\lin$ is furnished by $V$ sitting inside the linear local model $\hat{S}\times_{\Gamma} V$ 
as $v\mapsto [x, v]$, where $x\in \hat{S}$ is fixed. The restriction of the Lie groupoid  (\ref{fol-gpds-loc-mod}) to this transversal is 
isomorphic to the action groupoid $\Gamma\ltimes V\tto V$. Therefore, Proposition \ref{prop-red-compl-transv} immediately implies 
part (ii) in Proposition \ref{gpd-of-the-local-model}.
\end{example}

\subsection{Morita equivalence}\label{ssec:Morita}
Morita equivalence is relevant to our story for it is the equivalence that reflects the ``transverse geometry'' or the ``geometry of the leaf space''. Let
us recall its precise definition using bibundles \cite{Hae}. For more details we refer to \cite{ALR,BW, Ha08, Lerman, Mrcun}.

A  {\bf Morita equivalence} between two Lie groupoids $\cG_i\tto M_i$, $i\in \{1, 2\}$, also called a {\bf Morita bibundle},
denoted by $P: \cG_1\simeq \cG_2$ and illustrated by the diagram
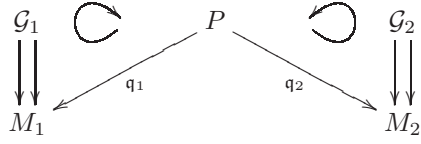
\[
\xymatrix{
 \G_1 \ar@<0.25pc>[d] \ar@<-0.25pc>[d]  & \ar@(dl, ul) & P\ar[dll]^-{\qq_1}\ar[drr]_-{\qq_2} & \ar@(dr, ur)   & \cG_2 \ar@<0.25pc>[d] \ar@<-0.25pc>[d]\\
M_1 & & & & M_2}
\]
is given by a smooth manifold $P$, endowed with:
\begin{itemize}
\item surjective submersions $\qq_1:P\to M_1$ and $\qq_2:P\to M_2$;
\item  commuting groupoid actions on $P$ of $\cG_1$ from the left, making $\qq_2:P\to M_2$ into a principal $\cG_1$-bundle, 
and of $\cG_2$ from the right, making $\qq_1:P\to M_1$ into a principal $\cG_2$-bundle. 
\end{itemize}

Given a Morita equivalence $P: \cG_1\simeq \cG_2$ one finds that:
\begin{enumerate}[(i)]
\item there is a homeomorphism of the orbit spaces $M_1/\cG_1$ and $M_2/\cG_2$, where two orbits
$\mathcal{O}_i\subset M_i$ correspond to each other iff
$\qq_{1}^{-1}(\mathcal{O}_1)= \qq_{2}^{-1}(\mathcal{O}_2)$;
\item if $x_1\in\O_1$ and $x_2\in \O_2$ are points in orbits in this correspondence, then the isotropy 
Lie groups $\cG_{1, x_1}$ and $\cG_{2, x_2}$ are isomorphic;
\item the groupoid $\cG_1$ is proper/Hausdorff iff the groupoid $\cG_2$ is.
\end{enumerate}
% These properties illustrate how Morita equivalence codifies the ``geometry of the leaf space". 

\begin{example}[Gauge groupoids]\label{ex:gauge} 
Given a Morita equivalence $P$ as above, the groupoid $\cG_1$ can be recovered from $\cG_2$ together with $P$ and its structure of principal $\cG_2$-bundle over $M_1$: $\cG_1$ will be isomorphic to the {\bf gauge groupoid} 
\[ \Gauge{\cG_2}{P}:=\left( P\times_{M_2} P/\cG_2 \tto M_1 \right),\]
the quotient of the submersion groupoid associated to $\qq_2: P\to M_2$ (Example \ref{ex-simple-foliations}) modulo the (diagonal) action of $\cG_2$. The isomorphism is induced by the division map $P\times_{M_2}P\to \cG_1$.
\end{example}

For a foliation groupoid $\cE\tto M$ and any complete transversal $T$, the groupoids $\cE$ and $\eE_{T}$ (see \eqref{ET}) are Morita equivalent: $P:= t^{-1}(T)$ defines a Morita bibundle, where $\qq_1$ and $\qq_2$ are the restrictions of $s$ and $t$, respectively, and the actions are given by
the multiplication of $\cE$. This leads to the following characterization of foliation groupoids, which is a refinement of Proposition \ref{fol-groupoid-isotropy}:

\begin{proposition}[\cite{CrMo}] 
A Lie groupoid $\cE$ is a (proper) foliation groupoid iff  it is Morita equivalent to a (proper) \'etale groupoid.
\end{proposition}

\begin{remark}[Haefligers's approach to transverse geometric structures]\label{rmk-Transversal geometric structures} Let us call {\bf a Haefliger sheaf on $\R^q$} any sheaf on $\R^q$ that comes together with an action of local diffeomorphisms $\phi: U\to V$ between opens in $\R^q$; that means that any such $\phi$ induces a bijection $\phi_*: \cS(U)\to \cS(V)$ and $\phi_*$ is compatible with the restriction maps. More formally, $\cS$ is a $\Gamma^q$-sheaf, where $\Gamma^q\tto \R^q$ is the Haefliger groupoid, whose space of arrows consists of germs of local diffeomorphisms, with the sheaf topology. A good example to keep in mind is %$\cS=\Omega^\bullet$, 
the sheaf of differential forms.

Fix such a Haefliger sheaf $\cS$. Using local charts, $\cS$ extends to all $q$-dimensional manifolds, giving rise to a functor defined on the category $\textrm{Man}^q$ consisting of $q$-dimensional manifolds and local diffeomorphisms between them. 
This extension is unique if we require it to have the same properties as $\cS$, but now with respect to diffeomorphisms between manifolds; it will be denoted by the same letter $\cS$.

Given an \'etale groupoid $\eE\rightrightarrows T$ over a $q$-dimensional manifold, the sheaf property allows us to define $\cS(\eE)$ as the set of $\eE$-invariant structures on the base,
\[ \cS(\eE):= \cS(T)^{\eE} .\]
More precisely, for any arrow $g$ of  $\eE$ we have a germ $\sigma_g$ of a local diffeomorphism of $T$ from $s(g)$ to $t(g)$, see
(\ref{germ-of-bisection}), and $u\in \cS(T)^\eE$ iff $\sigma_g$ takes $\textrm{germ}_{s(g)}(u)$ to $\textrm{germ}_{t(g)}(u)$, for any $g\in\eE$. In other words, $\cS(T)$ is a $\eE$-sheaf and $\cS(\eE)$ is the space of its invariant sections. Moreover, any Morita equivalence $\eE_1 \simeq \eE_2$ of 
\'etale groupoids induces a bijection $\cS(\eE_1)\cong \cS(\eE_2)$, and this construction is natural with respect to composition of Morita equivalences.

The last property allows us to further extend $\cS$ to arbitrary foliation groupoids $\cE\tto M$ by making use of the restrictions
$\eE_{T}$ (see (\ref{ET})). To make the definition independent of the choice of $T$, we define $\cS(\cE)$ as the set of collections 
\[  u= \{u^{T}\}\]
of elements $u^T\in \cS(\eE_T)$, one for each complete transversal $T$, with the property that 
for any two such transversals $T_1$ and $T_2$, $u^{T_1}$ to $u^{T_2}$ correspond to each other via the map
$\cS(\eE_{T_1}) \cong \cS(\eE_{T_2})$ induced by the natural Morita equivalence between 
$\eE_{T_1} $ and $\eE_{T_2}$ (i.e. the composition of the Morita equivalences $\eE_{T_1} \simeq \cE \simeq \eE_{T_2}$ or, more directly, the Morita equivalence defined by the Morita bibundle $\eE(T_1, T_2)$ of arrows starting in $T_1$ and landing in $T_2$). With this, it is clear that once a complete transversal $T$ is fixed, the obvious map 
$\cS(\cE) \to \cS(\eE_T)$ is 1-1.

Therefore, given a codimension $q$ foliated manifold $(M, \cF)$ one can define the set $\cS(M/\cF)$ 
of {\bf transverse $\cS$-structures} on $(M, \cF)$ as $\cS(\Hol(M,\cF))$. For instance, if one considers differential forms on manifolds, the space of transverse forms for $(M, \cF)$ is the space $\Omega^{\bullet}(T)^{\Hol(M,\cF)}$ of differential forms on a complete transversal $T$, invariant under holonomy. Similarly for transverse Riemannian metrics, transverse measures, transverse symplectic forms, etc. 

One can often remove the ambiguity coming from the choice of a complete transversal by representing a transverse $\cS$-structure 
directly at the level of $M$. For instance, in the case of $\cS= \Omega^{\bullet}$ one looks at the basic forms 
\[ \Omega^{\bullet}(M)_{\cF-\bas}:= \{ \omega\in \Omega^{\bullet}(M): i_{V}\omega= 0, \ \Lie_{V}\omega= 0 \ \ \ \textrm{for\ all}\ V\in \Gamma(\cF)\}.\]
Restriction from $M$ to any complete transversal $T$ induces an isomorphism:
% \begin{equation}\label{basic-forms-on-T}
\[
\Omega^{\bullet}(M)_{\cF-\bas}\cong \Omega^{\bullet}(T)^{\Hol(M,\cF)}, \omega\mapsto \omega|_{T}.
\]
% \end{equation}
Hence, $\Omega^{\bullet}(M)_{\cF-\bas}$ yields a concrete realization of $\cS(M/\cF)$ at the level of $M$. 
\end{remark}

\subsection{Foliations of $\mathcal{C}$-types} 
We turn now to the study of foliations of compact types. Any such foliation has linear holonomy. In fact, we have:

\begin{lemma} \label{lemma-gen-pr-gpds} 
Let $(M,\cF)$ be a foliation of proper (respectively, s-proper) type. Then
its leaves  are closed embedded (respectively, compact) submanifolds, the holonomy groups are finite,
and the orbit space is Hausdorff. If $(M,\cF)$ is strong proper, then the leaves
have finite fundamental group.
\end{lemma} 

\begin{proof}
A simple topological argument implies that for a proper (respectively, s-proper) Lie groupoid all the orbits are closed embedded (respectively, compact) submanifolds, the isotropy groups are compact, and the orbit space, furnished with the quotient topology, is Hausdorff (see e.g. \cite{CS,We3,Zu}). For an s-connected integration $\cE$ of $\cF$, the isotropy groups of $\cE$ surject onto the holonomy groups. 
\end{proof}

% We will discuss later whether the properties stated Lemma  \ref{lemma-gen-pr-gpds}  characterize properness and/or s-properness of a foliation.

While strong $\mathcal{C}$-types were defined using the monodromy groupoids, we claim that the $\mathcal{C}$-types can be checked using the holonomy groupoids:

\begin{theorem}\label{prop-fol-crit-C} 
A foliation is of $\mathcal{C}$-type iff its holonomy groupoid has property $\mathcal{C}$.
\end{theorem}

The proof of Theorem \ref{prop-fol-crit-C} is deferred until the next section, where we discuss normal forms for foliations. For now, we look at some examples. One should keep in mind that a foliation $(M,\cF)$ is of compact type iff it is s-proper and $M$ is compact.

\begin{example}\label{ex-simple-foliations-2} A simple foliation $(M,\cF)$ as in Example \ref{ex-simple-foliations} is always of proper type, it is of s-proper type iff $p$ is proper and it is of compact type iff $M$ is compact.
\end{example}

%v \begin{example}\label{ex-simple-foliations-2} A simple foliation $(M,\cF)$ is:
% \begin{itemize}
% \item always of proper type;
% \item of s-proper type if $p$ is proper;
% \item of compact type if $M$ is compact.
% \end{itemize}
% \end{example}

\begin{example}\label{ex-local-linear-models-2}  For the linear foliation $(\hat{S}\times_{\Gamma} V, \cF_{\lin})$ the explicit integrations \eqref{fol-gpds-loc-mod} and \eqref{mon-loc-mod} together with  Proposition \ref{gpd-of-the-local-model} imply that $\cF_{\lin}$ is:
\begin{itemize} 
 \item proper (respectively, $s$-proper) iff $\Gamma$ is finite (respectively, $\Gamma$ is finite and $S$ is compact);
 \item strong proper (respectively, strong s-proper) iff $\pi_1(S)$ is finite (respectively, $\pi_1(S)$ is finite
 and $S$ is compact);
 \item never of compact type.
\end{itemize}
\end{example}

Note that in the definition of $\cC$-type one requires the foliation groupoid to be s-connected. One should be aware of the following phenomena:
\begin{enumerate}[(i)]
\item If a foliation groupoid $\cE\tto M$ is proper, then passing to its source connected component $\cE^0\subset \cE$
may destroy properness.
  \item If a foliation $(M, \cF)$ is of proper type and $U\subset M$ is open, then $(U, \cF|_{U})$ may fail to be of proper type. 
\end{enumerate}
In fact, notice that for a foliation $(M, \cF)$ and an open set $U\subset M$, the leaves of $\cF|_{U}$
are the connected components of the intersections of the leaves of $\cF$ with $U$. Although the restriction $\Hol(M, \cF)|_U$ still integrates $\cF|_{U}$, it may fail to be s-connected. Passing to the associated s-connected groupoid, one gets precisely the holonomy groupoid of $\cF|_{U}$:
\begin{equation}\label{hol-restr-to-U} 
\Hol(U, \cF|_{U})= (\Hol(M, \cF)|_{U})^{0} .
\end{equation}
Examples illustrating (i) and (ii) can then be constructed even starting from a simple foliation.
For instance, consider $M= \R^2$ with the foliation induced by the second projection
and restrict it to $U= \R^2\setminus \{0\}$.
Then $(U, \cF|_{U})$ is not proper because its leaf space is not Hausdorff. In particular,
$\Hol(M, \cF)|_{U}$ is proper 
while its source connected component is not.

Lemma \ref{lemma-gen-pr-gpds} gives necessary conditions for the properness of a foliation, but they are not
sufficient as illustrated by the next two examples: 

\begin{example}\label{counter-ex-proper-fol}
There are foliations where all leaves are embedded, the holonomy groups are finite, and the leaf space is Hausdorff,
but are not of proper type:  consider the linear foliation of the M\"obius band $M$ by circles. The middle circle $C$
is the only leaf with non-trivial holonomy.

Restrict now this foliation to the open $U$ obtained from $M$ by removing one point in  $C$.
In this way, the leaf space remains unchanged, is Hausdorff, and the leaves clearly have the desired properties.
However, the holonomy groupoid is not proper, as can be seen by considering the holonomy group of the initial foliation, given by
(\ref{fol-gpds-loc-mod}), and then restricting as in (\ref{hol-restr-to-U}). From Theorem \ref{prop-fol-crit-C} we conclude 
that $(U, \cF|_{U})$ is not of proper type.
\end{example}

\begin{example} 
There are foliations where all leaves are embedded, the homotopy groups are finite, and the leaf space is Hausdorff,
but are not of strong proper type:  consider the first projection $p:\R^5\to\R$ and on the fiber $p=0$ remove the complement
of a tubular neighborhood of an embedding $\mathbb{P}^2\subset \{0\}\times\R^4$. The resulting submersion 
$p:M\to\R$ defines a simple foliation whose leaves have the desired properties. Since a curve in $\mathbb{P}^2$ which is non-trivial in homotopy 
is a vanishing cycle,  the monodromy groupoid cannot be Hausdorff. However, being an instance of a simple foliation, it is of proper type. 
\end{example}

These examples indicate that properness is more difficult to check directly. The situation is quite different for $s$-properness and compactness when the local Reeb stability implies a converse to Lemma  \ref{lemma-gen-pr-gpds}. This brings us to normal forms.

\subsection{Normal forms}
We start with the the standard local Reeb stability:

% The simplest local normal form for a foliation is the well-known local Reeb stability:

\begin{theorem}[Local Reeb stability]\label{local-Reeb-stab} Let $(M, \cF)$ be a codimension $q$ foliation. If $S$ is a compact leaf with finite holonomy 
group, then there exists a saturated neighborhood $U$ of $S$ and a foliated isomorphism
\[ (U, \cF|_{U})\overset{\cong}{\longrightarrow} (\hat{S}\times_{\Gamma} \nu_x(S) , \cF_{\lin}),\]
where the right hand side is the linear model associated to the linear holonomy action of the holonomy group
$\Gamma= \Hol_x(M, \cF)$ at some point $x\in S$ (Example \ref{ex-local-linear-models}). 
\end{theorem}

As promised, an immediate consequence is:

\begin{theorem}\label{thm-s-proper} A foliation is s-proper (respectively, strong s-proper) iff all its  leaves are compact and 
have finite holonomy (respectively, fundamental group). 
\end{theorem}

\begin{proof}
We are left to prove the converse implication. For that notice that it suffices to check s-properness on saturated neighborhoods, i.e., by the previous theorem, on the linear models.  But this was already remarked in Example \ref{ex-local-linear-models-2}. 
\end{proof}

Theorem \ref{thm-s-proper} shows that foliations of s-proper type are the same thing as classical
compact foliations (Remark \ref{classical-compact}) all of whose leaves have finite holonomy. While the leaves of 
an arbitrary classical compact foliation need not have finite holonomy, one of the main results in the subject 
states that for compact foliations the following two conditions are equivalent (\cite{EMS,Ep}):
\begin{enumerate}[(a)]
  \item All leaves have finite holonomy.
  \item The leaf space  is Hausdorff.
\end{enumerate}
% Therefore, we can restate the part of Theorem \ref{thm-s-proper} concerning s-properness as follows:

Hence, the part of Theorem \ref{thm-s-proper}  concerning s-properness  can be restated as:

\begin{corollary}
A foliation is s-proper iff all its leaves are compact and its leaf space is Hausdorff.
\end{corollary}

What about the proper case? For that we need a version of local Reeb stability which holds on saturated neighborhoods of non-compact 
leaves. To achieve such a ``normal form" one needs to enlarge the class of ``local models" allowed. 

\begin{example}[Non-linear local models]
The new local models start with the following data (compare with the linear local models in Example \ref{ex-local-linear-models}):
\begin{enumerate}[(i)]
\item A finite group $\Gamma$ acting linearly on a finite dimensional vector space $V$.
\item A connected manifold $P$ endowed with a free and proper action of $\Gamma$.
\item A $\Gamma$-equivariant submersion $\mu: P\to V$ with connected fibers.
\end{enumerate}
The foliation $\cF(\mu)$ by the fibers of $\mu$ descends to the quotient modulo $\Gamma$ and the new local model is
the resulting foliated manifold
\begin{equation}\label{new-model} 
(P/\Gamma, \cF(\mu)/\Gamma ) .
\end{equation}
The foliations arising in this way are still of proper type. To see this, we exhibit a proper integrating  foliation groupoid. We start with the submersion groupoid associated to $\mu$ (Example \ref{ex-simple-foliations}), denoted $P\times_{V} P$, and we consider the diagonal action of $\Gamma$. The action is free, proper and by groupoid automorphisms. Therefore
\begin{equation} \label{new-model-gpd} 
(P\times_V P)/\Gamma \tto P/\Gamma,
\end{equation}
is a Lie groupoid, which is easily seen to be a proper integration of $\cF(\mu)/\Gamma$. Of course, this is just an instance of the gauge construction from 
Example \ref{ex:gauge}. 

Notice that if one starts with a $\Gamma$-cover $\hat{S}\to S$ and a representation $\Gamma\to \GL(V)$, 
letting $P= \hat{S}\times V$ and $\mu:P\to V$ be the second projection, one recovers the local linear models of Example \ref{ex-local-linear-models}, together with their integration.
\end{example}

Here is our version of Reeb stability for non-compact leaves of proper foliations:

\begin{theorem}\label{thm-gen-model}
If $(M, \cF)$  is a foliation of proper type and $S$ is a leaf, then there exists  data (i)-(iii) as above,
a saturated neighborhood $U$ of $S$,  and a diffeomorphism of foliated manifolds
\[ (U, \cF|_{U})\cong (P/\Gamma, \cF(\mu)/\Gamma) \]
sending $S$ to the leaf $\mu^{-1}(0)/\Gamma$. 
\end{theorem}

This result is in fact  the ``improved local model" for proper groupoids of \cite{CS}, applied to foliation groupoids. For foliation groupoids however, both the local model as well as the proof, are simpler. We will  sketch here an argument which, modulo some small adaptations,  can be applied also in the Poisson context \cite{CFMc}.

%We will sketch here the argument since we will need in \cite{CFMc} an analogous result in Poisson Geometry which follows along the same steps. Moreover, the proof is simpler for foliation groupoids and it will also make clear where the data (i)-(iii) for the local model comes from.

\begin{proof} 
Fix an s-connected, proper integrating groupoid $\cE\tto M$ and let $x\in S$. One divides the proof into the following steps:
\vskip 5 pt

1) Choose a small transversal $T$ to the foliation with $T\cap S=\{x\}$. The restriction
$\eE= \cE|_{T}$ is a proper \'etale groupoid which has $x$ as a fixed point.
\vskip 5 pt

2) For any proper \'etale groupoid $\eE \tto T$ with a fixed point $x$, there exists a saturated neighborhood $V\subset T$ of $x$ 
together with an action of the (finite) isotropy group $\eE_x$ on $V$ such that 
\[ \eE|_{V} \cong \eE_x\ltimes V \cong  \eE_x\ltimes T_x V.\]
The first isomorphism follows from the linearization of proper \'etale groupoids around fixed points
(\cite[Proposition 5.30]{MM} or \cite[Section 6]{We3}), while the second one follows from Bochner's Linearization Theorem (\cite{DK}).

\vskip 5 pt

3) In our case, we have $\eE_x=\cE_x$ and $T_x V=\nu_x(S)$. Hence, if we consider the saturation $U\subset M$ of $V$, then there is a Morita equivalence: 
\[ \cE|_{U} \simeq \cE_x\ltimes \nu_x(S).\]
\vskip 5 pt

4) For a bibundle $P$ that implements this Morita equivalence,
\[
\xymatrix{
 \cE|_{U} \ar@<0.25pc>[d] \ar@<-0.25pc>[d]  & \ar@(dl, ul) & P\ar[dll]^-{\qq_1}\ar[drr]_-{\qq_2} & \ar@(dr, ur)   & \cE_x\ltimes \nu_x(S) \ar@<0.25pc>[d] \ar@<-0.25pc>[d]\\
U & & & & \nu_x(S)  }
\]
the right action amounts to a free and proper action of $\cE_x$ on $P$ for which $\mu$ is an equivariant surjective submersion.
\vskip 5 pt

5) Therefore, using the gauge construction (Example \ref{ex:gauge}), 
$\cE|_U$ can be recovered from the groupoid on the right and the bibundle $P$. In our case this translates simply into the desired isomorphism:
\begin{equation}\label{nform-groupoid-proper}
\xymatrix@=0.3pc{ 
\cE|_U \ar@<0.25pc>[dd] \ar@<-0.25pc>[dd] & & (P\times_{\nu_x} P)/\Gamma \ar@<0.25pc>[dd] \ar@<-0.25pc>[dd] \\
& \cong & \\
U & & P/\Gamma}
\end{equation}
which sends the class of a pair $(p, q)\in P\times_{\nu_x} P$ to the unique arrow $g\in \cE|_{U}$ with the property that $p= gp$.
\end{proof}

%\begin{remark}
One can recover the $s$-proper case as follows. Letting $\Gamma=\cE_x$, note that $\hat{S}=s^{-1}(x)$ is a $\Gamma$-covering of $S$ and that $ \Gamma$ acts on $\nu_x(S)$ via the linear holonomy. One then checks that $P=\hat{S}\times \nu_x(S)$ is a bibundle implementing the Morita equivalence between $\cE|_{U}$ and $\cE_x\ltimes \nu_x(S)$, giving rise to the isomorphism:
\begin{equation}\label{nform-groupoid-s-proper}
\xymatrix@=0.3pc{ 
\cE|_U \ar@<0.25pc>[dd] \ar@<-0.25pc>[dd] & & (\hat{S}\times \hat{S})\times_{\Gamma} \nu_x(S) \ar@<0.25pc>[dd] \ar@<-0.25pc>[dd] \\
& \cong & \\
(U, \cF|_{U}) & & (\hat{S}\times_{\Gamma} \nu_x(S), \cF_{\lin})  }
\end{equation}
In this way, we recover the linear local model of Example \ref{ex-local-linear-models}. We conclude that the Local Reeb Stability Theorem applies to any leaf $S$ of an $s$-proper foliations and, in fact, the diffeomorphism (\ref{nform-groupoid-s-proper}) is nothing but a groupoid version of Theorem \ref{local-Reeb-stab} for \emph{any} $s$-proper foliation.
%\end{remark}

% Once we have normal forms for proper foliations, we can show that the $\mathcal{C}$-type of a foliation can be tested via its holonomy groupoid.

\begin{proof}[Proof of Theorem  \ref{prop-fol-crit-C}] We have to show that if an integration $\cE$ of $(M,\cF)$ is of $\mathcal{C}$-type, then 
the same holds for the holonomy groupoid $\Hol(M,\cF)$. It is clear that such properties descend to quotients provided the latter are Hausdorff. To check that 
$\Hol(M, \cF)$ is indeed Hausdorff one proceeds, again, by restricting to small enough saturated opens and checking it for the local model (\ref{new-model}). 

If the action of $\Gamma$ on $V$ is effective, then so is the action of $\Gamma$ on $P\times_V P$. Proposition \ref{prop-red-compl-transv} then 
implies that (\ref{new-model-gpd}) must be the holonomy groupoid of (\ref{new-model}) and we are done. 
The general case can be reduced to the effective one as follows: the representation $\rho: \Gamma\to\GL(V)$ has image and kernel denoted by 
 $\Gamma_0$ and $K$, respectively. By construction, $\Gamma_0$ acts effectively on $V$. Consider $P_0= P/K$.
Then the action of $\Gamma$ on $P$ descends to an action of $\Gamma_0$ on $P_0$ (still free and proper) and $p$ descends to $p_0: P_0\to V$. It is clear that $P/\Gamma= P_0/\Gamma_0$. 
\end{proof}

\subsection{Orbifolds}
\label{ssec:Orbifolds}

% We recall that p
Proper foliation groupoids serve as atlases for orbifold structures:

\begin{definition}
\label{def-orbifold-fol} 
Let $B$ be a topological space. An {\bf orbifold atlas} on $B$ is a pair $(\cB, p)$ consisting of:
\begin{itemize}
\item a proper foliation groupoid $\cB\tto M$;
\item a homeomorphism $p: M/\cB{\to} B$ between the space of orbits of $\cB$  and $B$.
\end{itemize}
An {\bf orbifold} is a pair $(B,\cB)$ consisting of a space $B$ and an orbifold atlas $\cB$ on $B$.
\end{definition}

Two orbifold atlases $(\cB_i, p_i)$, $i\in \{1, 2\}$, are said to be {\bf equivalent} if there exists a Morita equivalence $\cB_1\simeq \cB_2$ with the property that the induced homeomorphism on the orbit spaces is compatible with $p_1$ and $p_2$.
%
%\comment{modified the second sentence to avoid repetition}
%{
%\begin{remark}[The \'etale viewpoint] 

It follows from the previous discussion that, by passing to a transversal $T$, one can always use orbifold atlases which are \'etale.
While nowadays one uses general {\it foliation} groupoids \cite{Moe02}, the first approaches to orbifolds via groupoids  
used only \'etale atlases (see e.g. \cite{ALR}). Although \'etale atlases are often advantageous, restricting to them is unnatural, not only conceptually, but also from the point of view of concrete examples. For example, quotients $M/G$ of proper, locally free, actions of Lie groups come with an obvious choice of orbifold atlas (the action groupoid), but not with a canonical \'etale one. PMCTs will provide similar examples. 

%\end{remark}
\begin{remark} \label{rk-orbifold-subtlety}
In the existing literature, orbifolds are often defined as ``a space with an equivalence class of orbifold atlases", while an atlas is interpreted as a
``presentation" of the orbifold \cite{Moe02, ALR}. However, note that: 
\begin{itemize}
\item two equivalent atlases can be equivalent in many, very distinct, ways, and
\item two different equivalences give rise to different ways of passing from one atlas to the other.
\end{itemize}
Hence, not fixing an atlas gives rise to subtleties. This shows up already when defining, in an atlas-independent way, a {\it morphism} between orbifolds, or a {\it vector bundle} over an orbifold.
This kind of problems can be solved by developing a rather heavy categorical language \cite{Lerman}. In practice, we do not have to deal with such issues and all orbifolds arise with a canonical orbifold atlas, in the sense of our definition. Given an orbifold $(B, \cB)$, it is sometimes advantageous to pass to a more convenient atlas $\cE$: however, such a passing will always be done via a {\it specified} Morita equivalence $Q_\cE:\cE\simeq \cB$. 
\end{remark}

\begin{example}[Manifolds and smooth orbifolds] Any manifold $B$ can be seen as an orbifold by using the trivial groupoid $B\tto B$ as an atlas.
Such an orbifold is called a {\bf smooth orbifold}. Note that an orbifold is smooth precisely when a/any defining atlas has no isotropy. Other equivalent \'etale atlases are provided by choosing a manifold atlas $\{U_i, \phi_i\}_{i\in I}$ for $B$ and considering the associated covering groupoid, i.e. the \'etale groupoid with space of objects the disjoint union $\coprod_{i} U_i$ and one arrow from $(x, i)$ to $(x, j)$ for each $x\in U_i\cap U_j$. 

On the other hand, for any orbifold $B$ one may also talk about the {\bf smoothness of the underlying topological space}: one requires the topological space $B$ to admit a smooth structure such that, for some orbifold atlas $\cB\tto M$, the quotient map $p: M\to B$ is a submersion. This condition does not depend on the choice of the atlas and determines a unique smooth structure on $B$, if it exists.  

Obviously, for a smooth orbifold the underlying topological space is smooth. However, it is important to keep in mind that the underlying topological space maybe smooth while the orbifold itself may still fail to be a smooth orbifold, for the orbifold atlases may have non-trivial isotropy. A simple example is obtained by taking the action groupoid $\Gamma\ltimes B\tto B$ associated with a trivial action of a finite group. Our study of PMCTs will give rise to much more interesting examples- where the isotropy information is important and cannot be disregarded.
\end{example}

\begin{example}[Classical orbifolds]\label{Classical orbifolds} Originally, orbifold structures on a space $B$ were defined in complete analogy with smooth structures, 
but using charts that identify the opens in $B$ with quotients $\R^n/\Gamma$ of finite groups $\Gamma$ acting effectively on $\R^n$ \cite{Satake,Thur}. 
With the appropriate compatibility between such charts, one obtains the notion of a \emph{classical orbifold atlas}.  Similarly to the covering groupoids above for manifolds, such an atlas can be organized into a proper \'etale groupoid whose space of orbits is $B$ \cite{ALR,MM}. Therefore, the classical notion of orbifold can be seen as particular classes of orbifolds in the sense of Definition \ref{def-orbifold-fol}. Here, following \cite{ALR, Moe02}, we will adopt the following equivalent working definition: a {\bf classical orbifold} is an orbifold for which the defining atlas is effective. 
\end{example}

Notice that the subtleties related to orbifolds atlases mentioned in Remark \ref{rk-orbifold-subtlety} are not present in the case of classical orbifolds: 

% \begin{lemma} For any two equivalent effective  orbifold atlases $\cB_{i}\tto T_i$ the Morita equivalence $Q: \cB_1\simeq \cB_2$ (inducing the identity on $B$) is unique up to isomorphism. 
% \end{lemma}
\begin{lemma}  
\label{lem:classical:orbifold}
For a classical orbifold $B$, an equivalence $Q: \cB_1\simeq \cB_2$ between two orbifold atlases for $B$ is unique up to isomorphism. 
\end{lemma}

Since any (proper) \'etale groupoid has an associated effective (proper) \'etale groupoid, we see that any orbifold has an {\bf underlying classical orbifold} structure. In this terminology, the smoothness of the underlying topological space of an orbifold is equivalent to the condition that its underlying classical orbifold is smooth. A general orbifold structure can be seen as a classical orbifold together with extra data, which is codified in the isotropy groups of the orbifold atlases.

\begin{example}[Good orbifolds]
A large class of examples of orbifolds arise as quotients $M/\Gamma$ for proper effective actions of discrete groups $\Gamma$: the action groupoid $\Gamma\ltimes M$ gives an orbifold atlas. 
Orbifolds of this type are called {\bf good orbifolds} \cite{ALR,Thur}.
% called {\bf developable orbifolds} or {\bf good orbifolds}.
\end{example}

\begin{example}[Foliations of $\cC$-type and orbifolds]
\label{ex:proper:foliations:orbifold}
For a foliation  $(M,\cF)$ of $\cC$-type, any s-connected, proper integration $\cE\tto M$ makes the leaf space $B=M/\cF$ into an orbifold $(B,\cE)$. 
Different integrations give different orbifold structures. However, the holonomy groupoid $\Hol(M,\cF)\tto M$ provides a smallest integration (Theorem \ref{seq-mon-G-hol}), which is proper (Theorem \ref{prop-fol-crit-C}) and effective  (Proposition \ref{prop-red-compl-transv}). Hence, the underlying classical orbifold of any orbifold defined by a foliation $(M,\cF)$ of $\cC$-type has atlas $\Hol(M,\cF)\tto M$.
\end{example}

\begin{remark}[Geometric structures on orbifolds] 
\label{rmk:structures on orbifolds} 
Haefliger's approach to transverse structures, discussed in Remark \ref{rmk-Transversal geometric structures}, when applied to the orbifold atlases allows one to consider various geometric structures on orbifolds, such as vector bundles, differential forms, Riemannian structures, etc: if $\cS$ is a Haefliger sheaf on $\R^q$ then $\cS(B)$ for a $q$-dimensional orbifold $(B,\cB)$ is defined by applying $\cS$ to the orbifold atlas $\cB\tto M$. 
% Note that 
Given some other atlas $Q_\cE:\cE\simeq \cB$ we have an induced isomorphism $\cS(\cE)\cong\cS(\cB)$, which in general depends on the Morita equivalence $Q_\cE$. 
\end{remark}

We can now return to regular PMCTs, the conclusion being that their leaf spaces are orbifolds. More precisely, each s-connected, proper integration gives rise to a orbifold structure on the leaf space, so one has the following more precise version of Theorem \ref{thm-reg-fol}: 

\begin{theorem}\label{thm-reg-fol2} If $(M, \pi)$ is a regular Poisson manifold of $\mathcal{C}$-type and $\cG$ is an s-connected integration of $(M, \pi)$ having property $\mathcal{C}$, then $\cG$ fits into a short exact sequence of Lie groupoids 
\[ \xymatrix{1\ar[r] & \cT(\cG) \ar[r] & \G \ar[r] & \cBG(\cG) \ar[r]& 1},\]
where:
\begin{enumerate}[(i)]
\item $\cT(\cG)$ is a smooth bundle of tori consisting of the identity connected components of the isotropy Lie groups $\G_x$.
\item $\cBG(\cG)$ is an s-connected foliation groupoid integrating $\cF_{\pi}$ satisfying property $\mathcal{C}$.
\end{enumerate}
In particular, $\cG$ induces an orbifold structure on the leaf space $B= M/\cF_{\pi}$, with $\cBG(\cG)$ as orbifold atlas.
The underlying classical orbifold has atlas $\Hol(M,\cF_\pi)\tto M$.
\end{theorem}

\begin{proof} This is basically proven in \cite{Moe03} in the context of regular groupoids. The main remark is that $\cT(\cG)$ is a 
\emph{closed} subgroupoid of $\cG$. This implies not only that $\cBG(\cG)$ is a Lie groupoid, but also  that it is of $\mathcal{C}$-type, and in particular
\emph{Hausdorff}, whenever $\cG$ is. In our case, since the isotropy Lie algebras are abelian and $\cG_x$ are compact, $\cT(\cG)$ will be a bundle of tori. 
\end{proof}

The reader will notice that the integration in Theorem \ref{thm-reg-fol2} does not need to be symplectic. We now turn to the implications of considering symplectic integrations.

%%%%%%%%%%%%%%%%%%%%%%%%
%%%%%%%%%%%%%%%%%%%%%%%%
%%%%%%%%%%%%%%%%%%%%%%%%
%%%%%%%%%%%%%%%%%%%%%%%%
%%%%%%%%%%%%%%%%%%%%%%%%
%%%%%%%%%%%%%%%%%%%%%%%%
%%%%%%%%%%%%%%%%%%%%%%%%
\section{Integral affine structure}
\label{sec:integral:affine}
%%%%%%%%%%%%%%%%%%%%%%%%
%%%%%%%%%%%%%%%%%%%%%%%%
%%%%%%%%%%%%%%%%%%%%%%%%
%%%%%%%%%%%%%%%%%%%%%%%%
%%%%%%%%%%%%%%%%%%%%%%%%
%%%%%%%%%%%%%%%%%%%%%%%%
%%%%%%%%%%%%%%%%%%%%%%%%
%%%%%%%%%%%%%%%%%%%%%%%%

Integral affine structures form another type of geometric structure that plays a crucial role in the study of compactness in Poisson Geometry. 
We initiate their study in this section. First, we start by recalling some basic definitions and
properties of integral affine structures on manifolds. Then we discuss transverse integral affine structures on foliated 
manifolds and their relation to integral affine structures on orbifolds. This will set the stage 
to prove the main result of this section, which improves on the orbifold structure on
the leaf spaces of PMCTs constructed in Theorem \ref{thm-reg-fol2}. A simplified version can be stated as follows:

\begin{theorem}
\label{thm-reg-fol-integ} For any regular Poisson manifold $(M, \pi)$ of $\mathcal{C}$-type its leaf space $B= M/\cF_{\pi}$ is
an integral affine orbifold: any s-connected symplectic integration $(\cG,\Omega)$ of $(M,\pi)$ having property $\cC$ gives rise to an integral orbifold structure on $B$. 
Moreover, the underlying classical orbifold is good.
\end{theorem}

For a foliated manifold $(M,\cF)$, a transverse integral affine structure is described by a collection of subgroups of its conormal bundle $\nu^*(\cF)$, as will be recalled below. On the other hand, the monodromy groups of a Poisson manifold \cite{CF2}, an invariant which characterizes its integrability, is  
another collection of subgroups of $\nu^*(\cF_\pi)$. Another goal of this section is to describe the role of the monodromy groups in the integral affine geometry of PMCTs. 
In this study, a new invariant of Poisson structures, called the \emph{extended monodromy groups}, will emerge.

%%%%%%%%%%%%%%%%%%%%%%%%%%%%%%
%%%%%%%%%%%%%%%%%%%%%%%%%%%%%%
%%%%%%%%%%%%%%%%%%%%%%%%%%%%%%
%%%%%%%%%%%%%%%%%%%%%%%%%%%%%%
%%%%%%%%%%%%%%%%%%%%%%%%%%%%%%
%%%%%%%%%%%%%%%%%%%%%%%%%%%%%%
\subsection{Integral affine structures on manifolds} 
\label{Integral affine structures on manifold} 
%%%%%%%%%%%%%%%%%%%%%%%%%%%%%%
%%%%%%%%%%%%%%%%%%%%%%%%%%%%%%
%%%%%%%%%%%%%%%%%%%%%%%%%%%%%%
%%%%%%%%%%%%%%%%%%%%%%%%%%%%%%
%%%%%%%%%%%%%%%%%%%%%%%%%%%%%%
%%%%%%%%%%%%%%%%%%%%%%%%%%%%%%

We will denote by $\Aff_\Z(\R^q)=\GL_\Z(\R^q)\ltimes \R^q$, the group of integral affine transformations, consisting of transformations of the type:
\begin{equation}
\label{IA-transf}
\R^q\to \R^q, \ \ x\mapsto A(x)+v,
\end{equation}
with $v\in \R^q$, $A\in \GL_\Z(\R^q)$. 
Integral affine structures on manifolds can be described in several
equivalent ways; we start with the most natural one, in terms of atlases. 
\begin{definition} An {\bf integral affine structure} on a $q$-dimensional manifold $B$ is a choice of a maximal
atlas $\{(U_i,\phi_i):{i\in I}\}$ with the property that each transition function 
\[ \phi_{j}\circ\phi^{-1}_i:\phi_i(U_i\cap U_j)\to \phi_j(U_i\cap U_j), \]
is (the restriction of) an integral affine transformation. Any chart $(U_i,\phi_i)$ is called an {\bf integral affine coordinate system}.
\end{definition}

More efficient descriptions are given in terms of latices.  By a {\bf lattice $\Lambda$ in a vector space} $V$
we mean a discrete subgroup of $(V, +)$ of maximal rank. We can always choose a basis $\{v_1, \ldots, v_q\}$ for $V$ such that:
\[ \Lambda= \Z v_1+ \cdots + \Z v_q. \]
The {\bf dual lattice} $\Lambda^{\vee}\subset V^*$ of the lattice $\Lambda$ is defined by 
\[ V^{\vee}= \{\xi\in V^*: \xi(\lambda) \in \Z,\ \ \forall\ \lambda\in \Lambda \}.\]
By a {\bf lattice on a vector bundle} $E\to B$ we mean a sub-bundle 
\[ E_\Z=\bigcup_{b\in B}\Lambda_b\subset E, \] 
consisting of lattices $\Lambda_b\subset E_b$. We say that it is smooth if, locally around each point $b_0\in B$, one can write
\[ \Lambda_b=  \Z\lambda_1(b)+ \ldots + \Z_q \lambda_q(b)\]
for some smooth local sections $\lambda_i$ of $E$.  An {\bf integral vector bundle} is a pair $(E,E_\Z)$ where $E_\Z$ is a smooth lattice in $E$.
An integral vector bundle $(E,E_\Z)$ comes with a {\bf canonical flat linear connection} $\nabla$: 
the one defined by the condition that all the local sections of $E_{\Z}$ are flat.

We can now state some alternative descriptions of integral affine structures, which will be useful later. 

\begin{proposition}\label{prop-int-affine-folklore} If $B$ is a q-dimensional
manifold, then there is a 1-1 correspondence between:
\begin{enumerate}[(i)]
\item An integral affine atlas $\{(U_i,\phi_i):{i\in I}\}$ on $B$;
\item A lattice $\Lambda^{\vee}\subset TB$ satisfying one of the following equivalent conditions:
\begin{enumerate}[(a)]
\item $\Lambda^{\vee}$ is smooth and any two local vector fields with values in  $\Lambda^{\vee}$ commute.
\item $\Lambda^{\vee}$ is smooth and the induced flat connection on $TB$ is torsion free.
\end{enumerate}
\item A lattice $\Lambda\subset T^*B$ satisfying one of the following equivalent conditions:
\begin{enumerate}[(a)]
\setcounter{enumii}{3}
\item $\Lambda$ is smooth and all its local sections are closed 1-forms.
\item $\Lambda$ is a Lagrangian submanifold of $(T^*B, \omega_\can)$.
\end{enumerate}
\end{enumerate}
In this 1-1 correspondence, $\Lambda$ and $\Lambda^{\vee}$ correspond to the lattices:
\[ 
\Lambda^{\vee}_b:= \Z\left.\frac{\partial}{\partial x_1}\right|_b+ \cdots+ \Z\left.\frac{\partial}{\partial x_q}\right|_b,\qquad
\Lambda_b:= \Z\left. \d x_1\right|_b+ \cdots+ \Z\left. \d x_q\right|_b,
\]
where $(x_1, \ldots, x_q)$ is any integral affine coordinate system around $b\in B$. 
\end{proposition}

% The proof is a simple exercise.

\begin{example}\label{ex-gen-IAS} Consider an integral affine group, i.e. a subgroup $\Gamma \subset \Aff_\Z(\R^q)$ of the group of transformations of type (\ref{IA-transf}). For explicit examples, it is useful to write an element $\gamma\in\Gamma$ in split form:
\[ \gamma= (u_\gamma, A_{\gamma}) \in \R^q\times \GL_\Z(\R^q). \]
The subgroup condition is equivalent to the following two conditions:
\begin{enumerate}[(i)]
\item $\gamma\mapsto A_\gamma$ defines a linear representation $\rho^{\lin}:\Gamma\to\GL_\Z(\R^q)$;
\item $\gamma\mapsto u_{\gamma}$ defines a 1-cocycle for $\Gamma$, i.e. 
$u_{\gamma\gamma'}= u_{\gamma}+ A_{\gamma}(u_{\gamma'}),\quad\forall\ \gamma, \gamma'\in \Gamma$.
\end{enumerate}
We say that the subgroup $\Gamma$ is of {\bf orbifold type} if the affine action on $\R^q$ is proper
and of {\bf smooth type} if the action is proper and free. In the smooth case,
\[ B:= \R^q/\Gamma\]
comes with an integral affine structure induced from the the standard integral affine structure on $\R^q$. Integral affine manifolds
which are quotients of $\R^q$ by smooth integral affine groups are called {\bf complete}.

The space $B$ can be obtained in stages. First, the split short exact sequence
% \begin{equation}\label{st-ex-seq-iAS} 
\[
\xymatrix{ 0\ar[r] & \R^q \ar[r] &\Aff_\Z(\R^q) \ar[r] & \GL_\Z(\R^q) \ar[r] & 0}
\]
% \end{equation}
restricts to $\Gamma$, yielding a short exact sequence: 
\[ \xymatrix{ 0\ar[r] & \Gamma^{\mathrm{tr}} \ar[r] &\Gamma \ar[r] & \Gamma^{\lin} \ar[r] & 0} \]
where:
\[ \Gamma^{\lin}= \{A_{\gamma}: \gamma\in \Gamma\},\quad
\Gamma^{\mathrm{tr}}= \{v_{\gamma}: \gamma\in \Gamma, A_{\gamma}= \mathrm{Id}\}.
\]
The translational part $\Gamma^{\mathrm{tr}}$ is a discrete subgroup of $(\R^q, +)$. Its rank $r$ is called the {\bf translational rank} of the integral affine group $\Gamma$ and defines a $\Gamma^\lin$-covering of $B$:
\[ B^{\lin}:= B/\Gamma^{\mathrm{tr}}\cong \T^r\times \R^{q-r}. \]

Here are two distinct examples of integral affine structures on the 2-torus. For the first one, we consider the subgroup $\Gamma\subset\Aff_{\Z}(\R^{2})$ generated by the translations $\gamma_1: (x, y)\mapsto (x+1, y)$ and $\gamma_2: (x, y)\mapsto (x, y+1)$ or, 
in the split notations, 
\[ \gamma_1= \left( (1, 0), \left[ \begin{array}{cc} 1 & 0 \\ 0 & 1 \end{array}   \right] \right), \ \gamma_2= \left( (0, 1), \left[ \begin{array}{cc} 1 & 1 \\ 0 & 1 \end{array}   \right] \right).\]
These two transformations commute and generate the abelian subgroup 
\[ \Gamma= \left\{ \gamma^{n}_{1}\gamma^{m}_{2}= \left((n,m),\left[ \begin{array}{cc} 1 & 0 \\ 0 & 1 \end{array}   \right]\right): n, m\in \Z\right\}. \]
Of course, this is just $\Z\times \Z$ with its standard action on $\R^2$, inducing the standard integral affine structure on the 2-torus $B= \R^2/\Gamma=\T^2$. Note that in this case $\Gamma$ has translational rank $2$ since we have:
\[ \Gamma^{\lin}=\left\{ \left[ \begin{array}{cc} 1 & 0 \\ 0 & 1 \end{array}   \right]\right\},\qquad \Gamma^{\mathrm{tr}}=\Z\times\Z\subset \R^2. \]

For the second example, we consider $\Gamma\subset\Aff_{\Z}(\R^{2})$ generated by 
\[ \gamma_1= \left( (1, 0), \left[ \begin{array}{cc} 1 & 0 \\ 0 & 1 \end{array}   \right]\right) , \ \gamma_2= \left( (0, 1), \left[ \begin{array}{cc} 1 & 1 \\ 0 & 1 \end{array}   \right] \right) .\]
Again, these commute, so they generate a subgroup isomorphic to $\Z\times \Z$, 
\[ \Gamma= \left\{ \gamma^{n}_{1}\gamma^{m}_{2}= \left((n+ \frac{m(m-1)}{2}, m),\left[ \begin{array}{cc} 1 & m \\ 0 & 1 \end{array} \right]\right): n, m\in \Z\right\}.\]
The quotient $B=\R^2/\Gamma$ is still \emph{diffeomorphic} to the 2-torus $\T^2$, but with a new integral affine structure. In this case $\Gamma$ has translational rank $1$ since we have:
\[ \Gamma^{\lin}= \left\{ \left[ \begin{array}{cc} 1 & m \\ 0 & 1 \end{array}   \right] : m\in \Z \right\},\ \Gamma^{\mathrm{tr}}= \Z(1, 0).\]. 
\end{example}

Integral affine structures look, at first, deceivingly simple. However, even some of the simplest 
questions are surprisingly hard to answer. For instance, we mention here an integral affine version of an
old conjecture in affine geometry:

\begin{conjecture}[Markus conjecture -- integral affine version]\label{Markus-conj}
Any compact integral affine manifold must be complete, i.e. of the form $\R^q/\Gamma$ for some smooth 
integral affine subgroup $\Gamma\subset  \Aff_\Z(\R^q)$. 
\end{conjecture}

\begin{remark}[Affine structures] \label{rk:affine-structures}
Affine structures on $B$ are defined, via atlases, as above, except that the changes of coordinates (\ref{IA-transf}) are only
affine (i.e., $v\in \R^q$ and $A\in\GL(\R^q)$). The analogue of the 1-1 correspondence from the previous proposition
states that they correspond to flat torsion free connections on $TB$. In this context one can talk about {\bf invariant} (or parallel) densities and volume forms by requiring invariance with respect to parallel transport of the connection. The standard Markus conjecture states that a compact affine structure with an invariant density has a complete connection.

Integral affine structures always admit invariant positive densities (or volume forms, in the orientable case): one sets $\mu:=|\d x_1\wedge \ldots \wedge \d x_q|$ for any choice of integral affine local coordinates $(x_1, \ldots, x_n)$.  An interesting question that seems to be still open is whether, conversely, any affine structure that admits an invariant density comes from an integral affine structure.
% It seems that it is not known if integral affine structures can be characterized as the affine structures admitting invariant densities (respectively, volume forms in the oriented case).
\end{remark}

 A integral structure $E_{\Z}$ on a vector bundle $E\to B$ gives 
rise to a bundle of tori $\cT:=E/E_{\Z}$. Conversely, given a bundle of tori $\cT\to B$, the Lie algebras of the fibers give rise
to a vector bundle $E\to B$, while the kernels of the exponential maps give rise to an integral structure $E_{\Z}$ on $E$ such that $\cT\equiv E/E_\Z$.

The very first indication of the close relationship between PMCTs and integral affine structures arises from the symplectic version of this correspondence, i.e., by considering proper integrations of the zero Poisson structure $\pi\equiv 0$ on $B$. Such an integration is the same thing as  a {\bf symplectic torus bundle over $B$}, i.e. a (smooth) bundle of tori $p: \cT\to B$ together with a symplectic form $\omega_{\cT}$ which is multiplicative in the sense that
\begin{equation}
\label{multiplicativity-omega-ct} 
m^*(\omega_{\cT})= \pr_{1}^{*}(\omega_{\cT})+ \pr_{2}^{*}(\omega_{\cT})
\end{equation}
where $m, \pr_1, \pr_2: \cT\times_{M} \cT\to \cT$ are the bundle multiplication and the two projections, respectively.

\begin{proposition}\label{prop:IAS-sympl-torus} If $\Lambda\subset T^*B$ defines an integral affine structure on $B$ then
\[ \cT_{\Lambda}:= T^*B/\Lambda\]
is a torus bundle and the standard symplectic form $\omega_\can$ on $T^*B$ descends to a symplectic form on $\cT_{\Lambda}$, 
making it into a symplectic torus bundle. Moreover, this gives rise to a bijection:
\[ 
\left\{\txt{integral affine\\ structures on $B$\\ \,}\right\}
\stackrel{1-1}{\longleftrightarrow}
\left\{\txt{isomorphism classes of\\ symplectic torus bundles over $B$\\ \,} \right\}
\]
\end{proposition}

\begin{proof} 
Let us start by remarking that given an integral affine structure $\Lambda\subset T^*B$ we have a smooth, free, and proper action of the bundle 
$\Lambda$ on the bundle of abelian groups $T^*B\to B$ by translations:
\[ \xi_b\mapsto \xi_b+ \alpha(b),\quad \alpha\in \Gamma(\Lambda). \] 
Hence, $\cT_{\Lambda}:= T^*B/\Lambda$ is a torus bundle. The canonical symplectic form $\omega_\can$ on $T^*B$ descends to $\cT_{\Lambda}$ iff $\omega_\can$ is invariant under this action. For a fixed $\alpha\in \Gamma(\Lambda)$, one checks easily that the translation by $\alpha$ is the time-1 flow of the vector field $X_{\alpha}$ given by $i_{X_\alpha}\omega_\can= p^*\alpha$, where $p: T^*B\to B$ is the projection. Denoting by $\phi_\a^t$ the flow of $X_\alpha$ and by $m_t:T^*B\to T^*B$ fiberwise multiplication by $t$, one finds that:
\[ m_t^*\omega_\can=t\omega_\can,\quad \phi_\a^{t}= m_t \circ \phi_{\alpha}^{1}\circ m_{1/t}.\]
So $\omega_\can$ is invariant under $\phi^1$ iff it is invariant under $\phi_a^t$, which will follow if:
\[ 0=\Lie_{X_\alpha}\omega_\can=\d i_{X_\alpha}\omega_\can=p^*\d\alpha, \quad \forall \alpha\in\Gamma(\Lambda). \]
But this follows from the fact that all sections of $\Lambda$ are closed (Proposition \ref{prop-int-affine-folklore}).
Since $\omega_\can\in\Omega^2(T^*B)$ is multiplicative, the same holds for the induced symplectic form $\w_{\cT_\Lambda}\in\Omega^2(\cT_{\Lambda})$.
We conclude that $(\cT_{\Lambda},\w_{\cT_\Lambda})$ is a symplectic torus bundle.

Conversely, let $(\cT, \w_\cT)$ be a symplectic torus bundle over $B$. It is a s-connected integration of the zero Poisson structure on $B$. Since the Weinstein groupoid of $(B, 0)$ is $(T^*B, \omega_\can)$, it follows that there is a morphism of symplectic groupoids:
\[ q: (T^*B, \omega_\can) \to (\cT, \omega_\cT),\]
which is a local diffeomorphism. The restriction of $q$ to a fiber gives a Lie group covering $q:T^*_bB\to \cT_b$, so 
its kernel is a lattice $\Lambda_b\in T^*_bB$. Since $q$ is a local diffeomorphism, for each $\alpha_0\in\Lambda_{b_0}$ 
there exists a unique smooth local section $\alpha\in\Gamma(T^*B)$ such that $\alpha(b_0)=\alpha_0$ and $\alpha(b)\in\Lambda_b$.
It follows that $\Lambda\subset T^*B$ is smooth. We conclude that the map $q$ factors through an isomorphism:
\[
\xymatrix{(T^*B,\w_\can)\ar[d] \ar[r]^q & (\cT,\w_\cT)\\ (\cT_\Lambda,\w_{\cT_\Lambda}) \ar[ru]_{\cong}} 
\]
This shows that  $\omega_\can$ descends to $T^*B/\Lambda$, i.e., that is invariant under the action of $\Lambda$. As in the first part of the proof, then every section of $\Lambda$ is closed, so $\Lambda$ is an integral affine structure on $B$, by Proposition \ref{prop-int-affine-folklore}.

To complete the proof it remains to show that if $\Lambda_1,\Lambda_2\subset T^*B$ are any integral affine structures and there is an 
isomorphism of symplectic torus bundle covering the identity, $\phi:(\cT_{\Lambda_1},\w_{\cT_{\Lambda_1}})\to (\cT_{\Lambda_2},\w_{\cT_{\Lambda_2}})$,  
then $\Lambda_1=\Lambda_2$. For that, observe that any such (possibly, non-symplectic) isomorphism, being continuous and additive on the fibers, must be induced by a bundle map $\hat{\phi}:T^*B\to T^*B$ of the form:
\[ \hat{\phi}:(b,\alpha)\mapsto (b,A_b(\alpha)), \quad (b\in B),\]
where $A_b:T^*_b B\to T^*_bB$ are linear isomorphisms with $A_b(\Lambda_1)=\Lambda_2$. 
To see that $A_b=$Id, so that $\Lambda_1=\Lambda_2$, one now uses that $\phi$ preserves the symplectic forms. 
\end{proof}

Integral affine structures are very closely related to Lagrangian fibrations. Indeed, any symplectic torus bundle fibers in a Lagrangian way over its base. Conversely, if  $q: (X, \Oga)\to B$ is a Lagrangian fibration with compact, connected fibers, then 
$B$ has an induced integral affine structure $\Lambda_{X}$ given by:
% Symplectic torus bundles are examples of Lagrangian fibrations and, indeed, integral 
% affine structures are very closely related to Lagrangian fibrations $q: (X, \Oga)\to B$ with compact, connected, fibers: 
% in this case $B$ has an induced integral affine structure $\Lambda_{X}$ which is defined by
\begin{equation}
\label{LambdaX-Lagr}
\Lambda_{X, x}:=\{\alpha\in T^*_xB: \phi^1_{\alpha_X}=\text{id}\},
\end{equation}
where $ \phi^t_{\alpha_X}$ denotes the flow of the vector field $\alpha_X$ on the fiber $q^{-1}(x)$ defined by:
\[
i_{\alpha_X}\Oga=q^*\alpha.
\]

We now give a Poisson geometric interpretation of this construction, which will serve as inspiration later on. First, the Lagrangian fibration condition is equivalent  to the fact that 
\[ q: (X, \Oga)\to (B, 0).\]
is a Poisson map into $B$ with the zero Poisson structure $\pi\equiv 0$, i.e. that $(X, \Oga)$ is a symplectic realization of the Poisson manifold $(B, 0)$. By the general properties of symplectic realizations \cite{CF2} (see also Section \ref{ssec:sympl-realizs-Hamilt} below), it follows that the Lie algebroid and the canonical integration act on the realization. In our case the Lie algebroid is $T^*B$ with the zero bracket and anchor (hence just a bundle of abelian Lie algebras) and the induced action on $X$ is 
\[ \sigma:\Omega^1(B)\to \X(X),\quad  \alpha\mapsto \alpha_X. \]
The canonical integration is the symplectic groupoid $(T^*B,\omega_\can)$, where $T^*B$ is viewed now as a bundle of abelian Lie groups, and the integration of the infinitesimal action $\sigma$ is the groupoid action:
\[
\xymatrix{
 (T^*B,\omega_\can) \ar[d]  & \ar@(dl, ul) & (X,\Oga),\ar[dll]^-{q}\\
(B,0) &   } \qquad\qquad \alpha\cdot u:=\phi^1_{\alpha_X}(u). 
\]
This action is locally free. Moreover, it is symplectic in the sense that (compare with (\ref{multiplicativity-omega-ct}) and with Appendix  \ref{App:Hamiltonian}):
\[
m^*(\Oga)= \pr_{1}^{*}(\omega_{\can})+ \pr_{2}^{*}(\Oga),
\]
where $m: T^*B\times_{B}X\to X$ is the action and $\pr_i$ are the projections.

Now the lattice (\ref{LambdaX-Lagr}) is precisely the isotropy of this action. Hence the corresponding symplectic torus bundle $\cT= T^*B/\Lambda_X$, a symplectic groupoid integrating $(B, 0)$, arises as the quotient of $T^*B$ which acts freely on $X$: 
\[
\xymatrix{
 (\cT,\omega_\cT) \ar[d]  & \ar@(dl, ul) & (X,\Oga)\ar[dll]^-{q}\\
(B,0) &   }
\]
The action is still symplectic, hence $q:X\to B$ is a symplectic principal $\cT$-bundle, or a free Hamiltonian $\cT$-space (see Appendix \ref{appendix:moment:maps}). Conversely, any such symplectic principal bundle is a Lagrangian fibration with compact, connected, fibers:

\begin{proposition}\label{pp:Lagr-fibr}
Any Lagrangian fibration $q:(X,\Oga)\to B$ with compact and connected fibers induces an integral affine structure $\Lambda$ on $B$, yielding a proper integration of $(B, 0)$, i.e. a symplectic torus bundle $\cT_{\Lambda}$ over $B$, for which it becomes a symplectic principal $\cT_{\Lambda}$-bundle.

Conversely, for any symplectic torus bundle $\cT_{\Lambda}$ over $B$, a symplectic principal $\cT_{\Lambda}$-bundle $q:(X,\Oga)\to B$ is a Lagrangian fibration with compact, connected fibers, inducing the integral affine structure $\Lambda$.
\end{proposition}

A classical result due to Duistermaat \cite{Duist} shows that Lagrangian fibrations with compact, connected fibers, are classified by the integral affine structure $\Lambda$ and the \emph{Lagrangian Chern class}. This will be recalled and generalized in Section \ref{ssec:torsors}.

%%%%%%%%%%%%%%%%%%%%%%%%%%%%%%
%%%%%%%%%%%%%%%%%%%%%%%%%%%%%%
%%%%%%%%%%%%%%%%%%%%%%%%%%%%%%
%%%%%%%%%%%%%%%%%%%%%%%%%%%%%%
%%%%%%%%%%%%%%%%%%%%%%%%%%%%%%
%%%%%%%%%%%%%%%%%%%%%%%%%%%%%%
\subsection{Integral affine structures on orbifolds and  foliations} 
\label{ssec:IAS-orbifol}
%%%%%%%%%%%%%%%%%%%%%%%%%%%%%%
%%%%%%%%%%%%%%%%%%%%%%%%%%%%%%
%%%%%%%%%%%%%%%%%%%%%%%%%%%%%%
%%%%%%%%%%%%%%%%%%%%%%%%%%%%%%
%%%%%%%%%%%%%%%%%%%%%%%%%%%%%%
%%%%%%%%%%%%%%%%%%%%%%%%%%%%%%

We define integral affine structures on orbifolds following Haefliger's approach (see Remark \ref{rmk:structures on orbifolds}):

\begin{definition}
\label{def:IAorbifold}
An {\bf integral affine structure on a orbifold} $(B,\cB)$ is an integral affine structure on the base of some \'etale orbifold atlas $\eE\tto T$ which is invariant under the action (\ref{germ-of-bisection}) by elements of $\eE$. 
\end{definition}

This definition only uses the linear part of the action (\ref{germ-of-bisection}), so an integral affine structure on an orbifold is the same things as one on its underlying classical orbifold. For this reason, the Morita equivalence $Q_\eE:\eE\simeq \cB$ plays here no role (see Lemma \ref{lem:classical:orbifold}): any Morita equivalence between two atlases allows us to move an invariant integral affine structure from one base to the other (pull-back to the bibundle, then push forward by the obvious quotient operation), and the result does not depend on the choice of equivalence.
 
\begin{example} 
If $\Gamma \subset \Aff_\Z(\R^q)$  is an integral affine group of orbifold type (see the previous example) then $B=\R^q/\Gamma$ will inherit the structure of integral affine orbifold. As a baby illustration, consider the subgroup $\Gamma\subset\Aff_{\Z}(\R)$ generated by 
\[ \gamma_1(x)= -x+ 1, \ \gamma_2(x)= -x .\]
As an abstract group, $\Gamma$ is the free group in two generators $\gamma_1$ and $\gamma_2$ subject to the relations $\gamma_1^2= \gamma_2^2= 1$, so that:
\begin{align*}
\Gamma\cong &~\Z_2\star \Z_2,\\
\Gamma^{\mathrm{tr}}=\{(0,0)\}, &\quad \Gamma^{\lin}= \{\textrm{Id}, - \textrm{Id}\}.
\end{align*}
The action of $\Gamma$ on $\R$ is proper and the only $x\in\R$ with non-trivial isotropy group are $x=\frac{n}{2}$ with $n\in\Z$, in which case we find:
\[ \Gamma_{\frac{n}{2}}=\{1, (\gamma_1\gamma_2)^{n-1}\gamma_1\}\cong\Z_2.\]
The quotient $B=\R/\Gamma= \S^1/ \Z_2$ is, topologically, just $[0, 1]$. This gives the interval $[0, 1]$ the structure of an integral affine orbifold.
\end{example}

To represent the integral affine structure in arbitrary, possibly non-\'etale, orbifold atlases, we need the notion of transverse integral affine structure. Recall that for a foliation $(M,\cF)$ of codimension $q$ a foliation atlas $\{(U_i,\phi_i):i\in I\}$ is an open cover $\{U_i:i\in I\}$ of $M$ together with submersions $\phi: U_i\to \R^q$ whose fibers are the plaques of $\cF$ in $U_i$. 

\begin{definition} 
\label{def:transverse:int:affine}
A {\bf transverse integral affine structure} on a foliation $(M,\cF)$ of codimension $q$ is a choice of a maximal foliation
atlas $\{(U_i,\phi_i):{i\in I}\}$ with the property that each transition function 
\[ \phi_{j}\circ\phi^{-1}_i:\phi_i(U_i\cap U_j)\to \phi_j(U_i\cap U_j), \]
is (the restriction of) an integral affine transformation in $\Aff_\Z(\R^q)=\GL_q(\Z)\ltimes \R^q$. A chart $(U_i,\phi_i)$ is called a {\bf transverse integral affine coordinate system}.
\end{definition}

More efficient descriptions of transverse integral affine structure can be given in terms of lattices: there is an analogue of Proposition \ref{prop-int-affine-folklore} where the lattices now live in the normal/conormal bundle to the foliation. The most useful characterization for us will be the one in terms of the conormal bundle $\nu^*(\cF)=(T\cF)^o\subset T^*M$, which we state as follows:

\begin{proposition}\label{prop-transv-int-affine-folklore} 
If $(M,\cF)$ is a foliation of codimension $q$, there a 1-1 correspondence between:
\begin{enumerate}[(i)]
\item A transverse integral affine atlas $\{(U_i,\phi_i):{i\in I}\}$ on $(M,\cF)$;
\item A lattice $\Lambda\subset \nu^*(\cF)$ which is a Lagrangian submanifold of $(T^*M,\omega_\can)$;
\item A lattice $\Lambda\subset \nu^*(\cF)$ locally spanned by $q$ closed, $\cF$-basic, 1-forms on $M$.
\end{enumerate}
In this 1-1 correspondence, $\Lambda$ is given by:
\[ 
\Lambda_x:= \Z\left. \d x_1\right|_x+ \cdots+ \Z\left. \d x_q\right|_x,
\]
where $(x_1, \ldots, x_q)$ is any transverse integral affine coordinate system around $x\in M$. 
\end{proposition}

Since basic forms are determined by their restriction to complete transversals, we deduce that a transverse integral affine structure on $(M, \cF)$ is the same thing as the choice of a holonomy invariant integral affine structure on a (any) complete transversal. This relates Definition \ref{def:transverse:int:affine} to Haefliger's approach (Remark \ref{rmk-Transversal geometric structures}).

\begin{example}[Simple foliations]\label{ex-trIAS-simple-fol} 
If  $(M,\cF)$ is simple, then transverse integral affine structures on $\cF$ are in 1-1 correspondence with integral affine structures on the smooth manifold $B= M/\cF$. In terms of lattices, they are related via pullback by $p:M\to B$.
\end{example}

\begin{example}[Orbifolds]\label{ex-trIAS-orbifolds} 
For foliations $(M, \cF)$ of proper type we know that the leaf space $B= M/\cF$ is an orbifold (see Example \ref{ex:proper:foliations:orbifold}).
We now have a bijection
\[ \left\{
                \txt{transverse integral affine structures\\
                   on the proper foliation $(M, \cF)$\\ \,}
              \right\}
\stackrel{1-1}{\longleftrightarrow}
 \left\{
                  \txt{integral affine structures\\
                    on the orbifold $B= M/\cF$\\ \, }
              \right\}\]
Strictly speaking, one has several orbifold structures on $B$, one for each proper s-connected integration $\cE$ of $\cF$. However, as we remarked before, the notion of integral affine structure only depends on the underlying classical orbifold.
\end{example}

% Strictly speaking, one has several orbifold structures on $B$, one for each proper s-connected integration $\cE$ of $\cF$ (see Example \ref{ex:proper:foliations:orbifold}). However, as we remarked before, the notion of orbifold integral structure only depends on the underlying classical orbifold, which has atlas $\Hol(M,\cF)$, and hence it is the same for all of them.

Starting with an arbitrary orbifold $B$, the previous example is relevant to the way one can represent integral affine structures on $B$ with respect to arbitrary orbifold atlases $\cE\tto M$ (not necessarily \'etale ones). While in this case $\cE$ may have disconnected s-fibers, we have to consider transverse integral affine structures $\Lambda$ for the foliation $\cF$ induced by $\cE$ on $M$ which satisfy the extra-condition that $\Lambda$ is invariant with respect to the induced action of $\cE$ on $\nu^*(\cF)$ (of course, this condition is superfluous if $\cE$ is s-connected). Such $\Lambda$s will be called {\bf $\cE$-invariant (transverse) integral affine}.

\begin{example}[Linear foliations]\label{ex-trIAS-linear-fol}
Let $\hat{S}\to S$ be a $\Gamma$-cover of a manifold $S$ and let $(V,V_\Z)$
be an integral vector space. If $\rho:\Gamma\to\GL_{V_\Z}(V)$ is an linear
representation that preserves the lattice, then the linear foliation
$(\hat{S}\times_{\Gamma} V, \cF_{\lin})$ (see Example \ref{ex-local-linear-models}) has a transverse integral affine structure.
\end{example}

The analogue of the relationship between integral affine structures and the zero-Poisson structure (Proposition \ref{prop:IAS-sympl-torus})
holds for transverse integral affine structures, provided one allows for Dirac structures into the picture. Let $(M,\cF)$
be a foliation with a transverse integral affine structure $\Lambda\subset\nu^*(\cF)$. Since $\Lambda\subset T^*M$ is Lagrangian, the
pullback to $\nu^*(\cF)$ of the canonical symplectic form $\omega_\can$ gives rise to a presymplectic torus bundle 
\[ (\cT_{\Lambda}= \nu^*(\cF)/\Lambda, \omega_{\cT}).\]

In general, by a {\bf presymplectic torus bundle} over a manifold $M$ we mean a bundle of tori $p: \cT\to M$ together with a closed 2-form $\omega_{\cT}\in\Omega^2(\cT)$ which is multiplicative and satisfies the non-degeneracy condition 
\[ \Ker(\omega_{\cT})\cap \Ker(\d p)= \{0\} .\]
For presymplectic torus bundles one has the following analogue of Proposition \ref{prop:IAS-sympl-torus}:

\begin{proposition}\label{prop:IAS-presympl-torus}  The correspondence $(\cF, \Lambda)\mapsto (\cT_{\Lambda}, \omega_{\cT})$ defines a bijection: 
\[ \left\{
                \begin{array}{ll}
                  \text{transverse integral affine}\\
                   \text{foliations}\ (\cF, \Lambda) \ \text{on}\ M
                \end{array}
              \right\}
\stackrel{1-1}{\longleftrightarrow}
 \left\{
                \begin{array}{ll}
                  \text{isomorphism classes of presymplectic}\\
                   \text{torus bundles}\ (\cT, \omega_{\cT})\ {over}\ M
                \end{array}
              \right\}\]
\end{proposition}

\begin{proof} 
We need to show that a presymplectic torus bundle defines a foliation $\cF$ with a transverse integral affine structure $\Lambda$. Let us mention the main changes in the arguments of the proof of Proposition \ref{prop:IAS-sympl-torus} 

For a vector bundle $E\to M$, closed multiplicative 2-forms on $E$, where multiplicativity is with respect to fiberwise addition, are necessarily of type 
\[ \omega_{\sigma}= \sigma^*(\omega_{\can}),\]
for some vector bundle map $\sigma: E\to T^*M$. This follows, e.g., from the integrability result of \cite{BCWZ} applied to $E$, viewed as a presymplectic groupoid. Hence, if $(\cT, \omega_{\cT})$ is a presymplectic torus bundle, and we apply this result to the bundle $\mathfrak{t}\to M$ consisting of the Lie algebras of the fibers of $\cT$, we find that:
\[ \exp^*(\omega_{\cT})= \omega_{\sigma},\]
for some $\sigma: \mathfrak{t}\to T^*M$, where $\exp: \mathfrak{t} \to \cT$ denotes the fiberwise exponential map. 

The non degeneracy condition continues to hold for the pull-back $\omega_\sigma$ and
when applied at elements $0_x\in \mathfrak{t}_x$ implies that $\sigma$ must be injective. Hence, there is a distribution $\cF\subset TM$ such that 
\[ \im(\sigma)= \left( TM/ \cF \right)^{*}= \nu^*(\cF).\]
The lattice $\Lambda_{\mathfrak{t}}:=\ker(\exp)\subset \mathfrak{t}$ will be moved by $\sigma$ into
a lattice $\Lambda\subset \nu^*(\cF)$. We have the extra-information that $\omega_{\sigma}$ descends to $\cT= \mathfrak{t}/\Lambda_{\mathfrak{t}}$; then, as in the proof of Proposition \ref{prop:IAS-sympl-torus}, this will imply that all the (local) sections of $\Lambda$ must be closed. In turn, this implies also that $\cF$ is integrable and then that $\Lambda$ is indeed a transverse integral affine structure for the foliation $\cF$. The rest of the arguments continue as for Proposition \ref{prop:IAS-sympl-torus}.
\end{proof}

\begin{example}
Assume that  $(M,\cF)$ is a simple foliation, as in Example \ref{ex-trIAS-simple-fol}, and $\Lambda$ is a
transversely affine structure that comes from an integral affine structure $\Lambda_B$ on the leaf space $B=M/\cF$. 
Then we have the presymplectic torus bundle $(\cT_\Lambda,\omega_{\cT_\Lambda})\to M$ and the symplectic torus bundle $(\cT_{\Lambda_B},\omega_{\cT_{\Lambda_B}})\to B$. The foliation defined by $\Ker(\omega_{\cT_\Lambda})$ on $\cT_\Lambda$ is simple as well, and its leaf space is precisely $\cT_{\Lambda_B}$. In other words, we have $\cT_\Lambda=p^*\cT_{\Lambda_B}$ and $\omega_{\cT_\Lambda}=p^*\omega_{\cT_{\Lambda_B}}$, where $p:M\to B$ is the projection onto the leaf space.
\end{example}

For the analogue of Proposition \ref{pp:Lagr-fibr} one replaces the Lagrangian fibrations by the symplectically complete isotropic fibrations of Dazord-Delzant \cite{DaDe}. This will be discussed in detail in Section \ref{sec:realizations}. 

% We do not give now an analogue of Proposition \ref{pp:Lagr-fibr}, but later in Section \ref{sec:realizations}
% we will introduce the symplectically complete isotropic fibrations of Dazord-Delzant \cite{DaDe}, which generalize Lagrangian fibrations, and allow one to extend this proposition.

The presymplectic torus bundle $(\cT_{\Lambda},\omega_{\cT})$ is also relevant for the integration of the Dirac structure $L_\cF$ associated with a foliation $\cF$ and understanding its $\cC$-type. Recall that this Dirac structure is defined by
\[ L_{\cF}:= \cF \oplus \nu^*(\cF) \subset TM\oplus T^*M, \]
and has presymplectic leaves consisting of the leaves of $\cF$ equipped with the zero-presymplectic form. We have an exact sequence of Lie algebroids
\[ \xymatrix{0\ar[r]& \nu^*(\cF)\ar[r]& L_{\cF} \ar[r]& \cF\ar[r]& 0},\]
which leads to explicit integrations of $L_{\cF}$. One such integration is obtained by observing that the linear holonomy action of $\Hol(M,\cF)$ on $\nu^*(\cF)$ descends to an action on 
$\cT_{\Lambda}$, so one obtains a groupoid
\[ \cT_{\Lambda}\Join \Hol(M, \cF)\tto M,\]
where an arrow $(\lambda,\gamma)$ consists of $\gamma\in \Hol(M, \cF)$ and $\lambda\in \cT_{\Lambda,\gamma(0)}$, 
and % the structure maps are given by:
\begin{equation}\label{Join-gpd} 
s(\lambda, \gamma)= s(\gamma),\quad t(\lambda, \gamma)= t(\gamma), \quad (\lambda, \gamma)\cdot (\lambda', \gamma')= (\lambda\cdot \hol_{\gamma}^{\lin}(\lambda'), \gamma\cdot \gamma'). 
\end{equation}
Together with the pull-back of $\omega_{\cT}$, this becomes a presymplectic groupoid integrating $L_{\cF}$. It is of $\cC$-type if $\cF$ is of $\cC$-type.

%%%%%%%%%%%%%%%%%%%%%%%%%%%%%%
%%%%%%%%%%%%%%%%%%%%%%%%%%%%%%
%%%%%%%%%%%%%%%%%%%%%%%%%%%%%%
%%%%%%%%%%%%%%%%%%%%%%%%%%%%%%
%%%%%%%%%%%%%%%%%%%%%%%%%%%%%%
%%%%%%%%%%%%%%%%%%%%%%%%%%%%%%
\subsection{From PMCTs to integral affine structures} 
\label{From PMCTs to integral affine structures}
%%%%%%%%%%%%%%%%%%%%%%%%%%%%%%
%%%%%%%%%%%%%%%%%%%%%%%%%%%%%%
%%%%%%%%%%%%%%%%%%%%%%%%%%%%%%
%%%%%%%%%%%%%%%%%%%%%%%%%%%%%%
%%%%%%%%%%%%%%%%%%%%%%%%%%%%%%
%%%%%%%%%%%%%%%%%%%%%%%%%%%%%%
We are now ready to describe the transverse integral affine structure associated with a PMCT, a fundamental geometric structure associated with such a of Poisson manifold.

If $(\G, \Omega)$ is a proper integration of $(M, \pi)$, then for any $x\in M$:
\begin{enumerate}[(i)]
\item the isotropy Lie group $\G_x$ is a compact Lie group with \emph{abelian} isotropy Lie algebra $\gg_x$, hence the kernel of the exponential defines a lattice
\[ \Lambda_{\cG,x}:=\Ker(\exp_{\gg_x})\subset \gg_{x}.\]
\item the symplectic form $\Omega$ induces an identification between the Lie algebroid $A(\G):=\Ker\d s$ and $T^*M$, which identifies $\gg_x$ with the conormal direction: 
\[ \gg_x\cong \nu_{x}^{*}(\cF_\pi), \quad v_x\mapsto  (i_{v_x}\Omega)|_{T_xM}\]  
\end{enumerate}
Putting (i) and (ii) together, we obtain lattices $\Lambda_{\cG, x}\subset \nu^{*}_x(\cF_\pi)$, and we set:
\[ \Lambda_{\cG}:=\bigcup_{x\in M}\Lambda_{\cG, x} \subset\nu^*(\cF_\pi). \]

An alternative description can be obtained by considering the torus bundle $\cT(\cG)$ made of the identity components of the isotropy groups (see Theorem \ref{thm-reg-fol2}), together with the restriction of $\Omega$. It is a presymplectic torus bundle so one can apply Proposition \ref{prop:IAS-presympl-torus} to obtain $\Lambda_{\cG}$. 

The relationship between these two approaches will be clear in the proof of the following basic result:

\begin{theorem}
\label{thm-lattice-proper-case} 
For each proper integration $(\cG,\Omega)$ of a regular Poisson manifold $(M, \pi)$ of proper type, $\Lambda_{\cG}$ defines a transverse integral affine structure on the symplectic foliation $\cF_\pi$.  
\end{theorem}

\begin{proof}
To show that $\Lambda_{\cG}$ is smooth we describe the lattices $\Lambda_{\cG,x}\subset \nu_{x}^{*}(\cF_\pi)$ as follows. Each  $\alpha\in\nu^*_x(\cF_\pi)=(T_x\cF_\pi)^0=\Ker(\pi^\sharp_x)$ 
corresponds to a right-invariant vector field $X_\alpha$ tangent to the isotropy group $\G_x$. By restricting $X_\alpha$ to $\G_x^0$,
the connected component of the identity of $\G_x$, we obtain an action of the bundle of abelian Lie algebras $\nu^*(\cF_\pi)$ 
on the bundle of tori $\cT(\cG)=\bigcup_{x\in M}\G^0_x$. The compactness of $\G^0_x$, implies that this action can be integrated to an action of the bundle of Lie groups $(\nu^*(\cF_\pi),+)$ on $\cT(\cG)$:
\[ \alpha\cdot g:=\phi^1_{X_\alpha}(g),\quad (\alpha\in (T_x\cF_\pi)^0,\ g\in \G^0_x), \]
where $\phi^\tau_{X_\alpha}$ denotes the flow of $X_\alpha$. Note that $\exp_{\gg_x}(\alpha)=\phi^1_{X_\alpha}(1_x)$, so we can identify
$\Lambda_\cG$ with the kernel of this action:
\[ \Lambda_{\cG,x}=\{\alpha\in \nu^*_x(\cF_\pi): \phi^1_{X_\alpha}=\text{id}\}. \]
This action is locally free, since the map $\alpha\mapsto X_\alpha$ is injective. This action is transitive on the fibers, 
since $\alpha\mapsto X_\alpha|_{1_x}\in T_{1_x}\cG_x$ is onto.
It follows that the kernel of the action $\Lambda_{\cG}$ is a smooth sub-bundle whose fibers $\Lambda_{\cG,x}$ are lattices in $\nu_x^*(\cF_\pi)$. 

In order to show that $\Lambda_{\cG}\subset T^*M$ is a  Lagrangian submanifold, note that $\dim\Lambda_{\cG}=\dim M=1/2\dim (T^*M)$ so we only need to check that $\omega_\mathrm{can}|_{\Lambda_{\cG}}=0$. By the fundamental property of $\omega_\mathrm{can}$, for any 1-form $\alpha:M\to T^*M$ we have:
\[ \alpha^*\omega_\mathrm{can}=\d \alpha. \]
Hence, it is enough to show that any 1-form $\alpha\in\Gamma(\Lambda_{\cG}|_U)$, defined on some open set $U\subset M$, is closed. To see this, 
observe that the associated vector field $X_\alpha$ satisfies:
\[ i_{X_\alpha}\Omega=t^*\alpha. \]
In fact, both sides are right invariant 1-forms and they coincide at the the identity section. Hence, when $\alpha_x\in\Lambda_{\cG,x}$ we find that:
\begin{align*}
0=(\phi^1_{X_\alpha})^*\Omega-\Omega&=\int_0 ^1 \frac{\d }{\d \tau} (\phi^\tau_{X_\alpha})^*\Omega~\d \tau\\
&=\int_0 ^1 (\phi^\tau_{X_\alpha})^*\Lie_{X_\alpha}\Omega~\d \tau\\
&=\int_0 ^1 (\phi^\tau_{X_\alpha})^*\d i_{X_\alpha}\Omega~\d \tau\\
&=\int_0 ^1 (\phi^\tau_{X_\alpha})^*t^*\d\alpha~\d \tau\\
&=\int_0 ^1 t^*\d\alpha~\d \tau=t^*\d\alpha,
\end{align*}
where we use $t\circ \phi^\tau_{X_\alpha}=t$. Since $t$ is a submersion we obtain, as claimed, $\d\alpha=0$.
\end{proof}

Using Theorem \ref{thm-reg-fol2} and Example \ref{ex-trIAS-orbifolds} we deduce:

\begin{corollary} 
Any s-connected, proper integration $(\cG,\Omega)$ of a regular Poisson manifold $(M,\pi)$ induces an integral affine orbifold structure on the leaf space $M/\cF_\pi$.
\end{corollary}

\begin{remark}[Twisted Dirac structures]\label{rem:twisted:tias}
The previous discussion extends to the Dirac case word by word. If $(M,L,\phi)$ is a regular $\phi$-twisted Dirac manifold, then a proper presymplectic integration $(\cG,\Omega,\phi)$ of $(M,L,\phi)$ defines a transverse integral affine structure $\Lambda_\cG\subset \nu^*(\cF_L)$. The reason is that the constraints on the kernel of $\Omega$ imply that its restriction to $\cT(\cG)$ still yields a presymplectic torus bundle.

For an s-proper, twisted presymplectic groupoid with smooth leaf space $B$, Zung \cite{Zu} described the integral affine
structure on $B$ by very different means.  Our approach, using transverse integral affine structures, allow us to deal with non-smooth 
leaf spaces as discussed in Example \ref{ex-trIAS-orbifolds} and, as we will see in \cite{CFMc}, and even with non-regular Dirac manifolds of compact types.
 \end{remark}

\subsection{The extended monodromy groups}
\label{sec:hol:monodromy}
%%%%%%%%%%%%%%%%%%%%%%%%%%%%%%
%%%%%%%%%%%%%%%%%%%%%%%%%%%%%%
%%%%%%%%%%%%%%%%%%%%%%%%%%%%%%
%%%%%%%%%%%%%%%%%%%%%%%%%%%%%%
%%%%%%%%%%%%%%%%%%%%%%%%%%%%%%
%%%%%%%%%%%%%%%%%%%%%%%%%%%%%%

We recall (see \cite{CF2}) that for any regular Poisson manifold $(M, \pi)$ there is a {\bf monodromy map} at $x\in M$:
\begin{equation}\label{partial-x} 
\partial_{\mon,x}: \pi_2(S,x) \rightarrow \nu_{x}^{*}(S)
\end{equation}
where $S= S_x$ is the symplectic leaf through $x$. The {\bf monodromy group} at $x$ is defined as the image of the monodromy map: 
\[ {\cN_\mon}|_{x}:= \im(\partial_{\mon,x}) \subset \nu_{x}^{*}(S),\]
and we set $\cN_\mon= \cup_{x\in M} {\cN_\mon}|_{x}$.

The origin of these Poisson invariants lies in the variation of symplectic areas, but they admit several interpretations, all of which will be useful in the sequel:
%In the context of regular Poisson manifolds, as explained in \cite{CF2}, there are several interpretations of these monodromy groups, all of which will be useful in the sequel:
\begin{itemize}
\item At the groupoid level, the Weinstein groupoid $\Sigma(M,\pi)$ yields a homotopy long exact sequence associated to $\s: \s^{-1}(x) \to S$ with first few terms:
\[ \xymatrix{\pi_2(S,x)\ar[r]^{\partial_{\mon,x}}  & \nu_{x}^{*}(S) \ar[r]^{\exp} &\Sigma_x(M,\pi) \ar[r] & 1}. \]
This gives a description of ${\cN_\mon}|_{x}$ as the kernel of the exponential map $\exp:\nu_x^*(S)\to \Sigma_x(M,\pi)$.
\item At the Lie algebroid level, any splitting  $\tau: TS\to T^*_SM$ of the short exact sequence of algebroids:
\begin{equation}
\label{pre-transitive-seq} 
\xymatrix{
0 \ar[r] & \nu^*(S) \ar[r] & \ T^*_{S}M \ar[r]^-{\pi^{\sharp}} &  TS \ar[r]  \ar@{-->}@/^1pc/[l]^{\tau} &  0 .}
\end{equation}
yields a curvature 2-form $\Omega_{\tau}\in \Omega^2(S, \nu^*(S))$ given by
\begin{equation}
\label{eq:curv-formula} 
\Omega_{\tau}(X, Y):= \tau([X, Y])- [\tau(X), \tau(Y)]\quad (\text{for}\  X, Y\in \X(S)).
\end{equation}
Viewing $\nu^*(S)$ as a flat vector bundle for the Bott connection, the 2-form $\Omega_{\tau}$ is closed as a form with coefficients $\nu^*(S)$. Its cohomology class does not depend on the choice of $\tau$ and defines a class $[\Omega_{\tau}]\in H^2(S, \nu^*(S))$, and:
% Then the monodromy map can be defined by:
\[ \partial_{\mon,x}: \pi_2(S,x) \rightarrow \nu_{x}^{*}(S),\quad [\sigma]\mapsto \int_\sigma \Omega_\tau. \]
\item The most geometric description of the monodromy arises as the variation of symplectic areas of leafwise spheres:  for any sphere $\sigma:(\mathbb{S}^2,N)\to (S_x,x)$ based at $x$ and a transverse direction $v\in\nu_x(S)$, one can find a foliated family of spheres $\sigma_t:(\mathbb{S}^2,N)\to (S_{x_t},x_t)$, such that $\sigma_0=\sigma$ and $v=[\dot{x}_t]$; one has:
\begin{equation}\label{eq:mon-as-var} 
\langle \partial_{\mon,x} [\sigma],v\rangle= \left.\frac{\d}{\d t}\right|_{t=0} \int_{\sigma_t} \omega_{S_{x_t}}.
\end{equation}
Hence, the quantity $\langle \partial_{\mon,x} [\sigma],v\rangle$ is the variation of the symplectic area of $\sigma$ in the normal direction $v$.
\end{itemize}

For integrable Poisson manifolds, each monodromy group ${\cN_\mon}|_x$ is a discrete subgroup of $\nu_{x}^{*}(S)$. As one could expect, they are closely related to the lattice $\Lambda_{\cG}$ of an s-proper integration. In the s-proper case the rank of ${\cN_\mon}$ does not depend on $x$, but it may fail to be a lattice, unless we are in the {\it strong} proper case. More precisely, we have:

\begin{theorem}\label{thm-lattice-strong-case} Let $(M, \pi)$ be a regular  Poisson manifold. Then:
\begin{enumerate}[(i)]
\item ${\cN_\mon}\subset \Lambda_{\cG}$ for any s-connected, proper integration $(\cG,\Omega)$ of $(M,\pi)$;
\item ${\cN_\mon}= \Lambda_{\Sigma(M, \pi)}$ if and only if $(M,\pi)$ is strong proper.  
\end{enumerate}
\end{theorem}

\begin{proof} We use the same notation of the proof of Theorem \ref{thm-lattice-proper-case}. Note that if $\alpha\in{\cN_\mon}|_x$ then the corresponding right invariant vector field $\widetilde{X}_\alpha$ in $\Sigma(M,\pi)$ satisfies 
\[ \phi^1_{\widetilde{X}_\alpha}=\text{id}.\] 
Under the covering projection $\Sigma(M,\pi)\to\cG$ this vector field is projected to  $X_\alpha$, which therefore also satisfies $\phi^1_{X_\alpha}=\text{id}$. It follows that $\alpha\in{\Lambda_\cG}|_x$, so (i) holds. Now (ii) follows from the definition of $\Lambda_{\cG}$ and the previous discussion on ${\cN_\mon}$.
\end{proof}

In particular, we have the following characterization of regular Poisson manifolds of strong $\cC$-type:

\begin{corollary}
\label{cor:strong:C:type}
A regular Poisson manifold $(M, \pi)$ is of strong $\cC$-type if and only if the foliation
$\cF_{\pi}$ is of strong $\cC$-type and $\cN_\mon$ is a transverse integral affine structure for the symplectic foliation $\cF_\pi$. 
\end{corollary}

\begin{proof}
The previous theorem, combined with Theorem \ref{thm-reg-fol}, proves the direct implication. 
For the reverse implication, note first that the lattice condition, together with the integrability criteria of \cite{CF1, CF2}, 
implies that $\Sigma(M, \pi)$ is smooth. Moreover, we have the short exact sequence of Lie groupoids:
\[
\xymatrix{1\ar[r]& \nu^{*}(\cF_\pi)/{\cN_\mon}\ar[r] & \Sigma(M, \pi) \ar[r]& \Mon(M, \cF_{\pi})\ar[r]& 1}
\]
where the second map associates to a cotangent path its base path, while the first map is induced from the exponential map 
$\exp:\nu^{*}_{x}(\cF_\pi)\rightarrow \Sigma_x(M,\pi)$ (see \cite{CF2}). This is a sequence of Lie groupoids in
which the extreme groupoids are of $\cC$-type. It follows immediately that the middle one is also of $\cC$-type, so the result follows.
\end{proof}

Conversely, one can look at regular Poisson manifolds $(M,\pi)$ for which $\cN_\mon=0$.
We have the following result, which includes as a particular case Proposition \ref{prop:IAS-sympl-torus}:

\begin{corollary}\label{cor:no:monodromy:strong:foliation}
 Let $(M,\pi)$ be a regular Poisson manifold such that:
 \begin{enumerate}[(i)]
  \item the monodromy groups are trivial: $\cN_\mon=0$;
  \item the symplectic foliation $\cF_\pi$ is of strong $\cC$-type.
 \end{enumerate}
Then each transversal integral affine structure $\Lambda$ on $(M,\pi)$ determines an s-connected, proper integration $(\cG_\Lambda,\Omega)$ such 
that $\Lambda_{\cG_\Lambda}=\Lambda$. Moreover, if the symplectic leaves are 1-connected, this establishes a bijection:
\[ \left\{
                \begin{array}{ll}
                  \text{transversal integral affine}\\
                   \text{structures}\ \Lambda \ \text{on}\ (M, \cF_\pi)
                \end{array}
              \right\}
\stackrel{1-1}{\longleftrightarrow}
 \left\{
                \begin{array}{ll}
                  \text{s-connected, proper  symplectic }\\
                   \text{integrations of}\ (M,\pi)
                \end{array}
              \right\}\]
\end{corollary}

\begin{proof}
When $\cN_\mon=0$, the exponential map $\exp:\nu^*(\cF_\pi)\to\Sigma(M,\pi)$ has no kernel. Therefore,
a transverse integral affine structure $\Lambda\subset \nu^*(\cF_\pi)$ yields a  subgroupoid $\exp(\Lambda)\subset\Sigma(M,\pi)$, which is
normal, discrete and Lagrangian. Hence, 
\[(\cG_\Lambda,{\Omega}):=(\Sigma(M, \pi),\Omega_{\Sigma(M,\pi)})/\exp(\Lambda)\]
is a symplectic groupoid integrating $(M,\pi)$.  As in the proof of Corollary \ref{cor:strong:C:type}, this symplectic groupoid is of $\cC$-type iff $\cF_\pi$ 
is of $\cC$-type. It is also clear that $\Lambda_{\cG_\Lambda}=\Lambda$.

In general, the isotropy groups $\Sigma_x(M,\pi)$ will not be connected, so there might be different symplectic
groupoids defining $\Lambda$. If the symplectic leaves are 1-connected, then $(\cG_\Lambda,{\Omega})$ is the only integration defining $\Lambda$.
\end{proof}

Our next aim is a more refined version of Theorem \ref{thm-lattice-strong-case}, and this will require a more refined version of the monodromy groups.  
These are new invariants, related to obstructions to s-properness, and they arise when revisiting the above descriptions of ${\cN_\mon}|_x$: the basic idea is to replace the spheres (2-homotopy classes) in the leaf $S$ by more general surfaces (2-homology classes) in $S$.

Consider an arbitrary regular Poisson manifold $(M, \pi)$ and fix $x\in M$. We choose a splitting $\tau: TS\to T^*_{S}M$ of the short exact sequence \eqref{pre-transitive-seq}. In order to integrate the resulting curvature 2-form $\Omega_\tau\in\Omega^2(S,\nu^*(S))$ over some surface we need first to pullback the vector bundle $\nu^*(S)\to S$ along $p_x:S^{\hol}_x\to S$, the $\Hol(\nabla)$-covering for the Bott connection based at $x$. 
This is the smallest cover over which the pullback of $\nu^*(S)$ becomes a trivial vector bundle $\nu_x^*(S)\times S^{\hol}_x\to S^{\hol}_x$. We can then set:

\begin{definition}
The {\bf hol-monodromy map} of the regular Poisson manifold $(M,\pi)$ at $x$ is the map:
\[ \partial_{\hol, x}: H_2(S^{\hol}_x, \Z) \to \nu_{x}^{*}(S), \quad [\sigma]\mapsto \int_\sigma p_{x}^{*}\Omega_\tau. \]
The  {\bf hol-monodromy group} ${\cN_{\hol}}|_x\subset\nu_x^*(S)$ is the image of this map.
\end{definition}

A version of the hol-monodromy map appears in the work of I.~M\u{a}rcut \cite{Marcut} on rigidity in Poisson geometry (see also \cite{CrMar}). For its geometric interpretation, we will consider smooth marked surfaces in the leaf $S$ through $x$, i.e. smooth maps 
\[ \sigma: (\Sigma, p) \to (S, x),\]
with $\Sigma$ a connected, oriented, compact surface without boundary, and $p\in \Sigma$.  By a {\bf leafwise deformation} of $\sigma$ we mean a family $\sigma_t: (\Sigma, p) \to (M, x_t)$ of smooth maps parametrized by $t\in (-\epsilon, \epsilon)$, starting at $\sigma_0= \sigma$ and such that for each fixed $t$ the surface $\sigma_t$  is inside the symplectic leaf through $x_t$. The transversal variation of $\sigma_t$ at $t= 0$ is the class of the tangent vector 
\[ \varnu(\sigma_t):= \left[ \left.\frac{\d}{\d t}\right|_{t=0} \sigma_t(p)\right]  \in \nu_x(S). \]

Note that given a smooth marked surface $\sigma: (\Sigma, p) \to (S, x)$ and a normal vector $v\in\nu_x(S)$ there may not exist
a leafwise deformation $\sigma_t$ with transversal variation $\varnu(\sigma_t)=v$. We will say that {\bf $\sigma$ is holonomy-trivial} with respect to the foliation $\cF_\pi$ if the holonomy of $\cF_\pi$ along loops of type $\sigma\circ \gamma$, with $\gamma$ a loop in $\Sigma$, is trivial. 

\begin{lemma}
\label{existence-defs} 
For any holonomy-trivial marked surface $\sigma: (\Sigma, p)\to (S,x)$ and $v\in \nu_x(\cF)$, one can find a leafwise deformation $\sigma_t$ with  $\varnu(\sigma_t)= v$. In this case, $\sigma$ admits a lift $\widetilde{\sigma}:(\Sigma, p)\to (S^{\hol}_x,\widetilde{x})$, where $\widetilde{x}$ denotes the class of the trivial loop.
\end{lemma}

\begin{remark}
\label{rem:holonomy-trivial}
Note that the notion of a holonomy-trivial surface uses the (non-linear) holonomy of the foliation $\cF_\pi$, while the covering $S^{\hol}_x\to S$ is relative to the Bott connection, i.e. the linear holonomy. Vanishing holonomy implies vanishing linear holonomy, but not the converse. Hence, a holonomy-trivial surface admits a lift to the cover $S^{\hol}_x$, but the converse is not true in general.
\end{remark}

\begin{proof}
Fix a complete Riemannian metric on $M$, split $TM= \cF_{\pi}\oplus E$ and consider:
\[ \phi: E_{\sigma}:=\sigma^*E\to M, \quad \phi(x, v):= \exp_{\sigma(x)}(v) .\]
Then $T_{0_x}E_{\sigma}= T_x\Sigma\oplus E_{\sigma(x)}$ and $(\d\phi)_{0_x}: T_{0_x}E_{\sigma}\to T_{\sigma(x)} M$ becomes $(i, (\d\sigma)_x)$, where $i: E\hookrightarrow TM$ is the inclusion. Hence $\phi$ is transversal to $\cF_\pi$ and we can take the pull-back foliation $\cF_{\phi}:= \phi^*\cF_\pi$ as a foliation on $E_{\sigma}$. The codimension remains the same, so the leaves of $\cF_{\sigma}$ are two-dimensional. Since $\Sigma$ is compact and tangent to $\cF_{\phi}$, it must be an entire leaf of $\cF_{\pi}$. On the other hand, as a general property of pull-back foliations, the holonomy of $\cF_{\sigma}$ at $p\in \Sigma$ factors through the holonomy of $\cF_\pi$ at $\sigma(p)= x\in S$. Therefore, since $\sigma$  is a holonomy-trivial surface, we deduce that the holonomy of $\cF_{\phi}$ along $\Sigma$ is trivial. By Reeb stability, $\cF_{\phi}$ is isomorphic, in a neighborhood of $\Sigma$, to the trivial foliation $\Sigma\times \R^q$. Then $\phi$ yields a smooth map $\Sigma\times \R^q\to M$ which takes leaves to leaves and induces isomorphisms at the level of the normal bundle. This allows one to construct for any normal vector $v\in\nu_x(S)$ a leafwise deformation $\sigma_t$ with transversal variation $\varnu(\sigma_t)=v$.
\end{proof}

We then have the following geometric interpretation of the hol-monodromy in 
terms of variations of symplectic areas whenever a class $[\widetilde{\sigma}]\in H_2(S^{\hol},\Z)$ is a lift of a holonomy-trivial $\sigma$:

\begin{proposition}
Let $(M,\pi)$ be a regular Poisson manifold. If $\sigma: (\Sigma, p)\to (S,x)$ is a holonomy-trivial marked surface then:
\[ \langle \partial_{\hol, x}[\widetilde{\sigma}],v\rangle =\left.\frac{\d}{\d t}\right|_{t=0} \int_{\sigma_t}\omega_{\cF_\pi}, \]
where $\sigma_t$ is a leafwise deformation of $\sigma$ with transversal variation $\varnu(\sigma_t)= v$ and $\widetilde{\sigma}:(\Sigma, p)\to (S^{\hol}_x,\widetilde{x})$ is a lift of $\sigma$.
\end{proposition}

\begin{proof}
The proof is the same as for variation of spheres, given in \cite{CF2} pp 97. % and so it will be omitted.
\end{proof}

Recall that $S^{\hol}_x$ is the smallest cover where the pullback of the flat bundle $\nu(S)\to S$ becomes trivial.
Of course, we can consider larger covers of $S$, i.e., covers $Q$ that factor through the holonomy cover:
\[
\xymatrix{ 
(\widetilde{S}_x,[x])\ar[d]\ar[ddrr]\\
(Q,q)\ar[d]\ar[drr] \\
(S^{\hol}_x,\widetilde{x})\ar[rr] & & (S,x)}
\]
where $\widetilde{S}_x$ is the universal covering space of $S$. Observing that $H_2(\widetilde{S}_x,\Z)\cong \pi_2(S,x)$, we obtain a diagram of monodromy maps:
\[
\xymatrix{ 
\pi_2(S,x)\ar[d]\ar[ddrr]^{\partial_{\mon,x}}\\
H_2(Q,\Z)\ar[d]\ar[drr]^{\partial_{Q,q}} \\
H_2(S^{\hol}_x,\Z)\ar[rr]^{\partial_{\hol,x}} & & \nu^*_x(S)}
\]
We set $\cN_{Q,x}:=\im\partial_{Q,q}$, so we have:
\[ \cN_{\mon}\subset \cN_Q\subset \cN_{\hol}. \]

Consider now a regular Poisson manifold $(M,\pi)$ whose symplectic foliation $\cF_\pi$ is proper.  In this case the linear and non-linear holonomy of $\cF_\pi$ coincide, hence:
\begin{itemize}
\item the geometric interpretation of the hol-monodromy is valid for every class in $H_2(S^{\hol}_x,\Z)$;
\item given any s-connected foliation groupoid $\cE$ integrating the symplectic foliation $\cF_\pi$, the s-fiber yields a covering space
$t:\cE(x, -)\to S_x$ which lies in between $\widetilde{S}_x$ and $S^{\hol}_x$. In particular, we have a corresponding monodromy map:
\[ \partial_{\cE,x}:H_2(\cE(x, -),\Z)\to \nu^*_x(S). \]
% whose image $\cN_{\cE}|_x$ will be called the {\bf $\cE$-monodromy group} at $x$.
\end{itemize}

\begin{definition} 
The  {\bf $\cE$-monodromy group} of $(M, \pi)$ at $x$ relative to the s-connected integration $\cE$ of the symplectic foliation $\cF_\pi$,
denoted $\cN_{\cE}|_x$, is the image of $\partial_{\cE,x}$.
\end{definition}

We can finally discuss the more refined version of Theorem \ref{thm-lattice-strong-case}:

\begin{theorem}\label{ext-mon-crit} Let $(M,\pi)$ be a regular Poisson manifold whose symplectic foliation $\cF_\pi$ is proper. For any s-connected integration $\cE$ of the symplectic foliation:  
\[ {\cN_\mon}\subset \cN_{\cE}\subset \cN_{\hol}, \]
where the first (respectively, second) inclusion becomes equality for $\cE= \Mon(M, \cF_\pi)$ (respectively, for $\cE= \Hol(M, \cF_\pi)$). 
Moreover, if $\cE= \cBG(\cG)$ is induced by an s-connected, proper integration $\cG$ as in Theorem \ref{thm-reg-fol2}, then 
\[ \cN_{\cE} \subset \Lambda_{\cG}.\]
\end{theorem}

\begin{remark}\label{rk-closed-ext-mon} When $\cE= \cBG(\cG)$ is induced by an s-proper integration $\cG$, it will follow from
Section \ref{sec:s-proper type} that $\cN_{\cE}$ will be not just a bundle of discrete subgroups of $\nu^*(\cF_\pi)$, but also a smooth, closed sub-bundle (in particular, of constant rank). In the maximal rank case, i.e. when $\cN_{\cE}$ is a lattice, one can show that $\cE$ is induced by an s-proper integration $\widetilde{\cG}$ with $\Lambda_{\widetilde{\cG}}= \cN_{\cE}$. Such integrations $\widetilde{\cG}$ deserve the name ``normalized".
\end{remark}

\begin{proof} The first part of the theorem follows from the remarks proceeding it. For the second part we 
use the Atiyah sequence (\ref{pre-transitive-seq}). Recall that similar sequences arise from principal bundles: if $q: P\rightarrow N$ is a principal $G$-bundle, then $A(P):= TP/G$ is  not only a vector bundle over $N$ but also a Lie algebroid with anchor induced by $\d q$ and with the bracket coming from the identification $\Gamma(A(P))= \X(P)^{G-\textrm{inv}}$. The short exact sequence associated to it is
\begin{equation}
\xymatrix{
0 \ar[r] & P\times_G\gg \ar[r] & A(P) \ar[r]^-{\d q} &  TN \ar[r] &  0 .}
\end{equation}
Splittings $\tau$ of this sequence are the same thing as connections on the principal bundle,
while the associated expression (\ref{eq:curv-formula}) is precisely the curvature of the connection.
When $G= \T$ is a torus, this closed form will represent the Chern class of the bundle, $c_1(P)\in H^2(N, \tt)$, which is integral:
the pairing of $c_1(P)$ with elements in $H_2(N, \Z)$ always lands in $\Lambda_\T$, the kernel of the exponential map of $\tt$.

Now, if $\cG$ is an s-connected, proper integration of $(M, \pi)$, then we obtain a principal bundle $q: \cG(x, -)\to \cE(x, -)$ with structure group the torus 
$\cT_x= \cT_x(\cG)$ (see Theorem \ref{thm-reg-fol2}). Moreover, the associated Atiyah sequence coincides with the pull-back of (\ref{pre-transitive-seq}) via the covering map $p_{\cE}: \cE(x, -)\to S$, which is just a translation of the fact that $\cG$ integrates the Lie algebroid $T^*M$. We deduce that 
\[ p_{\cE}^{*}[\Omega_{\tau}]\in H^2(\cE(x, -), \nu_{x}^{*})\]
coincides with the Chern class of the torus bundle $q: \cG(x, -)\to \cE(x, -)$. The integrality of this Chern class shows that evaluation on classes in $H_2(\cE(x, -))$ lands in $ \Lambda_{\cG}$, so we conclude that $\cN_{\cE}\subset \Lambda_{\cG}$. 
\end{proof}

\begin{corollary} \label{cor:finite-index}
If $(M, \pi)$ is of s-proper type then $\cN_{\hol}$ is a smooth closed sub-bundle of discrete subgroups of $\nu^*(\cF_\pi)$.
Moreover, if $\cE$ is induced by an s-connected, s-proper integration, then $\cN_{\cE}$ is of finite index in $\cN_{\hol}$.  
\end{corollary}

\begin{proof} Let $\cG$ be an s-connected, s-proper integration. Using the first part of Remark \ref{rk-closed-ext-mon}, it suffices to show that
$\cN_{\cE}|_x$ is of finite index in $\cN_{\hol}|_x$.  This follows from a general remark about finite covers applied to $\cE(x, -)\to S^{\hol}$: if $q: \widetilde{N}\to N$ is a finite cover between compact manifolds then 
\[ q_* (H_2(\widetilde{N}, \Z)/\text{torsion})\subset  H_2(N, \Z)/\text{torsion}\]
is of finite index. This is equivalent to the fact that, when working over $\Q$, $q_*: H_2(\widetilde{N})\to H_2(N)$ is surjective. In turn, this follows from the standard spectral sequence $E_{p, q}^{2}= H_p(\Gamma, H_q(\widetilde{N})) \Longrightarrow H_{p+q}(N)$ ($\Gamma$ the group of the cover): since the homology over $\Q$ of finite groups is trivial, we deduce that $H_{k}(N)$ is isomorphic to the space of $\Gamma$-coinvariants of $H_{k}(\widetilde{N})$ and $q_*$ becomes the quotient map. 
%is identified with the quotient map. 
\end{proof}

\begin{remark}
The notions of extended monodromy discussed here extend to non-regular Poisson manifolds and, in fact, to any Lie algebroid. We postpone this discussion to \cite{CFMc}.
\end{remark}

% \hspace*{.3in}

\subsection{Examples}  Here are some examples to illustrate the behavior of the groups $\cN_{\mon}$ and $\cN_{\hol}$, their computation and their relevance to the compactness-types.

\subsubsection{A non-proper example with $\cN_\hol$ discrete} 
The groups $\cN_\hol$ can be seen as Poisson invariants whose discreteness is a necessary condition for properness- cf. Corollary \ref{cor:finite-index}. Here is an example which shows that this condition is not sufficient and that $\cN_\hol$ provides interesting 
invariants also in the non-proper case. Start with the Reeb foliation $\cF$ of $\S^3$ and make it into a symplectic foliation by choosing a metric on $\S^3$ and considering the induced area forms on the leaves. The resulting Poisson structure is not of proper type since the Reeb foliation is not proper: for example, the linear and non-linear holonomy of the compact leaf are distinct. 

These Reeb type Poisson structures are always integrable since for any symplectic leaf $S$ we have $\pi_2(S,x)=\{0\}$, so that $\cN_\mon=\{0\}$. On the other hand, for points $x$ in the open leaves we obviously have $\cN_\hol|_x=\{0\}$, since the leaves are contractible. However, we claim that for points in the compact leaf, $\cN_\hol$ is not trivial. In order to see this, we consider as a model for a neighborhood of the compact leaf $\T^2$ the space $M=(\R\times\S^1\times\R)/\Z$, and we let $\R^+\times\S^1\times\R$, with coordinates $(r,\theta,z)$, be foliated by the level sets of the submersion:
\[ F(r,\theta,z)=(r^2-1)e^z, \]
and let $\Z$ act by translations in the $z$ coordinate. The compact leaf $\T^2$ corresponds to $r=1$. On $\R^+\times\S^1\times\R$ we consider the regular Poisson structure:
\[ \pi=\left((r^2-1)\frac{\partial}{\partial r}-2r\frac{\partial}{\partial z}\right)\wedge \frac{\partial}{\partial \theta}. \]
The function $F$ is a Casimir and $\pi$ is invariant under the $\Z$-action, so we obtain a Poisson structure on $M$ whose symplectic foliation is the Reeb foliation. Choosing the splitting $\tau:T\T^2\to T^*_{\T^2}M$ of $\pi^\sharp:T^*_{\T^2}M\to T\T^2$ defined by:
\[  \frac{\partial}{\partial \theta}\mapsto -\frac{1}{2}\d z,\quad \frac{\partial}{\partial z}\mapsto -\frac{1}{2}\d  \theta, \]
we find that its curvature 2-form is constant:
\[ \Omega_\tau\left(\frac{\partial}{\partial \theta},\frac{\partial}{\partial z}\right)=\left.\frac{1}{4}[\d z,\d\theta]_{\pi}\right|_{r=1}=-1. \]
The leaf $\T^2$ has trivial linear holonomy, so we conclude that
\[ \cN_\hol|_{r=1}=\left\{\int_\sigma\Omega_\tau: [\sigma]\in H_2(\T^2,\Z)\right\}=\Z\subset \R\]
is a discrete subgroup. Notice, by the way, that the (non-linear) holonomy of the compact leaf is non-trivial and
that there are no holonomy-trivial $\sigma:\Sigma\to \T^2$ with $[\sigma]\ne 0$, so one
cannot compute $\partial_\hol$ by transverse variations of symplectic areas.
In any case, the groups $\cN_\hol$ are discrete, but they do not form a smooth closed sub-bundle of $\nu^*(\cF_\pi)$ (compare with
Corollary \ref{cor:finite-index}).

\subsubsection{An s-proper but not strong proper example} 
Consider now 
\[ M= \T^2\times \R^+, \quad \pi=t \tfrac{\partial}{\theta_2}\wedge \tfrac{\partial}{\partial \theta_1}.\]
As above, $\cN_\mon=\{0\}$ and $\cN_{\hol}$ is clearly a lattice. 
In particular $(M, \pi)$ is not of strong proper type but, since the symplectic foliation is of proper type (even simple), one may expect that 
$(M, \pi)$ is of proper type. Let us prove all these in an explicit manner. It is useful to remark that the universal cover of $(M, \pi)$,
\[  \widetilde{M}= \R^2\times \R^+, \quad \widetilde{\pi}=t \tfrac{\partial}{\xi_2}\wedge \tfrac{\partial}{\partial \xi_1},\]
sits as an open Poisson submanifold of the linear $\hh(3)^*$, where $\hh(3)$ is the Lie algebra of the Heisenberg 
group $\mathrm{H(3)}$ of unipotent upper triangular $3\times 3$ matrices:
\[  H(3)= \{ \begin{pmatrix}
    1 & x & z\\ 0 & 1 & y \\ 0 & 0 & 1
   \end{pmatrix}: x, y, z\in \R\} \]
and where we use $\xi_1, \xi_2, t$ for the coordinates with respect to the canonical basis of $\hh(3)^*$, corresponding to 
$x, y$ and $z$, respectively. Hence the canonical integration of  $(\widetilde{M}, \widetilde{\pi})$ is the action groupoid arising from the 
{\it coadjoint} action, 
 \begin{equation}\label{eq:Hesienberg-action}
 \mathrm{H(3)}\curvearrowright \R^2\times \R^+,\quad (x,y,z)\cdot(\xi_1,\xi_2,t)=(\xi_1+ty,\xi_2-tx,t).
 \end{equation}
 The symplectic form on this groupoid comes from the Liouville form on $T^*\mathrm{H(3)}$. After trivializing $T^*\mathrm{H(3)}$ using left translations, in our coordinates, it becomes 
\begin{equation}\label{eq:Omega-Heisenberg} 
\Omega= d\xi_1\wedge\d x+\d\xi_2\wedge\d y+\d t\wedge\d z-t\d x\wedge\d y-x\d t\wedge \d y,
 \end{equation}
Of course, one can check directly that $(H(3)\ltimes \widetilde{M},\Omega)$ is a symplectic groupoid integrating $(\widetilde{M}, \widetilde{\pi})$. The action of $\pi_1(M)= \Z^2$ 
is Poisson hence it lifts to an action on the groupoid by symplectic groupoid automorphisms. Hence the canonical integration of $(M, \pi)$ 
 can be described exactly as above using (\ref{eq:Hesienberg-action}) and (\ref{eq:Omega-Heisenberg}) but with
 $\R^2\times \R^+$ replaced by $\T^2\times \R^+$ and the coordinates $(\xi_1, \xi_2)$ replaced by the coordinates $(\theta_1, \theta_2)$ of $\T^2$. It is clear that this groupoid is not proper. 

 Any other s-connected symplectic integration is obtained as a quotient modulo a discrete normal subgroupoid $\H$ of the isotropy bundle,
which is Lagrangian as a submanifold of $(\Sigma(M, \pi),\Omega)$. Now observe that the isotropy bundle is:
\[ \{(x,y,z,\theta_1,\theta_2,t)\in H_3\times M: yt\in\Z, xt\in \Z,z\in\R\}\cong (\Z^2\times \R)\times M, \]
so any local (bi)section of the isotropy bundle is of the form:
\[ (\theta_1,\theta_2,t)\mapsto \left(\frac{m}{t},\frac{n}{t},r(\theta_1,\theta_2,t),\theta_1,\theta_2,t\right), \]
where $m,n\in\Z$ and $r(\theta_1,\theta_2,t)$ is smooth on $M$. The section is Lagrangian iff:
\[ r(\theta_1,\theta_2,t)=\frac{1}{t^2}\left(m\theta_1+n\theta_2\right)+r_0(t).\]
In particular, it follows that the co-compact lattices
\[\H_{(\theta_1,\theta_2,t)}=\left\{\left(\frac{m}{t},\frac{n}{t},\frac{\theta_1n+\theta_2m+p}{t^2}\right),\,n,m,p\in \Z\right\}\]
fit into a subgroupoid so that $\cG:=\Sigma(M,\pi)/\H$ is s-proper and 
 $\Omega$ descends to $\G$. This gives an explicit s-connected, s-proper integration of $(M,\pi)$ with 
 $\Lambda_\cG=\cN_{\hol}$.

\subsubsection{The free Hamiltonian $\cT$-spaces perspective} A fundamental tool to construct Poisson structures is by Hamiltonian reduction of symplectic manifolds. This is recalled in Appendix \ref{appendix:moment:maps}, in the general context of symplectic groupoids. 
Integral affine manifolds $(B,  \Lambda)$ provide the simplest examples of symplectic groupoids: the symplectic torus bundle $\cT_{\Lambda}$. These give rise to the simpler theory of Hamiltonian $\cT_{\Lambda}$-spaces (see 
\ref{ssec:q-Hamiltonian spaces}) where one can take Corollary \ref{q-hamil-def-reform} as definition. For a free Hamiltonian $\cT_{\Lambda}$-space
$q: (X, \Oga)\to B$ with connected fibers, the reduced Poisson manifold $X_{\red}= X/\cT_{\Lambda}$ 
is of proper type, it will have $B$ as (smooth) leaf space and the induced integral structure is precisely
$\Lambda$ (cf. Corollary \ref{reduction-free-qHamilt}). Looking at the explicit s-connected integrating groupoid $\Gauge{\cT_\Lambda}{X}$ given by  
(\ref{cG-red2}), we see that:
\begin{enumerate}[(a)]
\item it is s-proper iff $q:X\to B$ is proper;
\item it is compact iff $X$ is compact;
\item it is the canonical integration iff  the $q$-fibers are 1-connected.
\end{enumerate}
Already when $B$ is 1-dimensional, with the standard integral affine structure, produces interesting examples.
For instance, the strong compact type example from \cite{Mar} arises via this procedure with $B=\S^1$. For $B=\R$, using 
maps $q$ with fibers which are not 1-connected may produce examples which are not strong proper but are proper (or $s$-proper if $q$ is proper).
The previous example fits into this scheme, but here is a slightly more general class of examples. 

Start with a symplectic manifold $(S_0,\omega_0)$ and consider the regular Poisson manifold $M= S_{0}\times\R^+$,
whose symplectic leaves are $(S_{0}\times \{t\}, t\omega_0)$, 
Using the geometric interpretation of the monodromy, we find that 
\[ \cN_{\mon}= \Per_{S^2}(\omega_0):=\left\{[\sigma]\in \pi_2(S_0,x_0): \int_\sigma \omega_0\right\}\subset\R, \]
the group of spherical periods of $\omega_0$. Since the holonomy of the foliation is trivial $\cN_{\hol}$ is computed similarly and gives the full group of 
periods $\Per(\omega_0)$. Hence, while the discreteness of $\Per_{S^2}(\omega_0)$ is the obstruction to integrability, the discreteness of $\Per(\omega_0)$ arises as an obstruction to properness. 
This time, this is the only obstruction. This can be seen by producing explicit proper integrations and that is done by realizing $M$ via reduction. Let us assume  
 $\Per(\omega_0)= \Z$, so that one can find a principal $\S^1$-bundle $p: P\to S_0$ whose Chern class is $[\omega_0]$. That means that
we find a connection 1-form $\theta\in \Omega^1(P)$ with $d\theta= p^*\omega_0$. The symplectization of $(P, \theta)$:
\[ X= P\times \R^{+}, \quad \Oga= \d(t\theta),\]
with $\S^1$-acting on the first coordinate and $q:(X, \Oga) \to \R$ the projection, is a Hamiltonian $\cT_{\Z}$-space. Its Poisson reduced space is $X_{\red}= M$ so, by the
general discussion, $M$ is of proper type (and of s-proper type iff $S_0$ is compact).

Here are some concrete examples, with various behavior of $\cN_{\mon}$ and $\cN_{\hol}$:
\begin{enumerate}[(a)]
\item if $(S_0,\omega_0)=(\S^2, a\,\omega_{\S^2})$, then $\cN_\mon=\cN_\hol=a\cdot \Z$;
\item if $(S_0,\omega_0)=(\Sigma_g, b\,\omega_{\Sigma_g})$, then $\cN_\mon=0$ and $\cN_\hol=b\cdot \Z$;
\item if $(S_0,\omega_0)=(\S^2\times\Sigma_g, a\,\omega_{\S^2}\oplus b\,\omega_{\Sigma_g})$, then $\cN_\mon=a\cdot \Z$ and $\cN_\hol=a\cdot \Z+ b\cdot \Z$;
\end{enumerate}
where $a, b\in \R\setminus \{0\}$, and $\omega_{\S^2}$, $\omega_{\Sigma_g}$, are normalized area forms on the sphere and on the closed surface of genus $g>0$. Cases (b) and (c) come
with infinite fundamental groups, hence they produce examples which are not strong proper. 
By the previous discussion,  they are s-proper except when $a/b\notin \Q$. In all the cases one can proceed as in the previous example ($g=1$),
and construct explicit s-connected, s-proper integrations, this time using the Lie theory of $\mathrm{SO}(3)$ (if $g=0$) or of $\mathrm{SL}(2)$ (if $g>1$).

%%%%%%%%%%%%%%%%%%%%%%%%
%%%%%%%%%%%%%%%%%%%%%%%%
%%%%%%%%%%%%%%%%%%%%%%%%
%%%%%%%%%%%%%%%%%%%%%%%%
%%%%%%%%%%%%%%%%%%%%%%%%
%%%%%%%%%%%%%%%%%%%%%%%%
%%%%%%%%%%%%%%%%%%%%%%%%
\section{The linear variation theorem I: 1-connected leaves}
%\section{\underline{The Duistermaat-Heckman linear variation theorem: 1-connected leaves}}
% \section{\underline{Linear variation of the symplectic forms I: 1-connected leaves}}
\label{sec:lin-var-1}
%%%%%%%%%%%%%%%%%%%%%%%%
%%%%%%%%%%%%%%%%%%%%%%%%
%%%%%%%%%%%%%%%%%%%%%%%%
%%%%%%%%%%%%%%%%%%%%%%%%
%%%%%%%%%%%%%%%%%%%%%%%%
%%%%%%%%%%%%%%%%%%%%%%%%
%%%%%%%%%%%%%%%%%%%%%%%%

{
%%%%%%%%%%%%%%%%%%%%%%%%%%%%%%
%%%%%%%%%%%%%%%%%%%%%%%%%%%%%%
%%%%%%%%%%%%%%%%%%%%%%%%%%%%%%
%%%%%%%%%%%%%%%%%%%%%%%%%%%%%%
%%%%%%%%%%%%%%%%%%%%%%%%%%%%%%
%%%%%%%%%%%%%%%%%%%%%%%%%%%%%%
\subsection{The classical  Duistermaat-Heckman Theorem}
\label{sec:DH:classical}
%%%%%%%%%%%%%%%%%%%%%%%%%%%%%%
%%%%%%%%%%%%%%%%%%%%%%%%%%%%%%
%%%%%%%%%%%%%%%%%%%%%%%%%%%%%%
%%%%%%%%%%%%%%%%%%%%%%%%%%%%%%
%%%%%%%%%%%%%%%%%%%%%%%%%%%%%%
%%%%%%%%%%%%%%%%%%%%%%%%%%%%%%

The linear variation in the titles of this section and the next one refers to a \emph{fundamental result} concerning PMCTs,
which is a \emph{generalization of the classical Duistermaat-Heckman Theorem}  \cite{DuHe} on the variation of the cohomology class of the symplectic form of symplectic reduced spaces. Let us recall this result in its simplest form.

Let a torus $\T$ act freely on a symplectic manifold $(S,\omega)$ in a Hamiltonian fashion with moment map $\mu:S\to \tt^*$. Then the symplectic reduced spaces $S_{\xi}=\mu^{-1}(\xi)/\T$ are all smooth symplectic manifolds with reduced symplectic form $\omega_\xi$. They are also diffeomorphic because one has a local model for $S$ around $\mu^{-1}(\xi_0)$ for any value $\xi_0\in\tt^*$ obtained as follows. Choose a connection 1-form $\alpha$ on the principal $\T$- bundle $q:\mu^{-1}(\xi_0)\to S_{\xi_0}$, then a local model for $S$ around $\mu^{-1}(\xi_0)$ is given by the product $\mu^{-1}(\xi_0)\times\tt^*$ furnished with the symplectic form:
\[ \pr_1^*q^*\omega_{\xi_0}+\d\langle \pr_2,\alpha\rangle, \]
where $\omega_{\xi_0}$ is the reduced symplectic form in $S_{\xi_0}$. The group $\T^n$ acts on $\mu^{-1}(\xi_0)\times\tt^*$ by acting on the first factor, and the action is Hamiltonian with moment map the second projection: $\mu=\pr_2:\mu^{-1}(\xi_0)\times\tt^*\to \tt^*$.

This local normal form leads to an identification of the symplectic reduced spaces $S_{\xi}\simeq S_{\xi_0}$, for $\xi$ close to $\xi_0$. Under this identification, the symplectic forms are linearly related:
\[ \omega_{\xi}=\omega_{\xi_0}+\langle F,\xi-\xi_0\rangle, \]
where $F\in\Omega^2(S_{\xi_0},\tt)$ is the curvature 2-form of the connection $\alpha$: 
\[ q^*F=\d\alpha. \] 

This identification of the symplectic reduced spaces depends on choices. However, any two identifications are related by an isotopy of $S_{\xi_0}$, so one can compare the cohomology classes of the symplectic forms, and this leads to:

\begin{theorem}[Duistermaat-Heckman  \cite{DuHe}]
\label{thm:DH:classic}
If a torus $\T$ acts freely on a symplectic manifold $(S,\omega)$ in a Hamiltonian fashion with proper moment map $\mu:S\to \tt^*$ then the cohomology class of the reduced symplectic form varies linearly:
\[ [\omega_{\xi}]=[\omega_{\xi_0}]+\langle c,\xi-\xi_0\rangle, \]
where $c\in H^2(S_{\xi_0})\otimes \tt$ is the Chern class of the principal $\T$-bundle $q:\mu^{-1}(\xi_0)\to S_{\xi_0}$.
\end{theorem}

From this result it follows also an important property of the measures or volume forms associated with the symplectic forms. In order to state it, consider the following volume forms:
\begin{itemize}
\item $\mu_\omega:=\frac{\omega^n}{n!}$, the Liouville form on $S$ ($2n=\dim S$);
\item $\mu_{\DH}^\omega:=\mu_*(\mu_\omega)$, the push-forward measure on $\tt^*$;
\item $\mu_{\Aff}$, the Lebesgue measure on $\tt^*$.
\end{itemize}
Then, we have the following corollary of the theorem above:

\begin{corollary}[\cite{DuHe}]
\label{cor:DH:classic}
For a free Hamiltonian $\T$-space $(S,\omega,\mu)$ with proper moment map, the Duistermaat-Heckman measure $\mu_{\DH}^\omega$ and the Lebesgue measure $\mu_{\Aff}$ are related by:
\[ \mu_{\DH}^\omega=\vol\cdot \mu_{\Aff}, \]
where $\vol:\tt^*\to \R$ is the function which associates to $\xi\in \tt^*$ the symplectic volume $\vol(S_{\xi})$ of the reduced symplectic space. Moreover, this function is a polynomial of degree at most $\frac{1}{2}\dim S_{\xi}=\frac{1}{2}\dim S-\dim \T$.
\end{corollary}

Notice that these results are really about the symplectic (or Poisson) geometry of the Poisson manifold $M=S/\T$, which has symplectic leaves the reduced spaces $S_\xi$ and leaf space the open subset $\mu(S)\subset \tt^*$. In the next sections we will provide remarkable generalizations of these results, valid for any PMCT. Our formulation of these results is made in terms of the \emph{developing map} associated with the integral affine structure on the leaf space, to be studied in this section. Our approach does not rely on a local normal form and hence gives the classical results above an entirely new perspective.
\smallskip 

Throughout this section we fix a Poisson manifold $(M, \pi)$ of s-proper type and an s-connected, s-proper integration $(\G,\Omega)\tto (M,\pi)$. Moreover:
\begin{itemize}
\item {\bf Standing assumption in this section:} The symplectic leaves of  $(M, \pi)$ are 1-connected.
\end{itemize}
 This assumption will be dropped in the next section, where we will consider the general case. The advantage of dealing first 
with 1-connected symplectic leaves is that there are no subtleties arising from the geometry of the leaf space: in this case
$B= M/\cF_{\pi}$ is smooth and there is only one  groupoid integrating $\cF_\pi$, namely the equivalence relation $M\times_B M\subset M\times M$
associated with the submersion $p:M\to B$. Therefore we do not have to worry about the foliation groupoid $\cBG(\cG)$ (cf. Theorem \ref{thm-reg-fol2}) or, equivalently, with the orbifold structure on $B$, which may be present even when $B$ is smooth but the leaves are not 1-connected. 

}

%%%%%%%%%%%%%%%%%%%%%%%%%%%%%%
%%%%%%%%%%%%%%%%%%%%%%%%%%%%%%
%%%%%%%%%%%%%%%%%%%%%%%%%%%%%%
%%%%%%%%%%%%%%%%%%%%%%%%%%%%%%
%%%%%%%%%%%%%%%%%%%%%%%%%%%%%%
%%%%%%%%%%%%%%%%%%%%%%%%%%%%%%
\subsection{The developing map for integral affine structures}
\label{sec:developing:map}
%%%%%%%%%%%%%%%%%%%%%%%%%%%%%%
%%%%%%%%%%%%%%%%%%%%%%%%%%%%%%
%%%%%%%%%%%%%%%%%%%%%%%%%%%%%%
%%%%%%%%%%%%%%%%%%%%%%%%%%%%%%
%%%%%%%%%%%%%%%%%%%%%%%%%%%%%%
%%%%%%%%%%%%%%%%%%%%%%%%%%%%%%
 The linear variation theorem, in one form or another, involves a basic concept of integral affine geometry which
we have not discussed so far. This is done in this section, where we fix an integral affine manifold $(B, \Lambda)$ and we discuss its {\bf developing map} \cite{GHL,Thur}. This is a  local diffeomorphism of integral affine manifolds 
\[ \dev:(\widetilde{B},\widetilde{\Lambda})\to(\R^q,\Z^q)\]
defined on the universal cover $\widetilde{B}$ endowed with the pull-back $\widetilde{\Lambda}$ of $\Lambda$. Let us first recall the standard definition:
\begin{quote}
{\it Fix a point $b_0\in B$ and an integral affine chart $(U_0,\chi_0)$ centered at $b_0$.
For any path $\gamma$ starting at $b_0$, cover it by a finite number of integral affine charts $\chi_i: U_i \rightarrow \chi_i(U_i)\subset \R^q$, $0\leq i \leq r$. Arrange the coordinates charts inductively so that each two consecutive ones match on the intersection (this can be done since the changes or coordinates, being affine, are defined on the entire $\R^q$).  Then $\dev([\gamma])$ is the image of $\gamma(1)$ by the last coordinate chart.}
\end{quote}
If one restricts to loops $\gamma$ and considers the entire change of coordinates between the first and the last chart, one obtains the {\bf integral affine holonomy representation} $h^{\Aff}:\pi_1(B)\to \Aff_\Z(\R^q)$, where $\pi_1(B)= \pi_1(B, b_0)$. Of course, 
\[ h^{\Aff}= (\dev(\gamma), h^{\lin}(\gamma))\]
where $h^{\lin}:\pi_1(B)\to \GL_\Z(\R^q)$ is the {\bf linear holonomy representation} of $(B, \Lambda)$, i.e. the linear holonomy of the flat connection on $B$ induced by $\Lambda$. These representations give rise to the {\bf affine holonomy group} $\Gamma^{\Aff}:=h^{\Aff}(\pi_1(B))\subset \Aff_\Z(\R^q)$ and similarly the {\bf linear holonomy group} $\Gamma^{\lin}$. Note that $\Gamma^{\Aff}$ is an integral affine group as in Example \ref{ex-gen-IAS}, whose linear part is $\Gamma^{\lin}$. The factorization
of the linear holonomy representation:
\[ \xymatrix{ \pi_1(B)\ar[r]^{h^{\Aff}} & \Gamma^{\Aff} \ar[r]^{\pr_{2}} & \Gamma^{\lin}} \]
gives rise to a sequence of covering spaces by integral affine manifolds:
\[ \xymatrix{ \widetilde{B}\ar[r] & B^{\Aff} \ar[r] & B^{\lin} \ar[r] &B} \]
where the middle and the last spaces are called the  {\bf affine and linear holonomy covers}, respectively. They are the smallest
 covers with trivial affine and  linear holonomy, respectively. For instance, in the situation of Example \ref{ex-gen-IAS}, $\widetilde{B}=  B^{\Aff}= \R^q$  and $B^{\lin}= \R^q/\Gamma^{\mathrm{tr}}$. 
 
%%%MARIUS: I REMOVE THIS COMMENT BECAUSE IT SHOWS UP AS A REMARK TWO PAGES LATER
% The image of $\dev:\widetilde{B}\to \R^q$ is an open subset which is invariant under the affine action of $\Gamma^{\Aff}$. In
% fact,  $(B, \Lambda)$ is complete
% {
% (see Conjecture \ref{Markus-conj})
% } %\end{blue}
% exactly when $\dev:\widetilde{B}\to \R^q$ is a diffeomorphism. Completeness can also be characterized infinitesimally as geodesic completeness of the associated flat connection.

The standard definitions given above for the developing map and the affine holonomy have a drawback: they both depend on a choice of a base point $b_0\in B$ and an integral affine chart around $b_0$. However, it is possible to give a more intrinsic definition, in the spirit of the present work, using the language of groupoids, as we now explain. This approach will turn out to be very useful in the sequel. We denote by $\GL_{\Lambda^\vee}(TB)\tto B$ (respectively, $\Aff_{\Lambda^\vee}(TB)\tto B$) the Lie groupoid whose arrows are the integral linear (respectively, integral affine) isomorphisms between the fibers of $TB$. Our convention is that an arrow $\phi:T_x B \to T_yB$ has source $y$ and target $x$. Notice that:
\begin{enumerate}[(i)]
\item Parallel transport for the canonical flat connection $\nabla$ of $(B,\Lambda)$ defines the linear holonomy, which can be seen as morphism of Lie groupoids: 
\begin{equation}
\label{eq-the-linear-action} 
h^{\lin}: \Pi_1(B)\rightarrow \GL_{\Lambda^\vee}(TB),\quad h^{\lin}([\gamma]): T_{\gamma(1)}B\to T_{\gamma(0)}B.
\end{equation}
%With our conventions, $h^{\lin}$ associates to a path homotopy class $[\gamma]$ a map
%\[ h^{\lin}([\gamma]): T_{\gamma(1)}B\to T_{\gamma(0)}B.\]
\item The connection $\nabla$ is torsion free and this can be interpreted as saying that the identity map
\[ \textrm{Id}: TB\to TB\]
is a 1-cocycle on the Lie algebroid $TB$ with coefficients in the representation $TB$. Hence, it integrates to a groupoid 1-cocycle in $\Pi_1(B)$ with values in $TB$:
% \begin{equation}
% \label{eq-the-developing-cocycle} 
\[
\dev: \Pi_1(B) \to TB, \ \ [\gamma]\mapsto \dev([\gamma])\in T_{\gamma(0)}B.
\]
% \end{equation}
The general formula for integrating algebroid 1-cocycles gives the expression:
\[ \dev([\gamma])= \int_{0}^{1} h^{\lin}(\gamma_{\epsilon})(\dot{\gamma}(\epsilon)) \d\epsilon, \]
and the cocycle condition means that for any two composable arrows in $\Pi_1(B)$:
\[ \dev([\gamma] \circ [\tau])= \dev([\tau])+ h^{\lin}([\tau])(\dev([\gamma])).\]
\end{enumerate}

These two pieces of structure can be organized together into an integral affine % (right) 
action of $\Pi_1(B)$ on $(TB,\Lambda^\vee)$: any $[\gamma]\in \Pi_1(B)$ induces an affine transformation
\[
T_{\gamma(1)}B \rightarrow T_{\gamma(0)}B, \ v\mapsto v\cdot [\gamma]= \dev([\gamma])+ h^{\lin}([\gamma])(v).
\]
Hence, one obtains a morphism of Lie groupoids 
\begin{equation}
\label{eq-the-affine-action} 
h^{\Aff}: \Pi_1(B)\rightarrow \Aff_{\Lambda^\vee}(TB).
\end{equation}

In order to recover the classical/based affine holonomy representation and developing map, one restricts 
to the isotropy group of $\Pi_1(B)$ at $b_0$, which is $\pi_1(B, b_0)$, and to the s-fiber above $b_0$, which is the model for the universal cover using paths starting at $b_0$:
$\widetilde{B}= s^{-1}(b_0).$ One obtains the linear and affine representations at $b_0$, 
\[ h^{\lin}|_{b_0}: \pi_1(B, b_0) \to\GL_{\Lambda_{b}^\vee}(T_{b_0}B) ,\quad h^{\Aff}|_{b_0}: \pi_1(B, b_0)\to \Aff_{\Lambda_{b_0}^\vee}(T_{b_0}B),\]
and the developing map at $b_0$, 
\[ \dev|_{b_0}:\widetilde{B}\to T_{b_0}B. \]
Finally, a choice of a basis $\mathfrak{b}_\Lambda$ for $\Lambda_{b_0}$ (which is equivalent to a choice of an integral affine chart centered at $b_0$) leads to an identification $(T_{b_0}B, \Lambda_{b_0}^\vee)\cong (\R^q,\Z^q)$, and we recover the original notions. Note also that, since $\Pi_1(B)$ is transitive, no information is lost by restricting at $b_0$.

% In order to recover the classical/based affine holonomy representation and developing map, one observes that the groupoid $\Pi_1(B)$ is transitive. Hence, $\Pi_1(B)$ is Morita equivalent to its isotropy group at $b$, i.e., to $\pi_1(B, b)$. The Morita equivalence is implemented (as for any transitive groupoid) by the s-fiber of $\Pi_1(B)$,  which is the model for the universal cover using paths starting at $b$: $\widetilde{B}= s^{-1}(b).$  As usual, the groupoid $\Pi_1(B)$ can be recovered from $\widetilde{B}$ and $\pi_1(B, b)$ as the quotient $\widetilde{B}\times_{\pi_1(B,b)}\widetilde{B}$. Hence, by restricting the groupoid actions  $h^{\lin}: \Pi_1(B)\rightarrow \GL_{\Lambda^\vee}(TB)$ and $h^{\Aff}: \Pi_1(B)\rightarrow \Aff_{\Lambda^\vee}(TB)$, as well as the 1-cocyle $\dev: \Pi_1(B) \to TB$, we obtain linear and affine representations \[ h^{\lin}|_b: \pi_1(B, b) \to\GL_{\Lambda_{b}^\vee}(T_{b}B) ,\quad h^{\Aff}|_b: \pi_1(B, b)\to \Aff_{\Lambda_{b}^\vee}(T_{b}B),\] and a map: \[ \dev|_b:\widetilde{B}\to T_bB. \] Finally, a choice of a basis $\mathfrak{b}_\Lambda$ for $\Lambda_b$ (which is equivalent to a choice of an integral affine chart centered at $b$) leads to an identification $(T_{b}B, \Lambda_{b}^\vee)\cong (\R^q,\Z^q)$, and we recover the original notion.

The description of $\dev:\widetilde{B}\to \R^q$ as a 1-cocycle for $\pi_1(B,b_0)$
with values in $\R^q$ appears in the work of Matsusima \cite{Mt}, who attributes the idea to Koszul.

\begin{remark}\label{rk-Markus-again} The image $\Omega\subset \R^q$ of $\dev:\widetilde{B}\to \R^q$ is an open subset which is invariant under the affine action of $\Gamma^{\Aff}$. When the action is free and proper, the induced map $B\to \Omega/\Gamma^{\Aff}$ will be 
a local diffeomorphism between integral affine manifolds. The conclusion of the Markus conjecture (Conjecture \ref{Markus-conj})
is equivalent to saying that $\Omega= \R^q$ and that the last map is a diffeomorphism. In turn, this is also equivalent to the condition that 
the linear connection is geodesically complete. 
\end{remark}

%%%%%%%%%%%%%%%%%%%%%%%%%%%%%%
%%%%%%%%%%%%%%%%%%%%%%%%%%%%%%
%%%%%%%%%%%%%%%%%%%%%%%%%%%%%%
%%%%%%%%%%%%%%%%%%%%%%%%%%%%%%
%%%%%%%%%%%%%%%%%%%%%%%%%%%%%%
%%%%%%%%%%%%%%%%%%%%%%%%%%%%%%
\subsection{The linear variation}
\label{sec:linear-var} 
%%%%%%%%%%%%%%%%%%%%%%%%%%%%%%
%%%%%%%%%%%%%%%%%%%%%%%%%%%%%%
%%%%%%%%%%%%%%%%%%%%%%%%%%%%%%
%%%%%%%%%%%%%%%%%%%%%%%%%%%%%%
%%%%%%%%%%%%%%%%%%%%%%%%%%%%%%
%%%%%%%%%%%%%%%%%%%%%%%%%%%%%%
We now return to Poisson geometry and define the linear variation of the leafwise symplectic cohomology class.

Under our standing assumption, the transverse integral affine structure 
$\Lambda_{\cG}\subset\nu^*(\cF_\pi)$ defined by $\G$ (see Theorem \ref{thm-lattice-proper-case}) is the pull-back of an integral 
affine structure on the manifold $B= M/\cF_{\pi}$. In this section we will only use the structure on $B$, for which
we use the same notation $\Lambda_{\cG}$.
% which will be also denoted by $\Lambda_{\cG}$. 

On the other hand, we can define a vector bundle $\cH\to B$ whose fibers are the degree 2 cohomology of the symplectic leaves:
\[ \cH_b:= H^{2}(S_{b})\]
There are two things to notice about the vector bundle $\cH\to B$:
\begin{enumerate}[(i)]
\item the integral cohomology yields a structure of an integral vector bundle $(\cH,\cH_\Z)$;
\item the leafwise symplectic form yields a canonical section $\varpi\in \Gamma(B, \cH)$:
\begin{equation}\label{omega-as-section} 
b\mapsto \varpi_b:=[\omega_{p^{-1}(b)}]\in H^2(S_{b}) .
\end{equation}
\end{enumerate}

There is a rich interplay between the integral affine structure $\Lambda_\cG$ on $B$ and the integral
vector bundle $(\cH,\cH_\Z)$. To make this precise note that, by (i) above, the bundle $\cH\to B$ has a canonical 
flat connection $\nabla$, the so-called {\bf Gauss-Manin connection}. Our first formulation of the variation of the symplectic form is:

\begin{definition}
The {\bf linear variation of $\varpi$} is the bundle map:
\[ \Ilin: TB \to \cH, \ v\mapsto  \Ilin(v):= \nabla_{v}\varpi.\]
 Its image is called the {\bf linear variation bundle of $\varpi$}, denoted:
\[ \Vspace^{\lin}:= \Ilin(TB)\subset \HH, \]
and it has an {\bf integral part}, denoted:
\[ \VspaceZ^{\lin}:= \Ilin(\Lambda^{\vee}_{\cG}), \] 
where $\Lambda^{\vee}_{\cG}\subset TB$ is the lattice induced by $(\cG,\Omega)$.
\end{definition}

Note that the linear variation $\varpi$, as well as its image $\Vspace^{\lin}$, does not depend on the specific integration $(\cG,\Omega)$, while the integral part $\VspaceZ^{\lin}$ does. 

We will see that we have always $\VspaceZ^{\lin}\subset \HH_{\Z}$.  A key ingredient in the proof is the following: since $\cF_\pi$ is regular, we can choose a splitting for all leaves at the same time, i.e., a bundle map $\tau:\cF_\pi\to T^*M$ which is a splitting of $\pi^\sharp:T^*M\to \cF_\pi$. The curvature $\Omega_\tau$ of this splitting gives a 1-form $[\Omega_\tau]\in\Omega^1(B;\cH)$ by setting:
\[ [\Omega_\tau]:TB\to\cH,\quad v_b\mapsto [\langle \Omega_\tau,v_b\rangle]\in H^2(S_b). \]
 Here we use the identification $T_{b}B\cong \nu_{x}(\cF_\pi)$, so $v_b$ can be thought of 
as a constant section of $\nu(S_b)$. Note that $[\Omega_\tau]$ is independent of the choice of splitting $\tau$ and: % we have:

\begin{proposition}
\label{lin-var-as-var-sympl-areas}
If $(M, \pi)$ is regular with compact, 1-connected, leaves, then:
\begin{equation}
\label{eq:linear:variation}
\nabla \varpi=[\Omega_\tau]. 
\end{equation}
In particular, if $v\in T_{b}B$ and $\sigma: (\S^2,p_N)\to (S_b,x)$, then:
\begin{equation}
\label{eq:int:linear:variation}
\int_{\sigma} \Ilin(v)= \langle {\partial_\mon}([\sigma]),v\rangle.
\end{equation}
\end{proposition}

\begin{proof} 
Clearly, the integral formula \eqref{eq:int:linear:variation} follows from \eqref{eq:linear:variation} and the definition of the monodromy in terms of the curvature.

To prove \eqref{eq:linear:variation}, let us a chose a distribution
$D\subset TM$ complementary to $\cF_\pi$, so that $D\cong \nu(\cF_\pi)$. 
This gives rise to a unique extension $\widetilde{\omega}\in\Omega^2(M)$ of the leafwise symplectic form satisfying $i_V\widetilde{\omega}=0$, for any $V\in D$. This extension, in turn, gives rise to a splitting $\tau: \cF_\pi\to T^*M$, $X\mapsto i_X\widetilde{\omega}$, with a curvature 2-form $\Omega_\tau$. We claim that if $V\in\X(D)$ is any vector field defined in a neighborhood of $S_b$ whose restriction to $S_b$ projects to $v_b\in TB$, then:
\begin{enumerate}[(a)]
\item for any section $\bar\eta\in\Gamma(\cH)$ represented by $\eta\in\Omega^2(M)$ with $i_V\eta=0$ one has:
\[ \nabla_v \bar\eta=[(\Lie_V\eta)|_{S_b}]; \]
% where $\eta\in\Omega^2(M)$ is a 2-form such that $i_V\eta=0$ and $\bar\eta=[\eta]$.
\item $\Lie_V\widetilde{\omega}=\Omega_\tau(V)$.
\end{enumerate}
These will imply \eqref{eq:linear:variation}.
% \newpage

Item (a) follows from the fact that the Gauss-Manin connection can be defined by lifting a vector field $X$ on $B$ to $M$ via a distribution $D$, since the flow of the horizontal lift gives a 1-parameter group of diffeomorphisms of the fibers, that preserves the integral cohomology. 

In order to prove item (b), we see that, for any $X,Y\in \cF_\pi$, the definition of the curvature 2-form gives:
\begin{align*}
\Omega_\tau(X,Y)(V)
&=\langle [\tau(X),\tau(Y)]_\pi-\tau([X,Y]),V\rangle \\
&=\langle [i_X\widetilde{\omega},i_Y\widetilde{\omega}]_\pi,V\rangle \\
&=-(\Lie_V\pi)(i_X\widetilde{\omega},i_Y\widetilde{\omega})
=(\Lie_V\widetilde{\omega})(X,Y),
\end{align*}
which shows that (b) holds. Here, we have used first that
\[ i_V[\a,\b]_\pi=i_{\pi^\sharp(\a)}\d i_V\b-i_{\pi^\sharp(\b)}\d i_V\a-(\Lie_V\pi)(\a,\b), \quad (\a,\b\in\Omega^2(M))\] 
together with $i_V\widetilde{\omega}=0$ and $\pi(i_X\widetilde{\omega},i_Y\widetilde{\omega})=-\widetilde{\omega}(X,Y)$, which yields:
\[ (\Lie_V\pi)(i_X\widetilde{\omega},i_Y\widetilde{\omega})=-(\Lie_V\widetilde{\omega})(X,Y). \]
\end{proof}

In what follows we make use of the following terminology related to integral vector bundles $(E, E_{\Z})$: a {\bf weak integral sub-bundle} of $E$ is an integral vector bundle $(F, F_{\Z})$ for which $F$ is a vector sub-bundle and $F_{\Z}\subset E_{\Z}$; it is called an  {\bf integral vector sub-bundle} if $F_{\Z}= F\cap E_{\Z}$. 

\begin{proposition}
\label{I-lin-iso} 
For any Poisson manifold $(M, \pi)$ with 1-connected leaves and an $s$-connected,  $s$-proper integration $\cG$, $\Ilin$  is a $\Pi_1(B)$-equivariant morphism of integral vector bundles,
\[ \Ilin: (TB, \Lambda^{\vee}_{\cG})\to (\HH, \HH_{\Z}).\]
Moreover, $(\Vspace^{\lin}, \VspaceZ^{\lin})$ is a weak integral affine vector sub-bundle of $(\HH, \HH_{\Z})$. In the strong s-proper case it is an integral affine sub-bundle.
\end{proposition}

\begin{proof}
We first show that $\Ilin(\Lambda^{\vee}_{\cG})\subset \HH_{\Z}$. For this, fix $b\in B$, let $v\in T_bB$ and choose $x\in p^{-1}(b)$. Using that ${\cN_\mon}\subset \Lambda_\cG$ and Proposition \ref{lin-var-as-var-sympl-areas}, we see that: 
\begin{align*}
v\in\Lambda^{\vee}_{\cG}\subset T_{p(x)}B \quad &\Longleftrightarrow \quad \lambda(v)\in \Z, \  \forall \ \lambda\in \Lambda_{\cG}\\
& \, \Longrightarrow \quad \partial_x(\sigma)(v)\in \Z ,\ \forall\ \sigma\in \pi_2(S,x) \\ 
&  \Longleftrightarrow  \quad \int_{\sigma}\Ilin(v)\in \Z ,\ \forall\ \sigma\in \pi_2(S,x) \\ 
& \Longleftrightarrow \quad \Ilin(v)\in H^2(S,\Z)
\end{align*}
where, for the last implication, we used that $S$ is simply connected.
This proves that $\Ilin(\Lambda^{\vee}_{\cG})\subset \HH_{\Z}$. Since the actions of $\Pi_1(B)$ on $TB$ and $\cH$ are by parallel transport relative to the flat connections 
determined by $\Lambda_{\cG}$ and $\cH_\Z$, we also obtain that $\Ilin$ is $\Pi_1(B)$-equivariant.

Next we prove that $\VspaceZ^{\lin}$ is a lattice in $\Vspace^{\lin}$. It is discrete since it sits inside $\HH_{\Z}$, hence it suffices to remark that $\Vspace^{\lin}/\VspaceZ^{\lin}$ is compact. But this follows from the fact that $\Lambda_{\cG}^{\vee}$ is a lattice in $TB$ and we have a surjective map:
\[ \Ilin: TB/\Lambda_{\cG}^{\vee}\to  \Vspace^{\lin}/\VspaceZ^{\lin}.\]

Finally, in the strong s-proper case, the only implication above becomes an equivalence, and we obtain that $\VspaceZ^{\lin}=\HH_{\Z}\cap \Vspace^{\lin}$.
\end{proof}

The previous proposition allows one to identify certain ``building blocks'' sitting inside $(M, \pi)$. The extreme cases follow easily from the previous two propositions:

\begin{corollary}
\label{cor-lin-iso-l}  
For any s-proper Poisson manifold $(M, \pi)$ with simply connected leaves one has:
\begin{enumerate}
\item[(i)] {\bf Zero-variation:} $\Ilin= 0$ iff $p:(M,\pi)\to B$ is a symplectic fibration;
\item[(ii)] {\bf Full-variation:} $\Ilin$ is injective iff $(M,\pi)$ is strong s-proper.
\end{enumerate}
\end{corollary}

\begin{proof}
By Proposition  \ref{lin-var-as-var-sympl-areas}, $\Ilin= 0$ is equivalent to the fact that all the classes $[\Omega_\tau|_{S_b}]$ vanish. These classes are the restrictions to the leaves of a global class $[\Omega_\tau]\in H^2(\cF_\pi, \nu^*)$, which can also be described more directly (see \cite{CF2}) by choosing an extension $\widetilde{\omega}\in \Omega^2(M)$ of the symplectic forms on the leaves, and then taking:
\[ \Omega_\tau(X,Y)(V)=\d_{\nu}\widetilde{\omega}(X, Y)(V)=\d\widetilde{\omega}(X, Y, V).\]
A spectral sequence argument and the fact that the leaves are 1-connected, implies that the vanishing of all the classes $[\Omega_\tau|_{S_b}]$ is equivalent to the existence of an extension $\widetilde{\omega}$ such that $\d_{\nu}\widetilde{\omega}= 0$. This last condition is a well-known cha\-rac\-te\-ri\-za\-tion of symplectic fibrations (see \cite{GLS}).

By Proposition \ref{lin-var-as-var-sympl-areas}, $\Ker(\Ilin)$ is the annihilator of ${\cN_\mon}$ and, by Proposition \ref{I-lin-iso}, $\im(\Ilin)$ is a discrete group. Hence the injectivity of $\Ilin$ is equivalent to ${\cN_\mon}$ being a lattice. By Theorem \ref{thm-lattice-strong-case}, this is equivalent to $(M,\pi)$ being of strong s-proper type. 
\end{proof}

When an s-proper Poisson manifold $(M, \pi)$ has full-variation, Proposition \ref{I-lin-iso} 
above shows that $\Ilin$ realizes $(TB,\Lambda^\vee)$ as an integral vector sub-bundle of $(\HH, \HH_{\Z})$ 
if and only if $\cG$ is the Weinstein groupoid of $(M, \pi)$.

% The linear variation map is in fact defined for any regular Poisson manifold $(M,\pi)$ 
% with compact 1-connected leaves,  with no reference to symplectic integrations.
% By Proposition \ref{lin-var-as-var-sympl-areas}, the condition $\Ilin= 0$ is equivalent to ${\cN_\mon}= 0$, so the latter is also a characterization
% of symplectic fibrations with compact simply connected leaves. In this case, by Corollary \ref{cor:no:monodromy:strong:foliation}, the
% s-proper symplectic integrations $(M,\pi)$ are in bijection with the integral affine structures on its base.

For the general case, we look at
\[ \cK:= \Ker(\Ilin) \subset TB \]
and this leads to a decomposition of $(M,\pi)$ into a foliation by Poisson submanifolds of zero-variation:

\begin{theorem}
\label{I-lin-iso-gen} 
For any s-proper Poisson manifold $(M, \pi)$ with simply connected leaves, $\cK$ defines an involutive distribution of constant rank.
If $(\cG,\Omega)$ is an $s$-connected, $s$-proper integration, then one has:
\begin{enumerate}[(i)]
\item The subgroup
\[ \{\xi\in T^*B: \xi(v)\in \Z\text{ for all }v\in (\Ilin)^{-1}(\cH_\Z)\}\subset T^*B\]
sits inside $\nu^*(\cK)$ and defines a transverse integral affine structure for $\cK$;   
\item Each leaf $K$ of $\cK$ is an integral affine submanifold of $B$ and the resulting Poisson submanifold 
\[ M_{K}:= p^{-1}(K) \subset M\]
has zero-variation. In particular, $p: M_{K}\to K$ is a symplectic fibration over the integral affine manifold $K$;  
\item For any transversal $T$ to $\cK$ of complementary dimension,  the resulting Poisson submanifold 
\[ M_{T}:= p^{-1}(T)\subset M\]
is a Poisson manifold of strong $s$-proper type.
\end{enumerate}
\end{theorem}

\begin{proof}
First of all, the fact that $\Ilin$ is $\Pi_1(B)$-equivariant implies that $\cK$ has constant rank. The Gauss-Manin connection being flat,
its curvature tensor vanishes, which obviously implies involutivity of $\cK$. 

Next, we turn to the proof of (i). We work at the level of $B$: under the isomorphism $\nu^{*}_{x}(\cF)\cong T^{*}_{p(x)}B$, the monodromy group of $(M, \pi)$ at any $x\in p^{-1}(b)$ becomes a subgroup ${\cN_\mon}|_{b}\subset T^*_b B$. Moreover, by Theorem \ref{thm-lattice-strong-case} (i), we have
\[ {\cN_\mon}\subset \Lambda_\cG \subset T^*B.\]
Since $\Lambda_{\cG}$ is a lattice, the fibers ${\cN_\mon}|_b$ will be discrete, hence also closed in $T^{*}_{b}B$. 

For a closed subgroup $\cC\subset V$ in a vector space, we set 
\[ \cC^{\vee}:=\{\xi\in V^*:\xi(C)\subset \Z\}. \]
Notice that $(\cC^{\vee})^{\vee}= \cC$. Also, to any such closed subgroup $\cC$ we associate two subspaces: $\Span(\cC)\subset V$ the span over $\R$, and $\Cospan(\cC)\subset V$ the cospan over $\R$, defined as the largest vector subspace of $V$ contained in $\cC$. Note that the span/cospan of $\cC^{\vee}$ coincides with the annihilator of the cospan/span of $\cC$. Moreover:
\begin{enumerate}[(a)]
\item if $\cC$ is discrete, then $\cC$ is a lattice in $\Span(\cC)$;
\item  if $\cC$ spans $V$, then $\cC/\Cospan(\cC)$ is a lattice in $V/\Cospan(\cC)$.
\end{enumerate}

Back to our situation, Proposition  \ref{lin-var-as-var-sympl-areas} shows that $\cK$ is the annihilator of ${\cN_\mon}$. Equivalently, $\Span(\cN_\mon)=\cK^0$ so  ${\cN_\mon}$ is a lattice in $\nu^*(\cK)=(\cK)^0$. Since ${\cN_\mon}\subset \Lambda_{\cG}$, sections of $\cN_\mon$ are necessarily closed forms, hence ${\cN_\mon}$ defines a transverse integral affine structure for $\cK$. In order to obtain the description of ${\cN_\mon}$ directly in terms of $\Ilin$, one notes that the sequence of implications in the proof of Proposition \ref{I-lin-iso} all become equivalences if we replace $\Lambda_{\cG}^{\vee}$ by $\cN_\mon^{\vee}$, so that:
\begin{equation}\label{N-via-Ilin} 
\cN_\mon^{\vee}= (\Ilin)^{-1}(\HH_{\Z}).
\end{equation}
Since ${\cN_\mon}= (\cN_\mon^{\vee})^{\vee}$ this proves the description in (i). 

While (i) is a statement about ${\cN_\mon}$, part (ii) of the proposition is about $\Lambda_{\cG}$ and its subtle interaction ${\cN_\mon}$. We need to prove that $\cK\cap \Lambda_{\cG}^{\vee}$ is a lattice in $\cK$. Since it is clearly discrete, it suffices to prove that the resulting quotient
\[  \cK/\cK\cap \Lambda_{\cG}^{\vee}\cong (\cK+ \Lambda_{\cG}^{\vee})/\Lambda_{\cG}^{\vee} \]
is compact. Under the previous identification, this quotient is the kernel of the map
\[ \cN_\mon^{\vee}/\Lambda_{\cG}^{\vee} \to \cN_\mon^{\vee}/ (\cK+ \Lambda_{\cG}^{\vee}).\]
Notice that the image of the last map is discrete, since it is a quotient of $\cN_\mon^{\vee}/\cK$, which is itself discrete by part (i). This implies that the kernel is closed in $\cN_\mon^{\vee}/\Lambda_{\cG}^{\vee}$, hence it suffices to remark that this last space is compact. Indeed, \eqref{N-via-Ilin} implies that  $\cN_\mon^{\vee}/\Lambda_{\cG}^{\vee}$ is the kernel of the map induced by $\Ilin$:
\[ {\Ilin}:\xymatrix{ TB/\Lambda_{\cG}^{\vee} \ar[r] & \Vspace^{\lin}/\VspaceZ^{\lin},}\]
and since $TB/\Lambda_{\cG}^{\vee}$ is compact, the result follows. To conclude the proof of (ii), notice that $M_{K}$ with the induced Poisson structure will have zero variation, since the associated normal bundles are precisely the kernel of $\Ilin$.

Finally, to prove (iii), the linear variation map for $M_{T}$ will be just the injective descent of $\Ilin$, defined on $TB/\cK$, modulo the obvious identifications:
\[ T_bB/\cK_b \cong T_b(T)\cong T^{*}_{x}M_T/\cF_x,\quad \text{with }b= p(x).\]
\end{proof}

\begin{example} For a Lie algebra $\gg$ of compact type, the dual $M=\gg^*$  is a Poisson manifold of proper type. In this case, we have the decomposition $\gg=\zz\oplus\gg_{\ss}$ into center and semisimple part, and $\zz=\{0\}$ if and only if $\gg^*$ is of strong-proper type. The passing from  $M$ to $M_T$ in Proposition \ref{I-lin-iso-gen} should be seen as a Poisson generalization of the passing from a compact Lie algebra to its semi-simple part. This example will be further discussed in Section \ref{ex-regular-coadjoint}.
\end{example}

%%%%%%%%%%%%%%%%%%%%%%%%%%%%%%%%%%%%%%%
%%%%%%%%%%%%%%%%%%%%%%%%%%%%%%%%%%%%%%%
%%%%%%%%%%%%%%%%%%%%%%%%%%%%%%%%%%%%%%%
%%%%%%%%%%%%%%%%%%%%%%%%%%%%%%%%%%%%%%%
%%%%%%%%%%%%%%%%%%%%%%%%%%%%%%%%%%%%%%%
%%%%%%%%%%%%%%%%%%%%%%%%%%%%%%%%%%%%%%%
\subsection{The linear variation theorem}
\label{sec:variation1} 
%%%%%%%%%%%%%%%%%%%%%%%%%%%%%%%%%%%%%%%
%%%%%%%%%%%%%%%%%%%%%%%%%%%%%%%%%%%%%%%
%%%%%%%%%%%%%%%%%%%%%%%%%%%%%%%%%%%%%%%
%%%%%%%%%%%%%%%%%%%%%%%%%%%%%%%%%%%%%%%
%%%%%%%%%%%%%%%%%%%%%%%%%%%%%%%%%%%%%%%
%%%%%%%%%%%%%%%%%%%%%%%%%%%%%%%%%%%%%%%

We now move to the study of the \emph{actual} variation of $\varpi$:

\begin{definition}
The {\bf variation of $\varpi$} is the bundle map:
% \begin{equation}
% \label{eq:Var}
\[
\Var: \Pi_1(B)\to \cH,\quad [\gamma]\mapsto \gamma^*\varpi_{\gamma(1)}, 
\]
% \end{equation}
where $\gamma^*$ stands for the parallel transport associated to the Gauss-Manin connection. The {\bf variation bundle} of $\varpi$
is the image of $\Var$:
 \[ \Vspace:= \im(\Var)\subset \HH \]
\end{definition} 

Our aim now is to show that the variation is linear. For this, as it will become apparent in the sequel, it is more natural to consider an \emph{affine} point of view. 
We define the {\bf affine variation} of $\varpi$ to be:
\[ \Iaff:= \varpi+ \Ilin: TB\to \HH, \]
and the {\bf affine variation bundle} to be its image:
\[ \Vspace^{\Aff}:= \Iaff(TB)= \varpi+ \Vspace^{\lin}.\]
From this point of view, the relevant action of $\Pi_1(B)$ on $TB$ is the one by affine transformations, as given by \eqref{eq-the-affine-action}. We will refer to it as the {\bf integral affine action} of $\Pi_1(B)$ on $TB$. On $\HH$, we will continue to use the linear action of $\Pi_1(B)$.

Our first version of the statement that the variation is linear or, more precisely, affine, is the following:

\begin{theorem}\label{thm-affine-iso}For any s-proper Poisson manifold $(M, \pi)$ with 1-connected leaves and an $s$-connected, $s$-proper
integration $(\cG,\Omega)$, the developing map $\dev$ of the integral affine structure on $B=M/\cF_\pi$ identifies the variation of $\varpi$ with its affine variation, i.e. one has a commutative diagram:
 \[ \xymatrix{ 
 \Pi_1(B) \ar[rr]^-{\Var} \ar[rd]_-{\dev} &    & \HH        \\
 & TB \ar[ru]_-{\Iaff}                                     &            
 }\]
In particular, $\Vspace$ is open in $\Vspace^{\Aff}$ and they are both $\Pi_1(B)$-invariant. Moreover, the variation $\Vspaceb\subset H^2(S_b)$ at each $b\in B$ sits inside the symplectic cone of the symplectic leaf $S_b$. 
\end{theorem}

\begin{remark}
\label{rk-comm-diagram}
The commutativity of the diagram in Theorem \ref{thm-affine-iso} is equivalent to saying that $\Iaff$ is equivariant. In this way, 
\[ \Iaff: (TB,\Lambda^{\vee}_{\cG}) \to (\HH,\HH_{\Z})\]
becomes a morphism of integral affine representations of $\Pi_1(B)$.
\end{remark}

\begin{proof} The commutativity of the diagram in the statement is equivalent to the commutativity of 
 \[ \xymatrix{ 
 \Pi_1(B) \ar[rr]^-{\delta(\varpi)} \ar[rd]_-{\dev} &    & \HH        \\
 & TB \ar[ru]_-{\Ilin}                                     &            
 }\]
where $\delta(\varpi): \Pi_1(B)\to \cH$ is defined by:
\[ \delta(\varpi) ([\gamma]):= \gamma^*\varpi_{\gamma(1)}- \varpi_{\gamma(0)}.\]
Note that in this diagram:
\begin{enumerate}[(a)]
\item $\delta(\varpi): \Pi_1(B)\to \cH$ is the 1-co\-cycle on $\Pi_1(B)$ coboundary of the 0-cycle $\varpi\in \Gamma(\HH)$; it differentiates to the algebroid 1-cocycle $v\mapsto \nabla_v\varpi$, i.e., $\Ilin$;
\item $\dev: \Pi_1(B)\to TB$ is the 1-cocycle on $\Pi_1(B)$ with values in $TB$ which integrates the algebroid 1-cocycle $\textrm{Id}:TB\to TB$. 
\item $\Ilin:TB\to\cH$ is a morphism of representations.
\end{enumerate}
It follows that the corresponding infinitesimal diagram is:
 \[ \xymatrix{ 
 TB  \ar[rr]^-{\Ilin} \ar[rd]_-{\textrm{Id}}  &    & \HH        \\
 & TB \ar[ru]_-{\Ilin}                                     &            
 }\]
which is trivially commutative (in this diagram $\Ilin$ appears in two distinct roles: as a Lie algebroid cocycle on the horizontal arrow and as a morphism of representations on the diagonal arrow).
\end{proof}

% The result above not only makes precise the statement that the variation of $\varpi$ is linear, but it also shows that there are restrictions on the possible symplectic integrations. In the commutative diagram, the two arrows given by $\Var$ and $\Iaff$ only depend on the Poisson structure, while the developing map $\dev$ depends on the affine structure on $B=M/\cF_\pi$, which in turn depends on the s-proper integration. 

In order to obtain a more concrete picture, let us fix
\begin{itemize}
\item a base point $b_0\in B$, and 
\item a $\Z$-basis $\mathfrak{b}_\Lambda=\{\lambda_1, \ldots, \lambda_q\}$ for $\Lambda_{b_0}$.
\end{itemize}
The affine holonomy representation becomes (for notations, see Example \ref{ex-gen-IAS}):
\[ h^{\Aff}_0: \pi_1(B,b_0)\rightarrow \Aff_\Z(\R^q),\quad \gamma \mapsto 
(v_{\gamma}, A_{\gamma}) .\]
Hence the main data consists of the $v_{\gamma}=(v_{\gamma}^1, \ldots, v_{\gamma}^q)$ and $A_{\gamma}=(A_{i}^{j}(\gamma))$,
where our convention is such that $A(\gamma)e_i= \sum_{j} A_{i}^{j}(\gamma)e_j$. 

If $(S, \omega_0)$ is the symplectic leaf corresponding to $b_0$, then the maps/actions in the previous discussion become: 
\begin{enumerate}[(i)]
\item a variation map with respect to paths that start at $b_0$:
\[ \Varb: \widetilde{B} \to H^2(S), \quad \Varb(\gamma)= \gamma^*[\omega_{\gamma(1)}]\]
\item a linear action of $\pi_1(B,b_0)$ on $H^2(S)$, that makes $\Varb$ equivariant. 
\item a $\pi_1(B,b_0)$-invariant weak integral affine subspace $V^{\Aff}_{0}= [\omega_0]+V^{\lin}_{0}\subset H^2(S)$.
\end{enumerate}

For any $x\in S$, let $P:=s^{-1}(x)$ be the $s$-fiber above $x$ of the s-proper integration $(\G,\Omega)\tto (M,\pi)$. The submersion $t:P\to S$ is a principal $\G_x$-bundle and the choice of basis $\mathfrak{b}_\Lambda$ gives an identification of the isotropy group with the standard $q$-torus $\T^q$. Hence $P\to S$ becomes a principal $\T^q$-bundle and we consider its Chern classes:
\[ c_1,\dots,c_q\in H^2(S)\quad (\text{integral classes})\]
Since $S$ is simply connected these classes do not depend on the base point $x\in S$.

\begin{corollary}\label{cor:DH2}
\label{cor:primitivity2} 
The Chern classes $c_1,\dots,c_q\in H^2(S)$ generate the space of linear variations of $\omega$ at $b_0$:
\[ V^{\lin}_{0}= \Span_{\R}(c_1, \ldots , c_q),\ \  V^{\lin}_{0, \Z}= \Span_{\Z}(c_1, \ldots , c_q).\]
The action of $\pi_1(B,b)$ on $V^{\Aff}_{0}\subset H^2(S)$ is given by 
\begin{equation}\label{for-constr-PMCTS} 
\gamma^*([\omega_0])= [\omega_0]+ \sum_{k} v_{\gamma}^{k} c_k,\quad
{\gamma}^{*}(c_i)= \sum_{k} A_{i}^{k} (\gamma) c_k,
\end{equation}
and for any path $\gamma$ in $B$ starting at $b_0$ one has 
\[ \gamma^*([\w_{\gamma(1)}])= [\omega_0]+ \dev^{1}_0(\gamma)c_1+ \ldots + \dev^{q}_0(\gamma)c_q,\]
where $\dev^{i}_0$ are the components of $\dev_0= \left.\dev\right|_{b_0,\mathfrak{b}_\Lambda}: \widetilde{B}\rightarrow \R^q$.
Hence, we have a commutative diagram:
\[ \xymatrix{ 
\widetilde{B}  \ar[rr]^-{\Varb} \ar[rd]_-{\dev_0}  &    &  V^{\Aff}_{0}
\subset   H^2(S)        \\
 & \R^q  \ar[ru]!<-20pt,0pt>_-{(v^i)\mapsto[\omega_0]+\sum_i v^i c_i}                                     & &             
}\]
where the image of $\Varb$ is an open, $\pi_1(B)$-invariant subset of $V^{\Aff}_{0}$, sitting inside the symplectic cone of $H^2(S)$.
\end{corollary}

Note that in the strong s-proper case the Chern classes $c_1,\dots,c_q$ are linearly independent, so $\Varb$ is a local diffeomorphism, and if $\cG$ is the canonical integration then they form a primitive family, in the sense that:
\[  \Span_{\Z}(c_1, \ldots , c_q)=  \Span_{\R}(c_1, \ldots , c_q)\cap H^2(S, \Z).\]
Hence, $V^{\Aff}_{0}$ is an integral affine subspace of $H^2(S)$, not only a weak one.

\begin{proof}
The corollary follows from Proposition \ref{I-lin-iso} and Theorem \ref{thm-affine-iso} once we realize that 
% on the fixed $\Z$-basis we have 
$\mathrm{var}^\lin_0(\lambda_i)= c_i$. This follows immediately from \eqref{eq:linear:variation}. 
\end{proof}

\begin{remark}\label{rem:smooth-lift2}
The corollary shows that, for any $v\in \im(\dev_0)\subset \R^q$, one can find a symplectic form $\omega_v$ on $S$ such that, in cohomology, we have:
\[ [\omega_v]= [\omega_0]+ v^1c_1+\cdots+v^qc_q.\]
This gives an explicit description for % $\left.{\Omega_{\DH}}\right|_b\subset H^2(S)$ 
the image of the variation map inside the symplectic cone as:
\[ \im(\Varb)= \{[\omega_v]: v\in \im(\dev_0)\}\subset H^2(S).\]
Note however that the symplectic forms $\omega_{v}$ are not unique. Also, while they can locally be chosen to depend smoothly on $v$,
it is not clear whether $v\mapsto \omega_v$ can be chosen smooth on the entire open $\im(\dev_0)$. 
\end{remark}

{
Theorem \ref{thm-affine-iso} and Corollary \ref{cor:primitivity2} should already remind the reader of the classic Duistermaat-Heckman Theorem (see Theorem \ref{thm:DH:classic}). We defer  to the next section the detailed explanation of this connection.}

For now we observe that the previous results suggests the following strategy to construct examples of PMCTs. 
For simplicity we restrict to integral affine manifolds which are complete (see Conjecture \ref{Markus-conj}). 

{
\begin{proposition}
\label{reconstr-2} 
Consider an integral affine manifold of type $B= \mathbb{R}^q/\Gamma$, with $\Gamma \subset \Aff_\Z(\R^q)$, and $S$ a compact 1-connected manifold.
Assume that the following conditions hold:
\begin{enumerate}[(i)] 
 \item $\Gamma$ acts on $S$ and there is a smooth $\Gamma$-equivariant map 
 \[ \R^q\ni v \mapsto \omega_v \in \Omega^{2}_{\mathrm{sympl}}(S).\]
 \item There exist linearly independent integral cohomology classes $c_1, \ldots, c_q$ in $H^2(S)$ such that:
\begin{equation}
\label{eq-lin-cond-crit} 
[\omega_v]= [\omega_0]+ v^1 c_1 + \ldots + v^q c_q,\quad \forall\ v\in \R^q.
\end{equation}
\end{enumerate}
Then $M:= S\times_{\Gamma} \mathbb{R}^q$ is a regular Poisson manifold of strong s-proper type 
(hence, if $B$ is compact, then  $M$ is of strong compact type).

More generally, condition (i) can be replaced by a smooth family of symplectic forms $w_v$ on $S$ 
and a lifting of the integral affine action
of $\Gamma$ on $\R^q$ to an action on $(S\times \R^q,\w_v)$
by Poisson diffeomorphisms (not necessarily a product action).
\end{proposition}

Note that the equivariance condition and (\ref{eq-lin-cond-crit}) imply that the action of $\Gamma$ on $[\omega_0]$ and the
$c_i$s is given by (\ref{for-constr-PMCTS}). Therefore, in practice, one starts from some smooth integral affine 
group $\Gamma \subset \Aff_\Z(\R^q)$  writing its elements in the split form $\gamma= (v_{\gamma}, A(\gamma))$
(cf. Example \ref{ex-gen-IAS}) and as a first step one tries to realize the identities (\ref{for-constr-PMCTS})
inside the cohomology of a compact manifold $S$. 
Observe that this already produces the cohomology bundle $\mathcal{H}=H^2(S)\times_\Gamma\R^q$, together with the section
$\varpi$. The second and much harder step is to represent the right hand side of (\ref{eq-lin-cond-crit}) by symplectic forms and to 
lift the action of $\Gamma$ on cohomology to an action by Poisson diffeomorphisms.

\begin{example} 
The simplest case to consider is $B= \S^1$ with its usual integral affine structure. The resulting problem turns out to be very closely related
to McDuff and Salamon's question on the existence of symplectic free circle actions with contractible orbits. Actually, 
the example given by Kotschick in that context \cite{Ko} turns out to be precisely the answer to our problem for $B= \S^1$ (see also \cite{Mar}).
That produces a very interesting example of Poisson manifold of strong compact type with leaf space $\S^1$ and $K3$ surfaces as symplectic leaves.
As we shall explain in future work, one can use the structure of the moduli space of marked K3 surfaces to 
apply Proposition \ref{reconstr-2} (the key feature is the strong Torelli theorem, which requires the most general version of condition (i))
and obtain similar PMCTs of strong compact type
with base the standard $\T^2$ and symplectic leaf the K3 surface (more generally, the Hilbert scheme of $n$ points on the K3 surface).
\end{example}
}

% \begin{proposition}
% \label{reconstr-2} 
% Consider an integral affine manifold of type $B= \mathbb{R}^q/\Gamma% $, with $\Gamma \subset \Aff_\Z(\R^q)$. Assume that $\Gamma$ acts on % a compact 1-connected manifold $S$ and:
% \begin{enumerate}[1.]
% \item one can find a symplectic form $\omega_0$ on $S$ and cohomology
% classes $c_1, \ldots, c_q\in H^2(S, \Z)$ on which $\Gamma$ acts according to (\ref{for-constr-PMCTS}) where, as in Example \ref{ex-gen-IAS}, we
% write the elements $\gamma\in \Gamma$ in the split form $\gamma= (v_{\gamma}, A(\gamma))$. 
% \item there is a smooth $\Gamma$-equivariant map % with values in the space of symplectic forms on $S$, 
% \[ \R^q\ni v \mapsto \omega_v \in \Omega^{2}_{\text{sympl}}(S),\]
% % with values in the space of symplectic forms on $S$ 
% such that, for all $v\in \mathbb{R}^q$, the following identity holds in $H^2(S, \R)$:
% \[ [\omega_v]= [\omega_0]+ v^1 c_1 + \ldots + v^q c_q.\]
% \end{enumerate}
% Then $M:= S\times_{\Gamma} \mathbb{R}^q$ is a regular Poisson manifold of strong s-proper type (hence, if $B$ is compact, then  $M$ is of strong compact type).
% \end{proposition}

%}

%%%%%%%%%%%%%%%%%%%%%%%%
%%%%%%%%%%%%%%%%%%%%%%%%
%%%%%%%%%%%%%%%%%%%%%%%%
%%%%%%%%%%%%%%%%%%%%%%%%
%%%%%%%%%%%%%%%%%%%%%%%%
%%%%%%%%%%%%%%%%%%%%%%%%
%%%%%%%%%%%%%%%%%%%%%%%%
\subsection{The twisted Dirac case}
\label{ssec:The twisted case} 
%%%%%%%%%%%%%%%%%%%%%%%%
%%%%%%%%%%%%%%%%%%%%%%%%
%%%%%%%%%%%%%%%%%%%%%%%%
%%%%%%%%%%%%%%%%%%%%%%%%
%%%%%%%%%%%%%%%%%%%%%%%%
%%%%%%%%%%%%%%%%%%%%%%%%
%%%%%%%%%%%%%%%%%%%%%%%%
We now briefly discuss the changes one needs to make so that the previous section applies also to twisted Dirac structures (for the motivation, please see the Introduction). Therefore we fix a closed 3-form $\phi\in \Omega^3(M)$ and a $\phi$-twisted Dirac structure $L$ on $M$ which, as before, 
we assume to be regular, of s-proper type, with 1-connected leaves, and with leaf space $B=M/\cF_L$.

To make sense of the linear variation, we interpret the class of fiberwise presymplectic forms as a section of a bundle over $B$, generalizing the section $\varpi\in \Gamma(\HH)$ used above. This forces us to consider the $\phi$-twisted version of $\HH$:
 
 \begin{definition} 
The $\phi$-twisted (second) cohomology at $b\in B$ is 
 \[
 \HH^\phi_b=H^2(S_b,\phi_b):=\frac{\{\beta\in \Omega^2(S_b)\,|\, \d\beta+\phi_{b}=0\}}{\{\d\Omega^1(S_{b})\}}.
 \]
 \end{definition}
 
Notice that this definition applies to any proper fibration $p: M\rightarrow B$ together with a closed 3-form $\phi\in \Omega^3(M)$ whose restriction $\phi_b$ to each fiber is exact (so that each $\HH^\phi_b$ is non-empty). 

The cohomology $\HH^\phi_b$ is not a vector space, but it is an integral affine space with underlying integral affine vector space $(\HH_{b}, \HH_{\Z, b})$. As  $\phi_b$ is exact,  Hodge theory implies that these affine spaces fit into an integral affine bundle $(\HH^\phi,\HH_\Z)$ with underlying integral vector bundle $(\HH,\HH_{\Z})$.

\begin{remark}[Integral affine bundles] 
\label{flat-integral-affine-bundles}
Given an affine bundle $E\rightarrow B$  we denote by $E^{\lin}$ its underlying vector bundle. An {\bf integral affine bundle} $(E,E^\lin_{\Z})$ is an affine bundle $E$ together with a lattice $E^{\lin}_{\Z}\subset E^{\lin}$.  As before, one can talk about 
% We also have the obvious notion of 
(weak) integral affine sub-bundle and of morphisms between integral affine bundles. 

The notion of {\bf affine connection} makes sense on any affine bundle $E\rightarrow B$: the space of sections $\Gamma(E)$ is an affine space with underlying vector space $\Gamma(E^\lin)$ and an affine connection is an affine map
% we consider affine maps
\[\n: \Gamma(E)\rightarrow \Omega^1(B,E^\lin)\]
whose linear part is a linear connection on $E^\lin$.  A (local) {\bf flat section} $s$ is a section satisfying $\n s=0$ and a {\bf flat affine connection} is one for which there exist local flat sections through every point in $E$. If the affine connection is flat, then so is it underlying linear connection on $E^{\lin}$. A {\bf flat integral affine bundle} is an integral affine bundle $(E,E^\lin_\Z)$ endowed with a flat affine connection $\n$ whose underlying linear connection $\n^{\lin}$ coincides with the one induced by the lattice $E^\lin_\Z$.

The notion of parallel transport and its basic properties extend to the setting of affine connections. In particular, a flat integral affine bundle is the same thing as an integral affine bundle $(E, E_{\Z}^{\lin})$ together with an action of $\Pi_1(B)$ on $E$ by integral affine transformations of the fibers.
Also, a {\bf morphism of flat integral affine bundles} is an integral affine morphism $f: E\to F$ which is $\Pi_1(B)$-equivariant.

While any vector space/bundle is canonically affine, any integral vector bundle is canonically a flat integral affine bundle, with the connection associated with the integral structure. An affine space/bundle is non-canonically isomorphic as affine spaces/bundles to its underlying vector space/bundle, and similarly if we add the adjective ``integral'', because an affine bundle always has a global section. However, a {\it flat} integral affine bundle $E$ is not isomorphic as flat integral affine vector bundles to its underlying integral affine vector bundle, unless $E$ has a global flat section. This is equivalent to the vanishing of the so-called {\bf radiance obstruction} of the flat affine bundle \cite{GH}. 
\end{remark}

\begin{example} 
Consider the tangent bundle $TB$ of an integral affine manifold $(B, \Lambda)$. Obviously, $(TB, \Lambda^{\vee})$ is an integral vector bundle, hence also a flat integral affine one: the corresponding action of $\Pi_1(B)$ on $TB$ is precisely the one induced by the linear holonomy representation (\ref{eq-the-linear-action}). However, $TB$ admits yet another flat integral affine structure, namely the one defined by the affine holonomy action of $\Pi_1(B)$ on $TB$ given by (\ref{eq-the-affine-action}). For $TB$ together with this flat integral affine structure we will reserve the notation $T^{\Aff}B$, and this is our realization of the affine tangent bundle from \cite{GH}. 
Of course, the only difference between $T^{\Aff}B$ and $TB$ lies on the $\Pi_1(B)$-action that one considers. % , hence they are different only as {\it flat} integral affine bundles. {GH}

In this framework, the identification of the variation with the affine variation from Theorem \ref{thm-affine-iso} is equivalent to saying that 
\[ \Iaff: (T^{\Aff}B,\Lambda^{\vee}) \to (\HH,\HH_{\Z})\]
is a morphism of flat integral affine bundles (see Remark \ref{rk-comm-diagram}). Moreover, the affine variation bundle $\Vspace^{\Aff}$ is a flat integral affine bundle, with underlying integral affine vector space $(\Vspace^{\lin}, \VspaceZ^{\lin})$ and the previous statement about $\Iaff$ can be split into two: (i) $\Iaff: T^{\Aff}B \to \Vspace^{\Aff}$ is a morphism of flat integral affine bundles and (ii) $\Vspace^{\Aff}$ is a weak sub-bundle of $(\HH, \HH_{\Z})$. 
\end{example}

% We can now introduce a flat connection on the integral affine bundle  $(\HH^\phi,\HH_\Z)$:
\begin{definition}
For any a proper fibration $p: M\rightarrow B$ with a closed 3-form $\phi\in \Omega^3(M)$ whose restriction to each fiber is exact, the  {\bf twisted Gauss-Manin connection}  $\nabla^\phi:  \Gamma(\HH^\phi)\rightarrow \Omega^1(B,\HH)$ is given by:
\begin{equation}\label{eq:twisted-variation}
 \langle \nabla^\phi_v(s),[\sigma_0]\rangle =\left.\frac{\d}{\d t}\right|_{t=0}\left(\int_{\partial ([0,t]\times \Sigma )}\sigma^*s+\int_{[0,t]\times \Sigma}\sigma^*\phi\right)
 \end{equation}
where $v\in T_bB$, $s\in\Gamma(B,\HH^{\phi})$, $[\sigma_0]\in H_2(S_b)$ is a homology class represented by a map $\sigma_0:\Sigma\rightarrow S_b$, and $\sigma: [0,1]\times \Sigma\rightarrow M$ is a map of fibrations extending $\sigma_0$ with base map a curve in $B$ representing $v$.
 \end{definition}
 
The independence of this definition on choices is a consequence of Stokes' theorem and the exactness of $\phi$ on fibers. One checks directly that $\nabla^\phi$ is a morphism of affine spaces, and that its linear part is the Gauss-Manin connection on $\HH$. 

A local flat section through any $c\in\cH^\phi_b$ can be constructed as follows: over a contractible neighborhood of $b$ the
 twisting form is exact: $\phi=\d\chi$. Therefore $[-\chi]_\phi$ defines a local flat section. It can be translated to attain the value $c$ at $b$ by adding the appropriate local flat section of $\HH$. Therefore $\HH^\phi$ is a flat integral affine bundle with underlying integral vector bundle $(\HH, \HH_{\Z})$. 

Now, given a $\phi$-twisted Dirac manifold $(M, L)$ we have the section $\varpi\in \Gamma(\HH^{\phi})$ and the twisted Gauss-Manin connection,
and this allows one to proceed as before. We can define the linear and affine variations by the same formulas:
\[ \Var^{\lin}:= \nabla^\phi \varpi: TB\to \HH, \quad \Var^{\Aff}:= \varpi+ \Var^{\lin}: TB\to \HH^{\phi}\]
and similarly for the linear/affine variation bundles $\Vspace^{\lin}$ and $\Vspace^{\Aff}$. Also, using the induced action of $\Pi_1(B)$ by parallel transport on $\cH^{\phi}$ one obtains the variation map 
\[ \Var: \Pi_1(B)\to \HH^\phi, \quad [\gamma]\mapsto \gamma^{*} \varpi.\]
and its image the variation bundle $\Vspace$. 

The main results from the previous section carry over to this context, with more or less obvious modifications. For example:

\begin{theorem}
\label{thm:twisted:variation}
For any s-proper $\phi$-twisted Dirac manifold $(M,L)$ with 
1-connec\-ted leaves and an $s$-connected, $s$-proper integration $(\cG,\Omega,\phi)$, the developing map $\dev$ of the integral affine structure on $B=M/\cF_\pi$ identifies the variation of $\varpi$ with its affine variation, i.e. one has a commutative diagram:
 \[ \xymatrix@R=15 pt{ 
 \Pi_1(B) \ar[rr]^-{\Var} \ar[rd]_-{\dev} &    & \HH^\phi        \\
 & TB \ar[ru]_-{\Iaff}                                     &            
 }\]
In particular, $\Vspace$ is open in $\Vspace^{\Aff}$ and they are both $\Pi_1(B)$-invariant.
\end{theorem}

\begin{proof}
The proof of Theorem \ref{thm-affine-iso} applies word by word, with one exception: one needs to be careful with the inclusion  $\Var^{\lin}(\Lambda^{\vee}_{\cG})\subset \HH_{\Z}$  of Proposition \ref{I-lin-iso}. For that, we need to make sure that Proposition \ref{lin-var-as-var-sympl-areas} still holds, and that sends us back to a description of the monodromy map in terms of variations of presymplectic areas. That was based on the choice of a splitting $\tau$ of \eqref{pre-transitive-seq} and the use of its curvature \eqref{eq:curv-formula}.  Such a splitting is provided by any extension of the foliated form $\omega$ to a 2-form on $M$ and the curvature $\Omega_\tau$ is computed using the $\phi$-twisted Dirac bracket. The resulting formula is precisely \eqref{eq:twisted-variation}, where $\Sigma$ is a sphere and the variation is determined by the corresponding vector at the image of its north pole.
\end{proof}

%%%%%%%%%%%%%%%%%%%%%%%%
%%%%%%%%%%%%%%%%%%%%%%%%
%%%%%%%%%%%%%%%%%%%%%%%%
%%%%%%%%%%%%%%%%%%%%%%%%
%%%%%%%%%%%%%%%%%%%%%%%%
%%%%%%%%%%%%%%%%%%%%%%%%
%%%%%%%%%%%%%%%%%%%%%%%%
\subsection{Two examples from Lie theory}
%%%%%%%%%%%%%%%%%%%%%%%%
%%%%%%%%%%%%%%%%%%%%%%%%
%%%%%%%%%%%%%%%%%%%%%%%%
%%%%%%%%%%%%%%%%%%%%%%%%
%%%%%%%%%%%%%%%%%%%%%%%%
%%%%%%%%%%%%%%%%%%%%%%%%
%%%%%%%%%%%%%%%%%%%%%%%%

\subsubsection{Regular coadjoint orbits}\label{ex-regular-coadjoint}
Let $G$ be a compact, connected Lie group with $\gg$ its Lie algebra. The symplectic groupoid $(T^*G,\w_\mathrm{can})$ is an s-proper
integration of the linear Poisson structure on $\gg^*$. The regular set $M:={\gg^*_\reg}$, consisting of those coadjoint orbits with stabilizer a maximal torus, is a regular Poisson manifold of s-proper type and $\cG:=(T^*G)|_{{\gg^*_\reg}}\tto {\gg^*_\reg}$ is a proper symplectic integration, inducing a transversal integral affine structure $\Lambda_\cG$. Compact coadjoint orbits are 1-connected, so our standing assumption holds and the leaf space of ${\gg^*_\reg}$ is a smooth integral affine manifold $(B,\Lambda_B)$. In this example, we can relate our previous discussion with some standard facts and constructions from Lie theory (see, e.g., \cite{BD,DK}). We will describe here this relationship, leaving the verifications to the reader. 

Fix a maximal torus $\T\subset G$ and let $\cc\subset\tt^*$ be the interior of a Weyl Chamber.  We recall that $\tt$ is a full slice to the adjoint action of $G$ on $\gg$: any regular orbit intersects $\tt$ transversely with tangent space $[\tt,\gg]$. Dually, the splitting
\[
 \gg=\tt\oplus [\tt,\gg]
\]
embeds $\tt^*$ into $\gg^*$ as a full slice to the coadjoint action. Each regular orbit intersects $\cc$ exactly once,
so we get a canonical diffeomorphism: 
\begin{equation}
\label{eq:coadj-triv}
 G/\T \times \cc  \to \gg^{*}_{\reg}\quad (g\T,\xi)\mapsto \Ad_g^*\xi
 \end{equation}
Under this diffeomorphism, the symplectic form of the orbit through $\xi\in \cc$ is the unique left $G$-invariant form  $\w_\xi\in \Omega^2(G/\T)$ satisfying at $\xi\cong e\T$:
\begin{equation}
\label{eq:coadj-form}
\w_\xi(u,v)=\xi([u,v]),\quad u,v\in \gg/\tt=T_\xi(G/\T).
\end{equation}

Let us fix a coadjoint orbit $S_0\subset \gg^*_\reg$ through some point $\xi_0\in\cc$, so that $S_0\cong G/\T$. If $\Lambda_G=\Ker(\exp:\tt\to \T)$ and $\Lambda_w$ is the weight lattice, we have isomorphisms:
\begin{itemize}
\item The leaf space: $(B,\Lambda_B)\cong(\cc,\Lambda_G^{\vee})$;
\item The normal space: $(\nu_{\xi_0}(S_0),\Lambda_\cG|_{\xi_0})\cong(\tt^*,\Lambda_G^{\vee})$;
\item The cohomology: $(H^2(S_0),H^2(S_0,\Z))\cong (\tt^*_{\ss},\Lambda_w)$, where  $\gg=\zz\oplus \gg_{\ss}$ is the decomposition into center and semisimple part, and $\tt=\zz\oplus \tt_{\ss}$. 
\end{itemize}
Explicitly, the last isomorphism associates to an element $\xi\in\tt^*_{\ss}$ the cohomology class of the form $\omega_\xi$ given by \eqref{eq:coadj-form}, hence we find that:
\begin{enumerate}[(i)]
\item The developing map $\dev_0:B\to \nu_{\xi_0}(S_0)$ is the inclusion:
 \[\dev_0: (\cc,\Lambda_G^\vee)\hookrightarrow (\tt^*,\Lambda_G^\vee);\]
 \item The linear variation $\Ilin:\nu_{\xi_0}(S_0)\to H^2(S_0)$ is the projection:
 \[\Ilin:(\tt^*,\Lambda_G^{\vee})\to (\tt^*_{\ss},\Lambda_w);\]
 \item The linear and affine variations match, so that $\Var:B\to H^2(S_0)$ is:
 \[\Var:(\cc,\Lambda_G^{\vee})\to (\tt^*_{\ss},\Lambda_w);\]
 \end{enumerate}
 
Notice that that the linear variation is injective iff $(\gg_{\reg},\w)$ is of strong s-proper type (cf.~Corollary \ref{cor-lin-iso-l}), and this happens iff $\gg$ is semisimple. Moreover, in this case, the linear variation is an isomorphism of integral vector spaces iff $G$ is the simply connected integration (cf.~Proposition \ref{I-lin-iso}). 

In general, the linear variation is not injective and its kernel is precisely $\zz^*$. If we consider the leaves through $K=\xi+\zz^*$ we obtain a Poisson submanifold $M_K\subset\gg^*_\reg$ of zero-variation (cf.~Theorem \ref{I-lin-iso-gen} (ii)). On the other hand, the leaves through $T=\xi+\tt^*_{\ss}$ yield a Poisson submanifold $M_T\subset\gg^*_\reg$ of full-variation (cf.~Theorem \ref{I-lin-iso-gen} (iii)), Poisson diffeomorphic to ${(\gg_{\ss}^*)}_\reg$.

 \subsubsection{Principal conjugacy classes}\label{ex-principal-orbits}
Let $G$ be a compact, connected Lie group with Lie algebra $\gg$, let $\langle\cdot,\cdot\rangle $ be an $\Ad$-invariant inner product, and let $L_G$ the corresponding Cartan-Dirac structure on $G$ with twisting $\phi$ the Cartan 3-form \cite{AMM, BCWZ, SW}. Recall that its leaves are the conjugacy classes and an s-proper integration is provided by the conjugacy action groupoid $G\ltimes G$ endowed with the multiplicative 2-form:
\begin{equation}\label{eq:mult-Cartan-Dirac}
\Omega_G(g,h)=\frac{1}{2}\left(\langle \Ad_h\pr_1^*\theta^L,\pr^*_1\theta^L\rangle+\langle 
\pr_1^*\theta^L,
\pr^*_2(\theta^L+\theta^R)\rangle \right),
\end{equation}
where $\theta^L$ and $\theta^R$ are the left and right-invariant Maurer-Cartan forms. We have the following basic result, relating $(G,L_G)$ and $(\gg^*,\pi_\lin)$ (see \cite[Theorem 3.13]{ABM}):

\begin{proposition}
\label{prop:Cartan-Dirac:linear}
Let $\overline{\exp}:\gg^*\to G$ be the composition of $\exp:\gg\to G$ with the isomorphism $\gg^*\cong\gg$ given be the inner product. The pullback Dirac structure $\overline{\exp}^*(L_G)$ is smooth and there is a 2-form $\chi\in\Omega^2(\gg^*)$ giving a gauge transformation:
\[ \overline{\exp}^*(L_G)=e^{\chi}L_{\pi_\lin}. \]
\end{proposition}

The 2-form $\chi$ in the proposition is an $\Ad^*$-invariant, canonical primitive of the pullback of the Cartan 3-form: $\overline{\exp}^*\phi=\d\chi$.

Let us now restrict to the regular set $G^\reg\subset G$, formed by the conjugacy classes of maximal dimension. We obtain an s-proper presymplectic $\phi$-twisted integration $(\cG,\Omega)=(G\ltimes G^\reg,\Omega_G)\tto G^\reg$, inducing a transverse integral affine structure $\Lambda_\cG$ to the foliation consisting of conjugacy classes in $G^\reg$. 

We recall that a regular orbit is called \emph{principal} if its isotropy is connected. A good example to keep in mind is $G=\mathrm{SO}(3)$, whose non-trivial conjugacy classes are all regular and among these there is only one which is non-principal, namely the conjugacy class of a non-trivial diagonal matrix. Principal orbits are 1-connected, and therefore $(G^\princ,L_G)$ is a connected regular twisted Dirac manifold of s-proper type, satisfying our standing assumption. Hence, the leaf leaf space $B=G^\princ/G$ is a smooth manifold carrying an integral affine structure $\Lambda_B$ such that $\Lambda_\cG=p^*\Lambda_B$.  Again, we can relate our previous discussion with some standard facts from the Lie theory of conjugacy classes of compact Lie groups (see.e.g., \cite{DK,BD}).

A maximal torus $\T\subset G$ is a full slice for the conjugation action, so $\T^\reg=\T\cap G^\reg$ and $\T^\princ=\T\cap G^\princ$ are also slices for the restricted action on $G^\reg$ and $G^\princ$. If $K$ is a connected component of $\T^{\reg}$ and $B$ is a connected component of $\T^\princ\cap K$, then we obtain a diffeomorphism:
% \begin{equation}\label{eq:conj-triv}
\[
G/\T \times B  \to  G^{\princ},\quad (g\T,k) \mapsto gkg^{-1},
\]
% \end{equation}
so $B$ is identified with the leaf space $G^\princ/G$. 

In general, $B$ is not 1-connected and one can identify its universal cover using $\exp:\tt\to\T$. A choice of positive roots determines a Weyl alcove of $\gg$, which is a connected component $\aa\subset\tt$ of $\exp^{-1}(\T^{\reg})$. A Weyl alcove of $G$ is a connected component $\aa_G\subset \aa$ of $\exp^{-1}(\T^{\princ})$ and the exponential $\exp:\aa_G\to B$ gives a covering map, so that $\widetilde{B}=\aa_G$.  It follows that we have a surjective local diffeomorphism:
\[ G/\T\times\aa_G\to G^\princ,\qquad (g\T,\xi)\mapsto g\exp(\xi)g^{-1}=\exp(\Ad_g\xi), \]
identifying the 2-form on the conjugacy class determined by $\xi\in\aa_G$ with:
\begin{equation}
\label{eq:twisted:2-form}
\widetilde{\chi}_\xi+\omega_\xi\in\Omega^2(G/\T).
\end{equation}
Here, $\omega_\xi$ is the symplectic form \eqref{eq:coadj-form} and $\widetilde{\chi}_\xi$ is the restriction to $G/\T\times\{\xi\}$ of the form $\widetilde{\chi}\in\Omega^2(G/\T\times\aa_G)$ obtained by pulling back the form $\chi$ in Proposition \ref{prop:Cartan-Dirac:linear} along the map \eqref{eq:coadj-triv}.

Now fix a conjugacy class $S_0\subset G^\princ$ through some point $g_0\in\exp(\aa_G)$, so that $S_0\cong G/\T$. If $\Lambda_G=\Ker(\exp:\tt\to \T)$, $\Lambda_w$ is the weight lattice, and 
\[ \Lambda_G^*:=\{v\in\tt:\langle v,\lambda\rangle\in\Z:\ \forall \lambda\in\Lambda_G\}, \] 
we find isomorphisms:
\begin{itemize} 
\item $(\nu_{g_0}(S),\Lambda_\cG|_{g_0})\cong (\tt,\Lambda^*_G)$;
\item $(H^2(S_0),H^2(S_0,\Z))\cong (\tt^*_{\ss},\Lambda_w)$.
\end{itemize}
We conclude that:
\begin{enumerate}[(i)]
 \item The developing map $\dev_0:\widetilde{B}\to \nu_{g_0}(S_0)$ is the inclusion:
 \[ \dev_0: (\aa_G,{\Lambda}^*_G)\hookrightarrow (\tt,\Lambda^*_G);\]
 \item The linear variation map  $\Ilin: \nu_{g_0}(S_0)\to H^2(S_0)$ is the composition of the isomorphism $\tt\cong \tt^*$ given by the inner product, with the projection onto $\tt^*_{\ss}$:
 \[\Ilin: (\tt,\Lambda^*_G)\to (\tt^*_\ss,\Lambda_w);\]
  \item The pullback to $\widetilde{B}=\aa_G$ of the bundle $\cH^{\phi}$ of twisted 2-cohomology groups trivializes and has the flat section $\xi\mapsto [\widetilde{\chi}_\xi]$, which allows to identify the linear and affine variation.
  \item The variation map $\Var:\widetilde{B}\to H^2(S_0)$ becomes the inclusion:
  \[\Var:(\aa_G,\Lambda^*_G)\to (\tt^*,\Lambda_w^\vee).\]
\end{enumerate}

Again, we note that the linear variation is injective iff $(G^\princ,L_G)$ is of strong s-proper type and this happens iff $\gg$ is semisimple. In this case the linear variation is an isomorphism of integral vector spaces iff $G$ is the simply connected integration.

%%%%%%%%%%%%%%%%%%%%%%%%
%%%%%%%%%%%%%%%%%%%%%%%%
%%%%%%%%%%%%%%%%%%%%%%%%
%%%%%%%%%%%%%%%%%%%%%%%%
%%%%%%%%%%%%%%%%%%%%%%%%
%%%%%%%%%%%%%%%%%%%%%%%%
%%%%%%%%%%%%%%%%%%%%%%%%
%%%%%%%%%%%%%%%%%%%%%%%%
%%%%%%%%%%%%%%%%%%%%%%%%
%%%%%%%%%%%%%%%%%%%%%%%%
%%%%%%%%%%%%%%%%%%%%%%%%
%%%%%%%%%%%%%%%%%%%%%%%%
%%%%%%%%%%%%%%%%%%%%%%%%
%%%%%%%%%%%%%%%%%%%%%%%%
\section{The linear variation theorem II: the general case}
\label{sec:lin-var-2}
%\section{\underline{The Duistermaat-Heckman linear variation theorem: the general case}}
% \section{\underline{Linear variation of the symplectic forms II: the general case}}
%%%%%%%%%%%%%%%%%%%%%%%%
%%%%%%%%%%%%%%%%%%%%%%%%
%%%%%%%%%%%%%%%%%%%%%%%%
%%%%%%%%%%%%%%%%%%%%%%%%
%%%%%%%%%%%%%%%%%%%%%%%%
%%%%%%%%%%%%%%%%%%%%%%%%
%%%%%%%%%%%%%%%%%%%%%%%%
  
We now extend the results from the previous section to general PMCTs of s-proper type, removing the assumption that symplectic leaves are 1-connected. The main difference is that now the leaf space is an orbifold. We will see how to state an appropriate version of the linear variation theorem, which will be a statement that holds on an orbifold bundle made of cohomologies of the symplectic leaves.

%%%%%%%%%%%%%%%%%%%%%%%%
%%%%%%%%%%%%%%%%%%%%%%%%
%%%%%%%%%%%%%%%%%%%%%%%%
\subsection{The developing map for transverse integral affine foliations}
\label{sec:IAS:orbifolds}
%%%%%%%%%%%%%%%%%%%%%%%%
%%%%%%%%%%%%%%%%%%%%%%%%
%%%%%%%%%%%%%%%%%%%%%%%%
%%%%%%%%%%%%%%%%%%%%%%%%
Since we do not have a smooth leaf space anymore, we are now forced to work with transverse integral affine structures. Let us point out how
the discussion in Section \ref{sec:developing:map}, on the developing map of integral affine manifolds, can be extended to the setting of transversally integral affine foliations \cite{GHL,Molino,Thur}. 

Given a foliation $(M,\cF)$ with a transverse integral affine structure $\Lambda$, in the intrinsic approach to the developing map one now has:
\begin{enumerate}[(i)]
\item An {\bf induced flat connection} $\nabla$ on $\nu(\cF)$ for which the local sections of $\Lambda$ are flat.
The connection gives rise to a {\bf linear holonomy} action (by parallel transport) of $\Pi_1(M)$ on $\nu(\cF)$:
% \begin{equation}\label{lin-action-transversal} 
\[
h^{\lin}: \Pi_1(M) \to \GL_\Lambda(\nu(\cF)).
\]
% \end{equation}
Its image will be denoted by $\Pi_1^{\lin}(M)\subset\GL_\Lambda(\nu(\cF))$.
\item The projection map $TM\to \nu(\cF)$ is an algebroid 1-cocycle and it integrates to the {\bf developing map}:
% \begin{equation}
% \label{dev-tr-fol} 
\[
\dev: \Pi_1(M)\to \nu(\cF).
\]
% \end{equation}
This, together with the linear holonomy action, gives rise to the {\bf affine holonomy action} 
% \begin{equation}\label{aff-action-transversal} 
\[
h^{\Aff}: \Pi_1(M)\to \Aff_\Lambda(\nu(\cF)).
\]
% \end{equation}
Its image will be denoted by $\Pi_1^{\Aff}(M)\subset \Aff_\Lambda(\nu(\cF))$.
\end{enumerate}
As in Section \ref{sec:developing:map}, to be more concrete one fixes
\begin{itemize}
\item a base point $x\in M$, and 
\item a $\Z$-basis $\mathfrak{b}_\Lambda=\{\lambda_1, \ldots, \lambda_q\}$ for $\Lambda_{x}$.
\end{itemize}
Upon restriction, we obtain the based/classical linear and affine holonomy representations (see \cite{GHL}):
\begin{equation}
\label{eq-hol-explic-TIAS}
h^{\lin}_0: \pi_1(M,x) \to\GL_{\Z}(\R^q) ,\quad h^{\Aff}_0: \pi_1(M, x)\to \Aff_{\Z}(\R^q),
\end{equation}
and the based developing map:
\[ \dev_0:(\widetilde{M},\widetilde{\Lambda})\to (\R^q,\Z^q), \]
which is a $\pi_1(M,x)$-equivariant integral affine submersion.
The images of these representations are the linear holonomy group $\Gamma^{\lin}\subset\GL_\Z(\R^q)$ and
the affine holonomy group $\Gamma^{\Aff}\subset\Aff_\Z(\R^q)$.

%%%%%%%%%%%%%%%%%%%%%%%%
%%%%%%%%%%%%%%%%%%%%%%%%
%%%%%%%%%%%%%%%%%%%%%%%%
\subsection{The linear variation theorem}
\label{sec:s-proper type}  
%%%%%%%%%%%%%%%%%%%%%%%%
%%%%%%%%%%%%%%%%%%%%%%%%
%%%%%%%%%%%%%%%%%%%%%%%%
%%%%%%%%%%%%%%%%%%%%%%%%
We assume now that $(M,\pi)$ is a regular Poisson structure and $(\cG,\Omega)\tto M$ is an s-connected, s-proper integration. We denote by $\Lambda= \Lambda_{\cG}$ the induced integral transverse integral affine structure and by $\cBG= \cBG(\cG)$ the induced integration of $\cF_{\pi}$, so we have a short exact sequence of Lie groupoids:
\[ \xymatrix{ 1\ar[r]& \nu^*(\cF_{\pi})/\Lambda \ar[r] & \cG\ar[r]& \cBG\ar[r] &1}. \]
We endow $B=M/\cF_\pi$ with the orbifold structure with atlas $\cBG$ (Theorem \ref{thm-reg-fol}). 
%In the leaf space $B=M/\cF_\pi$ we consider the orbifold structure with atlas $\cBG$ (see Theorem \ref{thm-reg-fol}). 

In order to study the variation of the leafwise symplectic forms, we need a better understanding of the ``vector bundle'' $\cH\to B$ with fiber $\cH_b= H^2(S_b)$, where the leafwise symplectic forms live. The problem is that $\cH$ is now only a set-theoretical vector bundle and even the ranks of the fibers may vary from point to point! To solve this problem we should proceed as follows:
\begin{itemize}
\item replace $\cH\to B$ by the representation $\cHE \to M$ of $\cBG\tto M$ defined by:
\[ \cHE_{x}= H^2(\cBG(x, -)),\]
where the action is the one induced from the right action of $\cBG$ on itself;
 \item replace $\Gamma(\cH)$ by the space $\Gamma( \cHE)^{\text{inv}}$ consisting of $\cBG$-invariant sections of the representation. Note that, while a priori $\Gamma(\cH)$ does not make sense, $\Gamma(\cHE)^{\text{inv}}$ sits inside the space of set-theoretical sections of $\cH$:
\begin{equation}
\label{incl-inv-set} 
\Gamma(\cHE)^{\text{inv}}\subset \Gamma_{\text{set}}(\cH).
\end{equation}
The image of this inclusion could be taken as definition of $\Gamma(\cH)$.    
\end{itemize}
The inclusion (\ref{incl-inv-set}) comes from the canonical isomorphisms 
\begin{equation}\label{identif-orbibundle} 
(\cHE_{x})^{\cBG_x-\text{inv}} \cong \cH_{b} ,
\end{equation}
valid for all $x\in M$, where $b= p(x)$ and $p: M\to B$ is the projection. This holds because $t: \cBG(x, -)\to S_{b}$ is a $\cBG_x$-covering projection and the isotropy $\cBG_x$ is finite. In this way, for any invariant section $\sigma$ of $\cHE$, $\sigma(x)$ makes sense as an element of $\cH_b$. Moreover, for any other $y$ with $p(y)= b$, there exists an arrow $g: x\to y$ in $\cBG$ and right action by $g$ becomes, after the identifications (\ref{identif-orbibundle}), the identity map on $\cH_b$. The invariance of $\sigma$ implies that $\sigma(x)\in \cBG_b$ only depends on $b$, therefore making sense of $\sigma$ as a section of $\cH$. \\

%%%%MARIUS: I HAD TO CHANGE A BIT HERE BEFORE WE STARTED TALKING ABOUT $\varpi$ BUT WE FORGOT TO SAY WHAT IT IS. 

Therefore, while (\ref{omega-as-section}) defines $\varpi$  only as a set-theoretical section of $\cH$, the previous discussion shows that the same information is obtained by considering the {\it smooth} section 
\[ \varpi \in \Gamma(\cHE)^{\text{inv}}\]
which, as a section of $\cHE$, is defined by pulling back along the target map of $\cB$ the cohomology classes of the leafwise symplectic forms.

\begin{remark}[Orbivector bundles]
\label{rem:cohom:orbifld}
The previous discussion belongs to the world of orbivector bundles (see \cite{ALR, Moe02} and Remark \ref{rmk:structures on orbifolds}):
an orbivector bundle over the orbifold $(B,\cB )$ is given by a linear representation $E^\cB  \to B$ of the groupoid $\cB$.  The bundle $\cHE$ is, of course, an example. Another example is provided by the tangent bundle to the orbifold: it is represented by the normal bundle $\nu(\cF)\to M$ of the foliation on the base, endowed with the linear holonomy action of $\cBG $. 

For any orbivector bundle $E^\cB  \to M$ over  $(B,\cB )$ its space of (smooth) sections is defined as $\Gamma(E^{\cB })^{\text{inv}}$. A morphism of orbivector bundles $E_1^\cB \to E_2^\cB $ (over the identity) is just a morphism of the representations and it induces a map between the space of sections $\Gamma(E_1^\cB )^{\text{inv}}\to \Gamma(E_2^\cB )^{\text{inv}}$. 
\end{remark}
\vskip .1 in

The previous remark suggests that the linear variation map $TB\rightarrow \cH$ of Section \ref{sec:variation1}, should now be replaced by a
morphism of orbivector bundles $TB\to \cH$, i.e., a morphism of representations $\nu^{*}(\cF_{\pi})\to \cHE$. Indeed, using the Gauss-Manin connection on $\cHE$ induced by $\cHE_{\Z}$, we define:

\begin{definition}
The {\bf linear variation} is the morphism of orbivector bundles:
\[ \Ilin: \nu(\cF_\pi)\to \cHE, \ \ v\mapsto \nabla_{v}\varpi. \]
\end{definition}

This map should now be seen as a morphism of integral representations of $\cBG $ or, equivalently, of integral orbivector bundles over $B$. We also consider the resulting linear variation space:
\[ (\Vspace^{\lin}, \VspaceZ^{\lin})= \Ilin(\nu(\cF_\pi), \Lambda_{\cG}^{\vee}).\]
We will see that this is an integral vector bundle sitting weakly inside $(\cHE, \cHE_{\Z})$. 

The integral structures make $\cHE$ and $\nu(\cF_{\pi})$ also into representations of the fundamental groupoid $\Pi_1(M)$ using the holonomy 
\[ h_{\nabla}: \Pi_1(M) \to\GL_\Z(\cHE)\]
induced by $\nabla$ (and similarly for $\nu(\cF_{\pi})$). It is not difficult to see that the actions of $\cBG $ and $\Pi_1(M)$ on $(\cHE, \cHE_{\Z})$ are compatible, since we have the commutative diagram:
\[ \xymatrix{
\Pi_1(M) \ar[r]^-{ h_{\nabla}} &\GL_\Z(\cHE) \\
\Mon(M, \cF_\pi)  \ar[u]^-{i_*}  \ar[r]^{h_{\cBG }} &  \cBG  \ar[u]_{\rho_{\cBG }} 
} \]
where $h_{\cBG }$ is the submersion associated to the foliation groupoid $\cBG $ and $\rho_{\cBG }$ is the action of $\cBG $ on $\cHE$ (and similarly for $\nu(\cF_{\pi})$). 
%Of course, one expects $\Ilin$ to be also $\Pi_1(M)$-equivariant and to preserve the integral structures but, this time, these statements are not consequences of ``abstract non-sense".  

We can now extend Proposition \ref{I-lin-iso}, Corollary \ref{cor-lin-iso-l} and then Theorem \ref{I-lin-iso-gen} to general s-proper Poisson manifolds.  We first formulate the analogues of Proposition \ref{I-lin-iso} and Theorem \ref{I-lin-iso-gen} together. Similar to the null-variation foliation $\cK$ introduced there, we now define $\cK_{M}\subset TM$ by 
\[ \Ker(\Var^{\lin})= \cK_{M}/\cF_{\pi} \subset \nu(\cF_{\pi}).\]
One should also recall the $\cBG $-monodromy groups $\cN_\cBG$ associated with any foliation groupoid $\cB$ integrating $(M,\cF_\pi)$, introduced in Section \ref{sec:hol:monodromy}.

\begin{theorem}\label{gen-I-lin-iso2} If $\cG$ is an s-connected, s-proper integration of $(M, \pi)$ then: 
\begin{enumerate}[(i)]
\item $\Ilin$  is a $\Pi_1(M)$-equivariant morphism of integral vector bundles,
\[ \Ilin: (\nu(\cF_\pi),\Lambda^{\vee}_{\cG})\to (\cHE, \cHE_{\Z})\]
with kernel $\cN_\cBG ^0$ and image $(\Vspace^{\lin}, \VspaceZ^{\lin})$.
\item $\cK_M$ is an integrable distribution and $\cN_{\cBG }\subset \nu^*(\cK_M)$ defines a transverse integral affine structure for $\cK_M$.
\item every leaf $\widetilde{K}$ of $\cK_M$ is a Poisson submanifold of $(M, \pi)$ saturated by symplectic leaves
and $(\widetilde{K}, \pi|_{\widetilde{K}})$ is of s-proper type with zero-variation. 
\item for any transversal $T$ to $\cK_{M}$ of complementary dimension its saturation $M_T$ with respect to the symplectic foliation is a Poisson submanifold of $s$-proper type with full-variation. It is of strong s-proper type if ${\cN_\mon}|_T$ is a lattice.
\end{enumerate}
\end{theorem}

\begin{proof}
The $\Pi_1(M)$-invariance is a consequence of the integrality $\Ilin(\Lambda^{\vee}_{\cG})\subset \cHE_\Z$ which we now prove. We use
the $\cBG $-variation map $\partial_\cBG $ whose image is precisely $\cN_{\cBG }$ (see Section \ref{sec:hol:monodromy}). By exactly the same arguments as in Proposition \ref{lin-var-as-var-sympl-areas}, one has 
\begin{equation}\label{lin-var-as-var-sympl-areas2} 
\langle \Ilin(v), \alpha\rangle= \partial_{\cBG }(\alpha)(v),\quad \forall\ v\in \nu_{x}(\cF_\pi),\ \alpha\in H_2(\cBG (x,-), \Z).
\end{equation}
Starting now with $v\in \nu_{x}(\cF_\pi)$, we have:
\begin{align*}
\quad \Ilin(v)\in H^2(\cBG (x, -),\Z)\quad &\Longleftrightarrow \quad
\langle \Ilin(v), \alpha\rangle \in \Z ,\ \forall\ \alpha\in H_2(\cBG (x,-), \Z)\\
& \Longleftrightarrow  \quad \partial_{\cBG , x}(\alpha)(v)\in \Z ,\ \forall\ \alpha\in H_2(\cBG (x,-), \Z) \\ 
& \Longleftrightarrow  \quad \lambda(v)\in \Z, \  \forall\ \lambda\in \cN_{\cBG , x}\\
& \Longleftrightarrow  \quad v\in \cN_{\cBG}^{\vee}.
\end{align*}
In other words one has the analogue of (\ref{N-via-Ilin}):
% \begin{equation}\label{preimage-HZ} 
\[
(\Ilin)^{-1}(\cHE_\Z)= \cN_{\cBG}^{\vee}. 
\]
% \end{equation}
The inclusion $\Ilin(\Lambda^{\vee}_{\cG})\subset \cHE_\Z$ follows now from the 
the last inclusion of Theorem \ref{ext-mon-crit}. The fact that $\VspaceZ^{\lin}$ is actually
a lattice in $\Vspace^{\lin}$ follows by the same argument as in the proof of Proposition \ref{I-lin-iso}: 
we now know that $\VspaceZ^{\lin}$ is discrete in $\Vspace^{\lin}$ and $\Ilin$ induces a map 
% $\nu(\cF_\pi)/\Lambda_{\cG}^{\vee}\to \Vspace^{\lin}/\VspaceZ^{\lin}$ 
from the compact space $\nu(\cF_\pi)/\Lambda_{\cG}^{\vee}$ onto $\Vspace^{\lin}/\VspaceZ^{\lin}$.

Note that (\ref{lin-var-as-var-sympl-areas2}) also implies that the annihilator of $\cN_{\cBG}$ is
\[ \cN_{\cBG}^0=\Ker(\Ilin).\]
The remaining statements are proven by exactly the same arguments as for Theorem \ref{I-lin-iso-gen}, 
but with $TB$ replaced by $\nu(\cF_\pi)$ and ${\cN_\mon}$ replaced by $\cN_{\cBG}$.
\end{proof}

The analogue of Corollary \ref{cor-lin-iso-l} (i), concerning zero-variation, holds without any further complications, once one makes precise sense of the notion of symplectic fibration over an orbifold- which we leave as an exercise for the reader.
 
The analogue of Corollary \ref{cor-lin-iso-l} (ii) states that the full-variation condition (i.e. the injectivity of $\Ilin$) is equivalent to the fact that $\cN_{\cBG}$ is a lattice in $\nu^*(\cF_\pi)$. This holds, by Theorem \ref{gen-I-lin-iso2} (i). One finds this situation, for example, in the strong s-proper case when $\cN_{\cBG}=\cN_\mon$ (but not only then!). Whenever the full-variation condition holds one obtains that % $\Ilin$ realizes $(\nu(\cF_{\pi}), \Lambda^{\vee}_{\cG})$ as an integral sub-bundles of $(\cHE, \cHE_{\Z}$), i.e. 
\[ \VspaceZ^{\lin}= \Vspace^{\lin}\cap \cHE_{\Z}.\]
This shows that the full-variation condition does not depend on the integrating groupoid. It can also be seen as an immediate consequence of the fact that for any finite covering the pull-back map in (real) cohomology is injective.

Finally, we can look at the variation of $\varpi$ and again prove its linear nature. 

First of all, in a similar fashion as in the previous section, we now have a variation map:
\[ \Var: \Pi_1(M)\to \cHE,\quad [\gamma]\mapsto \gamma^*{\varpi_{\gamma(1)}}.\]
On the other hand, we also have the affine version of $\Ilin$:
\[ \Iaff: \nu(\cF_\pi)  \to \cHE, \quad v\mapsto  \Iaff(v):= \varpi + \Ilin(v),\]
and its image:
\[ \Vspace^{\Aff}:= \varpi+ \Vspace^{\lin}\subset \cHE .\]

Second. the statement of the linear variation will use $\nu(\cF_\pi)$ and its structure of integral affine representation of 
$\Pi_1(M)$ (see Section \ref{sec:IAS:orbifolds}), or in the terminology of Remark \ref{flat-integral-affine-bundles}, the structure of flat integral affine bundle. In order to emphasize this structure, we will use the notation $\nu^{\Aff}(\cF_\pi)$. Moreover, we also use the developing $\dev: \Pi_1(M)\to \nu(\cF_{\pi})$ associated to the transverse integral affine structure (see Section \ref{sec:IAS:orbifolds}):

 The statement of the linear variation theorem is now as follows:

\begin{theorem}\label{prop-affine-iso-2} 
The image of the variation map $\Var$ is contained in $\Vspace^{\Aff}$ and $\Iaff: \nu^{\Aff}(\cF_\pi) \to \Vspace^{\Aff}$ is a $\Pi_1(M)$-equivariant morphism.
 Equivalently, there is a commutative diagram:
 \[ \xymatrix{ 
 \Pi_1(M) \ar[rr]^-{\Var} \ar[rd]_-{\dev} &    & \cHE      \\
 & \nu^{\Aff}(\cF_{\pi}) \ar[ru]_-{\Iaff}                                     &            
 }\]
In particular, $\Vspace^{\Aff}$ is a $\Pi_1(M)$-invariant weak integral affine sub-bundle of $\cHE$.
\end{theorem}

The proof follows exactly the same arguments as in Theorem \ref{thm-affine-iso}, and so it will be omitted. 
\medskip

More explicit descriptions, as in Corollary \ref{cor:primitivity2}, can be obtained by fixing a base point $x\in M$ and a $\Z$-basis $\mathfrak{b}_{\Lambda}= \{\lambda_1, \ldots, \lambda_q\}$ for $\Lambda_{x}$. Then $\mathfrak{b}_{\Lambda}$ induces an identification of $\cG_{x}^{0}$ with the standard torus $\T^q$, so the projection $\cG(x, -)\to \cBG(x, -)$ becomes a principal $\T^q$-bundle. We can then consider its Chern classes
\[ c_1,\dots,c_q\in H^2(\bar{S}, \Z)= \cHE_{\Z, x}.\] 
where we set $\bar{S}:= \cBG(x, -)$. As in Corollary \ref{cor:primitivity2}, we denote by 
\[ V^{\Aff}_{0}= \omega_0+ V^{\lin}_{0}\subset H^2(\bar{S})=\cHE_x, \] 
the fiber of $\Vspace^{\lin}$ at $x$ and by $\dev_0: \widetilde{M}\to \R^q$ the resulting developing map.  One then obtains the following extension of Corollary \ref{cor:primitivity2}:

\begin{corollary}
\label{cor:primitivity-2} 
The Chern classes $c_1,\dots,c_q\in H^2(\bar{S})$ generate the space of linear variations of $\varpi$:
\[ V^{\lin}_{0}= \Span_{\R}(c_1, \ldots , c_q),\ \  V^{\lin}_{0, \Z} = \Span_{\Z}(c_1, \ldots , c_q).\]
In the strong s-proper case, the classes $c_1,\dots,c_q$ are linearly independent and they form a primitive family, i.e., we have
\[  \Span_{\Z}(c_1, \ldots , c_q)=  \Span_{\R}(c_1, \ldots , c_q)\cap H^2(\bar{S}, \Z).\]

Moreover, for any path $\gamma$ in $M$ starting at $x$ one has 
\[ \gamma^*([\omega_{\gamma(1)}])= [\omega_{x}]+ \dev^{1}_0(\gamma)c_1+ \ldots + \dev^{q}_0(\gamma)c_q,\]
where $\dev^{i}_0$ are the components of $\dev_0$, so we have a commutative diagram:
\[\xymatrix{
\widetilde{M}  \ar[rr]^-{\Varb} \ar[rd]_-{{\dev_0}}  &    &  V^{\Aff}_{0}
\subset   H^2(\bar{S})        \\
 & \R^q  \ar[ru]!<-20pt,0pt>_-{(v^i)\mapsto[\omega_0]+\sum_i v^i c_i}                                     & &             
}\]

\end{corollary}
\medskip

\begin{remark}
One also has analogous results for twisted DMCTs with the appropriate modifications. The interested reader should be able to find the appropriate statements and its proofs.
\end{remark}
\vskip 60 pt

\subsection{Examples}

{
\subsubsection{The classical Duistermaat-Heckman theorem revisited}\label{ex:Duistermaat-Heckman}
The linear variation of cohomology, given by Theorems \ref{thm-affine-iso} and \ref{prop-affine-iso-2}, or their more explicit versions Corollaries \ref{cor:primitivity2} and \ref{cor:primitivity-2}, are the global versions of the classical Duistermaat-Heckman Theorem (cf.~Theorem \ref{thm:DH:classic}). More precisely, we can recover the classical theorem as follows. Given a free Hamiltonian $\T$-action on a connected symplectic manifold
$(S,\omega)$ with moment map $\mu:S\to\tt^*$
we consider the Poisson manifold $M= S/\T$, whose symplectic leaves are the symplectic reduced spaces $S_\xi=\mu^{-1}(\xi)/\T$. 
This kind of Poisson manifolds furnish examples of PMCTs, and were discussed in 
detail in \cite[Section 5.4]{CFMa}: there it is shown that if $\mu$ is proper an s-connected, s-proper, symplectic integration is given by:
\[ \cG=(S\times_\mu S)/\T\tto S/\T, \]
with symplectic form $\Omega$ induced from $\pr_1^*\omega-\pr_2^*\omega$. 

Now observe that for this symplectic integration $(\cG,\Omega)\tto S/\T$:
\begin{itemize}
\item the induced orbifold structure on the leaf space $B=\mu(S)\subset \tt^*$ is the submersion groupoid $\cB(\cG)=S/\T\times_\mu S/\T$, so it is smooth;
\item the induced integral affine structure on $\tt^*$ is the canonical integral affine structure $\Lambda$ for which $\T=\tt^*/\Lambda$;
\item the s-fiber of $\cG$ through a point in $S_\xi$ is (isomorphic to) the principal $\T$-bundle $\mu^{-1}(\xi)\to S_\xi$.
\end{itemize}
Since the orbifold $\cB(\cG)$ is actually smooth, Theorem \ref{prop-affine-iso-2} (respectively, Corollary \ref{cor:primitivity-2}), reduces to its smooth version, Theorem \ref{thm-affine-iso} (respectively, Corollary \ref{cor:primitivity2}). The conclusion is that the cohomology class $[\omega_\xi]$ of the symplectic form of the symplectic reduced space $S_\xi=\mu^{-1}(\xi)/\T$ satisfies:
\[ [\omega_\xi]=[\omega_{\xi_0}]+\langle c,\xi-\xi_0\rangle, \]
where $c\in H^2(S_{\xi_0},\Lambda)$ is the Chern class of the principal $\T$-bundle $\mu^{-1}(\xi_0)\to S_{\xi_0}$. This is precisely the classical result as stated in Theorem \ref{thm:DH:classic}.

Actually, in their paper  \cite{DuHe}, Duistermaat and Heckman allow for non-free actions, which leads to symplectic reduced spaces $\mu^{-1}(\xi)/\T$ which are orbifolds. It is not hard to see that our work on PMCTs can be extended to Poisson orbifolds, which allow to treat the non-free case. Moreover, our approach extends to the non-regular case \cite{CFMc}, showing that s-proper Poisson manifolds provide the right setting for the globalization of the Duistermaat-Heckman theorem.
}

\subsubsection{Regular conjugacy classes}\label{ex-regular-orbits}
Let us return to the Cartan-Dirac structure on a compact, connected Lie group $G$ discussed in Section \ref{ex-principal-orbits}, and look now at all the regular conjugacy classes. While the principal orbits are 1-connected, the regular non-principal orbits are not. As before, $(G^\reg,L_G)$ is a regular 
$\phi$-twisted Dirac structure, with an s-connected, s-proper integration provided by the $\phi$-twisted presymplectic groupoid $G\ltimes G^\reg$. It induces a transverse affine structure $\Lambda_\cG$ to the foliation $\cF_{L_G}$ made of the regular conjugacy classes and its leaf space $B=G^\reg/G$ is now an integral affine orbifold. 

The orbifold structure on $B$ is determined by the foliation groupoid $\cBG$ in the sequence (see Theorem \ref{thm-reg-fol}):
\[ \xymatrix{ 1\ar[r]& \nu^*(\cF_{L_G})/\Lambda \ar[r] & G\ltimes G^{\reg}\ar[r]& \cBG\ar[r] &1} \]
We claim that $\cBG$ is the holonomy groupoid of $\cF_{L_G}$ or, equivalently, that the action of $\cBG$  on $G^{\reg}$ is effective. In fact, if 
$(g,x)\in G\ltimes G^{\reg}$ maps to $h\in \cBG$, then the (local) action of $h$ is the one of the bisection $\{g\}\times G^{\reg}$ of $G\ltimes G^{\reg}$. This (local) action is trivial iff $g\in Z(G)$. Since the center is the intersection of all maximal tori, the arrow $h\in\cBG$ must be a unit.
  
The universal cover of $G^{\reg}$ is given by
\begin{equation}
\label{eq:regular:orbits:cover}
G/\T\times \aa\rightarrow G^{\reg}, \quad (g\T,\xi)\mapsto g\exp(\xi)g^{-1}=\exp(\Ad_g\xi), 
\end{equation}
where $\aa$ is a Weyl alcove for $\gg$, with covering group the quotient of the affine Weyl groups (see \cite{BD}):
\[ \Gamma=\pi_1(G^\reg)=W^{\Aff}_G/W^{\Aff}. \]
One has $\pi_1(G^\reg)=\pi_1(G)$, so there is a covering group $G^\lin\to G$ whose regular part is the holonomy cover of $(G^\reg,\cF_{L_G})$. Hence, the linear holonomy group is:
\[ \Gamma^\lin=\pi_1(G^\reg)/\pi_1(G^\lin)=W^{\Aff}_G/W^{\Aff}_{G^\lin}. \]

The covering map \eqref{eq:regular:orbits:cover} allows us to identify the twisted symplectic 2-forms on the leaves (conjugacy classes)
by the same formula \eqref{eq:twisted:2-form}. In fact, the analysis of the regular part of $G$ is essentially the analysis of the principal part of $G^{\lin}$ discussed  in Section \ref{ex-principal-orbits}. If we fix a non-principal conjugacy class $S_0\subset G^\reg\backslash G^\princ$ through some point $g_0\in\exp(\aa)$, we have $\bar{S}_0=\cBG(g_0,-)\cong G/\T$ and, again, we find: % isomorphisms:
\begin{itemize} 
\item $(\nu_{g_0}(S_0),\Lambda_\cG|_{g_0})\cong (\tt,\Lambda^*_G)$;
\item $(H^2(\bar{S}_0),H^2(\bar{S}_0,\Z))\cong (\tt^*_{\ss},\Lambda_w)$.
\end{itemize}
We conclude: 
  \begin{enumerate}[(i)]
  \item $B$ is an orbifold quotient of the integral affine manifold $B^{\lin}=\aa/\Gamma^{\lin}$ (endowed with integral affine structure $\Lambda^*_G$);
  \item The developing map $\dev_0:\widetilde{G}^\reg\to \nu_{g_0}(S_0)$ is now given by: 
   \[\dev_0:(G/\T\times\aa,\Lambda^*_G)\to (\tt,\Lambda^*_G),\]
   where one first projects onto $\aa$ and then takes the inclusion. 
   \item The linear variation $\Ilin: \nu_{g_0}(S_0)\to H^2(\bar{S}_0)$ is the identification $\tt\cong\tt^*$ given by the inner product followed by
   the projection onto the semisimple factor
     \[\Ilin: (\tt,\Lambda^*_G)\to (\tt^*_\ss,\Lambda_w);\] 
   \item The pullback to $\widetilde{G}^\reg=G/\T\times\aa$ of the bundle $\cH^\cBG$ of twisted 2-cohomology groups trivializes and has the flat section $\xi\mapsto [\widetilde{\chi}_\xi]$, which allows to identify the linear and affine variation.
  \item After the previous identification, the variation map $\Var:\widetilde{G}^\reg\to H^2(\bar{S}_0)$ becomes the projection to $\aa$,
  followed by the inclusion in $\tt$,  the identification $\tt\cong\tt^*$, and then the projection onto the semisimple factor:
  \[\Var:(G/\T\times\aa,\Lambda^*_G)\to (\tt^*_{\ss},\Lambda_w^\vee).\]
\end{enumerate}

%%%%%%%%%%%%%%%%%%%%%%%%
%%%%%%%%%%%%%%%%%%%%%%%%
%%%%%%%%%%%%%%%%%%%%%%%%
%%%%%%%%%%%%%%%%%%%%%%%%
%%%%%%%%%%%%%%%%%%%%%%%%
%%%%%%%%%%%%%%%%%%%%%%%%
%%%%%%%%%%%%%%%%%%%%%%%%
\section{Measures and the Duistermaat-Heckman formula}
% \section{\underline{PMCTs, orbifold measures and a Duistermaat-Heckman formula}}
\label{sec:measures}
%%%%%%%%%%%%%%%%%%%%%%%%
%%%%%%%%%%%%%%%%%%%%%%%%
%%%%%%%%%%%%%%%%%%%%%%%%
%%%%%%%%%%%%%%%%%%%%%%%%
%%%%%%%%%%%%%%%%%%%%%%%%
%%%%%%%%%%%%%%%%%%%%%%%%
%%%%%%%%%%%%%%%%%%%%%%%%
%%%%%%%%%%%%%%%%%%%%%%%%

Another fundamental property of PMCTs is the existence of natural invariant volume forms/measures. 
Given a regular Poisson manifold of proper type $(M, \pi)$ and an s-connected, proper integration $(\G,\Omega)\tto M$ we will 
see that the leaf space $B= M/\cF_{\pi}$ carries a natural measure. 
The basic idea is simple: an integral affine structure gives rise to a density,
hence to a measure. However, since $B$ is an orbifold we need a bit of care with the role of the groupoid in this construction. 
%%%MARIUS: I REMMOVED THE NEXT COMMENT BECAUSE, AT THIS POINT, WHEN WE SAY THAT $B$ IS SMOOTH, WE REALLY MEAN SMOOTH AS AN ORBIFOLD (THE COMMENT WAS FROM THE TIME WHEN, WHEN SAYING $B$, WE MEANT THE MANIFOLD.
% (even when $B$ is smooth!).
The outcome will be that any s-connected, proper integration $(\G,\Omega)$ induces a measure on $B$, called the {\bf integral affine measure} induced by $(\G,\Omega)$ and denoted $\mu_{\Aff}^{\G}$ (we omit the dependence in $\Omega$ in the notation, but the reader should keep in mind that this construction depends on having a proper \emph{symplectic} integration).

In the s-proper case there is yet another natural measure on $B$:
the one obtained by pushing down the Liouville measure associated to the symplectic form $\Omega$. 
One obtains a measure on $B$, called the {\bf Duistermaat-Heckman measure} 
induced by $(\G,\Omega)$ and denoted by $\mu_{\DH}^{\G}$. As we shall see, the relationship between $\mu_{\Aff}^\G$ and $\mu_{\DH}^\G$ can be described via a Duistermaat-Heckman formula, involving the volumes of the symplectic leaves.

\subsection{Measures on leaf spaces}
We start by fixing some notations and terminology. First of all, by a measure on a locally compact Hausdorff space $X$ we mean here a Radon measure in the sense of a positive 
linear functional 
\[ \mu: C_{c}(X)\to \R\]
defined on the space $C_{c}(X)$ of compactly supported continuous function on $X$. Although we will not use set-measures, 
it is still handy to use the notation 
\[ \mu(f)= \int_{X} f(x) \,\d\mu(x) .\]

When $M$ is a smooth manifold one can use $C_{c}^{\infty}(M)$ instead of $C_{c}(M)$. Moreover, in this case one can talk about 
{\bf geometric measures}: if $ \mathcal{D}_{c}(M)$ denotes the space of compactly supported sections of the density bundle $\mathcal{D}_M= |\wedge^{\top}T^*M|$, then each density $\rho\in \mathcal{D}(M)= \Gamma(M, \mathcal{D}_M)$ induces a linear functional
\[ \mu_{\rho}: C_{c}^{\infty}(M)\to \R, \ \mu_{\rho}(f):= \int_{M} f\rho ,\]
which is a measure whenever $\rho$ is positive. When $M$ is an oriented manifold, we can use the orientation to identify $\mathcal{D}_M$ with $\wedge^{\top} T^*M$, hence to integrate top forms instead of densities. 

An integral affine structure $\Lambda$ on a manifold $M$ induces a density $\rho_{\Lambda}$: locally, if $\lambda_{1}, \ldots , \lambda_{n}$ is a coframe that spans $\Lambda\subset T^*M$, then 
\[  \rho_{\Lambda}= |\lambda_1\wedge \ldots \lambda_n|.\]
Of course, if $M$ is oriented, then one case use oriented coframes to obtain a volume form  $\eta_{\Lambda}$ and $\rho_{\Lambda}= |\eta_{\Lambda}|$. 

Notice that forms, densities or measures on manifolds give sheaves
\[ X\mapsto \Omega^{\bullet}(X), \ \mathcal{D}(X), \textrm{or}\ \mathcal{M}(X) \]
to which one can apply Haefliger's transverse geometry approach (see Remarks \ref{rmk-Transversal geometric structures} and \ref{rmk:structures on orbifolds}). This leads to well-defined notions of differential forms $\Omega^{\bullet}(B, \cB)$, 
densities $\mathcal{D}(B, \cB)$ 
and measures $\mathcal{M}_{\textrm{orbi}}(B, \cB)$  
on any orbifold $(B, \cB)$:
if $\eE\tto T$ is an \'etale orbifold atlas then one considers invariant forms $\Omega^\bullet(T)^\eE$, invariant densities $\mathcal{D}(T)^\eE$ and invariant measures $\mathcal{M}(T)^\eE$. It is easy to see that, in these cases, the resulting objects depend only on the underlying classical orbifold. 

The following shows that $\mathcal{M}_{\textrm{orbi}}(B, \cB)$ can be identified with $\mathcal{M}(B)$- the space ordinary of measures 
on the locally compact Hausdorff space $B$.

\begin{lemma}
\label{lem:orbif-measures} 
Given a \'etale orbifold atlas $\eE\tto T$ with quotient map $p: T\to B$, there is a 1-1 correspondence:
\[ \{ \textrm{measures}\ \mu\ \textrm{on}\ B\} \stackrel{1-1}{\longleftrightarrow} \{\eE-\textrm{invariant measures}\ \widetilde{\mu}\ \textrm{on}\ T\}.\]
Explicitly, it is uniquely determined by $\widetilde{\mu}= \mu \circ p_{!}$, where
\[ p_{!}: C_{c}^{\infty}(T) \to C_{c}^{\infty}(B), \ \ p_{!}(f)(p(x))= \sum_{g\in s^{-1}(x)} f(t(g)) .\]
\end{lemma}

\begin{remark} It is instructive to realize that, in the resulting bijection 
\[ \mathcal{M}(B) \cong \mathcal{M}_{\textrm{orbi}}(B, \cB),\]
the left hand side depends only on the topological space $B$, the right hand side depends only on the classical orbifold underlying $(B, \cB)$,
but the isomorphism depends on the full orbifold structure ($p_!$ above depends on $\eE$!). A simple but already illustrative example is obtained when $B$ is a smooth manifold but we endow it with a non-smooth orbifold structure with orbifold atlas $\cB:= \Gamma\ltimes B\tto B$, where $\Gamma$ is a finite group acting trivially on $B$; while   
$\mathcal{M}_{\textrm{orbi}}(B, \cB)= \mathcal{M}(B)$, the previous isomorphism introduces a factor $|\Gamma|$, the cardinality of $\Gamma$. 
\end{remark}

More generally, this discussion extends to any foliation groupoid $\cE\tto M$, so one can talk about $\cE$-transverse
forms, $\cE$-transverse densities and $\cE$-transverse measures: one considers invariant forms $\Omega^\bullet(T)^{\eE_T}$, invariant densities $\mathcal{D}(T)^{\eE_T}$ and invariant measures $\mathcal{M}(T)^{\eE_T}$, where $T\subset M$ is any complete transversal to the $\cE$-orbit foliation.
Of course, for a proper foliation $(M, \cF)$ the structures on the orbifold $B=M/\cF$ coincide with the transverse structures on $(M, \cF)$. 

For any foliation groupoid $\cE\tto M$ there is a 1:1 correspondence:
\[ \left\{
                \begin{array}{ll}
                  \text{$\cE$-transverse forms}\\
                   \rho_T\in\Omega^\bullet(T)^{\eE_T}
                \end{array}
              \right\}
\stackrel{1-1}{\longleftrightarrow}
 \left\{
                \begin{array}{ll}
                  \text{invariant sections}\\
                   \rho^\nu\in\Gamma(\wedge^{\bullet}\nu^*(\cF)
                \end{array}
              \right\}\]
where by ``invariant" we mean invariant under linear holonomy, i.e., satisfying:
\[ \nabla_X\rho^\nu=0,\ \forall X\in \X(\cF),\]
where $\nabla$ is the Bott connection. Similarly, for densities we have:
\[ \left\{
                \begin{array}{ll}
                  \text{$\cE$-transverse densities}\\
                   \rho_T\in\mathcal{D}(T)^{\eE_T}
                \end{array}
              \right\}
\stackrel{1-1}{\longleftrightarrow}
 \left\{
                \begin{array}{ll}
                  \text{invariant sections}\\
                   \rho^\nu\in\Gamma(\cD_\nu)
                \end{array}
              \right\}\]
where $\cD_\nu=|\wedge^{\top}\nu^*(\cF)|$. These correspondences are obtained by considering the Morita equivalence:
\begin{equation}
\label{eq:Morita:transverse}
\xymatrix{
 \G \ar@<0.25pc>[d] \ar@<-0.25pc>[d]  & \ar@(dl, ul) & s^{-1}(T)\ar[dll]^-{t}\ar[drr]_-{s} & \ar@(dr, ur)   &  \G_T\ar@<0.25pc>[d] \ar@<-0.25pc>[d]\\
M & & & & T}
\end{equation}
For each $g\in s^{-1}(T)$, we obtain an isomorphism:
\[ \xymatrix{ \nu_{t(g)}(\cF)& \nu_g(s^{-1}(s(g)))\ar[l]_{\d_g t} \ar[r]^{\d_g s}&T_{s(g)} T}. \]
The isomorphisms determined by two arrows with the same source and target differ by the action of an element of $\cE$ on the normal space to the orbit, i.e., the linear holonomy action. This gives the desired 1-1 correspondence between elements $\rho_T\in\Omega^\bullet(T)^{\eE_T}$  and invariant sections $\rho^\nu\in\Gamma(\wedge^{\top}\nu^*(\cF))$.

Using this correspondence, we conclude:

\begin{proposition}
\label{prop:measure:affine}
Let $\cE\tto M$ be a proper foliation groupoid integrating $\cF$. Each transverse integral affine structure $\Lambda\subset \nu^*(\cF)$ determines a measure $\mu_{\Aff}$ on the orbifold 
$M/\cF$ which is represented by the invariant density $\rho_{\Aff}^\nu\in\Gamma(\cD_\nu)$ given by:
\[ \rho^\nu_{\Aff}|_x=|\lambda_1\wedge \ldots \lambda_n|,\]
where $\lambda_{1}, \ldots , \lambda_{n}$ is any basis of $\Lambda_x$.
\end{proposition}

Lemma \ref{lem:orbif-measures} tells us how to compute the resulting integrals by working on a transversal, but it is desirable to work directly at the level of $M$. For that, observe that the short exact sequence:
\[ \xymatrix{0\ar[r] & T\cF\ar[r] & TM\ar[r] & \nu(\cF) \ar[r]& 0} \]
induces an isomorphism between the associated density bundles:
\[ \cD_\nu\simeq \cD_{T^*\cF}\otimes \cD_M. \]
So we can decompose any invariant density $\rho^\nu\in\Gamma(\cD_\nu)$ as:
\[ \rho^{\nu}=\rho^*_{\cF}\otimes\rho_M, \]
where $\rho_\cF$ is a density along the leaves of $\cF$ and $\rho_M$ is a density on $M$.

\begin{proposition}
\label{prop:Fubini-foliations}
Let $\cE\tto M$ be a proper foliation groupoid integrating $\cF$. If  $\mu_B$ is a geometric measure on the orbifold $B=M/\cF$ represented by an invariant density $\rho^\nu\in\Gamma(\cD_\nu)$, then for any $f\in C_{c}^{\infty}(M)$ one has:
\[ \int_M f(x) \,\d\mu_{M}(x) =\int_B \left(\iota(b) \int_{S_b}f(y)\,\d\mu_{S_b}(y) \right)\,\d\mu_B(b),\]
where $\mu_{S_b}$ is the measure on the leaf $S_b$ and $\mu_M$ the measure on $M$ associated with any decomposition $\rho^{\nu}=\rho^*_{\cF}\otimes\rho_M$, while $\iota:B\to \N$ is the function that for each $b\in B$ counts the number of elements of the isotropy group $\cE_x$ ($x\in S_b$).
\end{proposition}

\begin{proof}
First we claim that it is enough to prove the theorem in the case where $(M,\cF)$ admits a complete transversal $T$ which intersects each orbit a finite number of times. In fact, recall that proper groupoids admit invariant partitions of unit, so it is enough to proof the theorem in the case where $M$ is the saturation of a small enough transversal $T$ to some orbit. Since the leaves of a proper groupoid are embedded and the leaf space is Hausdorff, we can choose the small transverse $T$ so that it intersects each orbit on a finite set. This proves the claim.

Now assume that we have fixed a $\mu_B$ is a geometric measure on the orbifold $B=M/\cF$ represented by an invariant density $\rho^\nu\in\Gamma(\cD_\nu)$, and that we have chosen some decomposition $\rho^{\nu}=\rho^*_{\cF}\otimes\rho_M$. We consider the Morita equivalence \eqref{eq:Morita:transverse}. Since $t:s^{-1}(T)\to M$ is a local diffeomorphism, on the space $s^{-1}(T)$ we have the pullback density $t^*\rho_M$.
We pick some $f\in C_{c}^{\infty}(M)$ and we compute the integral:
\[ \int_{s^{-1}(T)}\frac{1}{|\cE(t(g),T)|} f(t(g))\, \d t^*\mu_M(g), \]
in two different ways:
\begin{enumerate}[(i)]
\item If we apply fiber integration along the proper submersion $t:s^{-1}(T)\to M$, we obtain:
\[ \int_{s^{-1}(T)}\frac{1}{|\cE(t(g),T)|} f(t(g))\, \d t^*\mu_M(g) = \int_M f(x)\, \d\mu_M(x). \]
\item If we apply fiber integration along the proper submersion $s:s^{-1}(T)\to T$, we obtain:
\begin{align*}
\int_{s^{-1}(T)}\frac{1}{|\cE(t(g),T)|} f(t(g))\,& \d t^*\mu_M(g)= \int_{s^{-1}(T)}\frac{1}{|\cE(t(g),T)|} f(t(g))\, (t^*\rho_\cF\otimes t^*\rho^\nu)(g)\\
&=\int_T\left(\frac{1}{|\cE(x,T)|}\int_{s^{-1}(x)} f(t(g))\, \d\mu_{s^{-1}(x)}(g)\right)\d \mu_T(x)\\
&=\int_T\left(\frac{|\cE_x|}{|\cE(x,T)|}\int_{S_x} f(y)\, \d\mu_{S_x}(y) \right)\d \mu_T(x),
\end{align*}
where we first used that $\rho_M=\rho^{\nu}\otimes\rho_\cF$ and then that $t$ restricts to a cover on each fiber $s^{-1}(x)$ with covering group $\cE_x$. Using Lemma \ref{lem:orbif-measures}, we conclude that:
\[ \int_{s^{-1}(T)}\frac{1}{|\cE(t(g),T)|} f(t(g))\,\d t^*\mu_M(g)=\int_B\left(\iota(b)\int_{S_b} f(y)\, \d\mu_{S_x}(y) \right)\d \mu_B(b), \]
where $\iota(b)=|\cE_x|$ for any $x$ with $p(x)=b$.
\end{enumerate}
Putting (i) and (ii) together the proposition follows.
\end{proof}

\subsection{A Weyl type integration formula} 
\label{sec:Weyl:integration}

Let $(M, \pi)$ be a regular Poisson manifold. The leafwise symplectic form gives the leafwise Liouville volume form:
\[ \frac{\omega_{\cF_\pi}^{\top}}{\top !}\in \Omega^\top(\cF_{\pi}).\]
This induces a 1:1 correspondence between top degree forms $\eta\in\Omega^\top(M)$ and 
sections $\eta^{\nu}\in  \Gamma(\wedge^\top \nu^*(\cF_\pi))$ by setting:
\[ \eta= \frac{\omega_{\cF_\pi}^{\top}}{\top!} \otimes \eta^{\nu}.\]
It turns out that under this correspondence the transverse invariant densities/volu\-me forms correspond to the Hamiltonian invariant densities/volume forms in $(M,\pi)$, in the sense of the following definition:

\begin{definition}
 A {\bf Hamiltonian invariant volume form/density/measure} $\mu$ on a Poisson manifold $(M, \pi)$ is any volume form/density/measure $\eta$ on $M$ which is invariant under the flow of any Hamiltonian vector field $X_h$, i.e.:
\[ \Lie_{X_h}\eta=0, \quad \forall h\in C^\infty(M). \] 
\end{definition}

In fact, we have:

\begin{proposition} 
\label{prop:hamil:invariant:correspondence}
For a regular Poisson manifold $(M, \pi)$ the assignment $\eta\mapsto \eta^{\nu}$ gives a 1-1 correspondence between:
\begin{enumerate}[(i)]
\item Hamiltonian invariant volume forms $\eta\in\Omega^\top(M)$,
\item transverse volume forms $\eta^\nu\in \Gamma(\wedge^\top \nu^*(\cF_\pi))$.
\end{enumerate}
\end{proposition}

\begin{proof}
It is immediate to check that if $\eta=\frac{\omega_{\cF_\pi}^{\top}}{\top!} \otimes \eta^{\nu}$, then:
\[ \Lie_{X_h}\eta=\frac{\omega_{\cF_\pi}^{\top}}{\top!} \otimes \nabla_{X_h}\eta^{\nu}, \]
so the result follows.
\end{proof}

A similar discussion holds for densities where one replaces the foliated volume form $\omega_{\cF_\pi}^{\top}$ by the
foliated density $|\omega_{\cF_\pi}^{\top}|$. 

 Let now $(\G,\Omega)$ be an s-connected, proper integration of a Poisson manifold $(M, \pi)$.
It gives rise to a transverse integral affine structure $\Lambda\subset \nu^*(\cF_\pi)$
and an orbifold structure $\cB(\cG)$ on $B= M/\cF_{\pi}$. Hence, we obtain (see Propositions  \ref{prop:measure:affine} and \ref{prop:hamil:invariant:correspondence}):
\begin{itemize}
\item a  {\bf integral affine measure} $\mu_{\Aff}$ on the orbifold $(B,\cB(\cG))$;
\item a {\bf integral affine transverse density} $\rho_{\Aff}^{\nu}$ on $(M,\cF_\pi)$ representing $\mu_{\Aff}$;
\item a {\bf Hamiltonian invariant density} $\rho_M$ on $(M,\pi)$ corresponding to $\rho_{\Aff}^{\nu}$.
\end{itemize}
The integral affine transverse density and the Hamiltonian invariant density are related by:
\begin{equation}\label{proof-dec-0}  
\rho_{M}:= \frac{|\omega_{\cF_\pi}^{\top}|}{\top!}  \otimes \rho_{\Aff}^{\nu} .
\end{equation}
The resulting measure $\mu_{M}$ on $M$ is an incarnation of the integral affine measure $\mu_{\Aff}$ on $B$ at the level of $M$. 

As a consequence of Proposition \ref{prop:Fubini-foliations}, we obtain:

\begin{theorem}
\label{thm:Fubini:Poisson}
Given an s-connected, proper integration $(\G,\Omega)$ of a regular Poisson manifold $(M, \pi)$, one has for any $f\in C_{c}^{\infty}(M)$, 
\[ \int_M f(x) \,\d\mu_{M}(x) =\int_B \left(\iota(b) \int_{S_b}f(y)\,\d\mu_{S_b}(y) \right)\,\d\mu_{\Aff}(b),\]
where $\mu_{S_b}$ is the Liouville measure of the symplectic leaf $S_b$, and $\iota:B\to \N$ is the function that for each $b\in B$ counts the number of connected components of the isotropy group $\G_x$ ($x\in S_b$).
\end{theorem}

\begin{proof}
The result follows immediately by applying Proposition \ref{prop:Fubini-foliations} to the foliation groupoid $\cBG(\G)$ associated with $\G$ in the short exact sequence (see Theorem \ref{thm-reg-fol2}):
\[ \xymatrix{1\ar[r] & \cT(\cG) \ar[r] & \G \ar[r] & \cBG(\cG) \ar[r]& 1},\]
where $\cT(\cG)$ is the bundle of Lie groups consisting of the identity connected components of the isotropy Lie groups $\G_x$.  Notice that 
$|\cBG(\cG)_x|$ is exactly the number of connected components of $\cG_x$.
\end{proof}

In the s-proper case the leaves are compact, hence they have finite symplectic volume, and we obtain:
\begin{corollary}
\label{cor:Fubini}
If $(\G,\Omega)$ is an s-connected,  s-proper integration of a regular Poisson manifold $(M, \pi)$, then for any $h\in C^{\infty}_c(B)$:
\[ \int_M h(p(x)) \,\d\mu_{M}(x) =\int_B\iota(b) \vol(S_b)\,h(b) \,\d\mu_{\Aff}(b),\]
where $\vol(S_b)$ is the symplectic volume of $S_b$, and $\iota:B\to \N$ is the function that for each $b\in B$ counts the number of connected components of the isotropy group $\G_x$ ($x\in S_b$).
\end{corollary}

If $G$ is a compact, connected Lie group with Lie algebra $\gg$ and $\T$ is a maximal torus, Weyl's integration formula asserts that there is an isomorphism:
\[ 
C^\infty(\gg)\cong  C_c^{\infty}(G/\T\times \tt)^W\quad f(x)\mapsto F(g\T,u):=f(\Ad(g)(u))|\mathrm{det}(\Ad u)_{\gg/\tt}|,
\]
and for fixed $\Ad$-invariant measures $\mu_{\gg}$ and $\mu_{\tt}$:
\[\int f(x)\mu_{\gg}(x)=
\frac{1}{|W|}\int_{\tt}\left( \int_{G/\T}f(\Ad(g)(u))\mu_{\gg/\tt}(g\T)\right)|\mathrm{det}(\Ad u)_{\gg/\tt}|\mu_{\tt}(u).\]
Here $W=N(\T)/\T$ denotes the the Weyl group. In \cite{CFMc} we shall prove that Weyl's formula is the result of specializing to $(\gg^*,\pi_{\lin})$ (with integration $(T^*G,\w_{\can})$) an integration formula generalizing Corollary \ref{cor:Fubini} for arbitrary Poisson manifolds of s-proper type.

\subsection{The Duistermaat-Heckman measure}
\label{sec:Duistermaat-Heckman}
The discussion above was valid for general PMCTs. When the Poisson manifold is s-proper there is another natural measure associated with the PMCT: 
if $(\G,\Omega)$ is an s-proper integration of $(M, \pi)$, then it is natural to consider the measure on $B=M/\G$ obtained as the push-forward of the Liouville measure $\mu_{\Omega}$:
\[ \mu_{\DH}:= (p_B)_{*}(\mu_{\Omega})\]
along the proper map $p_B:= p\circ s= p\circ t: \G\to B$. We will show that:

\begin{theorem} 
\label{thm:measures:DH:Aff}
If $(\G,\Omega)$ is an s-connected, s-proper integration of $(M, \pi)$ then
\begin{equation}\label{eq:DH-measures}
\mu_{\DH}^\Omega= (\iota \cdot \vol)^2 \mu_{\Aff},
\end{equation}
where $\vol:B\to\R$ is the leafwise symplectic volume function and $\iota:B\to \N$
counts the number of connected components of the isotropy group of a symplectic leaf. Moreover, $(\iota \cdot \vol)^2$ is a polynomial for the
orbifold integral affine structure.
\end{theorem}  

The rest of this section is devoted to the proof of the theorem. First, assuming (\ref{eq:DH-measures}) to hold, 
we show that $(\iota \cdot \vol)^2$ is a polynomial in $(B,\Lambda)$. Note that its pullback to $M$ is a Casimir. Another Casimir, which 
we know to be a polynomial on $(B,\Lambda)$ by Theorem \ref{prop-affine-iso-2}, is the function $p^*\vol_{\cB}^2$,  
associating to a leaf $S_x$ the square of the symplectic volume of $\cBG(x, - )$. We now have two
non-zero Casimirs which on each leaf differ by an integer multiple; therefore their ratio is a constant.
   
%The rest of this section is devoted to the discussion of this theorem. 
% First, assuming (\ref{eq:DH-measures}) to hold, 
% it is easy to conclude that $(\iota \cdot \vol)^2$ is a polynomial in $(B,\Lambda)$.
% To start with, its pullback to $M$ is a Casimir. We can construct another Casimir $p^*\vol_{\cB}^2$ which to each leaf $S$ associates the square of the symplectic volume of $\cBG(x, - )$
% with respect to $t^*\w$. Theorem \ref{prop-affine-iso-2} implies that $p^*\vol_{\cB}^2$ defines a polynomial on $(B,\Lambda)$.
% But we have two non-zero Casimirs which on each leaf differ by an integer multiple, therefore their ratio is a constant.  

Now we turn to the proof of (\ref{eq:DH-measures}).
First of all, recall that the push-forward of measures is defined whenever we have a proper map $p: P\to B$ between locally compact Hausdorff spaces: it is the map given by
\[ p_*: \mathcal{M}(P)\to \mathcal{D}(B),\ (p_*\mu)(f)= \mu_{P} (f\circ p). \]
or in the integral formula notation:
\[ \int_{B} f(b) \ \d(p_*\mu)(b)= \int_{P} f(p(x)) \ \d\mu(x) .\]
When $p:P\to B$ is a proper submersion % between smooth manifolds, 
this operation transforms geometric measures into geometric measures, and amounts to fiber integration of densities:
\[ p_{!}= \int_{p-\textrm{fibers}} : \mathcal{D}_{c}(P) \to \mathcal{D}_{c}(B) .\] 
More precisely, the short exact sequence induced by the $\d p:TP\to TB$ yields, for each $x\in P$, a canonical decomposition: 
\[ \mathcal{D}_{P, x} \cong \mathcal{D}^{p}_{x} \otimes   \mathcal{D}_{B, p(x)}  \]
where $\mathcal{D}^{p}$ is the bundle of densities along the fibers of $p$. Hence, given $\rho\in \mathcal{D}(P)$, for any $b\in B$ we can view the restriction $\rho|_{p^{-1}(b)}$ as an element of $\mathcal{D}(p^{-1}(b))\otimes  \mathcal{D}_{B, b}$ and one can integrate along the fiber to obtain:
\[ p_{!}(\rho)(b): = \int_{p^{-1}(b)} \rho|_{p^{-1}(b)} \in \mathcal{D}_{B, b}.\]
By Fubini's theorem we conclude that $p_*(\mu_{\rho})= \mu_{p_!(\rho)}$. 

We apply this to an s-connected, $s$-proper integration $(\G, \Omega)$ of $(M, \pi)$. The Duistermaat-Heckman density $\mu_{\DH}^\Omega= p_* s_* (\mu_{\Omega})$
can be understood in two steps. The first one is integration along the $s$-fibers giving rise to a density on $M$:
\[ \rho^{M}_{{\DH}}:= \int_{s-\textrm{fibers}} \frac{|\Omega^n|}{n!} \in \mathcal{D}(M).\]

\begin{lemma}
$\mu_{\DH}^M$ is an invariant measure.
\end{lemma}

\begin{proof}
The source map $s:(\G,\Omega)\to(M,\pi)$ is Poisson. Hence, if $f\in C^\infty(M)$ the Hamiltonian vector fields $X_f$ and $X_{s^*f}$
are $s$-related and we have:
\[\Lie_{X_f}\mu_{\DH}^{M}=\int_s \Lie_{X_{s^*f}}\frac{|\Omega^n|}{n!}=0.\]
\end{proof}

The second step is to push-forward the measure $\mu_{\DH}^M$ along the map $p:M\to B$, resulting in $\mu_{\DH}^\Omega=p_*\mu_{\DH}^M$. Since Theorem \ref{thm:Fubini:Poisson} shows that $p_{*}(\mu_{M})= \iota \cdot \vol\cdot \mu_{\Aff}$, the proof of Theorem \ref{thm:measures:DH:Aff} is completed by proving the following:

\begin{lemma} One has $\rho^{M}_{{\DH}}= \iota \cdot \vol\cdot \rho_{M}$. In other words, at each $x\in M$, one has
\[ \mu^{M}_{\DH}(x)= \iota(x)\cdot \vol(S_x)\cdot \frac{\omega_{S_x}^m}{m!}\wedge \lambda_1\wedge\cdots\wedge\lambda_q, \]
where $\{\lambda_1,\dots, \lambda_q\}\subset \nu_x(\cF_\pi)$ is any basis of the transverse integral affine structure determined by $\G$.
\end{lemma}
 
\begin{proof} We fix a point $x\in M$, we denote by $S$ the leaf through $x$ and let
\[ p= \dim(S), \ q= \dim(\nu_x(\cF_\pi)), \ n= p+ q.\]
For any $g\in\G$ the short exact sequence:
\[ \xymatrix{ 0\ar[r] & T_g (s^{-1}(x)) \ar[r] & T_g\G \ar[r]^{\d_g s} & T_x M \ar[r] & 0}\]
gives a canonical isomorphism 
\[ \mathcal{D}_{T_g\G} \cong \mathcal{D}_{T_g (s^{-1}(x))}\otimes \mathcal{D}_{T_xM}.\]
This leads to a decomposition of the Liouville density 
\begin{equation}\label{proof-dec-1} 
\frac{|\Omega^{n}_{g}|}{n!}= \xi_g \otimes \rho_{M, x}, \quad \textrm{with}\ \xi_g\in  \mathcal{D}_{T_g s^{-1}(x)} .
\end{equation}
We conclude that:
\[ \rho^{M}_{{\DH}}(x)=\int_{s-\textrm{fibers}} \frac{|\Omega^n|}{n!}=\left(\int_{s^{-1}(x)}\xi\right) \rho_{M}(x).\] 

Next, there is a similar short exact sequence 
\[ \xymatrix{ 0\ar[r] & T_g\G(x, y) \ar[r] & T_g s^{-1}(x)  \ar[r]^{\d_g t} & T_y S \ar[r] & 0}\]
which induces a decomposition 
\[ \mathcal{D}_{T_gs^{-1}(x)}\cong \mathcal{D}_{T_g\G(x, y)} \otimes \mathcal{D}_{T_yS}.\]
Hence, we can write 
\begin{equation}\label{proof-dec-2} 
\xi_g= \eta_g \otimes \frac{|\omega_{S}^{\top}|}{\top!} , \ \ \ \textrm{with}\ \eta_g\in \mathcal{D}_{T_gs^{-1}(x)}
\end{equation}
Using this decomposition, we see that:
\[ \int_{s^{-1}(x)} \xi= \int_{S} \left(  \frac{|\omega_{S}^{\top}|}{\top!}(y) \int_{\G(x, y)} \eta(g) \right).\]

Therefore, to prove the lemma, it suffices to show that 
\[ \int_{\G(x, y)} \eta(g)= \iota(x),\] 
the number of connected components of $\G_x= \G(x, x)$. The compact Lie group $\G_x$ comes with its bi-invariant Haar density 
\[ \Haar(\G_x) \in \mathcal{D}(\G_x). \] 
Left translation $L_{a}: \G_x\to \G(x, y)$ by any $a: x\to y$ gives a similar density $\Haar(\G(x, y))$ on $\G(x, y)$. Because the total volume with respect to the Haar density is $1$, it suffices to show that $\eta= \iota(x) \Haar(\G(x, y))$. 

Notice that $\Haar(\G_x)= \frac{1}{\iota(x)} \Haar(\G_{x}^{0})$,
where the Haar density on the identity component (a torus) is induced by lattice given by the kernel of its exponential map. Denote by $\{\lambda^1, \ldots, \lambda^q\}$ a basis of the integral lattice $\Lambda_{\G, x}\subset \nu_{x}^{*}(\cF_\pi)$ and let $\{\widetilde{\lambda}^{1}, \ldots, \widetilde{\lambda}^{q}\}\subset \gg_x$ the corresponding basis of the kernel of the exponential map. Using left translations, the last vectors define vector fields on $s^{-1}(x)$ which we denote by $\{\overleftarrow{\lambda^1}, \ldots, \overleftarrow{\lambda^q}\}$. Then:
\[ \Haar(\G(x, y))=\frac{1}{\iota(x)}|\overleftarrow{\lambda^1}\wedge\cdots\wedge\overleftarrow{\lambda^q}|. \]
Therefore we are left with proving that:
\begin{equation}\label{to-prove-prop-DH} 
\eta_g (\overleftarrow{\lambda^1}\wedge \ldots \wedge \overleftarrow{\lambda^q})=1, \forall g\in \G(x, y).
\end{equation}

For that we have to unravel the construction of $\eta_g$, which goes via the decompositions
(\ref{proof-dec-1}) and (\ref{proof-dec-2}). First note that we can choose a basis $\{X^1,\dots,X^n,Y^1,\dots,Y^n\}$ for $T_g\G$ with the following properties:
\begin{enumerate}[(a)]
\item $X^1, \ldots, X^n$ is a basis of $\Ker(\d_g s)$;
\item $X^{p+1}=\overleftarrow{\lambda^1}|_g, \ldots, X^n=\overleftarrow{\lambda^q}|_g$ is a basis of $\Ker(\d_g s)\cap \Ker(\d_g t)$;
\item $Y^1, \ldots, Y^p,X^{p+1}, \ldots, X^n$ is a basis of $\Ker(\d_g t)$;
\item $\{\d_g s(Y^{p+1}),\dots,\d_g s(Y^{p+1})\}$ is the basis of the dual lattice $\Lambda^{\vee}_{\G, x}\subset\nu_x(\cF_\pi)$, dual to the basis $\{\lambda^1, \ldots, \lambda^q\}$.
\end{enumerate}
Then we see that:
\begin{enumerate}[(i)]
\item Decomposition (\ref{proof-dec-1}) gives:
\[
\frac{|\Omega^{n}_{g}(X^1, \ldots, X^n, Y^1, \ldots, Y^n)|}{n!}= \xi_g(X^1, \ldots, X^n) \cdot \rho_{M, x}(\d s(Y^1), \ldots, \d s(Y^n))
\]
\item Decomposition (\ref{proof-dec-2}) gives:
\[
\xi_g(X^1, \ldots, X^n) = \frac{|\omega_{S}^{\top}(\d t(X^1), \ldots, \d t(X^p))|}{\top!} \cdot \eta_g(X^{p+1}, \ldots, X^{n})
\]
\item Relation (\ref{proof-dec-0}) and Proposition \ref{prop:measure:affine} together with (d) gives:
\begin{align*}
\rho_M(\d s(Y^1), \ldots, \d s(Y^n))
&= \frac{|\omega_{S}^{\top}(\d s(Y^1), \ldots, \d s(Y^p))|}{\top!} \cdot \rho_{\nu}^{\Aff}(\d s(Y^{p+1}), \ldots , \d s(Y^n))\\
&= \frac{|\omega_{S}^{\top}(\d s(Y^1), \ldots, \d s(Y^p))|}{\top!} .
\end{align*}
\end{enumerate}
Putting (i), (ii) and (iii) together, we find that
\begin{align} 
\frac{|\Omega^{n}_{g}(X^1, \ldots, X^n, Y^1, \ldots, Y^n)|}{n!}=&\eta_g(X^{p+1}, \ldots, X^{n}) \cdot \nonumber \\
 \cdot   \frac{|\omega_{S}^{\top}(\d t(X^1), \ldots, \d t(X^p))|}{\top!}  \cdot & \frac{|\omega_{S}^{\top}(\d s(Y^1), \ldots, \d s(Y^p))|}{\top!} \label{forgotten-eq}
\end{align}

We now compute the left hand side of \eqref{forgotten-eq}. For that we use that the $s$ and $t$-fibers are $\Omega$-orthogonal, so by (b) it follows that:
\[ i_{X^{p+1}} \ldots i_{X^n}(\Omega^n)= \frac{n!}{q!p!} i_{X^{p+1}} \ldots i_{X^n}(\Omega^q)\cdot \Omega^{p}.\]
By (a) and (c) we have that  for $1\le j\le p$, the covector $\Omega(X^j, -)$ vanishes on all the $X^i$ and on all $Y^1, \ldots, Y^p$, so this last relation gives:
\begin{align*}
\frac{|\Omega^{n}_{g}(X^1, \ldots, X^n, Y^1, \ldots, Y^n)|}{n!}&=\\
=\frac{|\Omega^q(X^{p+1}, \ldots, X^n, Y^{p+1}, \ldots , Y^n)|}{q!} &\cdot
\frac{|\Omega^p(X^1, \ldots, X^p, Y^1, \ldots, Y^p)|}{p!},\\
=\frac{|\Omega^q(X^{p+1}, \ldots, X^n, Y^{p+1}, \ldots , Y^n)|}{q!} &\cdot
\frac{|\Omega^k(X^1, \ldots, X^p)|}{k!} \cdot  \frac{|\Omega^k(Y^1, \ldots, Y^p)|}{k!},
\end{align*}
where we have written $2k=p$. Moreover, since the restriction of $\Omega$ to the $s$-fibers coincides with the pull-back of $\omega_{S}$ via $t$, and similarly for the $t$-fibers, we find that
\[ \frac{\Omega^p(X^1, \ldots, X^p, Y^1, \ldots, Y^q)}{p!}= \frac{\omega^{k}_{S}(\d t(X^1), \ldots, \d t(X^p))}{k!} \cdot  \frac{\omega^{k}_{S}(\d s(Y^1), \ldots, \d s(Y^p))}{k!}.\]
It follows that \eqref{forgotten-eq} can be reduced to:
\[ \eta_g(X^{p+1}, \ldots, X^{n}) = \frac{\Omega^q(X^{p+1}, \ldots, X^n, Y^{p+1}, \ldots , Y^n)}{q!}. \]

Now we observe that by the multiplicativity of $\Omega$ we have:
\[
\Omega(X^{p+j},Y^{p+j})=\Omega_g(\overleftarrow{\lambda^{i}_{g}}, Y^{p+j}_{g})=\Omega_x(\lambda^j,\d s(Y^{p+j}_{g}))= \delta_{i, j},
\]
where we used (d). Since $\Omega(X^{p+i},X^{p+j})= 0$ for all $i,j=1,\dots,q$, we find that 
\[  \eta_g(X^{p+1}, \ldots, X^{n})=\frac{\Omega^q(X^{p+1}, \ldots, X^n, Y^{p+1}, \ldots , Y^n)}{q!}= 1,\]
which shows that \eqref{to-prove-prop-DH} holds and completes the proof. 
\end{proof}

As we shall show in \cite{CFMc} Theorem \ref{thm:measures:DH:Aff} holds for arbitrary Poisson manifolds of s-proper type. In fact, 
the polynomial  $(\iota\cdot \mathrm{vol})^2$ will play a fundamental role in the study of global properties of non-regular Poisson manifolds of s-proper type 
\cite{CFMc}. 

{
\begin{example}[The classical case]
Consider a free Hamiltonian $\T$-action on a connected symplectic manifold $(S,\omega)$ with a proper 
moment map $\mu:S\to\tt^*$, so that $M=S/\T$ is a Poisson manifold with
leaf space $\mu(S)\subset\tt^*$.  As we observed in Section \ref{ex:Duistermaat-Heckman},
the s-connected, s-proper symplectic integration $\cG=(S\times_\mu S)/\T$ induces on $\tt^*$ the integral 
affine structure $\Lambda$ for which $\T=\tt^*/\Lambda$. Hence, the integral affine measure
$\mu_{\Aff}$ on the leaf space is the usual Lebesgue measure on $\tt^*$.

On the other hand, $\mu_{\DH}^\Omega$ does not quite coincide with the classical Duistermaat-Heckman measure $\mu_{\DH}^\omega$: the latter is defined as the push-forward under the moment map $\mu:S\to \tt^*$ of the Liouville measure $\mu_\omega$ (see the discussion preceding Corollary \ref{cor:DH:classic}). However, as in that discussion, one can show using the local model that the two measures are related by:
\[ \mu_{\DH}^\Omega=\vol\cdot \mu_{\DH}^\omega. \]
Of course this also follows from the classical result of Duistermaat-Heckman and our Theorem \ref{thm:measures:DH:Aff}.

The isotropy groups of $\cG$ all coincide with $\T$, hence, are connected. Therefore, the function $\iota:B\to \N$ assumes the constant value 1, and Theorem \ref{thm:measures:DH:Aff} gives:
\[ \mu_{\DH}^\Omega=(\vol)^2\cdot \mu_{\Aff}. \]
\end{example}
We conclude that Theorem \ref{thm:measures:DH:Aff} recovers Corollary \ref{cor:DH:classic} and the polynomial nature of the classical Duistermaat-Heckman measure on $\tt^*$. Note that, for a general PMCT, while the function $\iota\cdot (\vol)^2:B\to\R$ is polynomial, the functions $\vol:B\to\R$ and $\iota\cdot \vol:B\to\R$ are not even smooth. This justifies our definition of the Duistermaat-Heckman measure.

}

\begin{remark}\label{rk-Joao}
This section is related to Weinstein's work on measures on stacks \cite{We}. According to his philosophy, 
the measures to consider in Poisson Geometry should arise by interpreting the symplectic groupoid as a stack. Our approach here is more direct approach, using the foliation groupoid instead of the full symplectic one. The precise relationship between the two is explained in \cite{CrMe}. 
\end{remark}

%%%%%%%%%%%%%%%%%%%%%%%%
%%%%%%%%%%%%%%%%%%%%%%%%
%%%%%%%%%%%%%%%%%%%%%%%%
%%%%%%%%%%%%%%%%%%%%%%%%
%%%%%%%%%%%%%%%%%%%%%%%%
%%%%%%%%%%%%%%%%%%%%%%%%
%%%%%%%%%%%%%%%%%%%%%%%%
\section{Proper isotropic realizations}
\label{sec:realizations}
%%%%%%%%%%%%%%%%%%%%%%%%
%%%%%%%%%%%%%%%%%%%%%%%%
%%%%%%%%%%%%%%%%%%%%%%%%
%%%%%%%%%%%%%%%%%%%%%%%%
%%%%%%%%%%%%%%%%%%%%%%%%
%%%%%%%%%%%%%%%%%%%%%%%%
%%%%%%%%%%%%%%%%%%%%%%%%
%%%%%%%%%%%%%%%%%%%%%%%%

Many algebraic or geometric objects can be studied via their representations. This philosophy also applies to Poisson Geometry, where the representations of a Poisson manifold take the concrete form of symplectic realizations (see below). For instance, the integrability of a Poisson manifold is equivalent to the existence of a complete symplectic realization \cite{CF2}. In this section we show that the properness of a Poisson manifold is closely related to the existence of \emph{proper isotropic realizations}. 

More precisely, to any proper isotropic realization $q:(X,\Oga)\to (M,\pi)$ we will associate a symplectic integration of $(M,\pi)$- 
the \emph{holonomy symplectic groupoid relative to $X$}, denoted by $\HolX$. It is the smallest integration that
acts on $X$ symplectically and it will play an import role in the last two sections of the paper. There we will introduce the Lagrangian Dixmier-Douady class of a proper integration and the ones with vanishing class are precisely the holonomy symplectic groupoids relative to some proper isotropic realization.

Proper isotropic realizations appeared first in the work of Dazord and Delzant \cite{DaDe}, under the name of \emph{symplectically complete isotropic fibrations}, as special fibrations of symplectic manifolds that generalize Lagrangian fibrations. From that point of view, this section generalizes the fact that the base of a proper Lagrangian fibration inherits an integral affine structure: we will show that the base of a proper symplectically complete isotropic fibration with connected fibers is a Poisson manifold of proper type.

\subsection{Symplectic realizations and Hamiltonian $\cG$-spaces} 
\label{ssec:sympl-realizs-Hamilt}
Recall that a {\bf symplectic realization} of a Poisson manifold $(M, \pi)$ is a symplectic manifold $(X, \Oga)$ together with a Poisson submersion
\[ q: (X, \Oga)\to (M, \pi).\]
The symplectic realization is called {\bf complete} if for any complete Hamiltonian vector field $X_{h}\in \X(M)$ the pullback $X_{h\circ q}\in \X(X)$ is complete. Of course, if $q$ is proper then it is complete. While every Poisson manifold admits a symplectic realization, for complete symplectic realizations one has:

\begin{theorem}[\cite{CF2}]\label{thm:integr-sympl-realiz}
A Poisson manifold is integrable if and only if it admits a complete symplectic realization.
\end{theorem}

Note however that, as the canonical integration $\Sigma(M,\pi)$ may already fail to be Hausdorff, 
in the previous theorem one has to allow for non-Hausdorff symplectic realizations (cf.~\cite[Remark 1]{CF2}). However, we will soon impose conditions that ensure that all the manifolds involved are Hausdorff. 

We recall in Appendix \ref{appendix:moment:maps} that for a symplectic integration $(\cG,\Omega)\tto(M,\pi)$ the moment map of an infinitesimally free Hamiltonian $\cG$-space $(X,\Oga)$ yields a symplectic realization of $(M,\pi)$. This motivates:

\begin{definition}
\label{X-compat-intgrts}
Given a symplectic realization $q: (X, \Oga)\to (M, \pi)$, a symplectic integration $(\cG,\Omega)\tto(M,\pi)$ is called {\bf $X$-compatible} if there is a symplectic $\cG$-action with moment map $q:X\to M$:
\[
\xymatrix@R=18 pt{
 (\cG,\Omega) \ar@<0.25pc>[d] \ar@<-0.25pc>[d]  & \ar@(dl, ul) & (X,\Oga)\ar[dll]^-{q}\\
(M,\pi) &   }
\]
\end{definition}

Every complete symplectic realization admits compatible integrations: the proof of Theorem \ref{thm:integr-sympl-realiz} given in \cite{CF2} shows that the Weinstein groupoid acts on every complete symplectic realization. In fact, we have (see also Appendix \ref{appendix:moment:maps}):

\begin{proposition}
For a Poisson manifold $(M, \pi)$, the complete symplectic realizations of $(M, \pi)$ are the same
thing as the moment maps of infinitesimally free $\Sigma(M,\pi)$-Hamiltonian spaces.
\end{proposition}

It will be useful to recall the construction from \cite{CF2}, that shows how the infinitesimal action determined by the realization $q: (X, \Oga)\to (M, \pi)$:
\[ \sigma:q^*T^*M\to\X(X),\quad i_{\sigma(\alpha)}(\omega)= q^*\alpha, \]
integrates to a symplectic action:
\[
\xymatrix{
 (\Sigma(M,\pi),\Omega) \ar@<0.25pc>[d] \ar@<-0.25pc>[d]  & \ar@(dl, ul) & (X,\Oga)\ar[dll]^-{q}\\
(M,\pi) &   }
\]
For that, let $a:I\to T^*M$ be a cotangent path with base path $\gamma_{a}:I\to M$ and choose $u\in X$ in the fiber over the initial point $\gamma(0)$. Since $q$ is complete, it follows that there is a unique path $\widetilde{\gamma}_{a}^{u}:I\to X$ with $q(\widetilde{\gamma}_{a}^{u}(t))=\gamma_{a}(t)$ and satisfying:
\[
\left\{
\begin{array}{l}
\frac{\d }{\d t}\widetilde{\gamma}_{a}^{u}(t)= \sigma(a(t)), \\
\\
\widetilde{\gamma}_{a}^{u}(0)=u
\end{array}
\right.
\]

We call $\widetilde{\gamma}_{a}^{u}$ the {\bf horizontal lift} of the cotangent path $a$ with initial point $u$.
It is easy to check that the horizontal lifts are leafwise paths in the symplectic orthogonal foliation $(\ker \d q)^\perp$.
It is proved in \cite{CF2} that two cotangent paths with the same initial point are cotangent homotopic if and only if their horizontal lifts are leafwise homotopic relative to the endpoints (and if this holds for some initial point $u$ it holds for any other point in the fiber). 

In summary, one can \emph{characterize} the canonical integration $\Sigma(M,\pi)\tto M$ from the realization $q:(X,\Oga)\to (M, \pi)$ as:
\begin{equation}\label{eq:alt-def-sigma} 
\Sigma(M,\pi)=\frac{ \{\text{cotangent paths $a:I\to T^*M$}\}}{\text{cotangent paths w/ lifts leafwise homotopic in $(\ker \d q)^\perp$}}. 
\end{equation}
If we denote by $[a]$ the class of a cotangent path, the symplectic action of $\Sigma(M,\pi)$ on $q:X\to M$ is then given by:
\begin{equation}\label{act-Sigma-on-X} 
\Sigma(M,\pi)_s\times_q X\to X, \quad ([a],u)\mapsto \widetilde{\gamma}_{a}^{u}(1). 
\end{equation}
The action gives an isomorphism of Lie groupoids:
\begin{equation}\label{identif-Sigma-ltimes-X} 
\Sigma(M,\pi)\ltimes X\cong \Mon((\ker \d q)^\perp), \quad ([a],u)\mapsto [\widetilde{\gamma}_{a}^{u}]. 
\end{equation}

Since $\Sigma(M,\pi)\tto M$ acts on any symplectic realization and it is the largest, s-connected, symplectic integration of $(M,\pi)$, it is natural to wonder:
\begin{itemize}
\item Given a symplectic realization $q:(X,\Oga)\to (M,\pi)$, is there is a ``smallest" $X$-compatible, s-connected, symplectic integration? 
\end{itemize}
The minimality property of the holonomy groupoid of a foliation suggests that, to construct such a groupoid, one should replace in the description (\ref{eq:alt-def-sigma}) of $\Sigma(M,\pi)$ ``homotopy'' by ``holonomy''. 

In other words, we define a new equivalence relation between cotangent paths $a_1,a_2:I\to T^*M$, which we call {\bf cotangent holonomy rel $X$},  by:
\[ a_1\sim_h a_2\text{ iff } \left\{ \begin{array}{l} \text{their horizontal lifts at any point $u$, }\widetilde{\gamma}_{a_1}^{u}, \widetilde{\gamma}_{a_2}^{u}:I\to X, \\
\text{have the same holonomy in } (\ker \d q)^\perp.\end{array}\right. \]
Notice that the base paths of cotangent holonomic paths have the same end points. Also, it is clear that:
\begin{enumerate}
\item[(a)] If $a_0$ and $a_1$ are cotangent homotopic then they are also cotangent holonomic rel $X$;
\item[(b)] If $a_0$ and $b_0$ are cotangent holonomic rel $X$ to $a_1$ and $b_1$, respectively, then the concatenations $a_0\cdot b_0$ and $a_1\cdot b_1$, if defined, are cotangent holonomic rel $X$.
\end{enumerate}

% Hence, we can associate to any complete symplectic realization a new groupoid as follows:
Therefore, we are led to the following:

\begin{definition}
\label{def:hol-sympl-gpd}
The {\bf holonomy symplectic groupoid} relative to $q:(X, \Oga)\to M$ is the groupoid $\HolX \tto M$ defined by:
\[ \HolX :=\frac{ \{\text{cotangent paths}\}}{\text{cotangent holonomy rel $X$}}, \]
with the obvious structure maps. Denote by $[a]_{h}$ the class of a cotangent path $a$. 
\end{definition}

There is an obvious groupoid action:
\[
\xymatrix{
 \HolX \ar@<0.25pc>[d] \ar@<-0.25pc>[d]  & \ar@(dl, ul) & X\ar[dll]^-{q}\\
M&   }
\]
which gives a morphism of groupoids:
\begin{equation}\label{HolX-from-sympl-orth}
\HolX \ltimes X\to \Hol((\ker \d q)^\perp), \quad ([a]_h,u)\mapsto [\widetilde{\gamma}_{a}^{u}]_h. 
\end{equation}
In good cases $\HolX$ will indeed be the ``smallest" $X$-compatible integration of $(M,\pi)$ and the last morphism will be an isomorphism of Lie groupoids.   

\begin{example}
For any symplectic groupoid $(\cG,\Omega)\tto (M,\pi)$, the target map $t:\cG\to M$ yields a complete symplectic realization of $(M,\pi)$.
We claim that in this case we have a natural isomorphism:

\[ \Hol_{\cG}(M,\pi)\cong \cG. \]
Indeed, the symplectic orthogonal foliation to the t-fibers is the foliation given by the s-fibers, which obviously has trivial holonomy. Hence, given a cotangent path $a:I\to T^*M$ starting at $x\in M$, if one denotes by $\widetilde{\gamma}_{a}:I\to \cG$ the unique horizontal lift  through $1_{x}$, then one has a well defined map $[a]_h\to \widetilde{\gamma}_{a}(1)$ and this defines the desired isomorphism from $\Hol_{\cG}(M,\pi)$ onto $\cG$.
\end{example}

\begin{example}
A Lagrangian fibration $q:(X,\Oga)\to B$ with compact connected fibers is a complete symplectic realization of the zero Poisson structure $\pi\equiv 0$.
We claim that in this case we have a natural isomorphism:

\[ \Hol_{X}(B,0)\cong \cT_\Lambda, \]
the symplectic torus bundle associated with the integral affine structure $\Lambda\subset T^*B$ (see Proposition \ref{prop:IAS-sympl-torus}). Indeed, in this case the symplectic orthogonal foliation to the fibers coincides with the fibers, and so has trivial holonomy. A cotangent path $a:I\to T^*B$ starting at $x\in M$ is just an ordinary path $a:I\to T^*_xB$ and it is cotangent homotopic to the constant path $\alpha=\int_0^1 a(t)\d t$ (see \cite{CF2}). For a constant path $\alpha\in T^*_xB$, the horizontal lift through $u\in X$ is the path $t\mapsto \phi^t_\alpha(u)$ (same notation as in the proof of Proposition \ref{prop:IAS-sympl-torus}) and we conclude that:
\[ a_1\sim_h a_2\text{ iff }  \int_0^1a_1(t)\d t -  \int_0^1a_2(t)\d t \in \Lambda. \]
Hence, the map $[a]_h\to \int_0^1 a(t)\d t\pmod{\Lambda}$ gives the desired isomorphism from $\Hol_{X}(B,0)$ onto $\cT_\Lambda$.
\end{example}

%%%%%%%%%%%%%%%%%%%%%%%
%%%%%%%%%%%%%%%%%%%%%%%
%%%%%%%%%%%%%%%%%%%%%%%
%%%%%%%%%%%%%%%%%%%%%%%
%%%%%%%%%%%%%%%%%%%%%%%
\subsection{The holonomy groupoid relative to an isotropic realization}
\label{sec:holonomy:groupd} 
%%%%%%%%%%%%%%%%%%%%%%%
%%%%%%%%%%%%%%%%%%%%%%%
%%%%%%%%%%%%%%%%%%%%%%%
%%%%%%%%%%%%%%%%%%%%%%%
%%%%%%%%%%%%%%%%%%%%%%%

In general, for a proper Poisson manifold $(M,\pi)$ the canonical integration $\Sigma(M,\pi)\tto M$ will fail to be proper. 
For a symplectic realization, the holonomy groupoid $\HolX$ constructed in the previous section 
is a smaller integration with better chances of being proper. When $\HolX$ is a proper symplectic groupoid then, according to Theorem \ref{thm-lattice-proper-case}, it will determine a transverse integral affine structure on $(M,\cF_\pi)$. We focus now on a class of
proper symplectic realizations to which one can always attach a transverse integral affine structure:

\begin{definition} An {\bf isotropic realization} of a Poisson manifold $(M, \pi)$ is a symplectic realization $q: (X, \Oga)\to (M, \pi)$ whose fibers are connected isotropic submanifolds of $(X, \Oga)$.
\end{definition}

\begin{remark}
\label{rk:symplectically complete isotropic fibrations}
Dazord and Delzant have studied in \cite{DaDe} the notion of a {\bf symplectically complete isotropic fibration} of a symplectic manifold $(X,\Oga)$. It is defined as a fibration $q:X\to M$ satisfying two properties:
\begin{enumerate}[(i)]
\item the fibers of $q$ are isotropic;
\item the symplectic orthogonal $(\ker\d q)^{\perp}$ is an integrable distribution.
\end{enumerate}
If one assumes additionally that the fibers of $q$ are connected it follows that $M$ carries a unique Poisson structure such that $q:(X, \Oga)\to (M, \pi)$ is a Poisson map, therefore making $q$ into an isotropic realization of $(M, \pi)$. Conversely, any isotropic realization of a Poisson manifold satisfies Dazord-Delzant's conditions. 

However, while the two notions are equivalent, they do reflect two different points of view, depending on whether one emphasizes the Poisson manifold $(M, \pi)$ or the symplectic manifold $(X, \Oga)$, respectively. The second point of view makes it clear that we are dealing with a generalization of the notion of a Lagrangian fibration of a symplectic manifold (see Section \ref{Integral affine structures on manifold}). 
\end{remark}
\vskip .1 in

Generalizing from proper Lagrangian fibrations, any proper isotropic realization also has an associated lattice. The infinitesimal action 
$\sigma:q^*T^*M\to\X(X)$ associated with the Poisson map $q:X\to M$ (see Appendix \ref{appendix:moment:maps}) restricts to an action $\sigma:\nu^*(\cF_{\pi})\to\X(X)$, which integrates to a global (bundle of groups) action: 
\begin{equation}\label{isotropic-action-normal} 
\xymatrix{
\nu^*(\cF_{\pi}) \ar[d]  & \ar@(dl, ul) & X\ar[dll]^-{q}\\
M&   }
\qquad 
\alpha\cdot u= \phi_{\sigma(\alpha)}^{1}(u).
\end{equation} 

Of course, this is just a particular case of the previous discussion: 
the exponential map $\exp:\nu_x^*(\cF_\pi)\to \Sigma(M,x)^0$ identifies this action with the restriction of the action of $\Sigma(M)$ on $X$ to
the connected component of its isotropy.  

\begin{definition} \label{def:LambdaX} 
The {\bf lattice associated to the proper isotropic realization} $q: (X, \Oga)\to (M, \pi)$ is the lattice
$\Lambda_X\subset \nu^*(\cF_\pi)$ given by:
% \begin{equation}\label{the-lattice-isotropic-case} 
\[
\Lambda_{X, x}:=\{\alpha\in \nu^{*}_{x}(\cF_\pi): \phi^1_{\sigma(\alpha)}=\text{id}\}.
\]
% \end{equation}
The associated torus bundle is $\cT_{X}:= \nu^*(\cF_\pi)/\Lambda_{X}$. 
% \[ \cT_{X}:= \nu^*(\cF_\pi)/\Lambda_{X}.\]
\end{definition} 

\begin{lemma}\label{lemma:Lambda-X-Mon} If $q: (X, \Oga)\to (M, \pi)$ is a proper isotropic realization, then 
$\Lambda_X$ defines a transverse integral affine structure for the symplectic foliation $\cF_{\pi}$, containing the monodromy 
group $\cN_{\rm{mon}}$ of $(M, \pi)$ (see Section \ref{sec:hol:monodromy}).

Moreover, the action (\ref{isotropic-action-normal}) induces an action of $\cT_{X}$ on $X$, 
\[ m: \cT_{X}\times_{M} X\to X, \ (\lambda, u)\mapsto \lambda\cdot u,\] 
which is free and proper, makes $q: X\to M$ into a principal $\cT_X$-bundle and the action is presymplectic in the sense that
\begin{equation}
\label{action-T-on-X} 
m^*(\Oga)= \pr_{1}^{*}(\omega_{\cT})+ \pr_{2}^{*}(\Oga),
\end{equation}
where $\omega_{\cT}$ the presymplectic form on $\cT_X$ (cf. Proposition \ref{prop:IAS-presympl-torus}).
\end{lemma}

\begin{proof} 

The fact that $\Lambda_X$ is a transverse integral affine structure is proved exactly as in the case of Lagrangian fibrations (see, e.g., \cite{DaDe}, or our proof of Theorem \ref{thm-lattice-proper-case}). The fact that $\Lambda_X$ contains $\cN_{\rm{mon}}$ is clear because we already know that the action of $\nu^{*}(\cF_{\pi})$ on $X$ factors through the action of the identity component of $\Sigma(M)$, which is $\nu^{*}(\cF_{\pi})/\cN_{\rm{mon}}$. 

The action of $\cT_{X}$ on $X$ is free since $\Lambda_X$ is precisely the kernel of the $\nu(\cF_\pi)$-action and the properness follows from the properness of $\cT_X$.  To check that the action is presymplectic, it suffices to observe that the action of $\Sigma(M)$ on $X$ is symplectic, together with the fact that the presymplectic forms on the conormal bundle coincides with the the pull-back of the symplectic form of $\Sigma(M)$ via the exponential map $\exp: \nu^{*}(\cF_{\pi})\to \Sigma(M)$.
\end{proof}

We can now state the main result of this section:

\begin{theorem}
\label{thm:isotropic:fib:grpd}
For any proper isotropic realization $q: (X, \Oga)\to (M, \pi)$: 
\begin{enumerate}[(i)]
\item $\HolX$ is an $X$-compatible, s-connected symplectic integration of $(M, \pi)$.
\item For any $X$-compatible, s-connected symplectic integration $(\cG,\Omega)\tto (M, \pi)$ there are \'etale morphisms of symplectic groupoids:
\[ \xymatrix{\Sigma(M,\pi)\ar[r] & \cG\ar[r] & \HolX}; \]
\item $\HolX$ is a proper Lie groupoid if and only if $\cF_{\pi}$ is of proper type.
\end{enumerate}
Moreover, one has a short exact sequence of Lie groupoids:
\[ \xymatrix{0\ar[r] & \cT_{X} \ar[r] & \HolX \ar[r] & \Hol(M,\cF_\pi)\ar[r] &0}.\]
\end{theorem}

For the proof of this theorem we start with an elementary but important 
property of isotropic realizations, which serves as starting point for reconstructing
the holonomy groupoid of $(\ker \d q)^\perp$ 
(hence, using (\ref{HolX-from-sympl-orth}, also the groupoid $\HolX$).

\begin{lemma}\label{lemma-start-isotr} 
For an isotropic fibration $q: (X, \Oga)\to (M, \pi)$, the foliation $(\ker \d q)^\perp$ coincides with 
the pull-back via $q$ of the symplectic foliation $\cF_{\pi}$.
\end{lemma}

\begin{proof}
The definition of the infinitesimal action $\sigma$ shows that $\im(\sigma)= (\ker \d q)^\perp$. Now, the fact that $q$ is a symplectic realization is equivalent to 
\[ \pi^{\sharp}(\xi)= \d q(\sigma(\xi)),\quad \forall \xi\in T^*M.\]
We deduce that, for a vector $v$ tangent to $X$, one has:
\begin{align*}
\d q(v)\in \cF_{\pi}&\Longleftrightarrow \d q(v)= \d q(\sigma(\xi)),  \text{ for some }\xi\in T^*M\\ 
& \Longleftrightarrow v\in \im(\sigma) +  (\ker \d q)=  (\ker \d q)^\perp+ (\ker \d q)= (\ker \d q)^\perp
\end{align*}
where for the last equality we used that the fibers are isotropic. % This shows that $(\d q)^{-1}(\cF_{\pi})= (\ker \d q)^\perp$. 
\end{proof}

Next, we look at the interaction between $\cT_{X}$ and the holonomy of the foliation.
Using the general properties of transverse integral affine structures (see Section \ref{sec:IAS:orbifolds}) and the discussion on presymplectic actions from Appendix \ref{App:twisted Dirac case}, one finds:
% how $\cT_{X}$ interacts with the holonomy of the foliation:

\begin{lemma} \label{lemma:action-Hol-on-T} 
The linear holonomy action (\ref{lin-hol-ref}) preserves $\Lambda_{X}$ so descends to an action
on $\cT_X$: for any leafwise path $\gamma: [0, 1]\to M$ from $x$ to $y$ one obtains
\[ \hol_{\gamma}: \cT_{X, x}\to \cT_{X, y}. \]
If $\Hol(M, \cF_{\pi})$ is endowed with the zero presymplectic form then the resulting 
action $m: \Hol(M, \cF_{\pi})\times_{M}\cT_X\to \cT_X$ is presymplectic:
\begin{equation}\label{action-Hol-on-T}
m^*(\omega_{\cT})= \pr_{2}^{*}(\omega_{\cT}).
\end{equation}
\end{lemma}

% \begin{proof} The first part follows from the general properties of transverse integral affine structures discussed in Section \ref{sec:IAS:orbifolds}.  The proof of \eqref{action-Hol-on-T} follows from the definitions.
% \end{proof}

We can now turn to the study of $\HolX$, with the aim of proving Theorem \ref{thm:isotropic:fib:grpd}.
Following the general discussion in Section \ref{ssec:sympl-realizs-Hamilt}, our strategy will be to understand the holonomy groupoid $\Hol((\ker \d q)^\perp)$ and then show that the morphism $\HolX \ltimes X\to \Hol((\ker \d q)^\perp)$ is actually in isomorphism.

Lemma \ref{lemma-start-isotr} and the fact that $q$ has connected fibers, shows that 
\begin{equation}
\label{comput-hol-isotr} 
\Hol((\ker \d q)^\perp){\cong} q^*\Hol(\cF_\pi),
\end{equation}
where $q^*\Hol(\cF_\pi)$ is the pullback groupoid:
\begin{equation}\label{pull-back-gpd-isotr}
q^*\Hol(\cF_\pi)=(X\times_{M} \Hol(M, \cF_{\pi})\times_{M} X\tto X).
\end{equation}
This groupoid consists of triples $(v, \gamma, u)$ with $q(v)= t(\gamma)$, $s(\gamma)= q(u)$. The source and target of the arrow $(v, \gamma, u)$ are $u$ and $v$, respectively, and the multiplication is given by:
\[ (w, \gamma_1, v)\cdot (v, \gamma_2, u)= (w, \gamma_1\cdot \gamma_2, u).\]
Here we will momentarily not distinguish between the leafwise path $\gamma$ and the element it represents in the holonomy groupoid. 

Next, we consider the projection $q^*\Hol(\cF_\pi)\to M$, $(v,\gamma,u)\mapsto q(u)$, and we define an action of the torus bundle $\cT_{X}$ on $q^*\Hol(\cF_\pi)\to M$ appealing to Lemma \ref{lemma:action-Hol-on-T}, by setting for each $\lambda\in\cT_X|_{q(x)}$:
\begin{equation}\label{act-t-hop} 
\lambda \cdot (v, \gamma, u)= (\hol_{\gamma}(\lambda) \cdot v, \gamma, \lambda\cdot u).
\end{equation}
Since the action of $\cT_X$ on $X$ is free and proper, it follows easily that the resulting quotient $q^*\Hol(\cF_\pi)/\cT_X$ is a Lie groupoid: 

\begin{lemma} 
If  $q: (X, \Oga)\to (M, \pi)$ is a proper isotropic realization, the quotient
\[ q^*\Hol(\cF_\pi)/\cT_X=\left( X\times_{M} \Hol(M, \cF_{\pi})\times_{M} X\right) /\cT_X \tto M \] 
is a smooth groupoid, which is Hausdorff whenever $\Hol(M, \cF_{\pi})$ is Hausdorff. Moreover, $q^*\Hol(\cF_\pi)/\cT_X$ is proper if $\cF_{\pi}$ is of proper type. 
\end{lemma}

\begin{proof} It is immediate to check from the definitions that the quotient is a groupoid. When $\Hol(M, \cF_{\pi})$ is Hausdorff, then we have a free and proper action of $\cT_X$ on a Hausdorff manifold, hence the quotient is smooth and Hausdorff. The last part on properness also follows immediately.

The only remaining question is to show that the quotient is smooth, when $\Hol(M, \cF_{\pi})$ is non-Hausdorff. 
Since the quotient map can be made into a submersion in at most one way, we only have to prove the local statement, around 
a neighborhood of an arrow $[v_0, \gamma_0, u_0]$ going from $x_0$ to $y_0$. 
Choosing two local sections of $\cT_{X}$, $\tau_{t}$ and $\tau_{s}$, defined on opens $U(y_0)$ containing $y_0$ and $U(x_0)$ containing $x_0$,
respectively, then 
\begin{align*}
\cT_X|_{U(y_0)}\times_{M}\Hol(M, \cF_{\pi})\times_{M} \cT_X|_{U(x_0)} &\to X\times_{M} \Hol(M, \cF_{\pi})\times_{M} X\\
(\lambda_2, \gamma, \lambda_1)&\mapsto (\lambda_2\cdot \tau_t(t(\gamma)), \gamma, \lambda_1\cdot {\tau_s}(s(\gamma)))
\end{align*}
defines an embedding into an open invariant subspace of $X\times_{M} \Hol \times_{M} X$. Finally, the quotient of the left hand side modulo the action of $\cT_{X}$ is clearly smooth. % we are done. 
\end{proof}

Next, we exhibit the symplectic structure of $q^*\Hol(\cF_\pi)/\cT_X$:

\begin{lemma} 
The 2-form $\widetilde{\Omega}:=\pr_{1}^{*}\Oga- \pr_{3}^{*}\Oga$ on $X\times_{M} \Hol(M, \cF_{\pi})\times_{M}X$, 
where $\pr_i$ is the projection on the $i$-th factor, descends to a 2-form $\Omega$
on the quotient groupoid $q^*\Hol(\cF_\pi)/\cT_X$, making it into a symplectic groupoid integrating $(M, \pi)$.
\end{lemma}

\begin{proof}
A more or less tedious computation shows that the kernel of the closed form $\widetilde{\Omega}= \pr_{1}^{*}\Oga- \pr_{3}^{*}\Oga$ 
is the image of the infinitesimal action induced by (\ref{act-t-hop}). Also, note that $\pr_3:(q^*\Hol(\cF_\pi),\widetilde{\Omega})\to (M,\pi)$ is f-Dirac.
It follows that $\widetilde{\Omega}$ descends to a symplectic form $\Omega$ on the quotient $q^*\Hol(\cF_\pi)/\cT_X$ and that the target map 
\[ t:(q^*\Hol(\cF_\pi)/\cT_X,\Omega)\to (M,\pi)\] 
is Poisson. Since $\widetilde{\Omega}$ is obviously multiplicative, so is $\Omega$, hence 
one obtains a symplectic groupoid $(q^*\Hol(\cF_\pi)/\cT_X,\Omega)$ integrating $(M, \pi)$.
\end{proof}

\begin{proof}[Proof of Theorem \ref{thm:isotropic:fib:grpd}]

We claim that $\HolX$ is a smooth quotient of $\Sigma(M,\pi)$, i.e., that it admits a smooth structure (necessarily unique) such that the canonical projection is a submersion. As first step we construct an isomorphism of groupoids
\begin{equation}
\label{eq-G(X,M)-alt} 
\Phi:\HolX \cong q^*\Hol(\cF_\pi)/\cT_X
\end{equation}
as follows: for an element $[a]_h\in \HolX$ represented by a cotangent path whose base path $\gamma_a$ starts at $x\in M$, we choose any $u\in q^{-1}(x)$ and we set
\[ \Phi([a]_h):= \left[ [a]\cdot u, \gamma_a, u\right] \in \HolX,\]
where we use the action of $[a]\in \Sigma(M,\pi)$ on $X$ (see (\ref{act-Sigma-on-X})) and we omit writing $[\gamma_a]_h$ for the middle element. This is independent of the choice of $u$.

The injectivity of $\Phi$ is clear: if $\gamma_{a_1}$ and $\gamma_{a_2}$ have the same holonomy with respect to $\cF_{\pi}$, Lemma \ref{lemma-start-isotr} shows that any lifts tangent to $(\ker\d q)^{\perp}$ and starting at the same point, such as $\widetilde{\gamma}_{a_1}^{u}$ and $\widetilde{\gamma}_{a_2}^{u}$, have the same holonomy with respect to $(\ker\d q)^{\perp}$.

For the surjectivity of $\Phi$, given any $[v, \gamma, u]$, one chooses a cotangent path $a$ with $[\gamma_a]_h= \gamma$. Since $[a]\cdot u$ and $v$ are in the same fiber of $q$, we can write 
\[ [v]= [\alpha]\cdot \left( [a]\cdot u\right) \] 
for some $\alpha\in\nu^*(\cF_\pi)$, so that:
\[ [v, \gamma, u]= \Phi([\alpha]\cdot [a]). \]

To check that  $\HolX$ is a smooth quotient of $\Sigma(M,\pi)$, we still have to check that the composition of the map $\Phi$ with
the projection $\Sigma(M,\pi)\to \HolX$ is a submersion. For that it suffices to prove the same property for its pull-back to $X$ via $q$.
On the pull-back, one has the identification (\ref{identif-Sigma-ltimes-X}) with the monodromy groupoid of
$(\ker\d q)^{\perp}$, the interpretation of (\ref{pull-back-gpd-isotr}) as the holonomy groupoid of the same foliation, 
and the map is just the canonical projection between the two. 

For the proof of Theorem \ref{thm:isotropic:fib:grpd} (ii), let $\cG$ be any other $X$-compatible, s-connected, symplectic integration of $(M,\pi)$.
The groupoid $q^*\cG$ is an s-connected integration of the foliation $(\ker\d q)^\perp$. Since $q^*\Hol(M,\cF_\pi)\simeq \cong((\ker\d q)^\perp)$,
and $\Mon((\ker\d q)^\perp)\cong q^*\Sigma(M,\pi)$, there are morphisms $q^*\Sigma(M,\pi)\to q^*\cG\to q^*\Hol(M,\cF_\pi)$ so that the following is a commutative diagram of surjective groupoid maps:
\[
\xymatrix{\\
q^*\Sigma(M,\pi)\ar[d]\ar[r] \ar@/^2pc/[rr]& q^*\cG \ar[d]\ar[r] & q^*\Hol(M,\cF_\pi)\ar[d] \\
\Sigma(M,\pi)\ar[r] \ar[r] \ar@/_2pc/[rr]& \cG \ar@{-->}[r] & \HolX \\
\quad
}
\]
The existence of the dotted arrow follows by surjectivity, so this proves (ii).

Finally, part (iii) of the theorem follows from the description of $\HolX$ as a quotient of $q^*\Hol(M, \cF_{\pi})$.
\end{proof}

We now look at some consequences of Theorem \ref{thm:isotropic:fib:grpd}. We concentrate on Poisson manifolds $(M, \pi)$ for which the symplectic foliation $\cF_{\pi}$ is of proper type, a condition that is necessary for $(M, \pi)$ to be of proper type. First of all we have:

\begin{corollary}
\label{cor:isotropic:proper:1}
Let $(M, \pi)$ be a Poisson manifold with $\cF_{\pi}$ of proper type. If  $q: (X, \Oga)\to (M, \pi)$ is a proper isotropic realization then $\HolX$ is an s-connected proper symplectic groupoid. Moreover:
\begin{enumerate}[(i)]
\item the orbifold structure induced on $B= M/\cF_{\pi}$ is a classical orbifold structure;
\item the transverse integral affine structure associated to the proper isotropic realization $q$ coincides with the one induced by the proper integration $\HolX$.
\end{enumerate}
\end{corollary}

In particular:

\begin{corollary}
Let $(M, \pi)$ be a Poisson manifold with $\cF_{\pi}$ of proper type. If $(M, \pi)$ admits a proper symplectic realization, then $(M, \pi)$ is of proper type.
\end{corollary}

One should be aware however that the properness of $(M, \pi)$ does not imply the existence of proper symplectic realizations. As we shall explain in the final sections of the paper, there is one more obstruction to the existence of proper symplectic realizations: the Lagrangian Dixmier-Douady class. 
\medskip

The previous corollary brings us to yet another aspect of the theory of isotropic realizations. After eventually passing to a cover, proper symplectic realizations can be made simple in the sense of the following:
%  the sense of the following remark. 
\begin{definition}
\label{def:simple:isotropic:realiz}
A {\bf simple isotropic realization} is an isotropic realization $q: (X,\Oga)\to (M, \pi)$ for which there exists an $X$-compatible symplectic integration
$(\cG, \Omega)\tto (M,\pi)$ whose action on $X$ is free.
\end{definition}

Using Theorem \ref{thm:isotropic:fib:grpd} and the discussion from Appendix \ref{appendix:moment:maps}, we deduce the following equivalent characterizations of simple isotropic realizations of PMCTs:

\begin{proposition}
\label{prop:simple:isotropic:realiz}
Let $q: (X, \Oga)\to (M, \pi)$ be a proper isotropic realization and assume that the symplectic foliation $\cF_{\pi}$ is of proper type. Then the following statements are equivalent:
\begin{enumerate}[(a)]
\item the symplectic foliation $\cF_{\pi}$ is simple;
\item $q: (X, \Oga)\to (M, \pi)$ is a simple isotropic realization;
\item there exists an $X$-compatible, s-connected, proper integration $(\cG,\Omega)\tto (M,\pi)$ that acts freely on $X$; 
\item $(X, \Oga)$ is a free Hamiltonian $\cT_{\Lambda}$-space for an integral affine manifold $(B, \Lambda)$ with reduced space $(M, \pi)\cong (X_{\red}, \pi_{\red})$ (see Corollary \ref{reduction-free-qHamilt}).
\end{enumerate}
\end{proposition}

Of course, under any of the equivalent conditions of the proposition, one has that $B=M/\cF_{\pi}$, $\Lambda$ the integral affine structure on $B$ induced by $\Lambda_X$ and $\cG$ is the gauge groupoid $\cG\cong\Gauge{\cT_{X}}{X}$ associated to the principal $\cT_X$-space $X$ (Appendix \ref{cG-red2}).

\begin{proof}
The fact that the holonomy of the symplectic foliation $\cF_{\pi}$ is an obstruction to the existence of simple isotropic realizations of $(M, \pi)$ follows from (\ref{comput-hol-isotr}), the minimality property of holonomy groupoids (Theorem \ref{seq-mon-G-hol}) and the fact that for an $X$-compatible integration $\cG\ltimes X$ is a foliation groupoid integrating $(\ker \d q)^\perp$. The rest of the statement should be clear.
\end{proof}

In order to show that eventually passing to a cover, proper symplectic realizations can be made simple, we need to appeal to the main result of Appendix \ref{appendix:molino}. This result applied to the proper foliation $(M, \cF_{\pi})$, together with the transversal integral affine structure $\Lambda_X$, yields the linear holonomy cover $(M^{\lin},\cF_\pi^\lin)$ with a smooth leaf space $B^\lin=M^\lin/\cF_\pi^\lin$- an integral affine manifold carrying an action of the linear holonomy group $\Gamma^{\lin}$ by integral affine transformations, with $B^{\lin}/\Gamma^{\lin}= B$.  The conclusion is that when $\cF_\pi$ is of proper type, any proper isotropic realization is obtained as quotient of a simple one:

\begin{corollary}
\label{cor:isotropic:proper:3}
Let $(M, \pi)$ be a Poisson manifold with $\cF_{\pi}$ of proper type. If $q: (X, \Oga)\to (M, \pi)$ is a proper isotropic realization then $X^\lin:=X\times_M M^\lin$ yields a $\Gamma^\lin$-equivariant proper isotropic realization
\begin{equation}
\label{eq-q-lin} 
q^{lin}: (X^{\lin}, \widetilde{\Omega}_X)\to (M^{\lin}, \widetilde{\pi}), 
\end{equation}
which is simple. 
\end{corollary}

This is summarized in the following diagram:
\[ \xymatrix{(X^\lin,\widetilde{\Omega}_X)\ar[d]\ar[r]^{q^\lin} & (M^\lin,\widetilde{\pi}) \ar[d]\ar[r] & B^\lin\ar[d]\\
(X, \Oga)\ar[r]_q& (M, \pi) \ar[r]& B}
\]
where $(X^\lin,\widetilde{\Omega}_X)\to B^\lin$ is the moment map of a free Hamiltonian $\cT_{B^\lin}$-space and (\ref{eq-q-lin}) is the resulting Hamiltonian quotient. 

% \begin{corollary}
% \label{cor:isotropic:proper:3}
% Let $q: (X, \Oga)\to (M, \pi)$ be proper isotropic realization of a Poisson manifold with $\cF_{\pi}$ of proper type. 
% Then the linear holonomy cover $M^\lin$ is a Poisson manifold with symplectic foliation $\cF^\lin_\pi$ and $X^\lin:=X\times_M M^\lin$ yields a $\Gamma^\lin$-equivariant proper isotropic realization
% \begin{equation}
% \label{eq-q-lin} 
% q^{lin}: (X^{\lin}, \widetilde{\Omega}_X)\to (M^{\lin}, \widetilde{\pi})
% \end{equation}
% which is simple. 
% \end{corollary}

%%%%%%%%%%%%%%%%%%%%%%%
%%%%%%%%%%%%%%%%%%%%%%%
%%%%%%%%%%%%%%%%%%%%%%%
%%%%%%%%%%%%%%%%%%%%%%%
%%%%%%%%%%%%%%%%%%%%%%%
\subsection{Isotropic realizations and Morita equivalence}
\label{sec:isotropic:Morita} 
%%%%%%%%%%%%%%%%%%%%%%%
%%%%%%%%%%%%%%%%%%%%%%%
%%%%%%%%%%%%%%%%%%%%%%%
%%%%%%%%%%%%%%%%%%%%%%%
%%%%%%%%%%%%%%%%%%%%%%%

The symplectic groupoids arising from proper isotropic realizations form a rather special class among all proper symplectic groupoids. A first illustration is the following result, where we use the presymplectic version of Morita equivalence (see Appendix \ref{appendix:moment:maps}):

\begin{proposition}
\label{prop:isotropic:Morita}
If $q:(X,\Oga)\to (M,\pi)$ is a proper isotropic realization then
$X\times_M\Hol(M, \cF_\pi)$, endowed with the pull-back of $\Oga$ by the first projection, defines a presymplectic Morita equivalence between the symplectic groupoid $\Hol_X(M, \pi)$ and the presymplectic groupoid integrating the Dirac structure $L_{\cF_{\pi}}$:
\[
\xymatrix{
 \Hol_X(M, \pi) \ar@<0.25pc>[d] \ar@<-0.25pc>[d]  & \ar@(dl, ul) & X\underset{M}{\times}\Hol(M, \cF_\pi)\ar[dll]^-{q\circ \pr_1}\ar[drr]_-{s\circ \pr_2}  & \ar@(ur, dr) & \cT_{X}\Join \Hol(M, \cF_\pi) \ar@<0.25pc>[d] \ar@<-0.25pc>[d] \\
(M,\pi)&  & & & (M,L_{\cF_{\pi}})}
\]
\end{proposition}

\begin{proof} For the groupoid on the right hand side we use the description given by \eqref{Join-gpd}, while for the groupoid $\Hol_X(M, \pi)$ we use \eqref{eq-G(X,M)-alt}. The two left and right actions are defined by
\[ [v, \gamma_0, u]\cdot (u, \gamma_1)= (v, \gamma_0\cdot \gamma_1), \quad (u, \gamma_1)\cdot  (\lambda, \gamma_2)= (\hol_{\gamma_1}(\lambda)u, \gamma_1\cdot \gamma_2) .\]
respectively. The actions obviously commute and it is straightforward to check that they are principal, so they define a Morita equivalence. 

We are left with proving that the actions are presymplectic. We will use the abbreviated notation $\Hol= \Hol(M, \cF_{\pi})$. For the right action, we have to check an equality of forms on the fiber product
\[ X\times_M\Hol\times_M\cT_X\times_M\Hol=\{ \left( (u, \gamma_1), (\lambda, \gamma_2)\right) : q(u)= t(\gamma_1), s(\gamma_1)= t(\gamma_2)= p(\lambda) \}.\]
We denote by $ \Hol^{(2)}$ the space of pairs of composable arrows in $\Hol$ and by letting $\eta= \hol_{\gamma_1}(\lambda)$ we reparametrize this fiber product space as
\[ X\times_{M}\cT_{X}\times_{M} \Hol^{(2)}= \{(u, \eta, \gamma_1, \gamma_2): q(u)= p(\eta)= t(\gamma_1)\}.\]
The two projections and the multiplication corresponding to the right action on the bibundle, on this new space become:
\begin{itemize}
\item $\pr_1:X\times_{M}\cT_{X}\times_{M} \Hol^{(2)}\to X\times_{M}\Hol$, $(u, \eta, \gamma_1, \gamma_2) \mapsto (u, \gamma_1)$. We need to consider the pull-back by this map of the form $\pr_{X}^{*}\Oga$.
\item $m:X\times_{M}\cT_{X}\times_{M} \Hol^{(2)}\to X\times_{M}\Hol$, $(u, \eta, \gamma_1, \gamma_2) \mapsto (\eta\cdot u, \gamma_1\cdot \gamma_2)$. We need to consider the pull-back by this map of the form  $\pr_{X}^{*}\Oga$. 
\item $\pr_2:X\times_{M}\cT_{X}\times_{M} \Hol^{(2)}\to \cT_{X}\Join \Hol$, $(u, \eta, \gamma_1, \gamma_2) \mapsto (\hol_{\gamma_1}^{-1}(\eta), \gamma_2)$.
We need to consider the pull-back by this map of the form  $\pr_{\cT_{X}}^{*}\omega_{\cT}$. However, due to (\ref{action-Hol-on-T}),
the same result is obtained if one pulls-back $\omega_{\cT}$ via $(u, \eta, \gamma_1, \gamma_2)\mapsto \eta$. 
\end{itemize} 
Combining these three terms, we find that the equation
\[ m^*\pr_{X}^{*}\Oga =\pr_1^*\pr_{X}^{*}\Oga+\pr_2^*\pr_{\cT_{X}}^{*}\omega_{\cT} \]
reduces to (\ref{action-T-on-X}), so the right action is presymplectic.

For the left action one needs to check that:
\[ m^*\pr_{X}^{*}\Oga=\pr_1^*\Omega+\pr_2^*\pr_{X}^{*}\Oga. \]
If one pulls back both sides of this equation along the submersion:
\begin{align*} 
X\times_M\Hol\times_M X\times_M X &\to \Hol_X\times_M (X\times_M \Hol)\\ 
(v, \gamma_0, u, \gamma_1)&\mapsto ([v, \gamma_0, u], (u, \gamma_1))
\end{align*} 
one obtains an obvious identity, so the the left action is presymplectic. 
\end{proof}

One can pass from presymplectic to symplectic Morita equivalences by restricting to complete transversals to the symplectic foliation. We obtain:

\begin{corollary}
\label{cor:Morita:complete:transversal}
If $q:(X,\Oga)\to (M,\pi)$ is a proper isotropic realization, for each choice of a complete transversal $T\subset M$ to $\cF_\pi$ there is a symplectic Morita equivalence:
\[
\xymatrix@R=20 pt@C=20pt{
 \Hol_{X}(M) \ar@<0.25pc>[d] \ar@<-0.25pc>[d]  & \ar@(dl, ul) & X\underset{M}{\times}\Hol(M, \cF_\pi)(-,T)\ar[dll]\ar[drr]  & \ar@(ur, dr) &  \cT_X|_T\Join \Hol(M, \cF_\pi)|_T\ar@<0.25pc>[d] \ar@<-0.25pc>[d] \\
(M,\pi)&  & & & (T,\pi\equiv 0)}
\]
\end{corollary}

The previous two results show that the groupoids associated with proper isotropic realizations are very special: they are Morita equivalent to torus bundles over the leaf space. Moreover, their symplectic Morita equivalence class does not depend on $X$ but only on the transverse integral affine structure it defines on $(M, \cF_{\pi})$.

When the foliation $\cF_\pi$ is proper we can do better and pass to torus bundle over classical orbifolds (see Corollary  \ref{cor:isotropic:proper:1}).
% we conclude that the groupoids associated with proper isotropic realizations are Morita equivalent to a torus bundle over a classical orbifold (see Corollary  \ref{cor:isotropic:proper:1}).
% In this case, one can do even better because the leaf space $B=M/\cF_\pi$ is a good orbifold. 
Here, as for Corollary \ref{cor:isotropic:proper:3}, we appeal again to Appendix \ref{appendix:molino}; the Morita point of view leads to an improvement of the corollary. More precisely, 
if we consider the linear holonomy cover $B^{\lin}$, together with the action of the linear holonomy group $\Gamma^{\lin}$,  we see that one has not only the symplectic torus bundle $\cT_{B^{\lin}}$ corresponding to the integral affine structure, but also the semi-direct groupoid
(itself a symplectic groupoid!)
\[ \cT_{B^{\lin}}\Join \Gamma^{\lin} \tto B^{\lin}.\]
% This is a symplectic groupoid, and we obtain (compare with Corollary \ref{cor:isotropic:proper:3}):

\begin{corollary}\label{cor:isotropic:proper:4} If $\cF_{\pi}$ is of proper type then, for any proper isotropic realization $q: (X, \Oga)\to (M, \pi)$, the symplectic holonomy groupoid $\Hol_X(M, \pi)$ is Morita equivalent to the symplectic groupoid $\cT_{B^{\lin}}\Join \Gamma^{\lin}\tto B^{\lin}$ associated with the linear holonomy cover of the classical orbifold $B=M/\cF_\pi$.  In particular, there is a 1-1 correspondence:
\[ 
\left\{\txt{Hamiltonian\\ $\Hol_X(M, \pi)$-spaces\\ \,}\right\}
\stackrel{1-1}{\longleftrightarrow}
\left\{\txt{$\Gamma^{\lin}$-equivariant\\ Hamiltonian $\cT_{B^{\lin}}$-spaces\\ \,} \right\}
\]
\end{corollary} 

The advantage of passing to $B^{\lin}$ instead of restricting to transversals is that the construction 
is choice-free and the base $B^{\lin}$ remains connected if $M$ is connected. 

%%%%%%%%%%%%%%%%%%%%%%%%%%%%%%
%%%%%%%%%%%%%%%%%%%%%%%%%%%%%%
\subsection{Isotropic realizations and $\cE$-integrations}
\label{Isotropic realizations and integrations}
%%%%%%%%%%%%%%%%%%%%%%%%%%%%%%
%%%%%%%%%%%%%%%%%%%%%%%%%%%%%%

If we start from a proper isotropic realization $q: (X, \Oga)\to (M, \pi)$, the corresponding holonomy groupoid fits into a short exact sequence (see Theorem \ref{thm:isotropic:fib:grpd}):
\[ \xymatrix{0\ar[r] & \cT_{X} \ar[r] & \HolX \ar[r] & \Hol(M,\cF_\pi)\ar[r] &0}.\]
One the other hand, if we consider an s-connected, proper symplectic integration $(\cG,\Omega)\tto (M,\pi)$ then we have the short exact sequence:
\[ \xymatrix{0\ar[r] & \cT(\cG)\ar[r] & \G \ar[r] & \cBG(\G) \ar[r] &0}, \]
where $\cT(\cG)=\nu^*(\cF_\pi)/\Lambda$ and $\cBG(\G)\tto M$ is a proper foliation groupoid integrating $\cF_\pi$. In general, 
the groupoid $\cBG(\cG)$ will be a larger integration than $\Hol(M,\cF_\pi)$. 

This shows that it is not possible to obtain every s-connected, proper symplectic integration $(\cG,\Omega)\tto (M,\pi)$ as the holonomy 
groupoid relative to an isotropic realization of $(M,\pi)$. However, one can extend the construction of $\HolX$,  replacing $\Hol(M, \pi)$ by \emph{any} foliation groupoid integrating $\cF_\pi$, even non s-connected integrations, if we adopt the identification (\ref{pull-back-gpd-isotr}) as definition of $\Hol_X(M, \pi)$:

\begin{definition}
\label{def:E:integration}
For any proper isotropic realization $q: (X, \Oga)\to (M,  \pi)$ and any integration $\cE\tto M$ of $(M, \cF_{\pi})$ the {\bf $\cE$-integration of $(M, \pi)$ relative to $X$} is  
\[ \cE_{X}(M, \pi):= \left( X\times_{M} \cE\times_{M} X\right) /\cT_X \tto M \]
with the symplectic structure induced from the 2-form $\widetilde{\Omega}:=\pr_{1}^{*}\Oga- \pr_{3}^{*}\Oga$.
\end{definition}

This definition recovers $\HolX$ when $\cE=\Hol(M,\cF_\pi)$. However, one should be aware that when $\cE=\Mon(M,\cF_\pi)$ the groupoid $\cE_{X}(M, \pi)$ and the Weinstein groupoid $\Sigma(M,\pi)$, in general, do not coincide, since $q^*\Mon(M,\cF_\pi)$ may be very different from $\Mon(q^*\cF_\pi)=\Mon((\ker\d q)^\perp)$.

\begin{remark} 
When the foliation groupoid $\cE$ is s-connected we can still define the $\cE$-integration of $(M, \pi)$ 
relative to $X$ geometrically, as in Definition \ref{def:hol-sympl-gpd}: the pullback groupoid $q^*\cE$ is also an s-connected integration of $(\ker\d q)^\perp$, and one can define cotangent $\cE$-equivalent paths, in a manner analogous to the way we defined cotangent holonomic paths, by requiring their horizontal lifts to induce the same element in the pull-back $q^*\cE$. However, this geometric definition fails in general, and later we do have to deal with non s-connected foliation groupoids.
\end{remark}

Exactly the same arguments as before imply the following more general version of Theorem \ref{thm:isotropic:fib:grpd}:

\begin{theorem}\label{thm:X-hol-E-version}
For any proper isotropic realization $q: (X, \Oga)\to (M, \pi)$ and any foliation groupoid $\cE$ integrating $(M,\cF_\pi)$:
\begin{enumerate}[(i)]
\item $\cE_{X}(M, \pi)$ is an an $X$-compatible symplectic integration of $(M, \pi)$.
\item $\cE_{X}(M, \pi)$ is an s-connected, proper Lie groupoid iff $\cE$ is s-connected and proper.
\end{enumerate}
Moreover, one has a short exact sequence of Lie groupoids:
\[ \xymatrix{0\ar[r] & \cT_{X} \ar[r] & \cE_{X}(M, \pi) \ar[r] & \cE\ar[r] &0}.\]
\end{theorem}

When $\cF_\pi$ is proper we still find:
\begin{itemize}
\item the transverse integral affine structure defined by $\cE_{X}(M, \pi)$ coincides with the lattice $\Lambda_X$ defined by the proper isotropic realization.
\item the associated orbifold structure on $B= M/\cF_{\pi}$ is the one induced by $\cE$. 
\end{itemize}

Also, similar to Proposition \ref{prop:isotropic:Morita}, $\cT_X\Join \cE\tto M$ is a presympectic groupoid integrating the Dirac structure $L_{\cF_\pi}$ associated with the foliation $\cF_\pi$, and one obtains:

\begin{proposition}
\label{prop:Morita:E:holonomy}
If $q:(X,\Oga)\to (M,\pi)$ is a proper isotropic realization and $\cE$ is a foliation groupoid integrating $(M,\cF_\pi)$, there is a presymplectic Morita equivalence:
\[
\xymatrix@R=15 pt@C=20pt{
 \cE_{X}(M, \pi) \ar@<0.25pc>[d] \ar@<-0.25pc>[d]  & \ar@(dl, ul) & X\underset{M}{\times}\cE\ar[dll]\ar[drr]  & \ar@(ur, dr) & \cT_X\Join \cE\ar@<0.25pc>[d] \ar@<-0.25pc>[d] \\
(M,\pi)&  & & & (M,L_{\cF_\pi})}
\]
\end{proposition}

By restricting to complete transversals $T\subset M$ to $\cF_\pi$, one can obtain symplectic Morita equivalences, leading to the analogue of Corollary \ref{cor:Morita:complete:transversal}.

%%%%%%%%%%%%%%%%%%%%%%%
%%%%%%%%%%%%%%%%%%%%%%%
%%%%%%%%%%%%%%%%%%%%%%%
%%%%%%%%%%%%%%%%%%%%%%%
%%%%%%%%%%%%%%%%%%%%%%%
%%%%%%%%%%%%%%%%%%%%%%%%%%%%%%%%
\subsection{The Dazord-Delzant class}
\label{sec:Dazord:Delzant}
%%%%%%%%%%%%%%%%%%%%%%%%%%%%%%%%
%%%%%%%%%%%%%%%%%%%%%%%
%%%%%%%%%%%%%%%%%%%%%%%
%%%%%%%%%%%%%%%%%%%%%%%
%%%%%%%%%%%%%%%%%%%%%%%
%%%%%%%%%%%%%%%%%%%%%%%

Given a regular Poisson manifold $(M, \pi)$, we have seen that any proper isotropic realization $q: (X, \Oga)\to (M, \pi)$ defines
a transverse integral affine structure $\Lambda_X$ on $(M,\cF_{\pi})$. We now address the converse problem. Namely, given:
\begin{itemize}
\item $(M, \pi)$ a Poisson manifold;
\item $\Lambda \subset \nu^*(\cF_\pi)$ a transverse integral affine structure;
\end{itemize}
is there a proper isotropic realization $q: (X, \Oga) \to (M, \pi)$ defining $\Lambda$? This leads to a cohomology class, 
which is essentially due to Dazord and Delzant \cite{DaDe}, and which we will call the Dazord-Delzant class associated to the data $(M, \pi, \Lambda)$. 

In order to describe this class, we will have to deal with sheaf cohomology. For a bundle $E\to M$ we will denote by $\underline{E}$
the associated sheaf of sections. For example,  $\underline{\cT}$ denotes the sheaf of section of the torus bundle $\cT\to M$
and  we will also denote by $\underline{\cT}_\cl$ the sheaf of closed sections, so for an open set $V\subset M$ we have:
\[ \underline{\cT}_\cl(U)=\{\alpha\in\underline{\cT}(U): \alpha^*\omega_\can=\d\alpha=0\}. \]

Let us denote by $(\Omega^\bullet(M,\cF_\pi),\d)$ the complex of forms whose pullback to the leaves of $\cF_\pi$ vanish. We have the short exact of sequence of sheaves (Poincar\'e Lemma):
\begin{equation}
\label{eq:Poincare}
\xymatrix{0\ar[r] & \Omega^\bullet_\cl(M,\cF_\pi) \ar[r] & \Omega^\bullet(M,\cF_\pi)\ar[r]^---{\d} &\Omega^{\bullet+1}_\cl(M,\cF_\pi)\ar[r] & 0}
\end{equation}
where $\Omega^\bullet_\cl(M,\cF_\pi)$ denotes the sheaf of closed forms. Notice that:
\[ \Omega^1(M,\cF_\pi)=\underline{\nu^*(\cF_\pi)}, \] 
and we can view $\Lambda\subset \Omega^1_\cl(M,\cF_\pi)$ as a locally constant subsheaf. If $\cT:=\nu^*(\cF_\pi)/\Lambda$, this leads to an identification of quotient sheaves:
\[ \underline{\cT}=\Omega^1(M,\cF_\pi)/\Lambda,\quad \underline{\cT}_\cl=\Omega^1_\cl(M,\cF_\pi)/\Lambda, \]
and to the short exact sequence of sheaves:
\[ \xymatrix{0\ar[r] & \underline{\cT}_\cl\ar[r] &\underline{\cT} \ar[r]^----{\d} &\Omega^2_\cl(M,\cF_\pi) \ar[r] & 0} \]
We are interested in the following piece of the corresponding long exact sequence in sheaf cohomology:
\begin{equation}
\label{eq:exa:sequence:obst}
\xymatrix{\ar[r] &H^1(M,\underline{\cT})\ar[r] &H^1(M,\Omega^2_\cl(M,\cF_\pi)) \ar[r]^---{\delta} &  H^2(M,\underline{\cT}_\cl)\ar[r] &}
\end{equation}

Now observe that associated to the leafwise symplectic form $\omega_{\cF_\pi}$ there is a class which is the obstruction to extend 
$\omega_{\cF_\pi}$ to a closed two 2-form on $M$: if $\{U_a\}_{a\in A}$ is a good cover of $M$ then we can find closed 2-forms $\omega_a\in\Omega^2(U_a)$ which locally extend $\omega_{\cF_\pi}$ and on $U_{ab}=U_a\cap U_b$ their difference is closed and vanishes on the symplectic leaves, so defines a 2-cycle $\omega_{ab}:=\omega_a-\omega_b\in \Omega^2_\cl(U_{ab},\cF_\pi)$.
Hence, we can set 
\[ \xi(M,\pi):=[\omega_{ij}]\in H^1(M,\Omega^2_\cl(M,\cF_\pi)). \]

\begin{remark}
Since $\Omega^\bullet(M,\cF_\pi)$ is a fine sheaf, the sequence \eqref{eq:Poincare} gives:
\[  
H^1(M,\Omega^2_\cl(M,\cF_\pi))\simeq \frac{H^0(M,\Omega^3_\cl(M,\cF_\pi))}{\d H^0(M,\Omega^2_\cl(M,\cF_\pi))}=H^3(M,\cF_\pi). 
\]
Under this isomorphism, the class $\xi(M,\pi)$ corresponds to the class $[\d\widetilde{\omega}]$ where $\widetilde{\omega}\in\Omega^2(M)$ is any 2-form extending the leafwise symplectic form $\omega_{\cF_\pi}$.
\end{remark}

\begin{definition}\label{def-DazDel}
If $(M,\pi)$ is a regular Poisson structure and $\Lambda\subset\nu^*(\cF_\pi)$ is a transverse integral affine structure, the {\bf Dazord-Delzant class} $\dd(M,\pi,\Lambda)$ is the image of $\xi(M,\pi)$ under the connecting homomorphism \eqref{eq:exa:sequence:obst}:
\[ \dd(M,\pi,\Lambda):=\delta(\xi(M,\pi))\in H^2(M,\underline{\cT}_\cl), \]
\end{definition}

Notice that an explicit 2-cocycle representing the obstruction class can be obtained by considering a good cover $\{U_a\}$ of $M$ and local extensions $\omega_a\in\Omega^2(U_a)$ of the leafwise symplectic form $\omega_{\cF_\pi}$, so that the difference $\omega_{ab}=\omega_a-\omega_b$ is exact:
\[ \omega_{ab}=\d \alpha_{ab},\quad (\alpha_{ab}\in\Omega^1(U_{ab},\cF_\pi)). \]
Then $\alpha_{ab}+\alpha_{bc}+\alpha_{ca}$ is a closed 1-form in $U_{abc}:=U_a\cap U_b\cap U_c$ vanishing on the leaves, and its projection gives a 2-cocycle representing $c_2(M,\pi,\Lambda)$:
% \begin{equation}
% \label{eq:obstr:class:cocycle}
\[
c_{abc}=[\alpha_{ab}+\alpha_{bc}+\alpha_{ca}]\in\Omega^1_\cl(U_{abc},\cF_\pi)/\Lambda. 
\]
% \end{equation}

The following result is essentially due to Dazord and Delzant:

\begin{theorem}[\cite{DaDe}]\label{thm:Dazord:Delzant}
The class $\dd(M,\pi,\Lambda)$ vanishes if and only if $\Lambda$ is defined by a proper isotropic realization $q:(X,\Oga)\to (M,\pi)$. 
\end{theorem}

The proof of this result can be split into two steps: 
\begin{enumerate}[(i)]
\item $\dd(M,\pi,\Lambda)=0$ if and only if there is a class $c_1\in H^1(M,\underline{\cT})$ which is mapped to $\xi(M,\pi)$ in the long exact sequence \eqref{eq:exa:sequence:obst};
\item  Given any class  $c_1\in H^1(M,\underline{\cT})$ there is a principal $\cT$-bundle $q:X\to M$ whose Chern class is $c_1$. This bundle has a symplectic structure $\Oga$ making $q$ into a symplectic complete isotropic fibration inducing $\pi$ on $M$ if and only if $c_1$ is mapped to $\xi(M,\pi)$ in the long exact sequence \eqref{eq:exa:sequence:obst}.
\end{enumerate}
We refer to \cite{DaDe} for the details.
\vskip .1 in

To our knowledge, it is an open problem to give an example of a Poisson manifold $(M,\pi)$ with Dazord-Delzant class
$\dd(M,\pi,\Lambda)\not=0$, where $\Lambda$ is a lattice defined by some proper integration $(\cG,\Omega)\tto (M,\pi)$. 

\begin{remark}[The Dirac setting]
\label{rmk:isotropic:proper:dirac}
The results on proper isotropic realizations extend to twisted Dirac structures with the appropriate modifications.
Now, the central notion is that of a {\bf presymplectic realization} (see \cite{BCWZ}):
\[q:(X,\Oga)\to (M,L,\phi),\]
where $\w$ is a 2-form such that $d\Oga+q^*\phi=0$, $q$ is a f-Dirac map, and one requires the non-degeneracy condition:
% \begin{equation}\label{eq:non:deg:Dirac}
\[
\Ker(\Oga)\cap \Ker(\d q)=\{0\}.
\]
%  \end{equation}
One defines $\HolXD$, the holonomy groupoid relative to the presymplectic realization $q:(X,\Oga)\to (M,L,\phi)$, 
as the quotient of cotangent paths modulo cotangent holonomy rel $X$ as in Definition
\ref{def:hol-sympl-gpd}. Then all fundamental properties of $\HolX$ still hold in the presymplectic setting, namely Theorem \ref{thm:isotropic:fib:grpd}, Proposition \ref{prop:simple:isotropic:realiz} and Corollary \ref{cor:isotropic:proper:4}. In the proofs one must use the (twisted) presymplectic version of Hamiltonian $\cG$-spaces 
(see Appendix \ref{appendix:moment:maps}). 

The Dazord-Delzant theory also extends in a more or less straightforward way to twisted Dirac structures
(see \cite{SaSe} for the case of twisted Poisson structures).
\end{remark}

\newpage

%%%%%%%%%%%%%%%%%%%%%%%%
%%%%%%%%%%%%%%%%%%%%%%%%
%%%%%%%%%%%%%%%%%%%%%%%%
%%%%%%%%%%%%%%%%%%%%%%%%
%%%%%%%%%%%%%%%%%%%%%%%%
%%%%%%%%%%%%%%%%%%%%%%%%
%%%%%%%%%%%%%%%%%%%%%%%%
%%%%%%%%%%%%%%%%%%%%%%%%
%%%%%%%%%%%%%%%%%%%%%%%%
%%%%%%%%%%%%%%%%%%%%%%%%
%%%%%%%%%%%%%%%%%%%%%%%%
%%%%%%%%%%%%%%%%%%%%%%%%
%%%%%%%%%%%%%%%%%%%%%%%%
%%%%%%%%%%%%%%%%%%%%%%%%
\section{Symplectic gerbes over manifolds}
\label{sec:gerbes:manifolds}
%%%%%%%%%%%%%%%%%%%%%%%%
%%%%%%%%%%%%%%%%%%%%%%%%
%%%%%%%%%%%%%%%%%%%%%%%%
%%%%%%%%%%%%%%%%%%%%%%%%
%%%%%%%%%%%%%%%%%%%%%%%%
%%%%%%%%%%%%%%%%%%%%%%%%
%%%%%%%%%%%%%%%%%%%%%%%%
%%%%%%%%%%%%%%%%%%%%%%%%
%%%%%%%%%%%%%%%%%%%%%%%%
%%%%%%%%%%%%%%%%%%%%%%%%
%%%%%%%%%%%%%%%%%%%%%%%%
%%%%%%%%%%%%%%%%%%%%%%%%
%%%%%%%%%%%%%%%%%%%%%%%%
%%%%%%%%%%%%%%%%%%%%%%%%
In the previous sections we have seen that Poisson manifolds of proper type %, with a choice of proper integration, give rise to various transversal structures. 
come with a rich transverse geometry.
In particular, the leaf space is an integral affine orbifold. We now fix an integral affine orbifold $(B, \Lambda)$ and we investigate the freedom one has in building Poisson manifolds of proper type with leaf space $(B, \Lambda)$. This problem is really about constructing (symplectic) groupoid extensions with kernel the torus bundle $\cT$ associated to $\Lambda$, and hence resembles the standard theory of $\S^1$-gerbes (\cite{BeXu,Lupercio, Moe03, Murray, Tu}). 

Recall that an $\S^1$-gerbe is a higher version of the notion of principal $\S^1$-bundle over $B$. While principal $\S^1$-bundles are classified by their Chern class $c_1\in H^2(B, \Z)$, $\S^1$-gerbes are classified by a similar class $\DD\in H^3(B,\Z)$, called the \emph{Dixmier-Douady class}. We will introduce a symplectic variant of the theory, consisting of \emph{symplectic $\cT_{\Lambda}$-gerbes} over $B$ and we will show that they are classified by their \emph{Lagrangian Dixmier-Douady class} living in $H^{2}(B, \underline{\cT}_{\Lagr})$. 

To achieve this, we will need the following variations of the standard theory of $\S^1$-gerbes:
\begin{enumerate}[\hspace{25pt}]
\item[{\bf v1:}]{\bf Replace $\S^1$ by general torus bundles $\cT$:} this is straightforward, but note that, while $\S^1$-gerbes arise as higher versions of principal $\S^1$-bundles, in the process of passing from $\S^1$ to torus bundles, principal $\S^1$-bundles will be replaced by $\cT$-torsors, a special class of principal $\cT$-bundles. Therefore, $\cT$-gerbes arise as higher versions of $\cT$-torsors; 
\item[{\bf v2:}]{\bf  A symplectic version of the theory:} this is the main novelty of our story. Remarkably, the lower version of the theory, the symplectic story, i.e. symplectic $\cT$-torsors, as well as its relevance to Lagrangian fibrations, has already appeared in \cite{Sjamaar};
\item[{\bf v3:}]{\bf  Gerbes over orbifolds:} although a large part of our discussion will be carried out in the case where $B$ is a smooth manifold, in general our leaf spaces $B$ are orbifolds. The passage from manifolds to orbifolds will be based again on Haefliger's philosophy (Remark \ref{rmk-Transversal geometric structures}). Note that $\S^1$-gerbes over orbifolds have already been considered- see e.g. \cite{Lupercio}.
\end{enumerate}
Some of the generalizations of $\S^1$-gerbes that we will consider could, in principle, be obtained by making use
of the general theory of gerbes with a given ``band'' over general ``sites" \cite{Breen,Giraud,LaStXu}. However, since all our bands will be abelian, and the most general sites we we need are the ones associated to orbifolds, we do not have to appeal to the general theory. 
And, more importantly, our symplectic gerbes can always be represented by extensions of Lie groupoids (as in, e.g., \cite{Murray, Lupercio, BeXu, Tu}).

In this Section we take care of symplectic gerbes over smooth manifolds, leaving the orbifold case for the next section. Regarding PMCTs, the outcome can be stated in a simplified form as follows:

\begin{theorem} \label{thm:DD-gerbe-smooth-intro}
Assume that $(M, \pi)$ is a Poisson manifold of proper type whose foliation $\cF_\pi$ has 1-connected leaves. Then 
each proper integration $(\cG, \Omega)\tto (M,\pi)$ gives rise to a symplectic gerbe over $B=(M/\cF_\pi,\Lambda_\cG)$, which is classified by a class
\[ \DD(\cG, \Omega)\in H^2(B, \ucT_{\Lagr}) .\]
Moreover, its pull-back via the projection $p:M\to B$ is precisely 
the Dazord-Delzant class $c_2(M,\pi,\Lambda_\cG)$ (Definition \ref{def-DazDel}). Furthermore, $\DD(\cG, \Omega)$ vanishes iff $\cG$ is the gauge groupoid of a free Hamiltonian $\cT$-space $q: (X, \Oga)\to B$ (Appendix \ref{cG-red2}).
\end{theorem}

%%%%%%%%%%%%%%%%%%%%%%%%
\subsection{Symplectic torsors}
\label{ssec:torsors}
%%%%%%%%%%%%%%%%%%%%%%%%
Since $\S^1$-gerbes are higher versions of principal $\S^1$-bundles, in order to get ready to deal with symplectic gerbes,
we first discuss how to implement the variations that we mentioned in the case of principal $\S^1$-bundles.
Here we concentrate on variations {\bf v1} and {\bf v2}. The formalism necessary for passing to orbifolds will be discussed in the next section. 

\subsubsection{From $\S^1$ to general torus bundles} 
\label{ssec:to general torus bundles}
We would like to replace $\S^1$ by a general torus bundle $\cT\to B$ and principal
$\S^1$-bundles by ``principal $\cT$-bundles''. Some care is needed, since we will not be dealing with general principal $\cT$-bundles.

Recall that a (right) principal $\cH$-bundle over $M$, for any Lie groupoid $\cH\tto N$ over some other manifold $N$, consists
of a bundle $P$ with two maps: the bundle projection $\pp: P\to M$, as well as a map $\qq: P\to N$ along which $\cH$ acts on $P$:
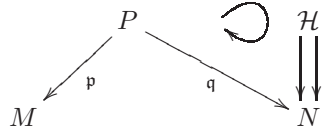
\[
\xymatrix{
 & P \ar[dl]^-{\pp}\ar[drr]_{\qq} & \ar@(ur, dr) & \cH \ar@<0.25pc>[d] \ar@<-0.25pc>[d] \\ M  & & & N}
\]
When $M= N= B$ and $\cG= \cT$ is a torus bundle over $B$, principal $\cT$-bundles over $B$ still come with two maps, $\pp$ and
$\qq$, which need not coincide. We will be interested in principal $\cT$-bundles over $B$, which in addition satisfy $\pp=\qq$. These
will be called {\bf $\cT$-torsors}. To distinguish them from other types of (groupoid) principal bundles, we will denote them by the letter $X$. 

Hence a $\cT$-torsor over $B$ is a manifold $X$ endowed with a (right) action of $\cT$ along a submersion $p_X: X\to B$, along which $\cT$ acts fiberwise,
freely, and transitively:
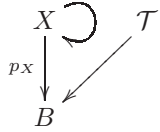
\[
\xymatrix{
 X  \ar[d]_-{p_X}  \ar@(ur, dr) & \cT\ar[dl] \\
 B  & }
\]

The {\bf fusion product} of two $\cT$-torsors $X_1$ and $X_2$ is the new $\cT$-torsor
\begin{equation}\label{fusion-T-bundles} 
X_1\star X_2:= (X_1\times_{B} X_2)/\cT,
\end{equation}
where $\cT$ acts on the fibered product by $(u_1, u_2)\cdot \lambda = (u_1\cdot \lambda, u_2\cdot \lambda)$, and where the action of $\cT$ on $X_1\star X_2$ is induced by the action on the second factor. 
Modulo the obvious notion of isomorphism, one obtains  an abelian group $\Tor_{B}(\cT)$.
%Factoring by the obvious notion of isomorphism of $\cT$-torsors, one obtains an abelian group which we will denote by $\Tor_{B}(\cT)$. 

\begin{remark}[$B$-fibered objects]\label{rk:B-fibered} A very useful interpretation to keep in mind of the condition $\pp=\qq$, distinguishing $\cT$-torsors as a particular class of principal $\cT$-bundle, is the following. The base $B$ is fixed from the start and all the objects that one considers are ``fibered'' over $B$ or ``parametrized'' by $B$, i.e. come with a submersion onto $B$. One should refer to them as pairs $(N, p_{N})$ with $p_{N}: N\rightarrow B$, but we will simply say that {\bf $N$ is $B$-fibered} without further mentioning $p_{N}$. They form a category $\textrm{Man}_{B}$. 

The objects that we consider are $B$-fibered versions of standard objects, which maybe recovered by letting $B$ be a point. For instance, a torus bundle $\cT\to B$ is just a compact, connected, abelian group object in $\textrm{Man}_{B}$. More generally, a Lie groupoid fibered over $B$, i.e. a Lie groupoid in $\textrm{Man}_B$, is just a Lie groupoid $\cH\tto N$ together with a submersion $p_N: N\rightarrow B$ such that $p_N\circ s= p_N\circ t$. These should be thought of as families of groupoids parametrized by $b\in B$, namely the restrictions of $\cH$ to the fibers $p_{N}^{-1}(b)$. One also has a $B$-fibered version of principal $\cH$-bundles over a manifold $M$: it is a principal $\cH$-bundles $P$ as above for which $p_{M}\circ \pp= p_{N}\circ \qq(= p_{P})$:
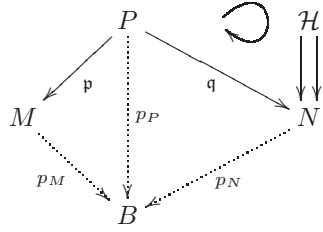
\[
\xymatrix{
 & P \ar[dl]^{\pp}\ar[drr]_{\qq} \ar@{.>}[dd]^-{p_{P}} & \ar@(ur, dr) & \cH \ar@<0.25pc>[d] \ar@<-0.25pc>[d] \\ 
 M \ar@{.>}[dr]_-{p_M} & & & N\ar@{.>}[dll]^-{p_N}\\
 & B & & }
\]
When $M= N= B$, this precisely means that $\pp=\qq$. Hence, $\cT$-torsors are the same thing as $B$-fibered principal $\cT$-bundles over $B$.
\end{remark}

 If $\cH\tto N$ is a Lie groupoid we denote by $\Bun_{\cH}(M)$ the set of equivalence classes of principal $\cH$-bundles over $M$. The description of principal bundles based on transition functions, yields an isomorphism:
\begin{equation}\label{Haeflig-coh-int} 
\Bun_{\cH}(M) \cong \check{H}^1(M, \cH), 
\end{equation}
 where $\check{H}^1(M, \cH)$ denotes Haefliger's first \v{C}ech cohomology with values in the groupoid $\cH$ \cite{Hae,Ha58}:
 the \v{C}ech cocycles are families $g= \{V_i, g_{ij}\}$, where $\{V_i\}_{i\in I}$ is an open cover of $M$ and
 \begin{equation}\label{coc-H1-Haefl}
 g_{ij}:V_{ij}\to \G,\quad g_{ij}(x)\cdot g_{jk}(x)=g_{ik}(x), \ \forall x\in V_{ijk}.
 \end{equation}
 Two such cocycles $g$ and $h$ are cohomologous if (after eventually passing to a refinement) there exist $\lambda_i:V_i\to\cH$ such that:
 \[ h_{ij}(x)=\lambda_i(x)\cdot g_{ij}(x)\cdot \lambda_j(x)^{-1},\ \forall x\in V_{ij}. \]
 
One has a completely similar $B$-fibered version of the previous discussion, obtained by requiring that all the maps involved, including isomorphisms of principal bundles, \v{C}ech cocycles $\{g_{ij}\}$, etc., commute with the projections into $B$. One obtains a 
$B$-fibered version of (\ref{Haeflig-coh-int}): 
\begin{equation}\label{eq:c1-gen} 
\Bun_{\cH, B}(M)\cong \check{H}^{1}_{B}(M, \cH). 
\end{equation}In the case we are interested in, when $M= N= B$ and $\cG= \cT$, notice that 
 \[ \check{H}^{1}_{B}(B, \cT)= H^1(B, \ucT),\]
 the cohomology with coefficients in the sheaf $\ucT$ of sections of $\cT$.

% For the $B$-fibered version of Haefliger's cohomology one has an isomorphism:
% \begin{equation}\label{eq:c1-gen} 
% \Bun_{\cH, B}(M)\cong \check{H}^{1}_{B}(M, \cH). 
% \end{equation}
% One now requires that all the maps involved, including isomorphisms of principal bundles, \v{C}ech cocycles $\{g_{ij}\}$, etc., commute with the projections into $B$. In the case we are interested in, when $M= N= B$ and $\cG= \cT$, notice that 
% \[ \check{H}^{1}_{B}(B, \cT)= H^1(B, \ucT),\]
% the cohomology with coefficients in the sheaf $\ucT$ of sections of $\cT$. 
 
 One can use the exponential sequence of $\cT$ to pass to the associated lattice $\Lambda$:
  \begin{equation}\label{exp-seseq}
 \xymatrix{1\ar[r]& \Lambda \ar[r] & \mathfrak{t}  \ar[r]^-{\exp}& \cT \ar[r] & 1}, 
  \end{equation}
 Then  (\ref{eq:c1-gen}) becomes the Chern-class isomorphism
  \begin{equation}\label{c1-for-cT}
  c_1: \xymatrix{\Tor_{B}(\cT)\ar[r]^-{\sim} &  H^1(B, \ucT)\cong H^2(B, \Lambda)}
  \end{equation}
 associating to a $\cT$-torsor its Chern class.  Concretely, given a $\cT$-torsor $p: X\to B$, any open cover $\{V_i\}$ of $B$ with local sections $s_i: V_i\to X$, yields on overlaps
\begin{equation}
\label{eq-ef-lambdaij}
s_i(x)= s_{j}(x)\lambda_{ij}(x),\quad (x\in V_{ij}),
\end{equation}
 where the $\{\lambda_{ij}\}$ is a \v{C}ech 1-cocycle representing the Chern class of $X$. Similarly, the construction of the real representatives of the Chern class of principal $\S^1$-bundles via connections extends without any problem to the setting of $\cT$-torsors.

\subsubsection{A symplectic version of the theory:} There is very little left to be to be done to obtain the symplectic torsors:
\begin{itemize}
\item restrict to symplectic torus bundles $(\cT, \omega_{\cT})\to B$ (see Section \ref{Integral affine structures on manifold}), which are determined by integral affine structures $\Lambda$ on $B$ (cf.~Proposition \ref{prop:IAS-sympl-torus});
\item consider $\cT$-torsors $X$ endowed with a symplectic form $\Oga$ and require that the action of $\cT$ to be symplectic in the sense of Appendix \ref{appendix:moment:maps}. 
\end{itemize}
 The resulting objects $(X, \Oga)$ 
are called {\bf symplectic $(\cT,\w_\cT)$-torsors}. Note that for two such $(X_1, \Omega_1)$ and $(X_2, \Omega_2)$, their composition (\ref{fusion-T-bundles}) is again symplectic: the form $\pr_{1}^{*}\Omega_1- \pr_{2}^{*}\Omega_2$ on $X_1\times_{B} X_2$ descends to a symplectic form on $X_1\star X_2$. 
We denote by $\Tor_{B}(\cT, \omega_{\cT})$ the resulting group of symplectic torsors. 

Proposition \ref{pp:Lagr-fibr}, in this language, says that Lagrangian fibrations are the same thing as symplectic torsors:

\begin{proposition}\label{Lagr-fibr-as-torsors} For any symplectic $\cT$-torsor the projection into $B$ is a Lagrangian fibration. Conversely, any Lagrangian fibration over $B$ is a symplectic $\cT_{\Lambda}$-torsor, where $\Lambda$ is the integral affine structure induced by the fibration.
\end{proposition}

The construction of the Chern class class has a natural symplectic version, that dates back to Duistermaat's work on global action-angle coordinates \cite{Duist}. This was further clarified and generalized by Delzant and Dazord \cite{DaDe} and Zung \cite{Zu2}, and rephrased in the language of symplectic torsors by Sjamaar \cite{Sjamaar}. The relevant sheaf is no longer $\ucT$, but rather the subsheaf $\ucT_{\Lagr}$ of local Lagrangian sections:
% \begin{equation}\label{eq-cT-Lagr} 
\[
\ucT_{\Lagr}(U)=\{\alpha\in\underline{\cT}(U): \alpha^*\omega_\can=\d\alpha=0\}. 
\]
% \end{equation}
The {\bf Lagrangian Chern class} $c_1(X,\Oga)\in H^1(B, \ucT_{\Lagr})$ of a symplectic $\cT$-torsor $X$ is represented
by a \v{C}ech-cocycle $\{\lambda_{ij}\}$, constructed as before (see (\ref{eq-ef-lambdaij})), but using now 
local \emph{Lagrangian} sections $s_i:V_i\to X$. 

To realize the Lagrangian Chern class as a degree two cohomology class, one needs the symplectic analogue of the exact sequence (\ref{exp-seseq}). For that one considers
\[   \cO_{\Lambda}\subset \cO_{\Aff}\subset  \cO,\]
where $\cO$ is the sheaf of smooth functions on $B$, $\cO_{\Aff}$ is the subsheaf of affine functions (i.e. of type 
$x\mapsto  r+ \sum_{i} c^i x_i$ in integral affine charts), and $\cO_{\Lambda}$ is the subsheaf of integral affine functions (obtained by requiring $c^i\in \Z$ in the previous expressions). The DeRham differential gives the short exact sequence of sheaves:
\[ \xymatrix{1\ar[r]& \cO_{\Lambda}\ar[r] & \cO \ar[r]^-{\d}& \ucT_{\Lagr} \ar[r] & 1}, \] 
and this leads to the Chern class map for symplectic torsors (cf. (\ref{c1-for-cT})):
\[ c_1: \xymatrix{\Tor_{B}(\cT, \omega_{\cT})\ar[r]^-{\sim} &  H^1(B, \ucT_{\Lagr})\cong H^2(B, \cO_{\Lambda})}.\]

Finally, the forgetful map from $\Tor_{B}(\cT, \omega_{\cT})$ to $\Tor_{B}(\cT)$ corresponds to the map
\[ H^2(B, \cO_{\Lambda})\to H^2(B, \Lambda)\]
 induced by the DeRham differential $\d: \cO_{\Lambda}\to \Lambda$, which interpreted as a sheaf morphism gives rise to the short exact sequence:
\begin{equation}\label{eq:ses-O-Lambda} 
\xymatrix{1\ar[r]& \R\ar[r] & \cO_{\Lambda} \ar[r]^-{\d}& \Lambda \ar[r] & 1}. 
\end{equation}

For the later use we point out the following:

\begin{corollary}\label{lemma:help-c-constr} 
Let $\cT$ be a torus bundle over a manifold $B$. If $B$ is contractible, 
then any $\cT$-torsor is trivial (or, equivalently, admits a global section). The same holds for symplectic $\cT$-torsors (i.e, they admit Lagrangian sections over contractible open sets).
\end{corollary}

\begin{remark}
In the construction of the Chern class of a principal $\S^1$-bundle via connections, one passes from 
integral to real coefficients. Similarly, in our case, we have to pass from $\cO_{\Lambda}$ to $\cO_{\Aff}$ or, equivalently, from $\cT_{\Lagr}\cong \cO/\cO_{\Lambda}$ to $\cO/\cO_{\Aff}$. In other words, to find a representative of the Chern classes in terms of differential forms, we are interested in the image of $c_1(X, \Oga)$ under the map:
\begin{equation}\label{grps-for-real-Chern}
H^1(B, \ucT_{\Lagr})\to H^1(B,\cO/\cO_{\Lambda})\cong H^2(B, \cO_{\Aff}) .
\end{equation}
To work with these groups, it is useful to use 
the following fine resolution:
\[ \xymatrix{\cO/\cO_{\Aff} \ar[r] & \Omega^{1}_{\partial}(B, T^*B) \ar[r]^-{\d_{\Lambda}} & \Omega^{2}_{\partial}(B, T^*B) \ar[r]^-{\d_{\Lambda}} & \ldots }\]
where $\Omega^{k}_{\partial}(B, T^*B)$ is the kernel of the antisymmetrization map 
$\partial:  \Omega^{k}(B, T^*B) \to \Omega^{k+1}(B)$,
and where $\d_{\Lambda}$ is the covariant derivative induced by the flat connection associated to $\Lambda$. Hence, $ H^2(B, \cO_{\Aff})=H^2(\Omega_{\partial}^{\bullet}(B, T^*B), \d_{\Lambda})$. 
Now, given a symplectic $\cT$-torsor $(X, \Oga)$, one chooses a \emph{Lagrangian connection} 
\[ \theta\in \Omega^1(X, \mathfrak{t})\]
(i.e. the horizontal spaces that it defines are Lagrangian) and one observes that its curvature  
\[ k_{\theta}\in \Omega^2(X, \mathfrak{t})_{\Lambda-\bas}= \Omega^2(B, T^*B)\]
actually lives in $\Omega^{2}_{\partial}(B, T^*B)$. The resulting class $[k_\theta]\in H^2(B, \cO_{\Aff})$ is precisely the class induced by $c_1(X, \Oga)$ under (\ref{grps-for-real-Chern}).
\end{remark}

%%%%%%%%%%%%%%%%%%%%%%%%%%%%
\subsection{Gerbes and their Dixmier-Douady class}
%%%%%%%%%%%%%%%%%%%%%%%%%%%%
We are now ready to move to gerbes. Our exposition will be self-contained, 
overviewing the standard theory of $\S^1$-gerbes and explaining at the same time how to take care of {\bf v1},
replacing $\S^1$ by a general torus bundle. Throughout this section we fix the base manifold $B$ and all the objects that we will
consider will be fibered over $B$ (see  Remark \ref{rk:B-fibered}).
 
\subsubsection{Definition of a $\cT$-gerbe} Among the various approaches to gerbes, the most relevant 
one for us is via extensions of groupoids \cite{BeXu, Lupercio, Murray, Tu}:  an $\S^1$-gerbe over $B$ is, up to Morita equivalence, an $\S^1$-extension of $B$, where $B$ is interpreted as a the identitiy groupoid. The groupoids that are Morita equivalent to $B$ are the groupoids $M\times_{B}M$ associated to submersions $p_M:M\to B$, and this leads one to consider 
central $\S^1$-extension of groupoids over $M$:
\[ \xymatrix{1\ar[r]& \S^1_M\ar[r] & \cG \ar[r]& M\times_B M\ar[r] & 1}, \]
where $\S^1_M=M\times\S^1$ is the trivial $\S^1$-bundle over $M$. There is an appropriate notion of Morita equivalence between two such extensions (see below) and the resulting equivalence classes are called $\S^1$-gerbes over $B$.

Next we replace the trivial circle bundle $\S^1_B= B\times \S^1$ by a general bundle of tori $\cT\to B$. Therefore, we consider {\bf central groupoid extensions} of type
\begin{equation}\label{basic-ext}
\xymatrix{1\ar[r] & \cT_M \ar[r] & \cG \ar[r]& M\times_B M\ar[r] & 1}, 
\end{equation}
where $\cT_M=p_M^*\cT$. Here, by central we mean that:
\[ \lambda\cdot g= g\cdot \lambda\]
for any arrow $g: x\rightarrow y$ of $\G$ and $\lambda\in \cT_b$, with $b= p_M(x)= p_M(y)$, and where we use $(\cT_M)_x= \cT_b=(\cT_M)_y$.

\begin{example}\label{ex:gauge-extension}
Recall that any (right) principal $\cT$ bundle over $M$:
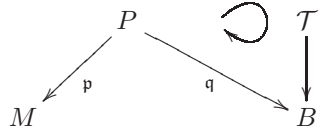
\[
\xymatrix{
  & P \ar[dl]^{\pp}\ar[drr]_{\qq} & \ar@(ur, dr) & \cT \ar[d] \\
M  & & & B}
\]
gives rise to the gauge groupoid $\Gauge{\cT}{P}$ (see Appendix \ref{appendix:moment:maps}). When $P$ is $B$-fibered, i.e when $p_{M}\circ \pp= p_{P}=\qq$, it follows that the gauge groupoid defines a central extension as above. Such {\bf gauge extensions} will soon be considered ``trivial".
\end{example}

All the groupoids appearing in the extension (\ref{basic-ext}) are fibered over $B$, and this is relevant for the right notion of equivalence. Two such extensions, associated with projections $p_i:M_i\to B$ and groupoids $\cG_i\tto M_i$, $i=1,2$, are said to be {\bf Morita equivalent extensions} if there exists a Morita $(\cG_1,\cG_2)$-bibundle $P$, in the sense of Section \ref{ssec:Morita}, such that:
\begin{enumerate}[(i)]
\item $P$ is a $B$-fibered Morita equivalence;
\item $P$ is central, i.e. the actions of $\cT_{M_i}$ on $P$ inherited from the $\cG_i$-actions coincide: $\lambda\cdot u= u\cdot \lambda$ for 
all $u\in P$, $\lambda\in \cT_b$, with $b= p_1(\qq_1(u))= p_2(\qq_2(u))$. 
\end{enumerate} 
The first condition says that the map induced by $P$ between the $\cG_i$-orbit spaces, i.e. $B$, is the identity. Or, with the notations from Section \ref{ssec:Morita}, that $p_1\circ \qq_1= p_2\circ \qq_2$. Moreover, this condition is used to make sense of the second condition. With these:

\begin{definition} \label{def-st-gerbes} Given a torus bundle $\cT$ over a manifold $B$, a {\bf $\cT$-gerbe over $B$} is a Morita equivalence class of extensions of type (\ref{basic-ext}).
\end{definition}

\subsubsection{The group of gerbes: fusion product}The set of $\cT$-gerbes over $B$ has an abelian group structure. 
It is based on the notion of {\bf fusion product} of two extensions $\cG_1$ and $\cG_2$ of type (\ref{basic-ext}), which is the extension associated with the submersion $p_{12}:M_{12}=M_1\times_B M_2\to B$:
\[
\xymatrix{1\ar[r] & \cT_{M_{12}} \ar[r] & \cG_1\star\cG_2\ar[r]& M_{12}\times_B M_{12}\ar[r] & 1},
\]
where the groupoid $\cG_1\star\cG_2\tto M_{12}$ has space of arrows:
\[ \cG_1\star\cG_2:= (\cG_1\times_B \cG_2)/\cT.\]
Here, the action of $\lambda\in\cT_b$ on a pair $(g_1,g_2)$ of arrows in the orbits of $\G_1$ and $\G_2$ corresponding to $b\in B$, is given by $(g_1,g_2)\cdot \lambda=(g_1\cdot \lambda, \lambda\cdot g_2)$.

The {\bf trivial $\cT$-gerbe} is the one represented by the trivial extension of the identity groupoid $B\tto B$ (hence $M= B$):
\begin{equation}\label{eq:trivial-gerbe}
\xymatrix{1\ar[r] & \cT \ar[r] & \cT \ar[r]& B \ar[r] & 1}.
\end{equation}

One checks easily that there are canonical isomorphisms of extensions:
\begin{align*} 
\cG\star\cT\cong\cT\star\cG\cong \cG,&\qquad
\cG_1\star\cG_2\cong \cG_2\star\cG_1,\\
(\cG_1\star\cG_2)\star \cG_3&\cong \cG_1\star(\cG_2\star \cG_3).
\end{align*}

The {\bf inverse $\cT$-gerbe} of the gerbe defined by an extension (\ref{basic-ext}) is represented by the opposite extension:
\[ \xymatrix{1\ar[r] & \cT_M \ar[r] & \cG^{\opp} \ar[r]& M\times_B M\ar[r] & 1}, \]
where $\cG^{\opp}$ is $\cG$ with the  opposite multiplication and the source/target interchanged. This is still a $\cT$-gerbe because on $\cT$ the multiplication is unchanged. Note that the inversion gives an isomorphism $\cG\cong \cG^\opp$ of groupoids but not one of extensions since it does not induce the identity on $\cT$! The fact that $\G\star\G^\opp$ represents the trivial gerbe follows from the straightforward Morita equivalence:
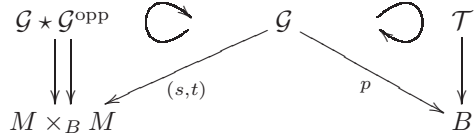
\[
\xymatrix{
\G\star\G^\opp\ar@<0.25pc>[d] \ar@<-0.25pc>[d]  & \ar@(dl, ul) & \G \ar[dll]^{(s,t)}\ar[drr]_{p} & \ar@(ur, dr) & \cT  \ar[d] \\
M\times_B M&  & & & B}
\]
where the left action is given by: $(g_1,g_2)\cdot h=g_1\cdot h\cdot g_2^{-1}$. We conclude
that the fusion product induces an abelian group structure on the set of $\cT$-gerbes over $B$. \\

It is useful to be able to recognize more directly when an extension (\ref{basic-ext}) represents the trivial gerbe. That means that there exists a Morita bibundle $P$ implementing a Morita equivalence with (\ref{eq:trivial-gerbe}): 
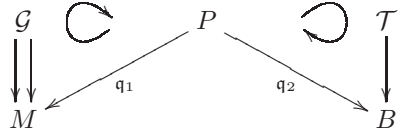
\[
\xymatrix{
\G\ar@<0.25pc>[d] \ar@<-0.25pc>[d]  & \ar@(dl, ul) &  P \ar[dll]^{\qq_1}\ar[drr]_{\qq_2} & \ar@(ur, dr) & \cT  \ar[d] \\
M&  & & & B}
\]
Since $P$ is central, the structure of (right) principal $\cT$-bundle on $P$ is determined by the action of $\cG$: $u\cdot \lambda= \lambda\cdot u$. Hence, the only thing that matters is the existence of a principal $\cG$-bundle over $B$:

\begin{lemma}\label{lemma:trivial-gerbe}
An extension (\ref{basic-ext}) represents the trivial $\cT$-gerbe iff $\check{H}^{1}_{B}(B, \G)\neq \emptyset$, i.e. iff there exists a $B$-fibered principal $\cG$-bundle $P$ over $B$.
\end{lemma}

Changing the point of view we also see that $\cG$ is itself determined, via the gauge construction, by $P$ and its structure of principal $\cT$-bundle. Hence:

\begin{corollary}\label{cor:trivial-gerbe}
An extension (\ref{basic-ext}) represents the trivial $\cT$-gerbe iff it is the gauge extension associated to a $B$-fibered principal $\cT$-bundle (Example \ref{ex:gauge-extension}).
\end{corollary}

\subsubsection{The Dixmier-Douady class:} We now recall the construction of the Dixmier-Douady class of the gerbe represented by the extension
(\ref{basic-ext}). It is the obstruction to triviality that arises from the characterization given in Lemma \ref{lemma:trivial-gerbe}: any principal bundle $P$ as in the lemma is pushed forward via the map $\cG\rightarrow M\times_B M$ to the principal $M\times_B M$-bundle over $B$ which is $M$ itself, hence the triviality question amounts to deciding whether $[M]\in \check{H}^{1}_{B}(B, M\times_{B}M)$ comes from $\check{H}^{1}_{B}(B, \cG)$. This lifting problem can be translated to the language of cocycles. One chooses a  good cover $\{V_i\}_{i\in I}$ of $B$ with local sections $s_i: V_i \rightarrow M$ of the projection into $B$ (maps in $\textrm{Man}_{B}$!). Then the \v{C}ech cocycle describing $M$ as an element in $\check{H}^{1}_{B}(B, M\times_{B}M)$  
% principal $M\times_B M$-bundle over $B$, 
is given by
\[ g_{ij}= (s_i, s_j):V_{ij}\to M\times_B M, \]
and the issue is wether one can lift this cocycle along the projection $\cG\to M\times_B M$ of the extension, to a cocycle with values in $\cG$.
% representing the desired principal $\cG$-bundle $[P]\in \check{H}^{1}_{B}(B, \cG)$. 

Viewing $\cG$ as a $\cT$-torsor over $M\times_{B} M$, we denote by $\cG_{i, j}$ its pull-back via $(s_i, s_j)$:
\[ \cG_{i, j}:= \{ g\in \cG: t(g)= s_i(x), s(g)= s_j(x)\textrm{ for some } x\in V_{ij}\}.\]
Since the $V_{ij}$ are contractible, this $\cT$-torsor is trivializable and we can find a section:
\[ \widetilde{g}_{ij}: V_{ij} \to \cG_{i, j},\quad \widetilde{g}_{ij}\mapsto g_{ij}=(s_i,s_j). \] 
The only issue is that $\{\widetilde{g}_{ij}\}$ may fail to be a cocycle. The obstruction arises by looking at triple intersections, on which we define
\begin{equation}\label{eq-def-c-ijk} 
c_{ijk}:= \widetilde{g}_{ij} \cdot \widetilde{g}_{jk} \cdot \widetilde{g}_{ki}\in \Gamma( V_{ijk}, \cT).
\end{equation}
This is indeed a section of $\cT$ since everything fibers over $B$! 

\begin{definition}
The \textbf{Dixmier-Douady class} of the extension is the cohomology class represented by the \v{C}ech 2-cocycle:
\[ \DD(\cG):= [\{c_{ijk}\}]\in \check{H}^{2}_{B}(B, \ucT) \cong H^3(B, \Lambda_{\cT}).\]
\end{definition}

When $\cT=\S^1\times B\to B$ with associated lattice $\Lambda_\cT=\Z\subset\R$, we recover the usual Dixmier-Douady class of an $\S^1$-gerbe living in $H^3(B,\Z)$.

It is clear that $\DD(\cG)= 0$ if and only if $\cG$ represents the trivial gerbe. Indeed, the assumption that the 2-cocycle $c_{ijk}$ is exact gives us, eventually after passing to a refinement, smooth functions $\lambda_{ij}:V_{ij}\to\cT$, fibered over $B$, such that:
\[ c_{ijk}(x)=\lambda_{ij}\cdot \lambda_{jk}\cdot \lambda_{kj}. \]
Then we can use the action of $\cT_M$ on $\cG$ to correct the $\widetilde{g}_{ij}$:
\[ \overline{g}_{ij}:=\widetilde{g}_{ij}\cdot\lambda_{ij}^{-1}. \]
We still have that $\overline{g}_{ij}\mapsto g_{ij}$ and that $\overline{g}_{ij}\cdot \overline{g}_{jk}\cdot \overline{g}_{ki}=0$, so that 
$\overline{g}_{ij}:V_{ij} \to \cG$ is a 1-cocycle representing an element $[P]\in \check{H}^{1}_{B}(B,\cG)$ with $i([P])=[M]$.

The previous constructions can be interpreted as a connecting ``homomorphism'' construction, $\DD(\cG)= \delta([M])$, where $\delta$ arises from the short exact sequence (\ref{basic-ext}):
\begin{equation}
\label{eq:ex:sequence:gerbe}
\xymatrix{\check{H}^{1}_{B}(B,\cG) \ar[r]^----{i} & \check{H}^{1}_{B}(B,M\times_B M)\ar[r]^---{\delta}&  \check{H}^{2}_{B}(B, \cT)= H^{2}(B,\underline{\cT})}.
\end{equation}
A formal argument shows that this is exact, i.e. $\im(i)=\delta^{-1}(0)$. Here and in what follows we will use the additive notation
for the group structure of $H^2(B,\underline{\cT})$, hence $0$ for its identity.
Lemma \ref{lemma:trivial-gerbe}
combined with the remark that $\check{H}^{1}_{B}(B,M\times_B M)$ contains only one element, namely $[M]$, shows that 
the statement that ``$\cG$ represents the trivial gerbe iff $\DD(\cG)=0$" is equivalent to the exactness of (\ref{eq:ex:sequence:gerbe}).

The last  interpretation of the Dixmier-Douady class, via (\ref{eq:ex:sequence:gerbe}), makes it rather clear that it only depends on 
the Morita equivalence class of the extension. Indeed, a Morita $(\G_1,\G_2)$-bibundle $Q$ allows one to transport a principal
$\G_1$-bundle $P\to B$ along the Morita equivalence yielding a principal $\G_2$-bundle $P\otimes_{\cG_1} Q\to B$. 
This leads to a bijection in Haefliger cohomology:
\[ Q_*:\check{H}^{1}_{B}(B,\cG_1)\cong \check{H}^{1}_{B}(B,\cG_2). \]
Because of the naturally of the construction, one obtains a commutative diagram 
\[
\xymatrix{\check{H}^{1}_{B}(B,\cG_1) \ar[r] \ar[d]_{\cong}& \check{H}^{1}_{B}(B,M_1\times_B M_1)\ar[r] \ar[d]_{\cong}& \check{H}^{2}_{B}(B, \cT)  \ar[r]^{\cong} \ar@{=}[d] & H^{2}(B,\underline{\cT})\ar@{=}[d]\\
\check{H}^{1}_{B}(B,\cG_2) \ar[r] & \check{H}^{1}_{B}(B,M_2\times_B M_2)\ar[r] &     \check{H}^{2}_{B}(B, \cT)  \ar[r]^{\cong}   &   H^{2}(B,\underline{\cT})},
\]
which implies that $\DD(\cG_1)= \DD(\cG_2)$. Of course, this can also be proved using \v{C}ech cocycles.
The details for the cocycle argument will be given in the next subsection in the context of symplectic gerbes.

\subsubsection{The Dixmier-Douady class as a group isomorphism.} The additivity of the Dixmier-Douady class:
\[ \DD(\cG_1\star\cG_2)= \DD(\cG_1)+ \DD(\cG_2), \]
can be checked using a formal argument based on the interpretation of $\DD$ via the exact sequence (\ref{eq:ex:sequence:gerbe}),
starting from the remark that $\cG_1\times_{B}\cG_2$ defines a $\cT\times \cT$-gerbe over $M_1\times_{B} M_2$. Alternatively,
one can also give a ``down to earth'' argument in terms of cocycles. Again, the details of the direct approach
will be given in the next section in the context of  symplectic gerbes.

Since $\DD(\cG)= 0$ if and only if $\cG$ represents the trivial gerbe, we conclude that $\DD$ is an 
injective group homomorphism. Since the base manifold $M$ of our extensions is allowed to be disconnected, 
$\DD$ is also surjective: any class $u\in H^2(B, \ucT)$ is represented by a \v{C}ech cocycle $\{c_{ijk}\}$ with respect to a good cover $\mathcal{V}= \{V_i\}_{i\in I}$ of $B$, so taking $M$ to be the disjoint union of the $V_i$, we see that $\{c_{ijk}\}$ becomes a 2-cocycle on the resulting groupoid $M\times_{B}M$ with coefficients in $\cT_M$, and we let $\cG$ to be the corresponding extension. 

Putting everything together, one obtains the central result of the theory:

\begin{theorem} 
\label{DD-clas} 
The Dixmier-Douady class induces an isomorphism between the group of $\cT$-gerbes over $B$ and $H^2(B, \underline{\cT})$.
\end{theorem}

%%%%%%%%%%%%%%%%%%%%%%%%
\subsection{Symplectic gerbes and their Lagrangian class}
\label{Symplectic gerbes and their Lagrangian class}
%%%%%%%%%%%%%%%%%%%%%%%%

We are interested in the extension of gerbes to the symplectic world. Similar to the passage from torsors to symplectic torsors
(Section \ref{ssec:torsors}), we replace torus bundles by symplectic torus bundles over $B$. Hence, our starting point is an integral affine manifold 
$(B, \Lambda)$ with its the associated symplectic torus bundle $(\cT,\omega_\cT)=(T^*B,\omega_\can)/\Lambda$.  

\begin{definition}
Let $(B, \Lambda)$ be a smooth integral affine manifold with associated torus bundle $(\cT,\omega_{\cT})$. A {\bf symplectic $(\cT,\omega_\cT)$-gerbe over} $B$ 
is a symplectic Morita equivalence class of central extensions of the form:
\begin{equation}
\label{eq:sympl-gerbe} 
\xymatrix{1\ar[r] & \cT_M \ar[r]^{i} & (\cG, \Omega) \ar[r]& M\times_B M\ar[r] & 1},
\end{equation}
where $p_M:M\to B$ is a surjective submersion, $\cT_{M}= p_{M}^{*}\cT$ and $(\cG,\Omega)$ is a symplectic groupoid with $i^*\Omega=p_M^*\omega_{\cT}$.
% We call an extension \eqref{eq:sympl-gerbe}, or the symplectic groupoid $(\cG,\Omega)$, a {\bf representative} of the symplectic $(\cT,\omega_{\cT})$-gerbe over $B$.
\end{definition}

\begin{remark} Symplectic Morita equivalence of extensions is the symplectic version of the notion from the previous section
(see Appendix \ref{appendix:moment:maps}). We continue to allow symplectic groupoids over a disconnected base 
(and possibly with disconnected $s$-fibers). This does not affect the basic property that $M$ carries a Poisson structure
$\pi$. When dealing with extensions (\ref{eq:sympl-gerbe}) 
with connected s-fibers, then $\cG$ will make $(M, \pi)$ into a Poisson manifold of proper type for which the associated 
leaf space is the integral affine manifold $(B, \Lambda)$.
\end{remark}

\begin{example}\label{ex:gauge-extension2} 
The symplectic analogue of Example \ref{ex:gauge-extension} is the (symplectic) gauge groupoid $\Gauge{\cT}{X}$ associated 
to a free Hamiltonian $\cT$-spaces $q: (X, \Oga)\to B$, described in Appendix \ref{cG-red2}). 
\end{example}

The fusion product of extensions has a straightforward symplectic version: in  the product $\cG_1\times_B\cG_2$ one considers
the closed 2-form $\pr_{1}^{*}\Omega_1-\pr_{2}^{*}\Omega_2$ and a simple computation shows that the kernel of this form is precisely
the orbits of the diagonal $\cT$-action on $\cG_1\times_B\cG_2$, so this form induces a multiplicative symplectic form $\Omega_1\star\Omega_2$ 
on the quotient $\cG_1\star\cG_2$. We define the {\bf fusion of symplectic $(\cT,\omega_\cT)$-gerbes} by:
\[ (\cG_1,\Omega_1)\star(\cG_2,\Omega_2):=(\cG_1\star\cG_2,\Omega_1\star\Omega_2). \]

As in the case of $\cT$-gerbes, the trivial symplectic $(\cT,\omega_\cT)$-gerbe is represented by $(\cT, \omega_{\cT})$ and the inverse
of the symplectic $(\cT,\omega_{\cT})$-gerbe defined by \eqref{eq:sympl-gerbe} is the one represented by the 
$\cG^{\opp}$ with the same symplectic form $\Omega$. Moreover, we have 
obvious symplectic isomorphisms:
\begin{align*}
(\cG,\Omega)\star (\cT,\omega_{\cT})&\cong (\cT,\omega_\cT)\star(\cG,\Omega)\cong (\cG,\Omega)\\
(\cG_1,\Omega_1)\star(\cG_2,\Omega_2)&\cong (\cG_2,\Omega_2)\star(\cG_1,\Omega_1) \\
((\cG_1,\Omega_1)\star(\cG_2,\Omega_2))\star(\cG_3,\Omega_3)&\cong (\cG_1,\Omega_1)\star((\cG_2,\Omega_2)\star(\cG_3,\Omega_3))
\end{align*}
and there is a symplectic Morita equivalence:
\[
\xymatrix{
(\G\star\G^\opp,\Omega\star\Omega) \ar@<0.25pc>[d] \ar@<-0.25pc>[d]  & \ar@(dl, ul) & (\G,\Omega) \ar[dll]^{(s,t)}\ar[drr]_{p} & \ar@(ur, dr) & (\cT,\omega_{\cT})  \ar[d] \\
M\times_B M&  & & & B}
\]
We conclude that the set of symplectic $(\cT,\omega_{\cT})$-gerbes over $B$ is an abelian group with the operation induced by the fusion product. 

For the symplectic version of the Dixmier-Douady class and of Theorem \ref{DD-clas} it is not surprising, given the discussion from Section \ref{ssec:torsors}, that we now have to replace $\underline{\cT}$ by the subsheaf $\underline{\cT}_{\Lagr}$ of local Lagrangian sections:

\begin{theorem}
\label{thm:Lagrangian:class:symp:gerbe}
Given a symplectic groupoid $(\cG,\Omega)\tto M$ fitting into an extension \eqref{eq:sympl-gerbe} there is an associated cohomology class:
\begin{equation}\label{c-cG-sympl} 
\DD(\cG,\Omega)\in H^2(B, \underline{\cT}_{\Lagr}).
\end{equation}
Moreover:
\begin{enumerate}[(i)]
\item this construction induces an isomorphism between the group of symplectic $(\cT,\w_\cT)$-gerbes over $B$ and $H^2(B, \underline{\cT}_{\Lagr})$. 
% \item $\DD(\cG, \Omega)=0$ iff $\cG\cong \Gauge{\cT}{X}$, the gauge groupoid of a free Hamiltonian $\cT$-space $q:(X, \Oga)\to B$ 
\item $\DD(\cG, \Omega)=0$ iff $\cG$ is the gauge extension associated to a free Hamiltonian $\cT$-space (Example \ref{ex:gauge-extension2}). 
\end{enumerate}
\end{theorem}

The class $\DD(\cG,\Omega)$ is called {\bf the Lagrangian Dixmier-Douady class} of the symplectic gerbe. 
The rest of this section is devoted to the proof of this theorem.

\subsubsection{Construction of the Lagrangian Dixmier-Douady class.} We proceed like in the previous section, 
and using the same notations. We choose the local sections $s_i:V_i\to M$ giving rise to the 1-cocycle $g_{i, j}= (s_i, s_j)$ 
and we form the $\cT$-torsor $\cG_{ij}$ over $V_{ij}$.
This is now a symplectic torsor, with the symplectic form inherited from $(\cG, \Omega)$, so the bundle projection $\cG_{ij}\to V_{ij}$ 
is a Lagrangian fibration with connected fibers. Since the $V_{ij}$ are contractible, we can choose the lifts $\widetilde{g}_{ij}$ 
to be Lagrangian. Using the multiplicativity of the symplectic form on $\G$, it is straightforward to check that
the resulting 2-cocycle (\ref{eq-def-c-ijk}) is made of Lagrangian sections:
\begin{equation}\label{eq:c-ijk-Lagr} 
c_{ijk}=\widetilde{g}_{ij}\cdot \widetilde{g}_{jk}\cdot \widetilde{g}_{ki}\in \Gamma( V_{ijk}, \cT_\Lagr). 
\end{equation}
A tedious but straightforward argument shows that the resulting cohomology class does not depend on the choices involved; this defines our Lagrangian Dixmier-Douady class:
\[ \DD(\cG,\Omega)\in H^2(B, \underline{\cT}_{\Lagr}).\]
Note that the characterization of the extensions which represent the trivial gerbes given by Lemma \ref{lemma:trivial-gerbe}, has a straightforward version in the symplectic case. This implies part (ii) of Theorem \ref{thm:Lagrangian:class:symp:gerbe}.
% In particular, part (ii) of Theorem \ref{thm:Lagrangian:class:symp:gerbe} follows immediately from the rest.

\begin{remark}\label{rk-sympl-DD} A symplectic version of the exact sequence \eqref{eq:ex:sequence:gerbe} leads to an interpretation
of the Lagrangian Dixmier-Douady class as the image of a connecting morphism of the class of $(M,\pi)$, as in the standard case:
\[ \DD(\cG,\Omega)=\delta([(M,\pi)])\in H^2(B,\underline{\cT}_\Lagr). \]
% Since the direct approach above is enough for our purposes, we shall not pursue this point of view.
\end{remark}

\subsubsection{Independence of the Morita class} In order to show invariance under symplectic Morita equivalence, assume we are given two extensions:
\[ \xymatrix{1\ar[r] & p_a^*\cT \ar[r] & \cG_a \ar[r]& M_a\times_B M_a\ar[r] & 1}\quad (a=1,2)\]
associated with submersions $p_a:M_a\to B$, and a symplectic Morita equivalence:
\[
\xymatrix{
(\G_1,\Omega_1)\ar@<0.25pc>[d] \ar@<-0.25pc>[d]  & \ar@(dl, ul) & (P,\Omega_P) \ar[dll]_{\qq_1}\ar[drr]^{\qq_2} & \ar@(ur, dr) & (\G_2,\Omega_2) \ar@<0.25pc>[d] \ar@<-0.25pc>[d] \\
M_1&  & & & M_2}
\]
Start with the construction of $c_2(\cG_a, \Omega_a)$: a good cover $\{V_i\}_{i\in I}$ of $B$,  $s^a_i:V_i\to M_a$ of $p_a:M_a\to B$, 
$g^a_{ij}=(s^a_i,s^a_j):V_{ij}\to M_a\times_B M_a$ and Lagrangian lifts:
\[ \widetilde{g}^a_{ij}: V_{ij} \to (\cG_1)_{i, j}.\]

Now, $(\qq_1, \qq_2): P\to M_1\times_{B} M_2$ is a symplectic $\cT$-torsor which, when pulled-back via $(s_i^1,s_j^2):V_i\to M_1\times_{B} M_2$, gives a symplectic $\cT$-torsor over $V_i$. Since $V_i$ is contractible, it has a Lagrangian section, i.e.~we find 
$u_i:V_i\to P$ such that $u_{i}^{*}\Omega_P= 0$, $\qq_1\circ u_i=s^1_i$ and $\qq_2\circ u_i=s^2_i$. For each $x\in V_{ij}$ the elements $u_i(x)$ and
$\widetilde{g}^1_{ij}(x)\cdot u_j(x)\cdot \widetilde{g}^2_{ji}(x)$ lie in the same fiber of $\qq_1$ and of $\qq_2$. By principality, it follows that there exist 
unique $\lambda_{ij}:V_{ij}\to \cT$ such that:
\[ u_i(x)=\lambda_{ij}(x)\cdot \widetilde{g}^1_{ij}(x)\cdot u_j(x)\cdot \widetilde{g}^2_{ji}(x)\quad (x\in V_{ij}). \]
Because $u_{i}$, $u_j$ and $\widetilde{g}^a_{ji}$ are Lagrangian sections,  the actions of $\cG_a$ on $P$ are symplectic, and
$\w_\cT$ is multiplicative, it follows  that $\lambda_{ij}$ is also Lagrangian.

Looking at triple intersections, and using that the actions of $\cT_{M_a}$ commute, we obtain for $x\in U_{ijk}$:
\[ 
\lambda_{ij}(x)\cdot\widetilde{g}^1_{ij}(x)\cdot\lambda_{jk}(x)\cdot \widetilde{g}^1_{jk}(x)\cdot \lambda_{ki}(x)\cdot \widetilde{g}^1_{ki}(x)=\widetilde{g}^2_{ij}(x)\cdot \widetilde{g}^2_{jk}(x)\cdot \widetilde{g}^2_{ki}(x),\]
which can be written as:
\[ c^2_{ijk}(x)-c^1_{ijk}=\lambda_{ij}(x)+\lambda_{jk}(x)+\lambda_{ki}(x).\]
Hence, passing to cohomology we have $\DD(\cG_1,\Omega_1)= \DD(\cG_2,\Omega_2)$.

\subsubsection{The Lagrangian Dixmier-Douady class as a group isomorphism} We now turn to the additivity of the Dixmier-Douady class:
\[ \DD((\cG_1,\Omega_1)\star(\cG_2,\Omega_2))= \DD(\cG_1,\Omega_1)+ \DD(\cG_2,\Omega_2).\]
To represent the Dixmier-Douady class of $(\cG_1\star\cG_2,\Omega_1\star\Omega_2)$ we start with the data used to construct
$c_2(\cG_a,\Omega_a)$ for $a\in \{1, 2\}$: a covering $\{V_i\}_{i\in I}$ and sections $s_i^a:V_i\to M_a$ of $p_a:M_a\to B$, yielding the 1-cocycle $g^a_{ij}=(s^a_i,s^a_j)$. The lifts $\widetilde{g}^a_{ij}: V_{ij} \to (\cG_a)_{i, j}$ lead to the 2-cocycles $c_{ijk}^{a}$ given by (\ref{eq:c-ijk-Lagr}). 
But now observe that these give similar data for $\cG_1\star\cG_2$: the sections $s_i=(s_i^1,s_i^2): V_i\to M_1\times_B M_2$ with associated 1-cocycle $g_{ij}=(g^1_{ij},g^2_{ij})$ and then the Lagrangian lift 
\[ \widetilde{g}_{ij}: V_{ij} \to (\cG_1\star\cG_2)_{i, j}\]
obtained by composing 
$(\widetilde{g}^1_{ij},\widetilde{g}^2_{ij}):V_{ij}\to \cG_1\times_B\cG_2$
with the projection into $\cG_1\star\cG_2$. We then find that the associated 2-cocycle is given by:
\[ c_{ijk}=\widetilde{g}_{ij}\cdot \widetilde{g}_{jk}\cdot \widetilde{g}_{ki}=c^1_{ijk}+c^2_{ijk} . \]
This proves that $c_2(\cG_1\star\cG_2,\Omega_1\star\Omega_2)=c_2(\cG_1,\Omega_1)+c_2(\cG_2,\Omega_2)$.

Finally, note that the argument on the injectivity and the surjectivity of $\DD$ from the previous subsection straightforwardly adapts to 
the present context.

\subsection{Lagrangian Dixmier-Douday class vs Dazord-Delzant class}
Given an extension \eqref{eq:sympl-gerbe} let us observe now that there is an obvious isomorphism of sheaves:
\[ p^{-1}(\underline{\cT}_{\Lagr})\cong \underline{\cT_M}_\cl, \]
so there is an induced map at the level of cohomology:
\begin{equation}
\label{eq:map:classes}
p^*:\check{H}^2(B, \underline{\cT}_{\Lagr})\to \check{H}^2(M,\underline{\cT_M}_\cl).
\end{equation}
This map allows to express the precise relationship between the Lagrangian Dixmier-Douady class of the extension and the Delzant-Dazord 
class of the underlying Poisson manifold (see Section \ref{sec:Dazord:Delzant}):

\begin{proposition}
\label{prop:lagr:obst:class}
Under \eqref{eq:map:classes} the class $c_2(\cG,\Omega)\in \check{H}^2(B, \underline{\cT}_{\Lagr})$
of a symplectic $(\cT,\w_\cT)$-gerbe with representative $(\cG,\Omega)\tto (M,\pi)$ is mapped to the obstruction 
class $c_2(M,\pi,\Lambda_\cG)\in\check{H}^2(M,\underline{\cT_M}_\cl)$.
\end{proposition}

\begin{proof}
Start with the data (\ref{eq:c-ijk-Lagr}) to construct the cocyle $c_{ijk}$ representing $c_2(\cG, \Omega)$: a cover $\{V_i\}_{i\in I}$ of $B$, sections $s_i:V_i\to M$ of $p:M\to B$ and the Lagrangian lifts  
$\widetilde{g}_{ij}: V_{ij} \to \cG$ of the 1-cocycle $g_{ij}= (s_i, s_j)$. Since $p:M\to B$ is open, to prove the proposition it is enough to show that 
\[ c_{ijk}\circ p:p^{-1}(V_{ijk})\to p^*\cT_\cl, \]
is a 2-cocycle representing $c_2(M,\pi,\Lambda_\cG)$.

For each $i\in I$, we cover the open sets $p^{-1}(V_i)$ by contractible open sets. The collection of all such open sets, denoted $\{U_a\}_{a\in A}$, is an open cover of $M$ which comes with a map of the indices $A\to I$, $a\mapsto i_a$, such that:
\[ p^{-1}(V_i)=\bigcup_{i_a=i} U_a. \]
The sections $s_i$ give transversals $T_i=\im(s_i)\subset M$ to the symplectic leaves and we have Morita equivalences:
\[
\xymatrix{
 \G|_{p^{-1}(V_i)}\ar@<0.25pc>[d] \ar@<-0.25pc>[d]  & \ar@(dl, ul) & \G(T_i,-) \ar[dll]_{t}\ar[drr]^{p\circ s}  & \ar@(ur, dr) & \cT|_{V_i}\ar[d]  \\
p^{-1}(V_i)&  & & & V_i}
\]
Since the $U_a$ are contractible, we can pick local sections $\sigma_a:U_a\to\G(T_{i_a},-)$ of the principal $\cT|_{V_{i_a}}$-bundle $t:\G(T_{i_a},-)\to p^{-1}(V_{i_a})$. Now observe that on the intersection $U_{ab}$, we have two sections of this principal bundle: the restriction $\sigma_a|_{U_{ab}}$ and the section:
\[ \sigma_b(x)\cdot  \widetilde{g}_{i_bi_a}(p(x)),\quad (x\in U_{ab}). \]
It follows that there exist $\lambda_{ab}:U_{ab}\to \cT_{V_{i_a}}$ such that:
\[ \sigma_a(x)=\sigma_b(x)\cdot  \widetilde{g}_{i_bi_a}(p(x))\cdot \lambda_{ab}(x). \]
Using this relation successively for $\sigma_a$, $\sigma_b$ and $\sigma_c$, we conclude that for $x\in U_{abc}$:
\[ 1= \widetilde{g}_{i_ai_c}(p(x))\cdot \lambda_{ca}(x)\cdot  \widetilde{g}_{i_ci_b}(p(x))\cdot \lambda_{bc}(x)\cdot  \widetilde{g}_{i_bi_a}(p(x))\cdot \lambda_{ab}(x)\]
Using the fact that $\cT_M$ is abelian, we conclude that:
\[ \widetilde{g}_{i_a i_b}\cdot \widetilde{g}_{i_b i_c}\cdot \widetilde{g}_{i_b i_a}=\lambda_{ab}+\lambda_{bc}+\lambda_{ca},\]
where we are now viewing $\lambda_{ab}:U_{ab}\to \cT_M$. This says that under the map \eqref{eq:map:classes} the class $c_2(\cG, \Omega)$,
represented by the 2-cocycle $c_{ijk}$, is mapped to a class represented by the 2-cocycle $\lambda_{ab}+\lambda_{bc}+\lambda_{ca}:U_{abc}\to \cT_M$. But now observe that:
\begin{enumerate}[(a)]
\item The map $t:(\G(T_{i_a},-),\Omega)\to p^{-1}(V_{i_a})$ is an isotropic realization, so it follows that $\sigma_a^*\Omega$ is a closed 2-form on $U_a$ extending the leafwise symplectic form $\omega_{\cF_\pi}$;
\item Since the principle bundle action is symplectic, we find that:
\begin{align*}
\sigma_a^*\Omega&=\sigma_b^*\Omega+(\widetilde{g}_{i_ai_b}\circ p)^*\Omega+\lambda_{ab}^*\omega_\cT\\
		&=\sigma_b^*\Omega+p^*\widetilde{g}_{i_ai_b}^*\Omega+\d\lambda_{ab}\\
		&=\sigma_b^*\Omega+\d\lambda_{ab}.
\end{align*}
where we used the fundamental property of $\omega_\cT$ and that the $\widetilde{g}_{i_ai_b}$ are Lagrangian.
\end{enumerate}
It follows that the 2-cocycle $\lambda_{ab}+\lambda_{bc}+\lambda_{ca}$ takes values in $(\cT_M)_\cl$ and represents the obstruction
class $c_2(M,\pi,\Lambda_\cG)$, so the proof is completed.
\end{proof}

\begin{remark}
In general, the map \eqref{eq:map:classes} fails to be injective.
Hence, it is possible to have a proper symplectic integration $(\cG,\Omega)\tto (M,\pi)$ such that $c_2(\cG,\Omega)\neq 0$, but
$c_2(M,\pi,\Lambda_\cG)= 0$. 
In this case, the Poisson manifold $(M,\pi)$ admits rather different proper integrations defining the same lattice:
one of them arising from free Hamiltonian reduction (Theorem \ref{thm:Dazord:Delzant}) and the other one not. 
When the fibers of $p:M\to B$ (the symplectic leaves) are 1-connected, a spectral sequence argument shows that \eqref{eq:map:classes}
is injective, and we conclude that either all proper integrations of $(M,\pi,\Lambda)$ arise from free Hamiltonian $\cT_\Lambda$-reduction,  
or none of them does.
\end{remark}

%%%%%%%%%%%%%%%%%%%%%%%%%%%%%%
%%%%%%%%%%%%%%%%%%%%%%%%%%%%%%
%%%%%%%%%%%%%%%%%%%%%%%%%%%%%%
%%%%%%%%%%%%%%%%%%%%%%%%%%%%%%
%%%%%%%%%%%%%%%%%%%%%%%%%%%%%%
%%%%%%%%%%%%%%%%%%%%%%%%%%%%%%
%%%%%%%%%%%%%%%%%%%%%%%%%%%%%%
%%%%%%%%%%%%%%%%%%%%%%%%%%%%%%
%%%%%%%%%%%%%%%%%%%%%%%%%%%%%%
%%%%%%%%%%%%%%%%%%%%%%%%%%%%%%
\section{Symplectic gerbes over orbifolds}
\label{sec:gerbes:orbifolds}
%%%%%%%%%%%%%%%%%%%%%%%%%%%%%%
%%%%%%%%%%%%%%%%%%%%%%%%%%%%%%
%%%%%%%%%%%%%%%%%%%%%%%%%%%%%%
%%%%%%%%%%%%%%%%%%%%%%%%%%%%%%
%%%%%%%%%%%%%%%%%%%%%%%%%%%%%%
%%%%%%%%%%%%%%%%%%%%%%%%%%%%%%
%%%%%%%%%%%%%%%%%%%%%%%%%%%%%%
%%%%%%%%%%%%%%%%%%%%%%%%%%%%%%
%%%%%%%%%%%%%%%%%%%%%%%%%%%%%%
%%%%%%%%%%%%%%%%%%%%%%%%%%%%%%

We now extend the theory of symplectic gerbes over manifolds to the case of orbifolds. Of course, our motivation comes 
from the fact that the leaf spaces of PMCTs are, in general, orbifolds. As in Section \ref{sec:gerbes:manifolds}, we start with a simplified statement of our results, generalizing Theorem \ref{thm:DD-gerbe-smooth-intro}:

% Due to the fact that the orbit spaces of PMCTs are, in general, orbifolds, we will study now symplectic gerbes over orbifolds. 
% We will also apply the resulting theory to  PMCTs, obtaining invariants of proper integrations.
% These invariants are the obstructions to reconstruct the integration from the orbifold structure on the leaf space and some isotropic realization. 
% For now, we present a simplified statement of our results:

\begin{theorem} Let $(\cG, \Omega) \tto M$ be a proper integration of $(M, \pi)$, inducing the orbifold atlas $\cBG= \cBG(\cG)$ on the leaf space $B=M/\cF_\pi$. Then $(\cG, \Omega)$ gives rise to a symplectic gerbe over $(B, \cBG)$, which is classified by a cohomology class
\[ \DD(\cG, \Omega)\in H^2(\cBG, \ucT_{\Lagr}) .\]
Moreover, $\DD(\cG, \Omega)$ vanishes iff $\cG=\cBG_X(M, \pi)$, the $\cBG$-integration of $(M, \pi)$ relative to a proper isotropic realization $q:(X, \Oga)\to (M, \pi)$ (cf. Definition \ref{def:E:integration}).
\end{theorem}

The passage from symplectic gerbes over manifolds to symplectic gerbes over orbifolds 
% The passage from manifolds to orbifolds 
is based on Haefliger's philosophy explained
in Remark \ref{rmk:structures on orbifolds}.  So throughout this section we fix an orbifold $(B, \cB)$, with atlas $\cB\tto N$. The reader should keep in mind, as main 
examples: the smooth case of the previous section where $\cB$ is $B\tto B$ (hence $N= B$), and the leaf space of a proper Poisson manifold $(M, \pi)$ with proper integration $\cG$, where $\cBG(\cG)\tto M$ will be the foliation groupoid from Theorem \ref{thm-reg-fol2} (hence $N= M$).

To handle various geometric structures on the orbifold $(B,\cB)$ one may, in principle, represent them with respect to the atlas $\cB$. 
However, for some purposes this may not be the most convenient atlas, so we will be looking at other orbifold atlases $\cE\tto M$ 
related to $\cB$ through a specified Morita equivalence
\begin{equation}\label{fixed-XCE} 
Q_{\cE}: \cE\simeq \cB.
\end{equation}

The organization of this section is similar to that of Section \ref{sec:gerbes:manifolds}, so we will describe first torsors over orbifolds and then gerbes over orbifolds. 

% Our discussion will follow the same pattern as in the smooth case: we will describe first torsors over orbifolds and then gerbes over orbifolds. 

%%%%%%%%%%%%%%%%%%%%%%%%%%%%%%
\subsection{Symplectic torsors over integral affine orbifolds}
\label{Symplectic torsors over integral affine orbifolds}
%%%%%%%%%%%%%%%%%%%%%%%%%%%%%%

\subsubsection{Torus orbibundles}
\label{ssec:Torus orbibundles}
 The notion of a {\bf torus orbibundle} $\cT$ over $(B,\cB)$ is very similar to the notion of vector orbibundle
(see Remark \ref{rem:cohom:orbifld}): they are represented by a $\cB$-torus bundle $\cT$, i.e. a torus bundle over $M$ endowed with an action of $\cB$:
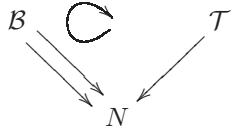
\[
\xymatrix{
\cB\ar@<0.25pc>[dr] \ar@<-0.25pc>[dr]    & \ar@(dl, ul) &  \cT \ar[dl]\\
 & N &  }
\]
We can represent the torus orbibundle $\cT$ with respect to any other atlas $\cE\simeq \cB$ by a $\cE$-torus bundle $\cT_{\cE}$ over
the base of $\cE$. Indeed, any Morita equivalence 
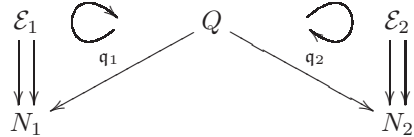
\[
\xymatrix{
\cE_1\ar@<0.25pc>[d] \ar@<-0.25pc>[d]  & \ar@(dl, ul) & Q \ar[dll]_{\qq_1}\ar[drr]^{\qq_2} & \ar@(ur, dr) & \cE_2 \ar@<0.25pc>[d] \ar@<-0.25pc>[d] \\
N_1&  & & & N_2}
\]
gives rise to a 1-1 correspondence between (isomorphism classes of) $\cE_1$-torus bundles $\cT_1$ and $\cE_2$-torus bundles $\cT_2$.
The correspondence is defined by the condition that there is a left $\cE_1$-equivariant and right $\cE_2$-equivariant isomorphism of torus bundles:
% \begin{equation}\label{torus-by-Morita} 
\[
\qq_{1}^{*}\cT_1\cong \qq_{2}^{*}\cT_2.
\]
% \end{equation} 
Here, the left action of $\cE_1$ on $\qq_{1}^{*}\cT_1$ it is the lift of the original action of $\cE_1$ on $\cT_1$, while the left action on
$\qq_{2}^{*}\cT_2$ is the tautological one: the action of an arrow $g: x\to x'$ of $\cE_1$ takes an element $(u, \lambda)$ 
in the fiber of $\qq_{2}^{*}\cT_2$ above $x$, $\qq_1(u)= x$, $\lambda\in \cT_{2, y}$ where $y= \qq_2(x)= \qq_2(x')$, 
to the element $(g\cdot u, \lambda)$ in the fiber of $\qq_{2}^{*}\cT_2$ above $x'$. Similarly for the right actions. Explicitly,
starting from $\cT_1$ we obtain $\cT_2$ as $\qq_{1}^{*}\cT_1/\cE_1$.

As in \ref{ssec:to general torus bundles}, it is useful to look at the more general notion of {\bf $\cH$-principal bundle over $(B,\cB)$}, for an arbitrary Lie groupoid $\cH\tto M$. By that we mean a (right) $\cH$-principal bundle $P$ over $M$, together with a left action of $\cB$, 
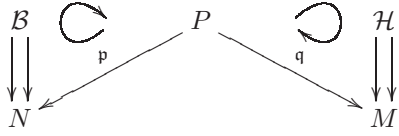
\[
\xymatrix{
\cB\ar@<0.25pc>[d] \ar@<-0.25pc>[d]  & \ar@(dl, ul) & P \ar[dll]_{\pp}\ar[drr]^{\qq} & \ar@(ur, dr) & \cH \ar@<0.25pc>[d] \ar@<-0.25pc>[d] \\
N&  & & & M}
\]
such that all the axioms from the notion of Morita equivalence are satisfied, except for the condition that the action of $\cB$ is principal. 
We denote by $\Bun_{\cH}(\cB)$ the set of isomorphism classes of such bundles. 

A description in terms of transition functions/\v{C}ech cocycles similar to (\ref{coc-H1-Haefl}) works well when $\cB$ is \'etale. 
For that we need the {\bf embedding category $\Emb_{\cU}(\cB)$} associated to a basis $\cU$ for the topology of $N$ \cite{Moe03}.
This is the discrete category whose objects are all open sets $U\in\cU$, and whose arrows $\sigma:U\to V$ are bisections $\sigma:U\to\cB$ such that $s\circ \sigma={\rm id}_U$ and $t\circ \sigma:U\to V$ is an embedding. The composition of $\sigma$ with another arrow $\tau$ from $V$ to $W$ is given by 
\[ (\tau\circ \sigma)(x):= \tau(t(\sigma(x)) \cdot \sigma(x).\]
Given a principal $\cH$-bundle $P$ over $(B,\cB)$, one choses $\cU$ such that $\pp: P\to N$ has local sections $s_U$ over each $U\in \cU$. Fixing the $\{s_{U}\}_{U\in\cU}$, one obtains:
\begin{itemize}
\item for each $U\in \cU$, a smooth map $f_{U}: U\to M$, namely $f_U:= \qq\circ s_{U}$. 
\item for each arrow $\sigma: U\to V$ in $\Emb_{\cU}(\cB)$, a smooth map $g_{\sigma}: U\to \cH$ so that
% \begin{equation}\label{eq:defining-tr-fcts} 
\[
g_{\sigma}(x): f_U(x) \to f_{V}(t\sigma(x)),\quad \sigma(x)\cdot s_{U}(x)= s_{V}(t(\sigma(x)))\cdot g_{\sigma}(x).
\]
% \end{equation}
\end{itemize}
The collection $\{f_{U}, g_{\sigma}\}$ will satisfy the cocycle condition
\[ g_{\tau\circ \sigma}(x)= g_{\tau}(t(\sigma(x)))\cdot g_{\sigma}(x).\]
Conversely, $P$ can be recovered from the family $\{f_{U}, g_{\sigma}\}$ and it is clear that the usual discussion on transition functions/\v{C}ech cocycles extends to this setting.

\subsubsection{Symplectic torus orbibundles} For the notion of {\bf symplectic torus orbibundles} the discussion is simpler when one works 
with an \'etale atlases $\cE$: we just require that $\cT_{\cE}$ comes with a symplectic structure which is multiplicative and which is invariant under the action of $\cE$. Passing to a non-\'etale $\cB\tto N$, the non-degeneracy of the 2-form is lost and the symplectic torus orbibundle is represented by a {\it presymplectic} torus bundle $(\cT, \omega_{\cT})\tto N$ satisfying (see Proposition \ref{prop:IAS-presympl-torus}):
\begin{enumerate}[(i)]
\item the kernel of $\omega_\cT$ coincides with the image of the infinitesimal action of $\cB$.
\item the action of $\cB$ on $(\cT, \omega_{\cT})$ is presymplectic (exactly as in Lemma \ref{lemma:action-Hol-on-T}).
\end{enumerate}
The first condition holds if and only if holds at units, i.e. if
\begin{enumerate}[(i)']
\item the foliation induced by  $(\cT, \omega_{\cT})$ on $M$ coincides with the orbits of $\cB\tto M$.
\end{enumerate}
Moreover, when $\cB$ is s-connected then (ii) follows from (i) (see Appendix \ref{appendix:moment:maps}). 

As in the smooth case, one obtains a 1-1 correspondence: 
\[ 
\left\{\txt{integral affine\\ structures on $(B, \cB)$\\ \,}\right\}
\stackrel{1-1}{\longleftrightarrow}
\left\{\txt{isomorphism classes of\\ symplectic torus bundles over $(B, \cB)$\\ \,} \right\}
\]

\begin{example}\label{lemma:action-Hol-on-T-reform} 
The first part of Lemma \ref{lemma:Lambda-X-Mon}, together  with  Lemma \ref{lemma:action-Hol-on-T} can be reformulated 
as saying that an isotropic realization gives rise to a symplectic torus orbibundle $(\cT,\omega_\cT)$ over
the classical orbifold $M/\cF_{\pi}$ . 
%(hence endowed with the atlas $\Hol(M,\cF_\pi)\tto M$). 
% (with atlas $\Hol(M,\cF_\pi)\tto M$).
\end{example}

\subsubsection{$\cT$-torsors over an orbifold} Given a torus orbibundle $\cT$ over an orbifold $(B, \cB)$, a {\bf $\cT$-torsor} over $(B, \cB)$ is 
a $\cT$-torsor $X$ over the base $N$ of $\cB$, together with a (left) action of $\cB$ on $X$ which is compatible with 
the action of $\cT$ in the sense that for $u\in X$ and $\lambda\in \cT$ above $x\in N$ and $b: x\to y$ in $\cB$, one has: 
\[
\xymatrix{
\cB  \ar@<0.25pc>[dr] \ar@<-0.25pc>[dr]  & \ar@(dl, ul)    X \ar[d] \ar@(ur, dr)  &  \cT \ar[dl] \\
 & N &  }, \quad b\cdot (u\cdot \lambda)= (b\cdot u)\cdot b(\lambda),
\]
where $b(\lambda)$ refers to the action of $\cB$ on $\cT$. 

This notion can be transported along Morita equivalences, hence $X$ can be represented similarly with respect to any other atlases (\ref{fixed-XCE}). To see this directly is rather tedious, but the ``fibered point of view'' allows for a simple formal argument.

\begin{remark}[Fibered point of view] 
\label{rk:Morita-category}
As in the previous section when $(B,\cB)$ was a manifold, it is useful to realize that all our objects are ``$\cB$-fibered", i.e. they come with a ``map'' into $\cB$. For example, while torus bundles over a manifold $B$ can be looked at as group-like objects in the category $\textrm{Man}_{B}$ of manifolds fibered over $B$ (see Remark \ref{rk:B-fibered}), one should think of a torus bundle $\cT$ over $(B,\cB)$ as encoding a group-like object in the category of groupoids fibered over $\cB$, namely $\cT\Join \cB \tto N$ with the obvious projection into $\cB$. 

Furthermore, it is useful to relax the notion of ``map'' between groupoids from smooth functors to Morita maps \cite{Hae} (also called Hilsum-Skandalis maps \cite{Ha08,Mrcun}). For these generalized maps, the resulting ``generalized isomorphisms" will be precisely the Morita equivalences. One thinks of a principal $\cG_2$-bundle over $\cG_1$ (in the sense described above) as a graph of a map from $\cG_1$ to $\cG_2$. Hence such a principal bundle will also  be called a {\bf Morita map} from $\cG_1$ to $\cG_2$ and we set:
\[ \cM(\cG_1, \cG_2)= \Bun_{\cG_2}(\cG_1) .\]
Given a third groupoid $\cG_3$ one has a composition operation: 
\[ \cM(\cG_1, \cG_2)\times \cM(\cG_2, \cG_3)\to \cM(\cG_1, \cG_3),\quad
(P, Q)\mapsto Q\circ P= P\times_{M_2}Q/\cG_2.\]
Moreover, any smooth morphism $F: \cG_1\to \cG_2$ can be thought  of as a  Morita map
\[ P_{F}:= M_1\times_{M_2} \cG_2 \in \cM(\cG_1, \cG_2),\]
where the actions of $\cG_i$ on $P_{F}$ are the obvious ones. More details on Morita maps can be found in \cite{ALR, Ha08, Lerman,MM, Mrcun}.

Now, while a torus bundle $\cT$ over $(B,\cB)$ may be thought of as encoding the groupoid $\cT \Join \cB$ endowed with the obvious projection into $\cB$, a $\cT$-torsor $X$ over $(B,\cB)$ should be thought of as encoding a left principal bibundle $X\times_{N}\cB$:
\[ \xymatrix{
\cB \ar@<0.25pc>[d] \ar@<-0.25pc>[d]  & \ar@(dl, ul) & X\times_{N}\cB \ar[dll]_{\qq_1}\ar[drr]^{\qq_2} & \ar@(ur, dr) & \cT \Join \cB \ar@<0.25pc>[d] \ar@<-0.25pc>[d] \\
N&  & & & N}\]
where $\qq_1(u, b)= p_X(u)$, $\qq_2(u, b)= s(b)$ and the left/right actions on an element $(u, b)\in  X\times_{N}\cB$ are:
\[ b_0\cdot (u, b)= (b_0\cdot u, b_0b), \quad  (u, b)\cdot (\lambda, b_1)= (u\cdot b(\lambda), bb_1).\]
This bibundle describes a Morita map from $\cB$ to $\cT\Join \cB$, which is a right inverse to the projection 
$\cT\Join \cB\to \cB$. Conversely, from a Morita map $P\in \cM(\cB, \cT\Join \cB)$ right inverse to the projection $\pr_{\cB}: \cT\Join \cB\to \cB$, 
one recovers $X$ as $X:= P/\cB$, with $p_X$ induced by $\qq_1$, and the obvious left action of $\cB$. For the right action of $\cT$ on $X$
one remarks that the condition that $P$ be a right inverse makes $P$ into a principal $\cT$-bundle over $\cB$, with some projection map $\pr: P\to \cB$. 
Using this, for $u\in X$ and $\lambda\in \cT$ in the fiber above $p_X(u)$, one defines 
\[ u\cdot \lambda:= p\cdot b^{-1}(\lambda)\]
where $p\in P$ is any representative of $u$ and $b=\pr(p)$. All together, one obtains:

\begin{lemma}\label{lemma:reinterpret-torsors} The construction $X\mapsto X\times_{N}\cB$ describes, up to isomorphism, a 1-1 correspondence between:
\begin{enumerate}[(i)]
\item $\cT$-torsors $X$ over $(B,\cB)$.
\item Morita maps $P\in \cM(\cB, \cT\Join \cB)$ which are sections of $\pr_{\cB}: \cT\Join \cB\to \cB$.
\end{enumerate}
\end{lemma}
\end{remark}

\subsubsection{Symplectic $\cT$-torsors over an orbifold}
In the case of presymplectic torus bundles $(\cT, \omega_{\cT})$ one can talk about {\bf symplectic $(\cT, \omega_{\cT})$-torsors over $(B,\cB)$}: 
then $X$ comes with a closed 2-form $\Oga$ and one requires that both actions $X\times_{N} \cT\to X$ and $\cB\times_{N}X \to X$ are compatible with the 2-forms that are present: 
\[
\xymatrix{
(\cB,0)  \ar@<0.2pc>[drr] \ar@<-0.2pc>[drr]  & \ar@(dl, ul) &   (X,\Oga) \ar[d] & \ar@(ur, dr)  &  (\cT,\omega_\cT) \ar[dll] \\
&  & N & & }
 \]
%\[
%\xymatrix{
%(\cB,0)  \ar@<0.25pc>[dr] \ar@<-0.25pc>[dr]  & \ar@(dl, ul)  &  (X,\omega_X) \ar[d] \ar@(ur, dr)  &  (\cT,\omega_{\cT}) \ar[dl] \\
% & N &  }
%\]
%Here $\cB$ is equipped with the zero form. 
Again, the situation is simpler when $\cB$ is \'etale, when the condition on the 
left action simply says that the form on $X$ is $\cB$-invariant. 
% For the fibered point of view, one has a version of 
Lemma \ref{lemma:reinterpret-torsors} has an obvious ``symplectic version", which uses 
(pre)symplectic bibundles (Appendix \ref{appendix:moment:maps}).

%%%%%%%%%%%%%%%%%%%%%%%%%%%%%%
\subsection{Symplectic gerbes over integral affine orbifolds}
\label{Symplectic gerbes over integral affine orbifolds}
%%%%%%%%%%%%%%%%%%%%%%%%%%%%%%

\subsubsection{Gerbes over orbifolds} We now fix  a torus bundle $\cT$  over the orbifold $(B, \cB)$.
% ; recall that  $\cB$ and $\cT$ are defined over $N$.  
Given an arbitrary atlas $Q_{\cE}: \cE\simeq \cB$, defined over some other manifold $M$ (so $\cE\tto M$), we are interested in extensions of type
\begin{equation}
\label{eq:gerbe-orb} 
\xymatrix{1\ar[r] & \cT_{\cE} \ar[r]^{i} & \cG \ar[r]^-{\pr} & \cE \ar[r]& 1}
\end{equation}
Here $\cT_{\cE}\to M$ is the $\cE$-torus bundle which represents $\cT$ in the new atlas. 
Such an extension induces an action of $\cE$ on $\cT_{\cE}$ obtained by lifting the elements of $\cE$ to elements of $\cG$ and conjugating by the lifts.
A {\bf central extension} is one for which the induced action coincides with the original action of $\cE$ on $\cT_{\cE}$. From now on we consider only central extensions, with no further notice.

For the notion of Morita equivalence, suppose we are given two such extensions $\cG_i$, relative to two foliated orbifold atlases $\cE_i\tto M_i$. Notice that the two atlases come with a specific Morita equivalence $Q$ between them
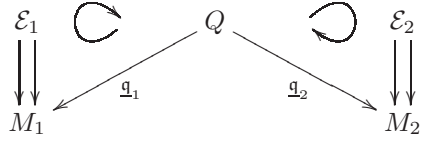
\[
\xymatrix{
\cE_1\ar@<0.25pc>[d] \ar@<-0.25pc>[d]  & \ar@(dl, ul) & Q \ar[dll]^{\underline{\qq}_1}\ar[drr]_{\underline{\qq}_2} & \ar@(ur, dr) & \cE_2 \ar@<0.25pc>[d] \ar@<-0.25pc>[d] \\
M_1&  & & & M_2}
\]
where $Q= Q_{\cE_1, \cE_2}$ is obtained by composing $Q_{\cE_1}: \cE_1\simeq \cB$ with the inverse of $Q_{\cE_2}: \cE_2\simeq \cB$. Via this Morita equivalence, $\cT_{\cE_1}$ corresponds to $\cT_{\cE_2}$, i.e. 
one has an isomorphism between the pull-back bundles
\begin{equation}\label{T-pull-back-identific} 
Q_*: \underline{\qq}_1^*\cT_{\cE_1}\cong \underline{\qq}_2^*\cT_{\cE_2};
\end{equation} 
Hence, for every $u\in Q$, denoting $\underline{\qq}_1(u)= x_1$, $\underline{\qq}_2(u)= x_2$,  one has an isomorphism 
\begin{equation}\label{eq:identif-pb-tori} 
\cT_{\cE_2, x_2}\cong \cT_{\cE_1, x_1}, \ \textrm{denoted}\ \lambda \mapsto \lambda^u.
\end{equation}

\begin{definition}
A {\bf  Morita equivalence of extensions} \eqref{eq:gerbe-orb}  is a Morita equivalence $P$ between the two groupoids,
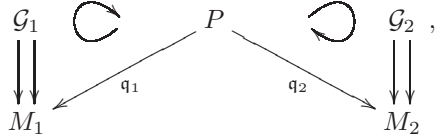
\[
\xymatrix{
\cG_1\ar@<0.25pc>[d] \ar@<-0.25pc>[d]  & \ar@(dl, ul) & P \ar[dll]^{\qq_1}\ar[drr]_{\qq_2} & \ar@(ur, dr) & \cG_2 \ar@<0.25pc>[d] \ar@<-0.25pc>[d] \\
M_1&  & & & M_2},
\]
together with a submersion $\pr: P\to Q$ satisfying the following properties:
\begin{enumerate}[(i)]
\item it is left $\cG_1$-equivariant and right $\cG_2$-equivariant, where the actions of $\cG_i$ on $Q$ are induced from the actions of $\cE_i$
via the submersions $\cG_i\to \cE_i$;
\item it is a (left) principal $\cT_{\cE_1}$-bundle and a (right) principal $\cT_{\cE_2}$-bundle;
\item the two torus actions on $P$ are compatible: if $\widetilde{u}\in P$ and $u=\pr(\widetilde{u})\in Q$, then
\[ \widetilde{u}\cdot \lambda= \lambda^{u}\cdot \widetilde{u},\quad \forall\  \lambda\in \cT_{\cE_2,\qq_2(u)}.\]
\end{enumerate}
\end{definition}

\begin{example}[Pullback of extensions]
\label{ex:gerbes:pull-back:restrictions}
Central extensions cannot be trans\-por\-ted along general Morita maps or equivalences, so they behave quite differently than the other types
of objects over orbifolds that we looked at so far. However, extensions can be pulled-back. First of all, an atlas $Q_\cE: \cE\simeq \cB$ for the orbifold $(B, \cB)$, with $\cE\tto M$, can be pulled-back along any smooth map 
\[ \qq: P\to M\] 
that is transverse to the leaves of the orbit foliation on $M$ induced by $\cE$. It gives rise to the new atlas $\qq^*\cE \tto P$, where $\qq^*\cE= P\times_{M} \cE\times_{M} P$ is the pull-back groupoid and $Q_{\qq^*\cE}= P\times_{M} Q_{\cE}$. Now, given an extension (\ref{eq:gerbe-orb}) over $\cE$, one has a pull-back extension over the orbifold atlas $\qq^*\cE$:
\[ \xymatrix{1\ar[r] & \cT_{\qq^*\cE}  \ar[r]& \qq^*\cG \ar[r] & \qq^*\cE \ar[r] & 1} ,\]
where $\qq^*\cG$ is the pull-back groupoid, while $\cT_{\qq^*\cE}$ is the pull-back bundle. Notice that $\cG\times_{M} P$ gives a Morita bibundle between $\cG$ and $\qq^*\cG$. 

In particular, extensions can be pulled-back via submersions, without affecting the Morita class. One can also take for $\qq$ the inclusion of a complete transversal, and conclude that any extension is Morita equivalent to one over an \'etale atlas.
\end{example}

The pull-back operation can be used to reinterpret and even redefine the notion of Morita equivalence of extensions. First, a Morita equivalence
$Q: \cE_1\simeq \cE_2$ induces an isomorphism of Lie groupoids 
\begin{equation}\label{ME-pull-back-identific} 
Q_*: \underline{\qq}_1^*\cE_1\cong \underline{\qq}_2^*\cE_2.
\end{equation}
Explicitly, a point $(u, \gamma_1, u')\in Q\times_{M_1} \cE_1\times_{M_1} Q= \underline{\qq}_{1}^{*}\cE_1$ corresponds to the point $(u, \gamma_2, u')$  where $\gamma_2$ is uniquely determined by the condition $\gamma_1\cdot u'= u\cdot \gamma_2$. In a general Morita equivalence between extensions, the conditions on $P$ ensure that the similar isomorphism $P_*:\qq_1^*\cG_1\cong \qq_2^*\cG_2$ is an isomorphism of extensions (i.e. not only of groupoids) that is compatible with $Q$. More precisely, one has an isomorphism of commutative diagrams: 
\[ \xymatrix{
1\ar[r] & \cT_{\cE_1}\ar[d]_{\pr^*Q_*}  \ar[r] & \qq_{1}^*\cG_1\ar[d]_-{P_*} \ar[r] & \qq_{1}^*\cE_1 \ar[d]_-{\pr^*Q_*} \ar[r] & 1\\
1\ar[r] & \cT_{\cE_2}  \ar[r] & \qq_{2}^*\cG_2 \ar[r] & \qq_{2}^*\cE_2 \ar[r] & 1},\]
where the left and right vertical maps are the pull-backs of (\ref{ME-pull-back-identific}) and (\ref{T-pull-back-identific}) via $\pr: P\to X$ and we use $\qq_{i}^{*}= \pr^*\circ \underline{\qq}_{i}^{*}$. From the previous example, we conclude:

\begin{corollary} Two central extensions $\cG_i$ over $M_i$ $(i=1, 2)$, are Morita equivalent if and only if there is a manifold $P$ 
together with submersions $\qq_i: P\to M_i$ such that the pull-backs $\qq_{i}^{*}\cG_i$ are isomorphic extensions. 
\end{corollary} 

% After these preparations, we are ready to introduce:

\begin{definition}
Given a torus orbibundle $\cT$ over the orbifold $(B, \cB)$, a {\bf $\cT$-gerbe} over $(B, \cB)$ is a Morita equivalence class 
of central extensions (\ref{eq:gerbe-orb}). We say that the extension represents the gerbe over the atlas $Q_\cE: \cE\simeq \cB$.
\end{definition}

As we explained above, a general gerbe can be represented only over {\it some} or\-bi\-fold atlas.
Still, Example \ref{ex:gerbes:pull-back:restrictions} 
shows that if a gerbe is represented by an extension over the atlas $Q_\cE: \cE\simeq \cB$, with $\cE\tto M$, then we can pullback along any map $\qq:M'\to M$ transverse to the orbits of $\cE$. Letting $\qq$ be the inclusion of a complete transversal, we see that any gerbe can be represented over an \'etale atlas. 

\subsubsection{Group structure} We start by observing that any atlas $Q_\cE:\cE\simeq \cB$ for the orbifold $(B,\cB)$ comes with a trivial central extension 
\begin{equation}\label{eq:triv:ext:orbif} 
\xymatrix{1\ar[r] & \cT_{\cE} \ar[r] & \cT\Join \cE \ar[r]^-{\pr} & \cE \ar[r] & 1}
\end{equation}
whose Morita class does not depend on the choice of atlas $\cE$. These define {\bf the trivial $\cT$-gerbe} over the orbifold $(B, \cB)$.

\begin{example}\label{ex:confused} 
As in the smooth case, there are non-trivial extensions that may still represent the trivial gerbe. Actually, one can give a full characterization of such extensions, generalizing Corollary \ref{cor:trivial-gerbe} to the orbifold context.
The starting point is an orbifold atlas $Q_\cE: \cE\simeq \cB$, with $\cE\tto M$, and a $\cT_{\cE}$-torsor $X$ over $M$. This gives rise to the groupoid $\cE_{X}(M)= (X\times_M\cE\times_{M} X)/\cT_{\cE}$ exactly as in Definition \ref{def:E:integration} (hence the quotient is modulo the action (\ref{act-t-hop})),
which fits into a central extension 
% \begin{equation}\label{eq:changed-gauge-ext} 
\[
\xymatrix{1\ar[r] & \cT_{\cE} \ar[r] & \cE_X(M) \ar[r] & \cE \ar[r] & 1}.
\]
% \end{equation}
Moreover, exactly as in Proposition \ref{prop:Morita:E:holonomy}, using $X\times_M\cE$ as Morita bibundle, it follows that this extension represents
the trivial gerbe. For the converse, one uses an argument similar to
that of Lemma \ref{lemma:reinterpret-torsors}. 
% We would like to emphasize that in this construction there is no assumed relationship between $\cE$ and $X$ at all. On the other hand, the constructions 
% makes $X$ into a $\cT_{\cE}$-torsor over $\cE_{X}(M)$. 
\end{example}

The {\bf inverse $\cT$-gerbe} of the gerbe defined by an extension (\ref{eq:gerbe-orb}) is represented by the opposite extension:
\[ \xymatrix{1\ar[r] & \cT_{\cE} \ar[r]^{i} & \cG^{\opp} \ar[r]^-{\pr} & \cE^{\opp} \ar[r]& 1}\]
Note that inversion gives a groupoid isomorphism $\cE\cong \cE^\opp$, 
so this extension represents a well defined gerbe over the same orbifold.
\\

Next, for the fusion product of extensions, we start with two extensions $\cG_i$ as above, relative to two orbifold atlases $\cE_i\tto M_i$. We still denote by $Q$ the induced Morita equivalence between $\cE_1$ and $\cE_2$, with projections denoted $\qq_i: Q\to M_i$ $(i=1, 2)$. While the pull-backs of $\cE_i$ to $Q$ are isomorphic by (\ref{ME-pull-back-identific}), there is a more symmetric way to represent the resulting groupoid over $Q$, namely as the fibered product over $M_1$ (for $\ltimes$) and $M_2$ (for $\rtimes$):
\[ \cE_{1,2}:= \cE_1\ltimes Q \rtimes \cE_2.\]
Indeed, $(\gamma_1, u, \gamma_2)\mapsto (\gamma_1\cdot u, \gamma_1)$ identifies $\cE_{1,2}$ with $\underline{\qq}_{1}^{*}\cE_1$ (and similarly for $\underline{\qq}_{2}^{*}\cE_2$). The groupoid $\cE_{1, 2}\tto Q$ is another atlas for $(B, \cB)$ and the advantage of this point of view is that one can define similarly the groupoid 
\[ \cG_{1, 2}= \cG_1\ltimes Q\rtimes \cG_2\tto X \]
and the torus bundle:
\[ \quad \cT_{1, 2}= \cT_{\cE_1}\ltimes Q \rtimes \cT_{\cE_2}= \underline{\qq}_{1}^{*}\cT_{\cE_1}\otimes \underline{\qq}_{2}^{*}\cT_{\cE_2}.\]
We now act with the smaller torus bundle (\ref{T-pull-back-identific}) using, e.g. its left hand side $\underline{\qq}_{2}^{*}\cT_{\cE_2}$. The action is along the map $\cG_{1, 2}\to M_2$, $(g_1, u, g_2)\mapsto \underline{q}_2(u)$ and it is given by 
\[ (g_1, u, g_2)\cdot \lambda:= (g_1\cdot \lambda^{u}, u, \lambda\cdot g_2),\]
where we use again the notation (\ref{eq:identif-pb-tori}). Finally, we define the {\bf fusion product of extensions} by considering the groupoid
\[ \cG_1\star \cG_2:= (\cG_1\ltimes Q\rtimes \cG_2)/\cT_{1,2}\]
and the associated central extension over the orbifold atlas $\cE_{1, 2}$. 

It is tedious but straightforward to extend the discussion from the smooth case to conclude that the fusion product gives rise to a group structure on the set of $\cT$-gerbes over the orbifold $(B, \cB)$ with the above identity gerbe and inverse operation on gerbes.

\subsubsection{The Dixmier-Douady class} We proceed as in the smooth case, starting with a characterization of triviality. This is the following generalization of Lemma \ref{lemma:trivial-gerbe} to the orbifold case, which should now be obvious:

\begin{lemma}\label{lemma:triv-gerbe-orbifold} Given a central extension (\ref{eq:gerbe-orb}), the following are equivalent:
\begin{enumerate}[(a)]
\item it represents the trivial gerbe over $(B,\cB)$;
\item there exists a principal $\cG$-bundle $P$ over $\cB$ which lifts the principal $\cE$-bundle over $\cB$ given by the atlas
$Q_{\cE}: \cE\simeq \cB$. Equivalently: $P\in \cM(\cB, \cG)$ fits into a commutative diagram of Morita maps
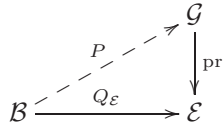
\[ \xymatrix{
 & & \cG \ar[d]^{\pr} \\
 \cB \ar[rr]^-{Q_{\cE}} \ar@{-->}[rru]^{P} & & \cE } \]
\end{enumerate}
\end{lemma}

Next, % we observe that, 
eventually after pulling back to a complete transversal $T\hookrightarrow N$, 
we may assume that the atlas $\cB\tto N$ is already \'etale. This allows us to use a \v{C}ech-type description of orbifold cohomology,
due to Moerdijk \cite{Moe03}, which we now recall. 

We consider abelian sheaves over $(B,\cB)$, i.e. sheaves $\cS$ of abelian groups on $N$ together with an action of $\cB$ from the right: hence any arrow  $x\to y$ of $\cB$ gives a group homomorphism between germs $\cS_y\to \cS_x$. These form an abelian category with enough injectives and the {\bf sheaf cohomology} groups $H^\bullet(\cB,\cS)$ are defined as the right derived functors associated to the functor $\Gamma(\cdot)^\cB$ of taking invariant sections. Hence:
\[ H^n(\cB,\cS):=H^n(\Gamma(I^\bullet)^\cB), \]
for some injective resolution $\xymatrix{0\ar[r]& \cS\ar[r]& I^0\ar[r]& I^1\ar[r] &\cdots}$. This sheaf cohomology is invariant under Morita equivalences. 

For the \v{C}ech cohomology description we use the embedding  category $\Emb_\cU(\cB)$ \cite[Section 7]{Moe03} (see Section \ref{ssec:Torus orbibundles}) . 
Similar to the transition functions for principal bundles over $\cB$, any abelian $\cB$-sheaf $\cS$ gives rise to a preseheaf
$\widetilde{\cS}$ on $\Emb_\cU(\cB)$, i.e. to a contravariant functor $\cA:\Emb_{\cU}(\cB)\to \mathbf{AbGrp}$. Explicitly, $\widetilde{\cS}$ associates to $U\in \cU$ the space $\cS(U)$ of sections of $\cS$ over $U$ and to an arrow $\sigma:U\to V$ of $\Emb_{\cU}(\cB)$ the map $\sigma^*: \cS(V)\to \cS(U)$ described as follows: for $\lambda\in \cS(V)$ the section $\sigma^*(\lambda)\in \cS(U)$ has germ at $x\in U$:
\[ \sigma^*(\lambda)_x=\lambda_{t(\sigma(x))}\cdot\sigma(x). \]
The \v{C}ech complex of $\cB$ relative to $\cU$ with coefficients in $\cS$, denoted $(\check{C}^{\bullet}_{\cU}(\cB ,\cS),\d)$, is defined as the standard complex computing the cohomology of the discrete category $\Emb_{\cU}(\cB)$ with coefficients in $\widetilde{\cS}$. That means: \begin{itemize}
\item a cocycle $c\in \check{C}^{n}_{\cU}(\cB ,\cS)$ is a map which associates to each string of n-composable arrows $\xymatrix@C=15pt{U_0&U_1\ar[l]_{\sigma_1}&\cdots \ar[l]_{\sigma_2} & U_n\ar[l]_{\sigma_n}}$ an element $c_{\sigma_1,\dots,\sigma_n}\in \cS(U_n)$.
\item the differential $\d: \check{C}^{n}_{\cU}(\cB ,\cS)\to \check{C}^{n+1}_{\cU}(\cB ,\cS)$ is given by 
\begin{align*} 
(\d c)_{\sigma_1,\dots,\sigma_{n+1}}=c_{\sigma_2,\dots,\sigma_{n+1}}+\sum_{i=1}^n (-1)^n &c_{\sigma_1,\dots,\sigma_i\sigma_{i+1},\dots,\sigma_{n+1}}+\\
&+(-1)^{n+1}\sigma_{n+1}^{*} \left( c_{\sigma_1,\dots,\sigma_n}\right) .
\end{align*}
\end{itemize}
We denote the resulting cohomology by $\check{H}^{\bullet}_{\cU}(\cB,\cS)$. The \v{C}ech complex and the \v{C}ech cohomology are functorial with respect to refinements of covers, so one can pass to colimits and define the {\bf \v{C}ech cohomology of the orbifold} $(B, \cB)$ 
with coefficients in $\cS$ as:
\[ \check{H}^\bullet(\cB,\cS)=\text{colim} \, H^{\bullet}_{\cU}(\cB,\cS).\] 

We now have the following result:

\begin{proposition}[\cite{CrMo}] 
If $H^i(U,\cS)=0$ for $i>0$ and all $U\in\cU$, then
\[ H^\bullet(\cB,\cS)\cong \check{H}^{\bullet}_{\cU}(\cB, \cS),\]
and these isomorphisms are compatible with taking refinements. In particular,
\[ \check{H}^\bullet(\cB,\cS)\cong H^\bullet(\cB,\cS).\]
\end{proposition}

And here is a result that is relevant for our discussion on orbifolds and which makes essential use of properness. 

\begin{lemma} For any proper \'etale groupoid $\cB$ and a torus bundle $\cT$ over $\cB$, with corresponding lattice $\Lambda$, one has 
\[ H^{\bullet}(\cB, \ucT)\cong H^{\bullet+1}(\cB, \Lambda) .\]
Similarly, if $(\cT,\omega_\cT)$ is a symplectic torus bundle over a proper \'etale groupoid $\cB$: 
\[ H^{\bullet}(\cB, \ucT_{\Lagr})\cong H^{\bullet+1}(\cB, \cO_{\Lambda}) .\]
\end{lemma}

\begin{proof} The key remark is that, due to the properness of $\cB\tto N$, for any $\cB$-sheaf $\cS$ of $\R$-vector spaces which is fine as a sheaf over the base manifold $N$, one has 
\[ H^{\bullet}(\cB, \cS)= 0,\quad \forall\ k \geq 1.\]
%The statement then follows by using the cohomology long exact sequences associated to (\ref{exp-seseq}) and (\ref{eq:ses-O-Lambda}). 
Use now the exact sequences induced by (\ref{exp-seseq}) and (\ref{eq:ses-O-Lambda}) in cohomology.
\end{proof} 

With all the cohomology apparatus in place, we can now proceed to define the Dixmier-Douady class of a gerbe over $(B,\cB)$. The construction is entirely similar to the smooth case, the main difference being that we have to start with an extension (\ref{eq:gerbe-orb}) over an atlas $\cE\tto M$ 
which may be different from $\cB$. We denote by $Q= Q_{\cE}$ the given Morita equivalence between $\cE$ and $\cB$ and we would like to measure the failure of extending $Q$, viewed as a principal $\cE$-bundle over $\cB$, to a principal $\cG$-bundle over $\cB$. This can be rephrased in terms of transition functions: we choose the basis $\cU$ of $N$ together with local sections $s_U: U\to Q$ of $\qq_2: Q\to N$, so that we can consider the transition system $\{f_U, g_{\sigma}\}$. We may assume that each $U\in \cU$ is contractible. For each arrow $\sigma: U\to V$ of $\Emb_{\cU}(\cB)$, the pull-back of $\pr: \cG\to \cE$ via $g_{\sigma}: U\to \cE$ gives a $\cT$-torsor $\cG_{\sigma}\to U$. Since $U$ is contractible it will admit a section, i.e. we obtain a lift $\widetilde{g}_{\sigma}: U\to \cG$ of $g_{\sigma}$. Of course, the cocycle condition may fail: for composable arrows $\xymatrix@C=15pt{U_0&U_1\ar[l]_{\sigma_1}&U_2 \ar[l]_{\sigma_2}}$ in $\Emb_\cU(\cB)$, one has a section of $\cT$ over $U$ given by 
\[ c_{\sigma_1, \sigma_2}:= \widetilde{g}_{\sigma_1\circ \sigma_2} \widetilde{g}_{\sigma_1} \widetilde{g}_{\sigma_2}\] 
and $(\sigma_1,\sigma_2)\mapsto c_{\sigma_1,\sigma_2}$ defines a cocyle in $C^{2}_{\cU}(\cB,\ucT)$. This gives rise to a class in cohomology which, by the usual arguments, is independent of the various choices:

\begin{definition}
The {\bf Dixmier-Douady class of the gerbe} represented by the extension (\ref{eq:gerbe-orb}) is
\[ \DD(\cG):=[c_{\sigma_1, \sigma_2}]\in H^2(\cB, \ucT)\cong H^3(\cB,\Lambda).\] 
\end{definition}

To see that $\DD(\cG)$ only depends on the Morita equivalence class, one first proves the additivity with respect to the fusion product:
\[  \DD(\cG_1\star \cG_{2})= \DD(\cG_1)+\DD(\cG_{2}). \]
This is done exactly like in the smooth case, but using the embedding category. Also, it is clear that $\DD(\cG^{\opp})= -\DD(\cG)$. Hence, if $\cG_1$ and $\cG_2$ are Morita equivalent then, since $\cG_1\star \cG_{2}^{\opp}$ is Morita equivalent to the trivial extension, we find:
\[ \DD(\cG_1)- \DD(\cG_2)= \DD(\cG_1\star \cG_{2}^{\opp})= 0.\]
Notice that, by construction, the vanishing of this class is equivalent to the fact that $\cG$ represents the trivial gerbe. In this way we have extended the discussion of gerbes from the smooth to the orbifold case. 

\begin{theorem} 
\label{DD-clas:orbi} 
Given a torus bundle $\cT$ over the orbifold $(B,\cB)$ with lattice $\Lambda$, the Dixmier-Douady class induces an isomorphism:
\[ \DD: \Gerbes_{(B,\cB)}(\cT) \to H^2(\cB, \ucT)\cong H^3(\cB,\Lambda). \]
\end{theorem}

\subsubsection{Symplectic gerbes}
The symplectic version of gerbes over an orbifold should now be obvious. One starts with a symplectic torus bundle $(\cT, \omega_{\cT})$ over the orbifold $(B, \cB)$ or, equivalently, with an integral affine structure. Then one looks at symplectic central extensions of the form
\begin{equation}
\label{eq:sympl-gerbe-orb-sympl} 
\xymatrix{1\ar[r] & \cT_{\cE} \ar[r]^{i} & (\cG, \Omega) \ar[r]^-{\pr} & \cE \ar[r]& 1},
\end{equation}
where $(\cG, \Omega)$ is now a symplectic groupoid and $i^*\Omega=\omega_{\cT}$. Similarly for the notion of symplectic Morita equivalence, using symplectic Morita bibundles (Appendix \ref{App:Symplectic Morita equivalence}).
Therefore a {\bf symplectic $(\cT, \omega_{\cT})$-gerbe} over $(B, \cB)$ is a symplectic Morita equivalence class of such symplectic central extensions. 

The group structure on gerbes is, this time, a bit more subtle; e.g., even the trivial extension (\ref{eq:triv:ext:orbif}) may fail to be symplectic. The most satisfactory solution is to consider presymplectic extensions (see  Section \ref{The twisted Dirac setting}). The shortest solution is to pass to an \'etale atlas. % , e.g. by restricting to a transversal. 
Indeed, if $\cB$ is \'etale, then:

% As explained above, after passing to a transversal, one may assume that $\cB$ is \'etale, since the resulting extension will still be symplectic. For example, one has a group structure on symplectic gerbes, since:
\begin{itemize}
\item The trivial central extension \eqref{eq:triv:ext:orbif} is now symplectic, yielding the trivial symplectic gerbe. Similarly, the inverse of a symplectic gerbe is symplectic.
\item The fusion product of symplectic extensions is a symplectic extension, and is independent of the symplectic Morita class.
\end{itemize}
For the construction of the Lagrangian Dixmier-Douady class of a symplectic gerbe, we assume again that $\cB$ is \'etale,
and we proceed exactly as in the non-symplectic case, except that in the construction above one now choses \emph{Lagrangian} sections
$s_U: U\to Q$, making use of Corollary \ref{lemma:help-c-constr}. This leads to:

\begin{theorem}
\label{thm:Lagrangian:class:symp:gerbe-orbi} 
Given a symplectic groupoid $(\cG,\Omega)\tto M$ fitting into an extension \eqref{eq:sympl-gerbe-orb-sympl} there is an associated cohomology class:
\[ \DD(\cG,\Omega)\in H^2(\cB, \underline{\cT}_{\Lagr}).\]
Moreover:
\begin{enumerate}[(i)]
\item this construction induces an isomorphism between the group of symplectic $\cT$-gerbes over $(B,\cB)$ and $H^2(\cB, \underline{\cT}_{\Lagr})$. 
\item $\DD(\cG, \Omega)= 0$ if and only if $(\cG, \Omega)$ is isomorphic to $\cE_X(M, \pi)$, the $\cE$-integration of $(M, \pi)$ relative to a proper isotropic realization $(X, \Oga)\to (M, \pi)$. % (cf. Definition \ref{def:E:integration}).
\end{enumerate}
\end{theorem}

Finally, let us look at the relationship between gerbes and symplectic gerbes. 
At the level of the Dixmier-Douady classes, the relationship is provided by the cohomology sequence induced by (\ref{eq:ses-O-Lambda}), which gives:
\begin{equation}\label{eq:ses-O-Lambda-conseq} 
\xymatrix{  H^3(\cB, \R)\ar[r]^-{i_*}& H^3(\cB, \cO_{\Lambda}) \ar[r]^-{\d_*} &  H^3(\cB, \Lambda) \ar[r]& H^4(\cB, \R)}
\end{equation}
Since $\d_*$ maps $\DD(\cG, \Omega)$ to $\DD(\cG)$, the preimage ${\d^*}^{-1}(\DD(\cG))$ is the affine space $\DD(\cG,\Omega)+i_*H^3(\cB, \R)$. 
A geometric interpretation will be given in the next section, using twisted Dirac structures, where the role of $H^3(B)$ will be to provide the background 3-forms. 
For now, let us recall a more explicit model for $H^{\bullet}(\cB, \R)$: in the DeRham complex 
$(\Omega^{\bullet}(N), \d)$ on the base $\cB\tto N$ we consider the subcomplex $(\Omega^{\bullet}(N)^{\cB}, \d)$ of $\cB$-invariant forms. Of course, this makes sense for any \'etale groupoid $\cB$ and defines a cohomology $H^{\bullet}_{\bas}(\cB)$. 
% We will give a geometric interpretation for this in the next section in the smooth setting, using twisted Dirac structures, where $H^3(B)$ will provide the background 3-forms. For now, let us recall that a more explicit model for $H^{\bullet}(\cB, \R)$: in the DeRham complex 
% $(\Omega^{\bullet}(N), \d)$ on the base $\cG\tto N$ we consider the subcomplex $(\Omega^{\bullet}(N)^{\cB}, \d)$ of $\cB$-invariant forms. Of course, this makes sense for any \'etale groupoid $\cB$ and defines a cohomology 
% \[ H^{\bullet}_{\bas}(\cB).\]

\begin{corollary} For any proper \'etale groupoid $\cB$ one has natural isomorphisms 
\[ H^{\bullet}_{\bas}(\cB)\cong H^{\bullet}(\cB, \R).\]
\end{corollary}

The concrete description of $H^{\bullet}(\cB, \R)$ can be extended to non-\'etale orbifold atlases $Q_{\cE}: \cE\simeq \cB$. If $\cE\tto M$ is any proper foliation groupoid one can talk about forms on $M$ which are $\cE$-basic: they are the forms on $M$ whose pull-back via the source and the target map of $\cE$ coincide. This defines a sub-complex $(\Omega^{\bullet}_{\textrm{bas}}(\cE), \d)$ of the DeRham complex of $M$ and gives rise to a (Morita invariant!) cohomology of $\cE$, denoted $H^{\bullet}_{\bas}(\cE)$. When $\cE$ is \'etale, this agrees with the previous definition. 

\begin{corollary}\label{cor:bas-obifold} For any proper \'etale groupoid $\cB$ and any orbifold atlas $\cE\simeq \cB$ one has natural isomorphisms 
\[ H^{\bullet}_{\bas}(\cE) \cong H^{\bullet}_{\bas}(\cB) \cong H^{\bullet}(\cB, \R).\]
\end{corollary}

In particular, when $H^3_{\bas}(\cE)=H^3_{\bas}(\cB)=0$, if the ordinary gerbe defined by a symplectic extension is trivial, then so is the induced symplectic gerbe. 

% We conclude, e.g., that when $H^3_{\bas}(\cE)=H^3_{\bas}(\cB)=0$, the Lagrangian Dixmier-Douady class vanishes iff and only the usual Dixmier-Douady class vanishes, or in other words, there is a symplectic Morita equivalence with the trivial extension iff there is an ordinary Morita equivalence.

%%%%%%%%%%%%%%%%%%%%%%%%
\subsection{The twisted Dirac setting}
\label{The twisted Dirac setting}
%%%%%%%%%%%%%%%%%%%%%%%%

We close our discussion on gerbes by explaining briefly how to pass 
from the Poisson to the (twisted) Dirac setting. The main outcome will not be a new theory of ``twisted presymplectic gerbes'', but a 
new way to represent symplectic gerbes by more general extensions. 

Let $(\cG, \Omega, \phi)\tto M$ be $\phi$-twisted presymplectic groupoid integrating a Dirac 
manifold $(M,L)$ with background 3-form $\phi$.
As discussed in Remark \ref{rem:twisted:tias}, in the proper regular case this still gives rise to a central extension 
\[ \xymatrix{1\ar[r] & \cT(\cG) \ar[r] & (\G, \Omega) \ar[r] & \cE(\cG) \ar[r]& 1}\]
inducing  an integral affine orbifold structure on the space $B$ of orbits, with associated presymplectic torus bundle 
$\cT(\cG)$ equipped with the restriction of $\Omega$.  

Start now with a sympletic torus bundle $(\cT,\omega_\cT)$ over the orbifold $(B, \cB)$ and look at central extensions
defined over some orbifold atlas $Q_{\cE}: \cE\simeq \cB$, of type 
% \begin{equation}\label{eq:extension-Dirac} 
\[
\xymatrix{1\ar[r] & \cT_{\cE} \ar[r] & (\G, \Omega, \phi) \ar[r] & \cE \ar[r]& 1},
\]
% \end{equation}
with $i^*\Omega=\omega_{\cT_\cE}$, but where $(\cG, \Omega, \phi)$ is now a twisted presymplectic groupoid.  Such an extension will be called a central {\bf twisted presymplectic extension} on the integral affine orbifold $(B,\cB)$. The notion of {\bf presymplectic Morita equivalence} between such extension is defined exactly as in the symplectic case, with the only difference that one now allows twisted presymplectic bibundles and the twisting will vary (see Appendix \ref{appendix:moment:maps}).  The construction of the {\bf Lagrangian Dixmier-Douady class} of such an extension carries over modulo some obvious modifications, giving 
\[ \DD(\cG, \Omega, \phi)\in H^2(\cB, \ucT_{\Lagr})\cong H^3(\cB, \cO_{\Lambda}).\]
However, %as we already mentioned, 
one does not obtain a new notion of ``twisted presymplectic gerbes'':

\begin{proposition}
For an integral affine orbifold $(B,\cB)$, one has that:
\begin{enumerate}[(i)]
\item any twisted presymplectic extension is presymplectic Morita equivalent to a symplectic extension;
\item two symplectic extensions are presymplectic Morita equivalent if and only if they are symplectic Morita equivalent.
\end{enumerate}
\end{proposition} 

\begin{proof} Part (i) follows by observing that:
\begin{enumerate}[(a)]
\item given a (regular) twisted presymplectic groupoid $(\cG, \Omega, \phi)$, after restricting to a complete transversal to the foliation induced on the base, one obtains a twisted presymplectic groupoid whose 2-form becomes non-degenerate. Moreover, as in Example \ref{ex:gerbes:pull-back:restrictions}, this operation does not change the (presymplectic) Morita equivalence class; 
\item given an extension $(\cG, \Omega, \phi)$ over $M$ and any 2-form $\tau\in\Omega^2(M)$, 
the original extension is presymplectic Morita equivalent to its $\tau$-gauge transform $(\cG, \Omega- t^*\tau+ s^*\tau, \phi- d\tau)$. In particular, if $\phi$ is exact, then by such a gauge transform one can pass to an untwisted presymplectic extension.
\end{enumerate}
To prove (ii) we invoke again the fact that a presymplectic Morita equivalence between symplectic groupoids is automatically symplectic. 
\end{proof}

However, it is still interesting to think about twisted presymplectic representations of symplectic gerbes. To illustrate that 
% over the integral affine orbifold. 
% The construction of the Lagrangian Dixmier-Douady class of such an extension,
% \[ \DD(\cG, \Omega, \phi)\in H^2(\cB, \ucT_{\Lagr})\cong H^3(\cB, \cO_{\Lambda}),\]
% carries over modulo some obvious modifications. 
let us assume first that $B$ is smooth, so that (\ref{eq:ses-O-Lambda-conseq}) gives us the exact sequence: 
\begin{equation}\label{eq:ses-O-Lambda-conseq2} 
\xymatrix{  H^3(B, \R)\ar[r]^-{i_*}& H^3(B, \cO_{\Lambda}) \ar[r]^-{\d_*} &  H^3(B, \Lambda) }. 
\end{equation}
From the construction of the Dixmier-Douady classes we immediately deduce:

\begin{proposition} Given an integral affine manifold $(B, \Lambda)$, any closed 3-form $\eta\in\Omega^3_\cl(B)$ and a twisted presymplectic extension over the submersion $p_M: M\to B$
\[ \xymatrix{1\ar[r] & (\cT_{\Lambda})_M \ar[r] & (\G, \Omega, \phi) \ar[r] & M\times_{B}M \ar[r]& 1},
 \]
then $(\G, \Omega, \phi+ p_{M}^{*}\eta)$ defines another twisted presymplectic extension with class:
\[ \DD(\G, \Omega, \phi+ p_{M}^{*}\eta)= \DD(\G, \Omega, \phi)+ i_{*}[\eta].\]
\end{proposition}

In particular, even if we are only interested in symplectic extensions, twisted (symplectic!) ones arise if one wants to understand the difference between symplectic gerbes and non-symplectic ones:

\begin{corollary} Given an integral affine manifold $(B, \Lambda)$, if two $\cT_{\Lambda}$-central symplectic extensions $(\cG_i, \Omega_i)$ induce the same (non-symplectic) gerbe, i.e. if there is a Morita equivalence of extensions:
\[ \cG_1\simeq \cG_2,\] 
then there exists a closed 3-form $\eta$ on $B$ and a twisted symplectic Morita equivalence of extensions:
\[(\cG_1, \Omega_1)\simeq (\cG_2, \Omega_2, \eta).\]
\end{corollary}

This discussion can be generalized from a smooth $B$ to orbifolds $(B, \cB)$ by using (\ref{eq:ses-O-Lambda-conseq2})
instead of (\ref{eq:ses-O-Lambda-conseq}) and the description of $H^3(\cB, \R)$ via basic forms, provided by Corollary
\ref{cor:bas-obifold}. In other words, one works with presymplectic extensions
\[ \xymatrix{1\ar[r] & \cT_{\cE} \ar[r] & (\G, \Omega, \phi) \ar[r] & \cE\ar[r] & 1}
 \]
defined over orbifold atlases $Q_{\cE}: \cE\simeq \cB$ and one replaces the closed 3-forms $\eta$ on $B$ by $\cE$-basic closed 3-forms on the base $M$ of $\cE$. One concludes that for any $\cE$-basic 3-form $\eta$ on $M$, the twisted presymplectic groupoid $(\G, \Omega, \phi+ \eta)$ has class satisfying the same formula as in the previous proposition.

\begin{remark}
It is not hard to see that all our definitions of (symplectic) gerbes over orbifolds extend to gerbes over any Lie groupoid (gerbes over stacks) and the definition of the (Lagrangian) Dixmier-Douady class makes sense at least for gerbes over any foliation groupoid.
\end{remark}

%%%%%%%%%%%%%%%%%%%%%%%%
%%%%%%%%%%%%%%%%%%%%%%%%
%%%%%%%%%%%%%%%%%%%%%%%%
\appendix
%%%%%%%%%%%%%%%%%%%%%%%%
%%%%%%%%%%%%%%%%%%%%%%%%
%%%%%%%%%%%%%%%%%%%%%%%%
%%%%%%%%%%%%%%%%%%%%%%%%
  
%%%%%%%%%%%%%%%%%%%%%%%%
%%%%%%%%%%%%%%%%%%%%%%%%
\section{Symplectic groupoids and moment maps}
\label{appendix:moment:maps}
%%%%%%%%%%%%%%%%%%%%%%%%
%%%%%%%%%%%%%%%%%%%%%%%%

In this appendix we recall some basic notions and results associated with symplectic groupoids and their
actions on symplectic manifolds, which are needed 
throughout the paper. Much of this material goes back 
to the work of Mikami and Weinstein \cite{WeMi} and Xu \cite{Xu}, complemented by the results in \cite{CF2}.

%%%%%%%%%%%%%%%%%%%%%%%%
\subsection{Hamiltonian $\cG$-spaces}
\label{App:Hamiltonian}
%%%%%%%%%%%%%%%%%%%%%%%%

Given a Poisson manifold $(M, \pi)$ and an integrating symplectic groupoid $(\cG, \Omega)$, a {\bf Hamiltonian $\cG$-space} 
is a symplectic manifold $(X, \Oga)$ endowed with a smooth map
\[ q: (X, \Oga)\to M,\]
as well as an action $m$ of $\cG$ on $X$ along $q$ which \emph{symplectic}. The condition that the action is symplectic can be expressed by saying that its graph: 
\[ \Graph(m)= \{(g, x,g\cdot x): g\in \cG, x\in X, \s(g)= q(x)\}\subset \cG\times X\times X\]
is a Lagrangian submanifold of $(\cG\times X\times X, \Omega\oplus\Oga\oplus -\Oga)$. Alternatively, this condition can be rewritten in the multplicative form: 
\[ m^*(\Oga)= \pr_{1}^{*}(\Omega)+ \pr_{2}^{*}(\Oga),\]
where $\pr_i$ are the two projections. 

Recall that a Poisson map $q:X\to M$ is called {\bf complete} if for any complete Hamiltonian vector field $X_{h}\in \X(M)$ the pullback $X_{h\circ q}\in \X(X)$ is complete. The very first basic fact about Hamiltonian $\cG$-spaces is:

\begin{lemma}\label{lemma-sympl-acts-infin} Given a symplectic groupoid $(\cG, \Omega)$ integrating the Poisson manifold $(M, \pi)$ and a symplectic action of $(\cG, \Omega)$ on $(X, \Oga)$, one has: 
\begin{enumerate}[(i)]
\item $q: (X, \Oga)\to (M, \pi)$ is a complete Poisson map; 
\item the induced infinitesimal action $T^*M$ on $X$,
\[ \sigma: q^*T^*M\to TX, \quad
\sigma(u, \alpha_x)= \left.\frac{\d}{\d t}\right|_{t=0} \exp(t\alpha_{x})\cdot u. \]
satisfies the moment map condition 
\begin{equation}\label{eq:gen-mom-map-cond0} 
i_{\sigma(\alpha)}(\Oga)= q^*\alpha\quad \forall\ \alpha\in \Omega^1(M).
\end{equation}
\end{enumerate}
\end{lemma}

Conversely, if one starts with a Poisson map $q: (X, \Oga)\to (M, \pi)$, it is immediate to check that the moment map condition \eqref{eq:gen-mom-map-cond0}  defines an infinitesimal Lie algebroid action $\sigma: q^*T^*M\to TX$. Moreover, one has:

\begin{lemma}\label{lemma-sympl-acts-infin-conv} Let  $q: (X, \Oga)\to (M, \pi)$ be a Poisson map with associate infinitesimal action $\sigma: q^*T^*M\to TX$. Then: 
\begin{enumerate}[(i)]
\item if for some symplectic groupoid $(\cG, \Omega)$ integrating $(M, \pi)$ the infinitesimal action $\sigma$ integrates to an action 
of $\cG$ on $X$, then the action is symplectic;
\item if $q: (X, \Oga)\to (M, \pi)$  is complete, the infinitesimal action always integrate to a symplectic action of the Weinstein groupoid 
$(\Sigma(M,\pi),\Omega)$ on $(X, \Oga)$.
\end{enumerate}
\end{lemma}

The standard theory of Hamiltonian $G$-spaces for a Lie group $G$ is recovered by letting $\cG= T^*G= G\ltimes \gg^*$ be the cotangent symplectic groupoid
integrating the linear Poisson structure on $\gg^*$, and $q:(X,\Oga)\to\gg^*$ be the usual moment map.

\subsection{Symplectic quotients of Hamiltonian $\cG$-spaces}
\label{App:Symplectic quotients of Hamiltonian}

Given a Hamiltonian $\cG$-space $(X,\Oga)$, we will say that:
\begin{enumerate}[(a)]
\item the {\bf action is free} at $u\in X$ if the isotropy group $\cG_u$ of the action is trivial.
\item the {\bf action is infinitesimally free} at $u\in X$ if the isotropy group $\cG_u$ of the action is discrete. Equivalentely, if the isotropy Lie algebra of the infinitesimal action $\sigma$ is trivial, or still if $\sigma_u$ is injective. By (\ref{eq:gen-mom-map-cond0}) this is also equivalent to $q$ being a submersion at $u$.
\end{enumerate}
%The moment map equation (\ref{eq:gen-mom-map-cond0}) shows that the action is infinitesimally free at $u$ if and only if $q$ is a submersion at $u$. 
Symplectic reduction makes sense in the general context of a Hamiltonian $\cG$-space $q: (X, \Oga)\to (M, \pi)$: the {\bf symplectic quotient} of $(X, \Oga)$ at a point $x\in M$ is:
\[ X//_{x}\ \cG:= q^{-1}(x)/\cG_x.\]
This carries a canonical symplectic form uniquely determined by the condition that its pull-back to $q^{-1}(x)$ is $\Oga|_{q^{-1}(x)}$. 

As in the classical case, to ensure smoothness one assumes that $\cG$ is proper and one restricts to points where the action is free, which form an open dense subspace of $M$. The following proposition shows that, in that case, if the fibers of $q$ are connected, then the symplectic reductions $q^{-1}(x)/\cG_x$ can be interpreted as the symplectic leaves of a second Poisson manifold:
\[ X_\red:= X/\cG, \]
which will still be of proper type, with the same space of symplectic leaves as $(M, \pi)$.

% As in the classical case, to ensure smoothness one assumes that $\cG$ is proper and one restricts to points where the action is free. More generally, one can consider points where the action is infinitesimally free, which amounts to allow mild singularities, so the symplectic quotients will then be orbifolds. Moreover, there are some very interesting phenomena taking place at points where the action fails to be free as we will see in \cite{CFMc}.

% Let us restrict ourselves to the case of free Hamiltonian spaces (e.g., by restricting to the open set where the action is free).
% The following proposition shows that in that case, if the fibers of $q$ are connected, then the symplectic reductions $q^{-1}(x)/\cG_x$ can be interpreted as the symplectic leaves of a second Poisson manifold:
% \[ X_\red:= X/\cG, \]
% which will still be of proper type, with the same space of symplectic leaves as $(M, \pi)$. 

\begin{proposition}\label{prop:free-reduction} 
Let $(\cG, \Omega)$ be a proper symplectic integration of $(M, \pi)$ and let $q: (X, \Oga)\to M$ be a free Hamiltonian $\cG$-space. Then:
\begin{enumerate}[(i)]
\item $X_\red$ is smooth and carries a unique Poisson structure $\pi_\red$ making the canonical projection
\[ p: (X, \Oga)\to (X_\red, \pi_\red)\]
into a Poisson submersion;
\item the gauge groupoid of the principal $\cG$-space $X$:
\begin{equation}\label{cG-red} 
\Gauge{\cG}{X}:=\left( X\times_{M} X/\cG \tto X_\red \right),
\end{equation}
with the 2-form induced from $\pr_{1}^{*}\Oga- \pr_{2}^{*}\Oga\in\Omega^2(X\times_{M} X)$, is a proper symplectic groupoid integrating $(M_\red, \pi_\red)$;
\item the connected components of the symplectic quotients $q^{-1}(x)/\cG_x$ are the symplectic leaves of $(X_\red, \pi_\red)$.
\end{enumerate}
\end{proposition}

\subsection{Symplectic Morita equivalence}
\label{App:Symplectic Morita equivalence}

Free Hamiltonian $\cG$-spaces are also the main ingredient in symplectic Morita equivalences. Moreover, the previous proposition becomes an immediate consequence of one of the main properties of such equivalences. We start by recalling the definition (see \cite{Xu}):

\begin{definition} 
A {\bf symplectic Morita equivalence} between two symplectic groupoids $(\cG_i, \Omega_i)\tto M_i$ is a Morita equivalence (see Section \ref{ssec:Morita}):
\[
\xymatrix{
 \cG_1 \ar@<0.25pc>[d] \ar@<-0.25pc>[d]  & \ar@(dl, ul) & X \ar[dll]^{\qq_1}\ar[drr]_{\qq_2}  & \ar@(ur, dr) & \cG_2\ar@<0.25pc>[d] \ar@<-0.25pc>[d]  \\
M_1&  & & & M_2}
\]
together with a symplectic form $\Oga$ on $X$ such that the actions of $\cG_1$ and $\cG_2$ on $(X, \Oga)$ are symplectic. 
\end{definition} 

Hence, the two legs in a symplectic Morita equivalence are left/right free Hamiltonian $\cG_i$-spaces. The two actions are proper, but the groupoids need not be proper. 
% Notice that here the groupoids $\cG_i$ need not need be proper, 
% only the actions need be proper, which is equivalent to requiring the action groupoids $\cG_i\ltimes X\tto X$ to be proper.

Similar to the non-symplectic case, one can recover one groupoid in a symplectic Morita equivalence from the other groupoid and the the bibundle $(X, \Oga)$ by the gauge construction:
\[ \cG_2\cong \Gauge{\cG_1}{X},\quad \cG_1\cong \Gauge{\cG_2}{X}. \]
This is precisely the construction of (\ref{cG-red}) in Proposition \ref{prop:free-reduction}.

A Morita equivalence allows to identify various ``transversal objects'' associated to $\cG_1$, such as leaf spaces, isotropy groups, isotropy Lie algebras, monodromy groups, etc., with the similar ones of $\cG_2$. Moreover, in the symplectic case, one obtains an equivalence between Hamiltonian $\cG_1$-spaces and Hamiltonian $\cG_2$-spaces (see \cite{Xu}). It is not difficult to see that all these fit nicely together in the proper case, when the homeomorphism between the two resulting leaf spaces will be a diffeomorphism of integral affine orbifolds and the variation of symplectic areas of symplectic reductions for Hamiltonian $\cG_1$-spaces will correspond to the ones for $\cG_2$.

\subsection{Hamiltonian $\cT_\Lambda$-spaces}
\label{ssec:q-Hamiltonian spaces}

The discussion above is interesting even in the case of proper integrations of the zero Poisson structure. 
By Proposition \ref{prop:IAS-sympl-torus}, these correspond to integral affine structures $\Lambda\subset T^*B$: the associated torus bundle $\cT_{\Lambda}= T^*B/\Lambda$ can be viewed as a symplectic groupoid 
integrating $(B, \pi\equiv 0)$.
% the zero Poisson structure on $B$. 
Therefore, for any integral affine manifold $(B, \Lambda)$ one can 
talk about Hamiltonian $\cT_\Lambda$-spaces in the sense discussed above, and we have:

\begin{corollary}\label{q-hamil-def-reform} Let $(B, \Lambda)$  be an integral affine manifold and let $(X, \Oga)$ be a symplectic manifold endowed with an action of $\cT_{\Lambda}$ along a smooth map $q: X\to B$. Then the following are equivalent:
\begin{enumerate}[(a)]
\item $(X, \Oga)$ is a Hamiltonian $\cT_\Lambda$-space;
\item the moment map condition $ i_{\sigma(\alpha)}(\Oga)= q^*\alpha$ holds for all $\alpha\in \Omega^1(B)$, where $\sigma: q^*T^*B\to TX$ is the infinitesimal action induced by the action of $\cT_{\Lambda}$ on $X$. 
\end{enumerate}
\end{corollary}

From Proposition \ref{prop:free-reduction} and the discussion on Morita equivalences we deduce:

\begin{corollary}\label{reduction-free-qHamilt} Given an integral affine manifold $(B, \Lambda)$ and a free $\cT_{\Lambda}$-Ha\-mil\-tonian space $q: (X, \Oga)\to B$, with quotient $X_\red:= X/\cT_{\Lambda}$, one has:
\begin{enumerate}[(i)] 
\item There a unique Poisson structure $\pi_{\red}$ on  $X_\red$ carries such that the projection is a Poisson submersion with connected fibers:
\begin{equation}
\label{red-q-Ham-pmap} 
p: (X, \Oga) \to (X_{\red}, \pi_{\red}).
\end{equation}
Moreover, $p$ makes $(X, \Oga)$ into a proper isotropic realization of $(X_{\red}, \pi_{\red})$. 
% \item $p$ makes $(X, \Oga)$ into a proper isotropic realization of $(X_{\red}, \pi_{\red})$ with connected fibers. 
%\item $(X_{\red}, \pi_{\red})$ is a Poisson manifold of proper type,
\item $(X_{\red}, \pi_{\red})$ admits the following proper integrating symplectic groupoid: 
\begin{equation}\label{cG-red2} 
\Gauge{\cT_\Lambda}{X}:= \left( X\times_{B} X\right)/\cT_{\Lambda} \tto X_\red ,
\end{equation}
with the 2-form induced from the 2-form $\pr_{1}^{*}\Oga-\pr_{2}^{*}\Oga\in\Omega^2(X\times_{B} X)$. 
\item $(X, \Oga)$ defines
a symplectic Morita equivalence:
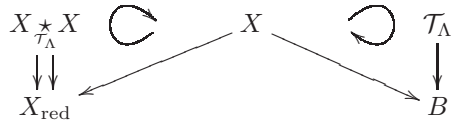
\[
\xymatrix@R=15 pt{
\Gauge{\cT_\Lambda}{X} \ar@<0.25pc>[d] \ar@<-0.25pc>[d]  & \ar@(dl, ul) & X \ar[dll]\ar[drr]  & \ar@(ur, dr) & \cT_\Lambda\ar[d] \\
X_\red&  & & & B}
\]
In particular, there is a 1-1 correspondence between Hamiltonian $\cT_\Lambda$-spaces and Hamiltonian $\Gauge{\cT_\Lambda}{X}$-spaces.
\end{enumerate} 
\end{corollary}

Note that the s-fibers of the gauge groupoid \eqref{cG-red2} are copies of the q-fibers, so we  we conclude that $X_\red$ is of proper type
when the fibers of $q$ are connected. There are examples where the fibers are not connected, the gauge groupoid is proper, but its source connected component fails to be proper.

We show in Section \ref{Integral affine structures on manifold} that Hamiltonian $\cT_\Lambda$-spaces are nothing more than Lagrangian 
fibrations $(X,\Oga)\to B$ inducing the integral affine structure $\Lambda$ on $B$. When $B=\T^n$ with its 
standard integral affine structure, $\cT_{\Lambda}$ is the trivial $\T^n$-bundle over $\T^n$ and Hamiltonian 
$\cT_{\Lambda}$-spaces are the same thing as quasi-Hamiltonian $\T^n$-spaces in the sense of (Reference \cite{AMM}).

\subsection{The twisted Dirac case}
\label{App:twisted Dirac case}
Let us mention briefly how to modify the previous discussion in the case of twisted Dirac manifolds. For details see \cite{Xu04}.

Given a presymplectic groupoid $(\cG,\Omega,\phi)$ integrating a twisted Dirac manifold $(M,L)$ with background 3-form $\phi\in\Omega^2(M)$, a {\bf Hamiltonian $\cG$-space}  consists of a manifold $X$, endowed with a 2-form $\Oga$, together with a smooth map
\[ q: (X, \Oga)\to M,\]
as well as an action $m$ of $\cG$ on $X$ along $q$, satisfying:
\begin{enumerate}[(a)]
\item multiplicativity: $m^*(\Oga)= \pr_{1}^{*}(\Omega)+ \pr_{2}^{*}(\Oga)$;
\item twisting: $\d\Oga-q^*\phi$ is horizontal relative to the action.
\end{enumerate}
For such a Hamiltonian $\cG$-space, the map $q:(X,\Oga)\to (M,L)$ is a forward Dirac map. One also has analogues of Lemmas \ref{lemma-sympl-acts-infin} and \ref{lemma-sympl-acts-infin-conv}.

For a free Hamiltonian $\cG$-space, the {\bf reduced space}:
\[ X_\red:= X/\cG, \]
now carries a unique $\phi_\red$-twisted Dirac structure $L_\red$ for which the projection $p:(X,\omega)\to (X_\red,L_\red)$ is 
a forward Dirac map and the twisting satisfies:
\[ \d\Oga=q^*\phi-p^*\phi_\red. \]
The gauge groupoid:
\[ \Gauge{\cG}{X}:=\left( X\times_{M} X/\cG \tto M \right),\]
% equipped 
with the 2-form induced from $\pr_{1}^{*}\Oga- \pr_{2}^{*}\Oga$, is a $\phi_\red$-twisted presymplectic groupoid integrating $(M,L_\red)$, with leaves the {\bf presymplectic quotients} $q^{-1}(x)/\cG_x$.

A {\bf presymplectic Morita equivalence} between two $\phi_i$-twisted presymplectic groupoids $(\cG_i, \Omega_i)\tto M_i$ is a Morita equivalence:
\[
\xymatrix{
 \cG_1 \ar@<0.25pc>[d] \ar@<-0.25pc>[d]  & \ar@(dl, ul) & X \ar[dll]^{\qq_1}\ar[drr]_{\qq_2}  & \ar@(ur, dr) & \cG_2\ar@<0.25pc>[d] \ar@<-0.25pc>[d]  \\
M_1&  & & & M_2}
\]
together with a 2-form $\Oga$ on $X$ such that the actions are presymplectic:
\[ m_i^*(\Oga)= \pr_{1}^{*}(\Omega_i)+ \pr_{2}^{*}(\Oga)\quad (i=1,2),\]
and the following twisting condition holds:
\[ \d\Oga=\qq_1^*\phi_1-\qq_2^*\phi_2. \]

Hence, the two legs in a presymplectic Morita equivalence are left/right free Hamiltonian $\cG_i$-spaces and one can recover one groupoid from the other groupoid and the bibundle $(X, \Oga)$ by the gauge construction:
\[ \cG_2\cong \Gauge{\cG_1}{X},\quad \cG_1\cong \Gauge{\cG_2}{X}. \]

Note that a twisted presymplectic groupoid maybe presymplectic Morita equivalent to a non-twisted symplectic groupoid. Moreover, it is easy to check that:
\begin{enumerate}[(i)]
\item Any $\phi$-twisted presymplectic groupoid $(\cG,\Omega,\phi)\tto(M,L)$ is presymplectic Morita equivalent to its restriction to a complete transversal $\cG|_T\tto T$, which is a twisted symplectic groupoid with twisting the restriction of $\phi$ to the transversal $T$;
\item A presymplectic Morita equivalence between two symplectic groupoids is actually a symplectic Morita equivalence.
\end{enumerate}
Finally, given a 2-form $B\in\Omega^2(M)$, there is a presymplectic Morita equivalence between a $\phi$-twisted presymplectic groupoid $(\cG,\Omega)\tto(M,L)$ and its $B$-transform, namely the $(\phi+\d B)$-twisted presymplectic groupoid $(\cG,\Omega')\tto(M,e^B L)$, where $\Omega'=\Omega+t^*B-s^*B$.

%%%%%%%%%%%%%%%%%%%%%%%%
%%%%%%%%%%%%%%%%%%%%%%%%
%%%%%%%%%%%%%%%%%%%%%%%%
\section{Proper transverse integral affine foliations}
\label{appendix:molino}
%\subsection{Transverse integral affine proper foliations: a global quotient theorem}
%%%%%%%%%%%%%%%%%%%%%%%%
%%%%%%%%%%%%%%%%%%%%%%%%
%%%%%%%%%%%%%%%%%%%%%%%%
%%%%%%%%%%%%%%%%%%%%%%%%
In this section we show that some of the ideas from Molino's structure theory of riemannian foliations (\cite{MM,Molino}) can be adapted to the case of transverse integral affine foliations 
proving in particular that the leaf spaces of proper transverse integral affine foliations are good orbifolds, i.e.~global quotients
modulo proper actions of discrete groups. In the terminology of \cite{Thur, GHL}, we show that proper integral affine foliations are complete.

Throughout this section we fix a foliated manifold $(M, \cF)$ together with a transverse integral affine structure $\Lambda\subset \nu^*(\cF)$. 

\subsection{Transverse integral affine structures and the foliation holonomy}

We start by investigating the influence of the transverse integral affine structure $\Lambda$ on the holonomy of $\cF$
Note first that the various groupoids associated to the foliation discussed in Section \ref{ssec:foliation-groupoids} fit into a sequence of groupoid morphisms 
\[
\xymatrix{ 
\Mon(M, \cF) \ar[r]^-{\hol} & \Hol(M, \cF) \ar[r]^-{\lin} &\Hol^{\lin}(M, \cF)\subset\GL_\Lambda(\nu(\cF)),}
\]
where the $\Lambda$ in the last factor is justified by the fact that the holonomy of the foliation preserves $\Lambda$. As pointed out in Section \ref{ssec:foliation-groupoids}, unlike the other groupoids in this sequence, $\Hol^{\lin}(M, \cF)$ does not have a smooth structure making $\lin$ smooth unless the holonomy is linear. 

\begin{proposition} \label{Mol:immersed1}
If $(M, \cF)$ admits a transverse integral affine structure then its holonomy is linear, i.e. $\lin$ is bijective. Moreover, with the induced smooth structure, $\Hol^{\lin}(M, \cF)$ is an immersed Lie subgroupoid of 
$\GL(\nu(\cF))$.
\end{proposition}

\begin{proof} 
For a leafwise path $\gamma$ from $x$ to $y$, choosing small transversals $S$ through $x$ and $T$ through $y$, the induced holonomy germ $\hol(\gamma): (S, x)\to (T, y)$ preserves the integral affine structures on the transversals. Using transverse integral affine charts, we obtain a
a germ of a diffeomorphism $\R^q\to\R^q$ around the origin which preserves the standard integral affine structure. Such germs are clearly linear.  
\end{proof}

Note that $\Hol^{\lin}(M, \cF)\subset\GL(\nu(\cF))$ may fail to be an embedding: one example is given by the the Kronecker foliation on the torus, with an irrational slope. This problem will soon disappear, once we assume properness of the foliation. 
\vskip 0.1 in

Next we compare the holonomy of the foliation, which we will refer to as $\cF$-holonomy, with the ones associated to the transverse integral affine. Recall from Section \ref{sec:IAS:orbifolds} that, associated to $\Lambda$, we consider:
\begin{itemize}
\item the linear holonomy $h^{\lin}: \Pi_1(M) \to \GL_\Lambda(\nu(\cF))$ with image denoted by 
\[ \Pi_1^{\lin}(M, \Lambda)\subset\GL_\Lambda(\nu(\cF)).\]
\item the affine holonomy $h^{\Aff}: \Pi_1(M)\to \Aff_\Lambda(\nu(\cF))$ with image denoted by 
\[ \Pi_1^{\Aff}(M, \Lambda)\subset \Aff_\Lambda(\nu(\cF)).\]
\end{itemize}
While the holonomies associated to $\Lambda$ are defined on $\Pi_1(M)$, the $\cF$-holonomy is defined on $\Mon(M, \cF)$. Hence, to compare the two, we will use the tautological map sending the leafwise homotopy class of a leafwise path to its homotopy class as a path in $M$:
\[ i_{*}: \Mon(M, \cF)\to \Pi_1(M).\]

\begin{proposition}
\label{prop-lin-hol-fol-hol} For a transverse integral affine structure $\Lambda$ on a foliation $(M,\cF)$, its holonomies are related to the $\cF$-holonomy through the following commutative diagrams:
\[ \xymatrix{
\Pi_1(M) \ar[r]^-{ h^{\lin}} &\GL_\Lambda(\nu(\cF)) \\
\Mon(M, \cF)  \ar[u]^-{i_*}  \ar[ru]_{\hol^{\lin}} & } 
\qquad
\xymatrix{
\Pi_1(M) \ar[r]^-{ h^{\Aff}} & \Aff_\Lambda(\nu(\cF)) \\ 
\Mon(M, \cF)  \ar[u]^-{i_*}  \ar[ru]_{(0, \hol^{\lin})} & } 
\]
In particular, the $\cF$-holonomy sits inside both the linear and the affine holonomy, and $\Hol(M, \cF)$ sits as an immersed subgroupoid:
\[j: \Hol(M, \cF) \hookrightarrow \Pi_1^{\lin}(M, \Lambda),\qquad (0, j): \Hol(M, \cF) \hookrightarrow \Pi_1^{\Aff}(M, \Lambda).\]
\end{proposition}

\begin{proof} 
For the  commutativity of the first diagram it suffices to check that the flat connection $\nabla$ on $\nu(\cF)$, whose parallel transport gives rise to $h^{\lin}$, when computed on vectors tangent to $\cF$, becomes the Bott $\cF$-connection, whose parallel transport gives rise to $\hol^{\lin}$. In other words, that we have:
\[ \nabla_{V}(\overline{X})= \overline{[V, X]},\quad \forall\ V\in \Gamma(\cF),\ \overline{X}\in \Gamma(\nu(\cF)).\] 
This is a local statement that follows right away using local vector fields $X^1, \ldots, X^q$ spanning the integral lattice and such that $[V, X^i]\in \Gamma(\cF)$ whenever $V\in \Gamma(\cF)$. 

The commutativity of the second diagram follows from the first one and the remark that the developing map $\dev: \Pi_1(M)\to \nu(\cF)$ vanishes on the image of $i_*$. To prove the remark observe that the projection $TM\to \nu(\cF)$, an algebroid 1-cocyle whose integration is $\dev$, is zero on the sub-algebroid $\cF\subset TM$; hence $h^{\Aff}\circ i_*$, as a groupoid cocycle integrating the zero algebroid cocycle, 
must be trivial. 
\end{proof}

% It suffices to check this locally. But then we can choose $X^1, \ldots, X^q$ local vector fields spanning the integral lattice and such that $[V, X^i]\in \Gamma(\cF)$ whenever $V\in \Gamma(\cF)$,%  so the claim follows.
%
% Next we claim that developing map $\dev: \Pi_1(M)\to \nu(\cF)$ vanishes on the image of $i_*$, so the commutativity of the second diagram follows from the first one. To prove this claim, we observe that the projection $TM\to \nu(\cF)$ viewed as an algebroid 1-cocyle, whose integration is $\dev$, when restricted to the sub-algebroid $\cF\subset TM$ is zero. Since $h^{\Aff}\circ i_*$ is a groupoid cocycle integrating the trivial algebroid cocycle, it must be trivial.

\subsection{Lifting to the linear holonomy cover}
As in Section \ref{sec:IAS:orbifolds}, one can be more concrete by fixing
\begin{itemize}
\item a base point $x\in M$, and 
\item a $\Z$-basis $\mathfrak{b}_\Lambda=\{\lambda_1, \ldots, \lambda_q\}$ for $\Lambda_{x}$.
\end{itemize}
Then one can represent the $\cF$-holonomy at $x$ as a map 
\[ \hol^\lin|_x:\pi_1(S,x)\to \GL_\Z(\R^q),\]
and similarly for the holonomies of $\Lambda$, $h^{\lin}|_x$ and $h^{\Aff}|_x$, defined on $\pi_1(M, x)$ (see (\ref{eq-hol-explic-TIAS})). 
The diagrams in Proposition \ref{prop-lin-hol-fol-hol} become the following diagrams: 
\[ 
\xymatrix{
\pi_1(M, x) \ar[r]^{h^{\lin}|_{x}} &\GL_\Z(\R^q) \\
\pi_1(L, x)  \ar[u]_{i_*} \ar[ru]_{\hol^{\lin}|_{x}} & },\qquad
\xymatrix{
\pi_1(M, x) \ar[r]^{h^{\Aff}|_{x}} &\Aff_\Z(\R^q) \\
\pi_1(L, x)  \ar[u]_{i_*} \ar[ru]_{(0,\hol^{\lin}|_{x})} & }
\]

The images of $h^{\lin}|_x$ and $h^{\Aff}|_x$ will be denoted by 
\[ \Gamma^{\lin}\subset\GL_\Z(\R^q),\quad \Gamma^{\Aff}\subset\Aff_\Z(\R^q),\]
respectively. These groups, being quotients of $\pi_1(M, x)$, give rise to a sequence of covering spaces endowed with pull-back foliations:
\[ \xymatrix{ (\widetilde{M},\widetilde{\cF})\ar[r] & (M^{\Aff},\cF^{\Aff}) \ar[r] & (M^{\lin},\cF^{\lin}) \ar[r] & (M,\cF)} \]
Each of these foliations has a (pull-back) transverse integral affine structure. Using the base point $x$, we can identify each of these spaces with the source fiber at $x$ of the corresponding groupoids:
\[ \widetilde{M}=\Pi_1(M)(x,-),\quad M^{\Aff}=\Pi^{\Aff}_1(M, \Lambda)(x,-),\quad M^{\lin}=\Pi^{\lin}_1(M, \Lambda)(x,-). \]

\begin{remark}
One can also use the basis $\mathfrak{b}_\Lambda$ to obtain a more concrete model for $M^{\lin}$, as the connected component
through $(x, \mathfrak{b}_\Lambda)$ of the $\Lambda$-frame bundle:
\[ \textrm{Fr}(\nu(\cF), \Lambda):=  \{(x, v_1, \ldots, v_q): x\in M, v_1, \ldots , v_q- \textrm{basis\ of}\ \Lambda_x \}\subset \textrm{Fr}(\nu(\cF)).\]
We will not use this description in what follows, but it provides some geometric insight into the linear holonomy cover and actions on them. Incidentally, it also shows that $\Pi_1^{\lin}(M)$ is the unit connected component of $\GL_\Lambda(\nu(\cF))$.
\end{remark}

Next, we state the main properties of the foliation $(M^{\lin},\cF^{\lin})$. An entirely similar result holds for $(M^{\Aff},\cF^{\Aff})$, but we leave the details to the reader.

\begin{lemma} 
\label{lem:linear:hol:cover}
If $p:M^{\lin}\to M$ is the covering projection, then the foliation $(M^{\lin}, \cF^{\lin})$ has the following properties:
\begin{enumerate}[(i)]
\item it has trivial $\cF$-holonomy;
\item the action of $\Gamma^{\lin}$ on $M^{\lin}$ takes leaves to leaves;
\item each leaf $L'$ of $\cF^\lin$ is isomorphic to the $\cF$-holonomy cover of a leaf $L$ of $\cF$, with covering projection the restriction of 
$p: M^{\lin}\to M$; 
\item there is a free, possibly non-proper, action of $\Hol(M,\cF)\tto M$ on $M^{\lin}\to M$ and $\Hol(M^{\lin},\cF^{\lin})\tto M^{\lin}$ is isomorphic to the resulting action groupoid $\Hol(M,\cF)\ltimes M^{\lin}\tto M^{\lin}$;
\end{enumerate}
\end{lemma}

\begin{proof} Property (i) should be clear since the transverse integral affine structure $p^*\Lambda$ on $(M^{\lin},\cF^{\lin})$ has linear holonomy map the composition of the linear holonomy map of $\Lambda$ on $(M,\cF)$ with $p_*: \pi_1(M^{\lin})\to \pi_1(M)$. Using Proposition \ref{prop-lin-hol-fol-hol}, we conclude that $(M^{\lin}, \cF^{\lin})$ must have trivial holonomy. 

Property (ii) follows from general properties of covers and pullback foliations.

For the proof of property (iii), we consider the s-fiber $\Pi_1^{\lin}(x, -)$ above $x$, together with the target map
\[ p: \Pi_1^{\lin}(x, -)\to M,\]
as a model for the covering $p: M^{\lin}\to M$. From the homotopy exact sequence of the $\Gamma^{\lin}$-cover $p: M^{\lin}\to M$, we obtain the
following short exact sequence:
\[ \xymatrix{ 0\ar[r] & \pi_1(M^{\lin})\ar[r] & \pi_1(M)\ar[r]^-{h^{\lin}} & \Gamma^{\lin}\ar[r] & 0}.\] 
For any embedded sub-manifold $L\subset M$ and for any connected component $L'$ of $p^{-1}(L)$, the restriction 
$p|_{L'}: L'\to L$ is a covering projection with group the image of $i_*(\pi_1(L))$ by $h^{\lin}$. In particular, $\pi_1(L')$ is isomorphic to the kernel of the composition $h^{\lin}\circ~i_*$ which equals the linear holonomy group of $(M, \cF)$ by Proposition \ref{prop-lin-hol-fol-hol}. This proves 
(iii) in the case where the leaves are embedded. With some care, the argument can be adapted to immersed leaves $L\subset M$. An alternative proof of (iii) follows also from the proof of (iv), to which we now turn.

In the model above, the action of $\Hol(M,\cF)$ on $M^{\lin}$ is induced from the inclusion $j: \Hol(M,\cF)\hookrightarrow \Pi_1^{\lin}$ (see Proposition \ref{prop-lin-hol-fol-hol}): this clearly gives a free, left, action 
\[ \Hol(M,\cF) \times \Pi_1^{\lin}(x, -) \to \Pi_1^{\lin}(x, -), \ \ (a, \gamma)\mapsto j(a)\gamma.\]
The orbit of the action through any $\gamma\in \Pi_1^{\lin}(x, -)$ is the image of the immersion:
\[ R_{\gamma}: \Hol(M,\cF)(x', -) \to \Pi_1^{\lin}(x, -), \ R_{\gamma}(a)= j(a)\gamma.\]
We claim that the tangent space at $\gamma$ to such an orbit coincides with $T_{\gamma}\cF^{\lin}$. Since these orbits are smooth, connected, immersed submanifolds of $M^{\lin}= \Pi_1^{\lin}(x, -)$, 
 it will follow that they are precisely the leaves of $\cF^{\lin}$. To prove the claim we compute
\[ (\d p)(\d R_{\gamma}(T_a\Hol(M,\cF)(x', -)))= (\d t)_{\gamma}(T_{\gamma} \Hol(M,\cF)(x, -))= T_{t(a)}\cF. \]
This shows that $\d R_{\gamma}(T_a\Hol(M,\cF)(x', -))\subset T_{j(\gamma)a}\cF^{\lin}$; by a dimension counting, this inclusion must be an equality, 
proving the claim. 
% and the claim follows. 
Therefore $\Hol(M,\cF)\ltimes M^{\lin}$ is a groupoid over $M^{\lin}$ integrating $\cF^{\lin}$, so it comes with a groupoid submersion 
\[ \Hol(M,\cF)\ltimes M^{\lin}\to \Hol(M^{\lin}, \cF^{\lin}).\]
Composing this map with the anchor $(s, t):  \Hol(M^{\lin}, \cF^{\lin})\to M^{\lin}\times M^{\lin}$ gives the anchor $(s, t):\Hol(M,\cF)\ltimes M^{\lin}\to M^{\lin}\times M^{\lin}$, which is injective by freeness of the action. Therefore the submersion is actually a diffeomorphism and (iv) follows.
\end{proof}

\subsection{Proper Foliations: a Molino type Theorem}
\label{Proper Foliations: a Molino Type Theorem}
In the proper case we have:
\begin{lemma} If $(M, \cF)$ is a proper foliation with a transverse integral affine structure then the immersions $\Hol(M, \cF)\hookrightarrow \GL(\nu(\cF))$, $\Hol(M, \cF) \hookrightarrow \Pi_1^{\lin}(M)$ and $\Hol(M, \cF) \hookrightarrow \Pi_1^{\Aff}(M)$ are embeddings of Lie groupoids.
\end{lemma}

\begin{proof} 
By Proposition \ref{prop-fol-crit-C}, the properness of $(M, \cF)$ implies that $\Hol(M, \cF)$ is a proper groupoid. It then suffices to remark that if $\cH$ is a proper groupoid and $F: \cH\to \cG$ is a morphism of Lie groupoids over $M$ covering the identity on the base, then $F$ is automatically a proper map: for $K\subset \cG$ compact, $F^{-1}(K)$ is closed inside the compact $\cH(s(K), t(K))$), hence it must be compact. In  particular, if $F$ is an injective immersion, than it is automatically a closed embedding. % and the lemma follows.
\end{proof}

As explained in Example \ref{ex-trIAS-orbifolds}, under the present assumptions, $B=M/\cF$ is an integral affine orbifold. 
Our final result shows that $B$ is actually a good orbifold:

\begin{theorem} 
\label{Molino-theorem}
If $(M, \cF)$ is a foliation of proper type with a transverse integral affine structure, then the linear holonomy cover $(M^{\lin}, \cF^{\lin})$ is simple and carries a transverse integral affine structure. Hence, its space of leaves
\[ B^{\lin}:= M^{\lin}/\cF^{\lin}\]
is a smooth integral affine manifold. Moreover, the action of $\Gamma^{\lin}$ on $M^{\lin}$ descends to a proper action on 
$B^{\lin}$ by integral affine transformations and $M^{\lin}$ yields a Morita equivalence:
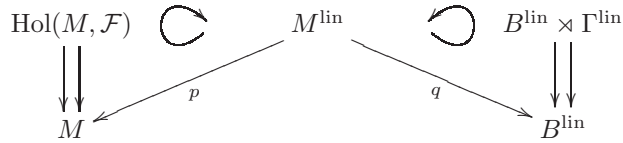
\[
\xymatrix{
 \Hol(M, \cF) \ar@<0.25pc>[d] \ar@<-0.25pc>[d]  & \ar@(dl, ul) & M^{\lin}\ar[dll]^-{p}\ar[drr]_-{q} & \ar@(dr, ur)   & B^{\lin}\rtimes \Gamma^{\lin} \ar@<0.25pc>[d] \ar@<-0.25pc>[d]\\
M & & & & B^{\lin}  }
\]
In particular, we have an isomorphism of integral affine orbifolds:
\[ M/\cF\cong B^{\lin}/\Gamma^{\lin}. \]
\end{theorem}

\begin{proof} 
By Lemma \ref{lem:linear:hol:cover} (iv), the holonomy groupoid of $(M^{\lin},\cF^{\lin})$ is proper and this foliation has trivial holonomy. Hence, it must be a simple foliation with smooth orbit space $B^{\lin}$. Equivalently, now the free action of $\Hol(M, \cF)$ on $M^{\lin}$ is also proper, hence $B^{\lin}= M^{\lin}/\Hol(M, \cF)$ is smooth and $q: M^{\lin}\to B^{\lin}$ is a principal $\Hol(M,\cF)$-bundle. By Lemma \ref{lem:linear:hol:cover} (ii) we have an action of $\Gamma^{\lin}$ on $B^{\lin}$ and $p: M^{\lin}\to M$ is a principal $B^{\lin}\rtimes \Gamma^{\lin}$-bundle. 

The two actions on $M^{\lin}$ clearly commute, hence we obtain a Morita equivalence. Since properness is a Morita invariant, it follows that the action of $\Gamma^{\lin}$ on $B^{\lin}$ must be proper. The properties concerning the integral affine structure are straightforward.
\end{proof}

\begin{remark} 
There is a version of the previous theorem where the linear holonomy cover $M^{\lin}$ is replaced by the affine holonomy cover $M^{\Aff}$, giving rise to a similar Morita equivalence of $\Hol(M, \cF)\tto M$ with $B^{\Aff}\rtimes \Gamma^{\Aff}\tto B^{\Aff}$. Here $B^{\Aff}$ is the affine holonomy cover of the integral affine manifold $B^{\lin}$. This version allows us to view the developing map $\dev: M^{\Aff}\to \R^q$ as the composition of the projection $M^{\Aff}\to B^{\Aff}$ with the developing map $\dev: B^{\Aff}\to \R^q$ of the integral affine manifold $B$. The argument uses the affine version of Lemma \ref{lem:linear:hol:cover}. 
\end{remark}

Since a classical integral affine orbifold can always be obtained as the leaf space of a foliation of proper type with a transverse integral affine structure, we conclude:

\begin{corollary}
Any classical integral affine orbifold is a good orbifold.
\end{corollary}

\end{document}